\documentclass{tac}

\usepackage[utf8]{inputenc} 
\usepackage[T1]{fontenc}
\usepackage{amsfonts, amssymb, amsmath}

\usepackage{enumitem}

\usepackage{mathtools} 
\usepackage{tensor} 
\usepackage{scalerel} 

\usepackage{perpage}
\MakePerPage{footnote}

\usepackage{chngcntr}
\counterwithin{table}{section}

\usepackage[ngerman, english]{babel}
\useshorthands{"}
\makeatletter
 \defineshorthand[ngerman]{"/}{\mbox{-}\bbl@allowhyphens}
\makeatother
\addto\extrasenglish{\languageshorthands{ngerman}}

\newcommand{\myprenote}{References to results and sections with their numbering prefixed with the capital letter `A', such as ``\auglemref{8.1}'' and ``\augsecref 7'', refer to results and sections of the prequel to the present article \cite{Koudenburg20}.\addcontentsline{toc}{section}{\protect\numberline{}References}}

\makeatletter
\renewenvironment{thebibliography}[1]
     {\section*{\refname}%
      \@mkboth{\MakeUppercase\refname}{\MakeUppercase\refname}%
      \myprenote
      \list{\@biblabel{\@arabic\c@enumiv}}%
           {\settowidth\labelwidth{\@biblabel{#1}}%
            \leftmargin\labelwidth
            \advance\leftmargin\labelsep
            \@openbib@code
            \usecounter{enumiv}%
            \let\p@enumiv\@empty
            \renewcommand\theenumiv{\@arabic\c@enumiv}}%
      \sloppy
      \clubpenalty4000
      \@clubpenalty \clubpenalty
      \widowpenalty4000%
      \sfcode`\.\@m}
     {\def\@noitemerr
       {\@latex@warning{Empty `thebibliography' environment}}%
      \endlist}
\makeatother

\hypersetup{
	pdfauthor={Seerp Roald Koudenburg},
	pdftitle={Formal category theory in augmented virtual double categories},
  pdfkeywords={formal category theory, Kan extension, Yoneda embedding, Yoneda structure, exactness, totality, free cocompletion, augmented virtual double category},
  allcolors=[rgb]{0.1,0.1,0.4}
}

\usepackage{tikz} 
\usetikzlibrary{decorations.markings, matrix, arrows, fit, calc}

\tikzstyle{map} = [->, font=\scriptsize]
\tikzstyle{linj} = [left hook->, font=\scriptsize]
\tikzstyle{rinj} = [right hook->, font=\scriptsize]
\tikzstyle{mono} = [>->, font=\scriptsize]
\tikzstyle{epi} = [->>, font=\scriptsize]
\tikzstyle{cell} = [double,double equal sign distance,-implies, shorten >= 3.75pt, shorten <= 3.75pt, font=\scriptsize]
\tikzstyle{eq} = [double,double equal sign distance]
\tikzstyle{ps} = [shorten >= 2pt]

\tikzstyle{iso} = [above, sloped, inner sep=1.5pt]
\tikzstyle{nat} = [above, sloped, inner sep=2pt]
\tikzstyle{desc} = [fill=white, inner sep=2pt]
\tikzstyle{dots} = [black, font=]

\tikzstyle{small} = [font=\scriptsize]

\tikzstyle{textbaseline} = [baseline=-3.2pt]

\tikzstyle{barred} = [decoration={markings, mark=at position 0.5 with {\draw[-] (0,-1.5pt) -- (0,1.5pt);}}, postaction ={decorate}]

\tikzstyle{math35} = [matrix of math nodes, row sep={3.25em,between origins}, column sep={3.5em,between origins}, text height=1.5ex, text depth=0.25ex, nodes in empty cells]
\tikzstyle{minimath} = [matrix of math nodes, row sep={3em,between origins}, column sep={3.25em,between origins}, font=\scriptsize, text height=1ex, text depth=0.25ex, nodes in empty cells]
\tikzstyle{scheme} = [textbaseline, x=1.6em, y=1.6em, yshift=-2.4em, font=\scriptsize, text depth=0ex, every node/.style={overlay}, execute at end picture = { \useasboundingbox ($(current bounding box.north west) + (0,0.4em)$) rectangle ($(current bounding box.south east) - (0,0.4em)$); }]

\makeatletter
\def\slashedarrowfill@#1#2#3#4#5{%
  $\m@th\thickmuskip0mu\medmuskip\thickmuskip\thinmuskip\thickmuskip
   \relax#5#1\mkern-7mu%
   \cleaders\hbox{$#5\mkern-2mu#2\mkern-2mu$}\hfill
   \mathclap{#3}\mathclap{#2}%
   \cleaders\hbox{$#5\mkern-2mu#2\mkern-2mu$}\hfill
   \mkern-7mu#4$%
}
\def\rightslashedarrowfill@{%
  \slashedarrowfill@\relbar\relbar\mapstochar\rightarrow}
\newcommand\xslashedrightarrow[2][]{%
  \ext@arrow 0055{\rightslashedarrowfill@}{#1}{#2}}
\makeatother

\def\slashedrightarrow{\xslashedrightarrow{}}

\newcommand{\conc}{%
  \mathbin{
    \mathchoice
    {\raisebox{1ex}{\scalebox{.7}{$\frown$}}}
    {\raisebox{1ex}{\scalebox{.7}{$\frown$}}}
    {\raisebox{.7ex}{\scalebox{.5}{$\frown$}}}
    {\raisebox{.7ex}{\scalebox{.5}{$\frown$}}}
  }
}

\providecommand{\cororef}[1]{Corollary~\ref{#1}}
\providecommand{\corosref}[2]{Corollaries~\ref{#1} and~\ref{#2}}
\providecommand{\defref}[1]{Definition~\ref{#1}}
\providecommand{\defsref}[2]{Definitions~\ref{#1} and~\ref{#2}}
\providecommand{\defthreeref}[3]{Definitions~\ref{#1}, \ref{#2} and~\ref{#3}}
\providecommand{\exref}[1]{Example~\ref{#1}}
\providecommand{\exrref}[2]{Examples~\ref{#1}--\ref{#2}}
\providecommand{\exsref}[2]{Examples~\ref{#1} and~\ref{#2}}

\providecommand{\lemref}[1]{Lemma~\ref{#1}}

\providecommand{\propref}[1]{Proposition~\ref{#1}}
\providecommand{\propsref}[2]{Propositions~\ref{#1} and~\ref{#2}}
\providecommand{\remref}[1]{Remark~\ref{#1}}
\providecommand{\thmref}[1]{Theorem~\ref{#1}}
\providecommand{\thmsref}[2]{Theorems~\ref{#1} and~\ref{#2}}
\providecommand{\tableref}[1]{Table~\ref{#1}}

\providecommand{\secref}[1]{Section~\ref{#1}}
\providecommand{\secrref}[2]{Sections~\ref{#1}--\ref{#2}}

\providecommand{\introref}{\hyperref[Introduction]{Introduction}}
\providecommand{\overref}{\hyperref[Overview]{Overview}}

\providecommand{\augcororef}[1]{Corollary~A#1}
\providecommand{\augdefref}[1]{Definition~A#1}

\providecommand{\augexref}[1]{Example~A#1}
\providecommand{\augexsref}[2]{Examples~A#1 and A#2}

\providecommand{\augthreeexref}[3]{Examples~A#1, A#2 and~A#3}

\providecommand{\auglemref}[1]{Lemma~A#1}
\providecommand{\auglemsref}[2]{Lemmas~A#1 and A#2}
\providecommand{\augthreelemref}[3]{Lemmas~A#1, A#2 and~A#3}
\providecommand{\augpropref}[1]{Proposition~A#1}

\providecommand{\augthmref}[1]{Theorem~A#1}

\providecommand{\augsecref}[1]{Section~A#1}

\newcommand\defeq{\mathrel{\vcentcolon\Leftrightarrow}}

\providecommand{\dfn}{\coloneqq}
\providecommand{\nfd}{\eqqcolon}

\providecommand{\of}{\circ}
\providecommand{\iso}{\cong}
\providecommand{\eq}{\simeq}

\providecommand{\brar}{\slashedrightarrow}
\providecommand{\xrar}{\xrightarrow}
\providecommand{\xlar}{\xleftarrow}
\providecommand{\xbrar}{\xslashedrightarrow}
\providecommand{\Rar}{\Rightarrow}
\providecommand{\xRar}{\xRightarrow}

\providecommand{\into}{\hookrightarrow}

\providecommand{\eps}{\varepsilon}

\DeclareMathOperator{\dash}{\makebox[1.3ex]{\text{--}}}

\providecommand{\tens}{\otimes}

\providecommand{\ul}[1]{\underline{#1}{}}

\providecommand{\brcs}[1]{\lbrace #1 \rbrace}
\providecommand{\bigbrcs}[1]{\bigl\lbrace #1 \bigr\rbrace}
\providecommand{\brks}[1]{\lbrack #1 \rbrack}
\providecommand{\bigbrks}[1]{\bigl\lbrack #1 \bigr\rbrack}

\providecommand{\pars}[1]{\left(#1\right)}
\providecommand{\bigpars}[1]{\bigl(#1\bigr)}

\providecommand{\lns}[1]{\lvert#1\rvert}

\providecommand{\angles}[1]{\langle#1\rangle}

\providecommand{\gen}[1]{\angles{#1}}

\providecommand{\set}[1]{\brcs{#1}}

\providecommand{\isect}{\cap}

\providecommand{\djunion}{\sqcup}

\providecommand{\downset}{\mathord\downarrow}
\providecommand{\upset}{\mathord\uparrow}

\DeclareMathOperator{\Dn}{Dn}
\DeclareMathOperator{\Dnp}{Dn^+}
\DeclareMathOperator{\Cl}{Cl}

\DeclareMathOperator{\gso}{\top} 

\providecommand{\natarrow}{\Rightarrow}

\providecommand{\map}[3]{#1\colon#2\to#3}
\providecommand{\mono}[3]{#1\colon#2\rightarrowtail#3}

\providecommand{\emb}[3]{#1\colon#2\into#3}

\providecommand{\nat}[3]{#1\colon#2\natarrow#3}
\providecommand{\cell}[3]{#1\colon#2\Rightarrow#3}

\providecommand{\hmap}[3]{#1\colon#2\slashedrightarrow#3}

\providecommand{\inv}[1]{{#1}^{-1}}

\providecommand{\rev}[1]{#1^\circ}

\DeclareMathOperator{\term}{!}

\DeclareMathOperator{\ob}{ob}
\newcommand{\id}{\mathrm{id}}
\newcommand{\yon}{\mathrm{y}}
\providecommand{\A}{\mathcal A}

\providecommand{\ladj}{\dashv}

\providecommand{\intl}{\int\limits}

\providecommand{\op}[1]{#1^\textup{op}}
\providecommand{\co}[1]{#1^\textup{co}}

\providecommand{\ps}[1]{\widehat{#1}}
\providecommand{\psps}[1]{\ps {\ps{#1\mspace{0 mu}}}}

\providecommand{\catvar}[1]{\mathcal{#1}}

\providecommand{\2}{\mathsf 2}
\providecommand{\A}{\catvar A}

\providecommand{\D}{\catvar D}
\providecommand{\E}{\catvar E}

\providecommand{\K}{\catvar K}
\renewcommand{\L}{\catvar L}
\providecommand{\M}{\catvar M}

\providecommand{\V}{\catvar V}

\providecommand{\Set}{\mathsf{Set}}
\providecommand{\PreOrd}{\mathsf{PreOrd}}

\providecommand{\Cat}{\mathsf{Cat}} 
\providecommand{\enCat}[1]{#1\text-\Cat}
\providecommand{\inCat}[1]{\Cat(#1)}
\providecommand{\twoCat}{2\text{-}\Cat}

\providecommand{\und}[1]{#1_0} 

\providecommand{\DblCat}{\mathsf{DblCat}}

\providecommand{\nlDblCat}{\DblCat_\textup{nl}}

\providecommand{\AugVirtDblCat}{\mathsf{AugVirtDblCat}}
\providecommand{\VirtDblCat}{\mathsf{VirtDblCat}}

\providecommand{\as}{\mathfrak a}
\providecommand{\lu}{\mathfrak l}
\providecommand{\ru}{\mathfrak r}
\providecommand{\sm}{\mathfrak s}
\providecommand{\inhom}[1]{\brks{#1}}
\providecommand{\dl}[1]{#1^\circ}
\providecommand{\flad}[1]{#1^\flat}
\providecommand{\shad}[1]{#1^\sharp}
\providecommand{\iotahom}{\inhom{\dl A, \ps I}_\iota}
\DeclareMathOperator{\ev}{ev}
\DeclareMathOperator{\coev}{coev}

\providecommand{\Rel}{\mathsf{Rel}}
\providecommand{\Span}[1]{\mathsf{Span}(#1)}

\providecommand{\Mod}{\mathsf{Mod}}
\providecommand{\ModRel}{\mathsf{ModRel}}
\providecommand{\Prof}{\mathsf{Prof}}

\providecommand{\enProf}[1]{#1\text-\Prof}
\providecommand{\ensProf}[1]{#1\text-\mathsf{sProf}}
\providecommand{\inProf}[1]{\Prof(#1)}
\providecommand{\spFib}[1]{\mathsf{spFib}(#1)}
\providecommand{\dFib}[1]{\mathsf{dFib}(#1)}
\providecommand{\ClModRel}{\mathsf{ClModRel}}
\providecommand{\ClOrdCls}{\mathsf{ClOrdCls}}
\providecommand{\CptClModRel}{\mathsf{CptClModRel}}

\providecommand{\wAlg}[3]{#3\text-\mathsf{Alg}_{(#1,\, #2)}}
\providecommand{\lbcwAlg}[2]{#2\text-\mathsf{Alg}_{(#1,\, \textup{ps},\, \textup{lbc})}}
\providecommand{\clx}{\textup c}
\providecommand{\lax}{\textup l}

\providecommand{\hc}{\odot}

\DeclareMathOperator{\vs}{\slash_\textup v}
\DeclareMathOperator{\hs}{\slash_\textup h}

\newcommand{\cocart}{\mathrm{cocart}}
\newcommand{\cart}{\mathrm{cart}}

\providecommand{\tab}[1]{\gen{#1}}

\providecommand{\cur}[1]{#1^{\scriptscriptstyle\lambda}}
\providecommand{\curp}[1]{#1^{\scriptscriptstyle{\lambda'}}}

\author{Seerp Roald Koudenburg}

\thanks{Parts of this article were written during visits of the author to Macquarie University, in September--November 2015, and Dalhousie University, in August 2016. I am grateful to the Macquarie University Research Centre and the @CAT-group for their funding of these visits. I would like to thank Ram\'on Abud Alcal\'a, Richard Garner, Mark Weber, and especially Bob Par\'e for helpful discussions. I thank the anonymous referee for their suggestions, which have led to several improvements in the readability of this work.}

\address{Mathematics Research and Teaching Group\\Middle East Technical University\\Northern Cyprus Campus\\ 99738 Kalkanl\i, G\"uzelyurt\\Turkish Republic of Northern Cyprus\\via Mersin 10, T\"urkiye}
\eaddress{roaldkoudenburg@gmail.com}

\title[Formal category theory in augmented virtual double categories]{Formal category theory in\\ augmented virtual double categories}

\copyrightyear{2024}


\begin{document}
	\maketitle
	\begin{abstract}
		In this article we develop formal category theory within augmented virtual double categories. Notably we formalise the classical notions of Kan extension, Yoneda embedding $\map{\yon_A}A{\ps A}$, exact square, total category and `small' cocompletion; the latter in an appropriate sense. Throughout we compare our formalisations to their corresponding $2$-categorical counterparts. Our approach has several advantages. For instance, the structure of augmented virtual double categories naturally allows us to isolate conditions that ensure small cocompleteness of formal presheaf objects $\ps A$.

		Given a monoidal augmented virtual double category $\K$ with a Yoneda embedding $\map{\yon_I}I{\ps I}$ for its monoidal unit $I$ we prove that, for any `unital' object $A$ in $\K$ that has a `horizontal dual' $\dl A$, the Yoneda embedding $\map{\yon_A}A{\ps A}$ exists if and only if the `inner hom' $\inhom{\dl A, \ps I}$ exists. This result is a special case of a more general result that, given a functor $\map F\K\L$ of augmented virtual double categories, allows a Yoneda embedding in $\L$ to be ``lifted'', along a pair of `universal morphisms' in $\L$, to a Yoneda embedding in $\K$.
	\end{abstract}
	
	\tableofcontents
	
	\section*{Introduction} \label{Introduction} \addcontentsline{toc}{section}{\protect\numberline{}Introduction}
	In this work we take a ``double"/dimensional'' approach to formal category theory by taking augmented virtual double categories, which have been recently introduced in \cite{Koudenburg20}, as a setting. The author's motivation for doing so is twofold. Firstly he considers double categorical structures to be a natural setting for the formalisation of classical categorical results that involve both profunctors and Yoneda embeddings. Consider for instance the classical result by Day (\cite{Day70}) asserting that any promonoidal category $A$ embeds into a monoidal category $P$. Denoting by $T$ the `free strict monoidal category' $2$"/monad, the promonoidal structure on $A$ can be regarded as given by a profunctor $\hmap\alpha A{TA}$ satisfying certain conditions, while the Yoneda embedding $\map\yon A{\Set^{\op A} \nfd P}$ underlies the promonoidal embedding $A \hookrightarrow P$, with the monoidal structure on $P$ given by `Day convolution'\footnote{For a formalisation, in augmented virtual double categories, of Day convolution when restricted to structure morphisms $A \brar TA$ that are representable, see Section~8 of \cite{Koudenburg15b}.} with respect to $\alpha$. Formalisation of Day's result potentially allows us to apply it to other category"/like objects, such as posets, double categories and double $2$"/categories (\cite{Cruttwell-Lambert-Pronk-Szyld22}), that are equipped with promonoidal"/like structures. Similarly it allows for generalisations to other $2$"/monads $T$, such as the ultrafilter monad on the $2$"/category of posets. In more detail, the latter generalisation isolates conditions on any `modular topological space' $A$ (\cite{Tholen09}), analogous to those satisfied by the profunctor $\alpha$, ensuring that $A$ embeds into an ordered compact Hausdorff space (\cite{Tholen09}).
	
	Another relevant classical categorical result is Ad\'amek and Rosick\'y's Theorem~2.6 of \cite{Adamek-Rosicky01}. Given a copresheaf $\map dA\Set$ one of its assertions is that the left Kan extension $\map{\textup{lan}_\yon\, d}P\Set$ of $d$ along the Yoneda embedding $\map\yon AP$ preserves finite products if and only if the category of elements $\int d$ is `cosifted'. Writing $S$ for the extension of the `free category with finite products' $2$"/monad to profunctors, with unit transformation $\cell\iota\id S$, and by regarding $d$ as a profunctor $\hmap D1A$, the cosiftedness of $\int d$ can be equivalently expressed as a `Beck"/Chevalley'"/like condition on the transformation of profunctors $\cell{\iota_D}D{SD}$. In \cite{Koudenburg14b} this observation is used to formalise Ad\'amek and Rosick\'y's result in terms of any  `double monad', acting on some double category, whose vertical part is a colax"/idempotent $2$"/monad; the latter in the sense of e.g.\ \cite{Kelly-Lack97}.
	
	Before describing the second part of the author's motivation we pause to describe the main difference between our formal notion of Yoneda embedding (\defref{yoneda embedding} below) and the $2$"/categorical approach taken by Street and Walters in \cite{Street-Walters78}. To do so we partly recall their notion of \emph{Yoneda structure} on a $2$"/category $\mathcal C$, which consists of a `right ideal' $\A$ of \emph{admissible} morphisms in $\mathcal C$ and, for each admissible object $A$ (that is $\id_A \in \A$), a formal Yoneda embedding $\map{yA}A{\mathcal PA}$ that is itself admissible. The collection of these formal Yoneda embeddings is required to satisfy three axioms. We recall only Axiom~2 here: for each admissible morphism $\map fAB$ a cell $\chi^f$ as on the left below, which exhibits $f$ as the absolute left lifting of $yA$ through $B(f, 1)$ in $\mathcal C$, is required to exist. For the prototypical example of the classical Yoneda embeddings $\map{yA}A{\Set^{\op A}}$, one for each locally small category $A$, take admissible functors $\map fAB$ to be those with all hom"/sets $B(fa, b)$ small and set $B(f, 1)(b) \dfn B(f\dash, b)$.
	\begin{displaymath}
  	\begin{tikzpicture}[baseline]
			\matrix(m)[math35, column sep={1.75em,between origins}]{A & & B \\ & \mathcal PA & \\};
			\path[map]	(m-1-1) edge node[above] {$f$} (m-1-3)
													edge[transform canvas={xshift=-1pt}] node[left] {$yA$} (m-2-2)
									(m-1-3) edge[transform canvas={xshift=1pt}] node[right] {$B(f, 1)$} (m-2-2);
			\path[transform canvas={xshift=0.2em,yshift=-0.1em}]	($(m-1-1)!0.5!(m-2-2)$) edge[cell] node[above left, inner sep=0pt] {$\chi^f$} (m-1-3);
		\end{tikzpicture} \qquad \qquad \qquad \begin{tikzpicture}[baseline]
			\matrix(m)[math35]{A & B \\ \ps A & \ps A \\};
			\path[map]	(m-1-1) edge[barred] node[above] {$J$} (m-1-2)
													edge[ps] node[left] {$\yon$} (m-2-1)
									(m-1-2) edge[ps] node[right] {$\cur J$} (m-2-2)
									(m-2-1) edge[barred] node[below] {$I_{\ps A}$} (m-2-2);
			\draw				($(m-1-1)!0.5!(m-2-2)$) node[font=\scriptsize] {$\cart$};
		\end{tikzpicture} \qquad \qquad \qquad \begin{tikzpicture}[baseline]
			\matrix(m)[math35, column sep={1.75em,between origins}]{A & & B \\ & \ps A & \\};
			\path[map]	(m-1-1) edge[barred] node[above] {$J$} (m-1-3)
													edge[transform canvas={xshift=-2pt}] node[left] {$\yon$} (m-2-2)
									(m-1-3) edge[transform canvas={xshift=2pt}] node[right] {$\cur J$} (m-2-2);
			\draw				([yshift=0.333em]$(m-1-2)!0.5!(m-2-2)$) node[font=\scriptsize] {$\cart$};
		\end{tikzpicture}
	\end{displaymath}
	
	In our double"/dimensional approach we take vertical morphisms $\map fAC$ to represent abstract functors and horizontal morphisms $\hmap JAB$ to represent abstract profunctors. In contrast to Street and Walters' approach our formalisation of the Yoneda lemma does not require a notion of admissibility: instead we consider \emph{all} horizontal morphisms to be admissible. For instance, to recover the classical Yoneda embeddings $\map\yon A{\Set^{\op A} \nfd \ps A}$, with $A$ locally small and $\ps A$ large in general (see \cite{Freyd-Street95}), we take the horizontal morphisms $\hmap JAB$ to be $\Set$"/profunctors $\map J{\op A \times B}\Set$: similar to the assignment $f \mapsto B(f, 1)$ for admissible functors $f$ above, \emph{every} such profunctor $J$ induces a functor $\map{\cur J}B{\ps A}$ given by $\cur J(b) \dfn J(\dash, b)$. Notice that a functor $\map fAB$ is admissible precisely if it induces a representable $\Set$"/profunctor $\hmap{f_*}AB$ (the \emph{companion} of $f$), and in that case $\cur{(f_*)}$ recovers $B(f, 1)$ above. 
	
	Regarding all horizontal morphisms as being admissible is the main feature of our approach, and we consider next the requirements that this imposes on our double"/categorical setting. First notice that in the prototypical example, of $\Set$"/profunctors between large categories (such as $\ps A$), we are not able to compose $\Set$"/profunctors in general. Thus, in general, we cannot require the vertical and horizontal morphisms to combine into a pseudo double category (see e.g.\ \cite{Grandis-Pare99}), which is equipped with composition for both vertical and horizontal morphisms. Instead we require them to form a virtual double category (see e.g.\ \cite{Cruttwell-Shulman10} or the `$\textbf{fc}$"/multicategories' of \cite{Leinster04}). This is a weaker structure that does not require horizontal composition; instead its cells are `multicells' $\cell\phi{(J_1, \dotsc, J_n)}K$, as on the left below, that have (possibly empty) paths as horizontal sources. Writing $\Set'$ for the category of large sets, the prototypical example is the virtual double category $\enProf{(\Set, \Set')}$ of $\Set$"/profunctors between large categories (i.e.\ categories internal in $\Set'$).
	
	Secondly notice that the classical Yoneda lemma supplies, for each $\Set$"/profunctor $\hmap JAB$, natural isomorphisms $J(a, b) \iso \ps A(\yon a, \cur Jb)$. In the pseudo double category $\enProf{\Set'}$ of $\Set'$"/profunctors between large categories these isomorphisms combine into a \emph{cartesian} cell of the form as in the middle above, where $I_{\ps A}$ denotes the \emph{horizontal unit} profunctor given by the hom"/sets of $\ps A$; for the universal properties of cartesian cells and horizontal units see e.g.\ Section~4 of \cite{Koudenburg20} or \defref{cartesian cells} below. It is natural to axiomatise the Yoneda lemma as the requirement that this cartesian cell exists for every horizontal morphism $\hmap JAB$. However, since the unit profunctor $I_{\ps A}$ is not a $\Set$"/profunctor in general, we cannot do so in the prototypical virtual double category $\enProf{(\Set, \Set')}$. We are thus led to the notion of augmented virtual double category \cite{Koudenburg20}, which extends that of virtual double category by adding in \emph{nullary} cells $\cell\psi{(J_1, \dotsc, J_n)}C$ of the form as on the right below, with empty horizontal targets. The virtual double category $\enProf{(\Set, \Set')}$ naturally extends to an augmented virtual double category whose nullary cells $\psi$ below are transformations that map into the (possibly large) hom"/sets of $C$. In particular we can, in the augmented virtual double category $\enProf{(\Set, \Set')}$, consider nullary cartesian cells as on the right above. Let now $\map\yon A{\ps A}$ be any morphism in any augmented virtual double category: the \emph{Yoneda axiom} of \defref{yoneda embedding} below requires that, for every horizontal morphism $\hmap JAB$, there exists a vertical morphism $\map{\cur J}B{\ps A}$ equipped with a nullary cartesian cell as on the right above.
	\begin{displaymath}
		\begin{tikzpicture}[baseline]
			\matrix(m)[math35, column sep={1.75em,between origins}]{& A_0 & \\ C & & D \\};
			\path[map]	(m-1-2) edge[transform canvas={xshift=-1pt}] node[left] {$f$} (m-2-1)
													edge[transform canvas={xshift=1pt}] node[right] {$g$} (m-2-3)
									(m-2-1) edge[barred] node[below] {$K$} (m-2-3);
			\path				(m-1-2) edge[cell, transform canvas={yshift=-0.25em}] node[right, inner sep=2.5pt] {$\phi$} (m-2-2);
		\end{tikzpicture} \mspace{24mu} \begin{tikzpicture}[baseline]
			\matrix(m)[math35, column sep={3.25em,between origins}]{A_0 & A_1 & A_{n-1} & A_n \\ C & & & D \\};
			\path[map]	(m-1-1) edge[barred] node[above] {$J_1$} (m-1-2)
													edge node[left] {$f$} (m-2-1)
									(m-1-3) edge[barred] node[above] {$J_n$} (m-1-4)
									(m-1-4) edge node[right] {$g$} (m-2-4)
									(m-2-1) edge[barred] node[below] {$K$} (m-2-4);
			\path[transform canvas={xshift=1.625em}]	(m-1-2) edge[cell] node[right] {$\phi$} (m-2-2);
			\draw				($(m-1-2)!0.45!(m-1-3)$) node {$\dotsb$};
		\end{tikzpicture} \mspace{24mu} \begin{tikzpicture}[baseline]
			\matrix(m)[math35]{A_0 \\ C \\};
			\path[map]	(m-1-1) edge[bend right=45] node[left] {$f$} (m-2-1)
													edge[bend left=45] node[right] {$g$} (m-2-1);
			\path				(m-1-1) edge[cell] node[right] {$\psi$} (m-2-1);
		\end{tikzpicture} \mspace{24mu} \begin{tikzpicture}[baseline]
			\matrix(m)[math35, column sep={1.625em,between origins}]
				{A_0 & & A_1 & & A_{n-1} & & A_n \\ & & & C & & & \\};
			\path[map]	(m-1-1) edge[barred] node[above] {$J_1$} (m-1-3)
													edge node[below left] {$f$} (m-2-4)
									(m-1-5) edge[barred] node[above] {$J_n$} (m-1-7)
									(m-1-7) edge node[below right] {$g$} (m-2-4);
			\path				(m-1-4) edge[cell] node[right] {$\psi$} (m-2-4);
			\draw				($(m-1-1)!0.48!(m-1-7)$) node {$\dotsb$};
		\end{tikzpicture}
	\end{displaymath}
	
	Returning to the author's motivation for this article, its second part is to contribute to formal approaches to higher dimensional category theory, as follows. In Chapter~9 of \cite{Riehl-Verity22} Riehl and Verity employ formal category theory in virtual double categories to define pointwise Kan extensions of functors between $\infty$"/categories; in fact their definition is recovered by one of the notions of Kan extension that we consider (see \exref{pointwise right extension for functors between infty-categories} below). Their formal approach does not however include a formal notion of Yoneda embedding, and the author believes that the theory of the present paper is likely to be of help towards obtaining such a notion, as is explained shortly. In the case of double categories, Grandis and Par\'e in \cite{Grandis-Pare07} introduce pointwise Kan extensions of lax double functors between pseudo double categories, as an instance of their formal notion of pointwise Kan extension in pseudo double categories \cite{Grandis-Pare08}. It is currently unclear to the author whether the former notion can be reconciled with the notions considered in the present article. The author is aware of three approaches to a notion of Yoneda embedding for double categories: the original approach by Par\'e in \cite{Pare11}; Street's formal approach for strict double categories \cite{Street17}, which uses the main result of \cite{Weber07}; and a formal approach using ``generalised Day convolution'', by applying the formalisation of Day's result, as described previously, to the `free strict double category' $2$"/monad. The precise relationship between these three approaches is currently unclear to the author; in particular he does not know if Par\'e's Yoneda embeddings satisfy any of the formal notions of Yoneda embedding.
	
	With the main aim of this work being the formalisation of category theory, notably that of the notions of Kan extension and Yoneda embedding, in augmented virtual double categories, our second aim is to provide necessary and sufficient conditions for the existence of formal Yoneda embeddings. This gives us a handle on the points raised in the preceding paragraph: the sufficent condition allows us to obtain formal Yoneda embeddings (such as for $\infty$"/categories) while, given a family of ``ad hoc'' Yoneda embeddings $\yon_A$ (such as Par\'e's Yoneda embeddings for double categories), the necessary condition facilitates constructing an augmented virtual double category in which the $\yon_A$ satisfy our formal notion of Yoneda embedding. The aforementioned conditions generalise as well as recover the following fact for finitely complete categories $\E$ with subobject classifier $\Omega$: $\E$ has power objects if and only if $\Omega$ is exponentiable; see e.g.\ Section~A2.1 of \cite{Johnstone02}. In some more detail, they apply to a monoidal augmented virtual double category $(\K, \tens, I)$ whose monoidal unit $I$ admits a Yoneda embedding $\map{\yon_I}I{\ps I}$, as follows. Given any \emph{unital} object $A$ in $\K$, i.e.\ $A$ admits a horizontal unit (\defref{cartesian cells}), we prove in \thmref{presheaf object equivalent to iota-small internal hom} below that, under mild conditions, the Yoneda embedding $\map{\yon_A}A{\ps A}$ exists if and only if the `inner hom' $\inhom{\dl A, \ps I}$ does, with $\dl A$ the unital `horizontal dual' of $A$ (formalising the notion of dual category), and in that case $\ps A \iso \inhom{\dl A, \ps I}$.
	
	The horizontal dual $\dl A$ here is defined by a `horizontal copairing' $\hmap\iota I{\dl A \tens A}$ (\defref{horizontal dual}), which induces an assignment that maps every horizontal morphism $\hmap JAB$ to its `adjunct' $\hmap{\flad J}I{\dl A \tens B}$. \thmref{presheaf object equivalent to iota-small internal hom} also applies in the cases where the assignment $J \mapsto \flad J$ is not essentially surjective onto the collection of morphisms of the form $I \brar \dl A \tens B$, e.g.\ in the case of \emph{small profunctors} between large categories in the sense of \cite{Day-Lack07}; see also Example~2.8 of \cite{Koudenburg20}. In such cases the universal property of the inner"/hom $\inhom{\dl A, \ps I}$ is restricted to morphisms $\dl A \tens B \to \ps I$ in the essential image of the composite below, where $J \mapsto \cur J$ is given by the formal Yoneda axiom as described previously; see \defref{internal iota-small hom} for the details.
	\begin{displaymath}
		\set{A \brar B} \quad \xrar{\flad{(\dash)}} \quad \set{I \brar \dl A \tens B} \quad \xrar{\cur{(\dash)}} \quad \set{\dl A \tens B \to \ps I}
	\end{displaymath}
	Moreover \thmref{presheaf object equivalent to iota-small internal hom} is obtained as a corollary of the following more general result, which combines the main results of \secref{yoneda embeddings in a monoidal augmented virtual double category}: \thmsref{universal morphism from a yoneda embedding}{yoneda embedding from universal morphisms theorem}. Given a functor $\map F\K\L$ of augmented virtual double categories and, in $\L$, a Yoneda embedding $\map{\yon_A}AP$ and a `universal morphism' $\hmap\iota A{FA'}$ (generalising the notion of copairing), under mild conditions these theorems show that a Yoneda embedding $\map{\yon_{A'}}{A'}{P'}$ exists in $\K$ if and only if there exists a `universal morphism' $\map\eps{FP'}P$ whose universal property is ``restricted'' like that of $\inhom{\dl A, \ps I}$ above. The author believes that these results will be useful in obtaining formal Yoneda embeddings.
	
	The third aim of this work is to compare, throughout, its double"/dimensional approach to formal category to the classical $2$"/categorical approaches that use Yoneda structures, as taken by Street and Walters in \cite{Street-Walters78} and by Weber in \cite{Weber07}. We close this introduction by outlining one of the advantages of our approach: it allows for isolating conditions that ensure the `small cocompleteness' of formal presheaf objects $\ps A$, as follows. Weber in Definition~3.17 of \cite{Weber07} defines an object $C$, of a $2$"/category equipped with a `good Yoneda structure', to be `cocomplete' whenever it admits pointwise left Kan extensions of all diagrams $C \xlar d X \xrar h Z$ such that $X$, the presheaf object $\ps X$, $Z$ and $h$ are admissible. All of our notions of left Kan extension, including the pointwise variant, that are defined in \secref{Kan extension section} below and used throughout this work, are along a (path of) horizontal morphism(s); see also the \overref\ below. Thus, as a consequence of our viewpoint of all horizontal morphisms being `admissible', we regard our notions of left Kan extension as being along `admissible' morphisms only. Analogous to Weber's definition we define an object $M$ of an augmented virtual double category to be \emph{cocomplete} (\defref{cocompletion}) whenever it admits pointwise Kan extensions (\defref{pointwise left Kan extension}) of all (or some pre"/specified class of) diagrams $M \xlar d A \xbrar J B$.
	
	Consider a formal Yoneda embedding $\map\yon M{\ps M}$ in an augmented virtual double category $\K$. An advantage of our approach is that the existence of pointwise Kan extensions into $\ps M$ is related to the existence of `pointwise horizontal composites' in $\K$, and this can be used to obtain a condition that ensures cocompleteness of $\ps M$. In some more detail, \emph{pointwise composites} of horizontal morphisms are defined by `pointwise cocartesian cells' (see Definition~9.1 of \cite{Koudenburg20} or \remref{right pointwise cocartesian and pointwise right cocartesian comparison} below), and in \defref{pointwise right cocartesian path} below the latter notion is weakened in two ways, resulting in that of `pointwise right unary"/cocartesian cell'. Given a diagram $\ps M \xlar d A \xbrar J B$ it is shown in \cororef{left Kan extension, exact cells, cocartesian cells} below that the pointwise left Kan extension of $d$ along $J$ exists if and only if there exists a pointwise right unary"/cocartesian cell of the form
	\begin{displaymath}
		\begin{tikzpicture}
			\matrix(m)[math35]{M & A & B \\ M & & B, \\};
			\path[map]	(m-1-1) edge[barred] node[above] {$\ps M(\yon, d)$} (m-1-2)
									(m-1-2) edge[barred] node[above] {$J$} (m-1-3)
									(m-2-1) edge[barred] node[below] {$K$} (m-2-3);
			\path				(m-1-1) edge[eq] (m-2-1)
									(m-1-3) edge[eq] (m-2-3)
									(m-1-2) edge[cell] (m-2-2);
		\end{tikzpicture}
	\end{displaymath}
	where $\ps M(\yon, d)$ denotes the \emph{restriction} of $\ps M$ along $\yon$ and $d$ (\defref{cartesian cells}) and where $K$ is any horizontal morphism. Using this result \thmref{presheaf objects as free cocompletions}, the main theorem of \secref{cocompleteness section}, isolates conditions that ensure that $\map\yon M{\ps M}$ defines $\ps M$ as the `free cocompletion' of $M$, in the sense of \defref{cocompletion} (see also the \overref\ below); in particular they ensure that $\ps M$ is cocomplete. In contrast, consider the analogous Theorem~3.20 of \cite{Weber07} which concerns a Yoneda embedding $\map yC{\ps C}$ of a good Yoneda structure. Like our \thmref{presheaf objects as free cocompletions} it proves that $y$ defines $\ps C$ as the free cocompletion of $C$, but it does not isolate conditions ensuring cocompleteness of $\ps C$; it instead assumes cocompleteness of $\ps C$. Further differences between Weber's result and our \thmref{presheaf objects as free cocompletions} are described in \remref{differences to Weber's result}. Given a monoidal augmented virtual double category $(\K, \tens, I)$, \thmref{free cocompletion of the monoidal unit} uses \thmref{presheaf objects as free cocompletions} to describe conditions ensuring that the Yoneda embedding $\map\yon I{\ps I}$ defines $\ps I$ as the `free cocompletion' of the monoidal unit $I$.
	
	\section*{Overview} \label{Overview} \addcontentsline{toc}{section}{\protect\numberline{}Overview}
	We start in \secref{Kan extension section} by introducing four notions of left Kan extension in an augmented virtual double category $\K$: a notion of \emph{weak Kan extension} (\defref{weak left Kan extension}); a notion of \emph{Kan extension} (\defref{left Kan extension}), which formalises enriched Kan extension (\exref{enriched left Kan extension}); a notion of \emph{pointwise weak Kan extension} (\defref{pointwise left Kan extension}), which is reminiscent of Street's $2$"/categorical notion of pointwise Kan extension \cite{Street74b}; and a notion of \emph{pointwise Kan extension} (\defref{pointwise left Kan extension}) which combines the latter two notions. Each of these defines the extension of a vertical morphism $\map d{A_0}M$ along a path $\hmap{\ul J}{A_0}{A_n}$ of horizontal morphisms, with the resulting left Kan extension $\map l{A_n}M$ being exhibited by a nullary cell
	\begin{displaymath}
		\begin{tikzpicture}[textbaseline]
			\matrix(m)[math35]{A_0 & A_1 & A_{n-1} & A_n \\};
			\draw	([yshift=-3.25em]$(m-1-1)!0.5!(m-1-4)$) node (M) {$M;$};
			\path[map]	(m-1-1) edge[barred] node[above] {$J_1$} (m-1-2)
													edge[transform canvas={yshift=-2pt}] node[below left] {$d$} (M)
									(m-1-3) edge[barred] node[above] {$J_n$} (m-1-4)
									(m-1-4) edge[transform canvas={yshift=-2pt}] node[below right] {$l$} (M);
			\path				($(m-1-2.south)!0.5!(m-1-3.south)$) edge[cell] node[right] {$\eta$} (M);
			\draw				($(m-1-2)!0.45!(m-1-3)$) node {$\dotsb$};				
		\end{tikzpicture}
	\end{displaymath}
	such cells we call \emph{(pointwise) (weak) left Kan}. Weakly left Kan extending along companions in $\K$ recovers the classical $2$"/categorical notion of left Kan extension in the vertical $2$"/category $V(\K)$ of objects, vertical morphisms and `vertical cells' of $\K$ (\propref{weak left Kan extensions along companions}). If $\K$ admits all horizontal units and restrictions on the right then the notions of left Kan extension and pointwise left Kan extension coincide (\remref{pointwise left Kan extension in the presence of horizontal units and restrictions on the right}). Using the results of \cite{Koudenburg18}, in \exref{continuous left Kan extensions} we construct pointwise left Kan extensions of morphisms of modular closure spaces \cite{Tholen09}. Given a functor $\map F\K\L$ of augmented virtual double categories the notion of a universal morphism \mbox{$\map\eps{FC'}C$}, from $F$ to an object $C \in \L$, is defined in \defref{universal vertical morphism}, analogously to the classical notion. \propref{taking adjuncts preserves left Kan cells} shows that if a nullary cell of the form $\cell\phi{F\ul J}C$ defines a left Kan extension in $\L$ then so does its ``$\eps$"/adjunct'' $\cell{\shad\phi}{\ul J}{C'}$ in $\K$.
	
	In \secref{pasting lemmas section} we prove two pasting lemmas for left Kan cells that are used throughout this work. The horizontal pasting lemma (\lemref{horizontal pasting lemma}) concerns the horizontal composite of two left Kan cells. It recovers the classical result for enriched iterated Kan extensions, and it forms the main reason for our choice of Kan extending along \emph{paths} of horizontal morphisms: without doing so the horizontal pasting lemma cannot be stated (\remref{extension along paths}). The vertical pasting lemma (\lemref{vertical pasting lemma}) concerns the vertical composite $\eta \of \ul\phi$ of a left Kan cell $\eta$ and a `cocartesian path of cells' $\ul\phi$. In fact considering the weakest requirements on the path $\ul\phi$ such that, for each left Kan cell $\eta$ composable with $\ul\phi$, the composite $\eta \of \ul\phi$ is again left Kan, leads to the weakened notion of \emph{right nullary"/cocartesian} path (\defref{cocartesian path}). The vertical pasting lemma for this weakened notion of cocartesian path is used throughout. The remainder of \secref{pasting lemmas section} consists of consequences of the pasting lemmas. Given a path $\hmap{(J_1, \dotsc, J_n)}{A_0}{A_n}$ and a `full and faithful morphism' $\map f{A_n}B$, \propref{pointwise left Kan extension along full and faithful map} for instance shows that if $\map lBM$ is the pointwise left Kan extension of some $\map d{A_0}M$ along the concatenation $\hmap{(J_1, \dotsc, J_n, f_*)}{A_0}B$, where  $\hmap{f_*}{A_n}B$ is the companion of $f$, then $l \of f$ forms the left Kan extension of $d$ along $\ul J$; this generalises the classical result on (enriched) left Kan extending along a full and faithful functor.
	
	The main theorem of \secref{pointwise Kan extensions section}, \thmref{pointwise Kan extensions in terms of pointwise weak Kan extensions}, shows that the notions of pointwise weak left Kan extension and pointwise left Kan extension coincide in augmented virtual double categories $\K$ that have all restrictions on the right as well as all `cocartesian tabulations'. The notion of tabulation (\defref{tabulation}) formalises that of graph of a functor. \propref{pointwise left Kan extensions along companions} then shows that pointwise left Kan extension along companions in such $\K$ coincides with pointwise left Kan extension in the vertical $2$"/category $V(\K)$, the latter in the classical sense of \cite{Street74b}.
	
	Using the notion of (weak) left Kan extension \secref{yoneda embeddings section} starts by introducing the notions of \emph{density} and \emph{weak density} for vertical morphisms (\defref{density definition}). \defref{yoneda embedding} then defines a \emph{(weak) Yoneda morphism} $\map\yon A{\ps A}$ to be a (weakly) dense morphism that satisfies the Yoneda axiom, as described in the \introref\ above. These conditions on $\yon$ do not imply that it is full and faithful, which instead is a consequence of the existence of the horizontal unit $\hmap{I_A}AA$ (\lemref{full and faithful yoneda embedding}); a full and faithful $\yon$ is called a \emph{(weak) Yoneda embedding}. Several of our results do not depend on the full and faithfulness of Yoneda morphisms (\remref{dependence on full and faithfulness}). A Yoneda morphism $\yon$ such that all restrictions $\ps A(\yon, f)$ exist, for any $\map fB{\ps A}$, induces, for every object $B$, an equivalence between the category of horizontal morphisms $A \brar B$ and that of vertical morphisms $B \to \ps A$ (\propref{equivalence from yoneda embedding}). Our notion of Yoneda embedding recovers that of enriched Yoneda embedding, that of enriched Yoneda embedding for small enriched presheaves in the sense of \cite{Day-Lack07}, that of power object in a finitely complete category, in the sense of Section~A2.1 of \cite{Johnstone02}, and that of upper Vietoris space of downsets in a closed"/ordered closure space, the latter in the sense of \cite{Tholen09}; see \exrref{yoneda embedding for unit V-category}{yoneda embeddings for internal preorders} and \exref{yoneda embedding for small enriched profunctors}. The Yoneda embeddings of the good Yoneda structure associated to a 2-topos \cite{Weber07} are instances of our notion too (\exref{yoneda embeddings in 2-topoi}). Given an augmented virtual double category $\K$, with vertical $2$"/category $V(\K)$, in \thmref{yoneda structures} we compare our notion of Yoneda embedding in $\K$ to the notion of Yoneda structure on $V(\K)$ (\cite{Street-Walters78}) and to the notion of good Yoneda structure on $V(\K)$ (\cite{Weber07}).
	
	In \secref{exact cells} the classical notion of exact square of functors, as considered by e.g.\ Guitart \cite{Guitart80}, is formalised as follows. Given a cell $\phi$ with horizontal target $\hmap KCD$ and a morphism $\map dCM$ we call $\phi$ \emph{left $d$"/exact} if, for every left Kan cell $\eta$ defining the left Kan extension of $d$ along $K$, the composite $\eta \of \phi$ is again left Kan (\defref{left exact}). In the presence of a Yoneda morphism $\map\yon C{\ps C}$ this notion relates to that of cocartesianness and that of cocompleteness of $\ps C$ as follows. If $\phi$ has the identity morphism $\id_D$ as vertical target then it is left $\yon$"/exact if and only if it is right unary"/cocartesian (\propref{left exactness and right unary-cocartesianness}). Moreover the left Kan extension of any morphism $\map d{A_0}{\ps C}$ along a path $\hmap{(J_1, \dotsc, J_n)}{A_0}{A_n}$ exists if and only if there exists a horizontal left $\yon$"/exact cell with the concatenation $\hmap{(\ps C(\yon, d), J_1, \dotsc, J_n)}C{A_n}$ as horizontal source (\propref{left Kan extensions along a yoneda embedding in terms of left y-exact cells}). \thmref{left exactness and Beck-Chevalley} describes left exactness in terms of a \emph{Beck"/Chevalley condition} (\defref{left Beck-Chevalley condition}); the latter condition, in turn, is used in \thmref{left Beck-Chevalley condition and absolute left Kan extensions} to characterise \emph{absolute} left Kan extensions (\defref{absolutely left Kan}).
	
	In \defref{totality} a morphism $\map fMN$ is defined to be \emph{total} if the pointwise left Kan extension of $f$ along every $\hmap JMB$ exists. An object $M$ is total if its identity morphism $\id_M$ is total; this recovers the classical notion of totality for enriched categories (\cite{Day-Street86} and \cite{Kelly86}). Given a Yoneda morphism $\map\yon M{\ps M}$ consider the morphism $\map{\cur{f_*}}N{\ps M}$ corresponding to the companion $\hmap{f_*}MN$, as given by the Yoneda axiom: \thmref{total morphisms} shows that the totality of $f$ is equivalent to the existence of a left adjoint to $\cur{f_*}$. The latter condition is analogous to that of the classical $2$"/categorical definition of totality introduced in \cite{Street-Walters78} (\exref{weak totality in an augmented virtual equipment with yoneda embeddings}). Any presheaf object $\ps M$ is total whenever the companion $\hmap{\yon_*}M{\ps M}$ of its Yoneda morphism exists (\exref{presheaf objects are total}) (the analogous result Corollary~14 of \cite{Street-Walters78} requires both $M$ and $\ps M$ to be admissible). Under mild conditions any morphism $\map fAC$ induces a morphism $\map{\ps f}{\ps C}{\ps A}$ of presheaf objects (\defref{restriction}); this formalises the classical functor $\ps f$ given by restricting presheaves along $f$. \propref{restriction and curry} describes the relation between the assignments $f \mapsto \ps f$ and $J \mapsto \cur J$, the latter given by the Yoneda axiom; in \cororef{uniqueness of yoneda embeddings} this is used to describe the uniqueness of Yoneda embeddings. Using the notion of totality, \corosref{right adjoint to ps f}{left adjoint to ps f} describe the right and left adjoints of $\ps f$.
	
	The aim of \secref{cocompleteness section} is to isolate conditions ensuring that a Yoneda embedding $\map\yon M{\ps M}$ defines $\ps M$ as the free `small' cocompletion of $M$. The appropriate notion of `smallness' here depends on the augmented virtual double category under consideration: while small cocompleteness in $\enProf{(\Set, \Set')}$ most naturally means ``admits all pointwise left Kan extensions along $\Set$"/profunctors $\hmap JAB$ with $A$ a small category'' (\exref{enriched free cocompletion}), in the pseudo double category $\ensProf{(\Set, \Set')}$ of small $\Set$"/profunctors between large categories the notion of ``admitting pointwise left Kan extensions along \emph{all} small $\Set$"/profunctors'' turns out to be more appropriate (\exref{small enriched free cocompletion}). This is why in \defref{cocompletion} we assume specified an `ideal' $\mathcal S$ of left diagrams $(d, J)$, consisting of pairs of morphisms $M \xlar d A \xbrar J B$, and then define an object $N$ to be \emph{$\mathcal S$"/cocomplete} whenever, for every $(d, J) \in \mathcal S$ such that $d$ has $N$ as target, the pointwise left Kan extension of $d$ along $J$ exists. Given an ideal $\mathcal S$ of left diagrams, the main result \thmref{presheaf objects as free cocompletions} uses the notion of pointwise right unary"/cocartesian cell to give conditions that ensure that a Yoneda embedding $\map\yon M{\ps M}$ defines $\ps M$ as the free $\mathcal S$"/cocompletion of $M$, as described previously at the end of the \introref.

	The main results of the final section (\secref{yoneda embeddings in a monoidal augmented virtual double category}) have already been described in the \introref. Its \thmref{presheaf object equivalent to iota-small internal hom} is used in obtaining some of the examples of Yoneda embedding in \secref{yoneda embeddings section}. \secref{yoneda embeddings in a monoidal augmented virtual double category} depends on \secrref{Kan extension section}{yoneda embeddings section} only, except for \thmref{free cocompletion of the monoidal unit} which depends on \secref{cocompleteness section}.
	
	\subsection*{References to the prequel.} This work is a sequel to the paper \cite{Koudenburg20}, which introduces the notion of augmented virtual double category. The results of the latter are used throughout this work. To save space we, when referring to such results, do not cite \cite{Koudenburg20} but instead refer to them by prefixing their numbering with the capital letter `A'; e.g.\ ``Definition~A1.2'' and ``Lemma~A8.1'' in this text refer to Definition~1.2 and Lemma~8.1 of \cite{Koudenburg20}. References to sections of the prequel use the same prefix, e.g.\ ``\augsecref 7'' refers to Section~7 of \cite{Koudenburg20}.
	
	Like the prequel this work is based on parts of the draft \cite{Koudenburg15b}, specifically its Sections~4 and~5. The material presented here is significantly more streamlined and expanded in many ways; in particular the material of the present Section~5 is new. The author encourages readers to consult the present article rather than the corresponding draft material of \cite{Koudenburg15b}.
	
	\subsection*{Notation} \label{Notation} \addcontentsline{toc}{section}{\protect\numberline{}Notation}
	We continue using the notation of \cite{Koudenburg20}. In particular:
	\begin{enumerate}[label=-]
		\item for any integer $n \geq 1$ we write $n' \dfn n - 1$ for its predecessor;
		\item given composable paths $\ul J = (J_1, \dotsc, J_n)$ and $\ul H = (H_1, \dotsc, H_m)$ of horizontal morphisms we write $\ul J \conc \ul H \dfn (J_1, \dotsc, J_n, H_1, \dotsc, H_m)$ for their concatenation;
		\item for a path $\ul J = (J_1, \dotsc, J_n)$ of horizontal morphisms we write $\id_{\ul J} \dfn (\id_{J_1}, \dotsc, \id_{J_n})$ for the corresponding path of identity cells;
		\item most cartesian and cocartesian cells (\augdefref{4.1} and \augsecref 7; see also \defref{cartesian cells} below) are left unnamed, and instead denoted by ``$\cart$'' and ``$\cocart$'';
		\item we assume fixed a category $\Set'$ of large sets and a subcategory $\Set \subsetneqq \Set'$ of small sets, such that the collection of morphisms of $\Set$ forms an object in $\Set'$.
	\end{enumerate}
	
	\section{Left Kan extension}\label{Kan extension section}
	We begin by introducing four notions of left Kan extension in augmented virtual double categories. The first of these, defined below, is that of `weak left Kan extension'. In \propref{weak left Kan extensions along companions} we will see that weak left Kan extension along companions (see \augdefref{5.1} or \defref{companion and conjoint} below) in an augmented virtual double category $\K$ corresponds to left Kan extension in the vertical 2"/category $V(\K)$ contained in $\K$ (\augexref{1.5}), the latter in the classical sense. The stronger notion of `left Kan extension', introduced in \defref{left Kan extension} below, recovers the classical notions of `weighted colimit' and `enriched left Kan extension', as introduced by Borceux and Kelly in \cite{Borceux-Kelly75} (see also Sections~3 and 4 of \cite{Kelly82}), as we will see in \exsref{weighted colimits}{enriched left Kan extension}. On the other hand the notion of `pointwise weak left Kan extension' of \defref{pointwise left Kan extension} is reminiscent of that of pointwise left Kan extension in a $2$"/category, as introduced by Street in \cite{Street74b}. The same definition also introduces the notion of `pointwise left Kan extension', which combines the latter two strengthenings. In \secref{pointwise Kan extensions section} we will see that the notions of `pointwise weak left Kan extension' and `pointwise left Kan extension' coincide in augmented virtual double categories $\K$ that have `cocartesian tabulations' as well as restrictions on the right (\thmref{pointwise Kan extensions in terms of pointwise weak Kan extensions}); moreover in that case they recover Street's $2$"/categorical notion of pointwise left Kan extension in $V(\K)$ (\propref{pointwise left Kan extensions along companions}).
	
	\subsection{Weak left Kan extension}
	The notion of weak left Kan extension below generalises the notion of `Kan extension' in a double category that was introduced in Definition~3.1 of \cite{Koudenburg14a}, by allowing extensions of a vertical morphism $\map dAM$ along a path of horizontal morphisms $\hmap{J_1}A{A_1}$, \dots, $\hmap{J_n}{A_{n'}}{A_n}$ instead of a single morphism $\hmap JAB$. The latter notion in turn specialises that of Kan extension given in Section~2 of \cite{Grandis-Pare08}, which allows extension of $\map dAM$ both at its source, along a morphism $\hmap JAB$, as well as at its target, along some $\hmap KMN$.
	
	Recall from \auglemref{1.3} the notion of horizontal composition $\phi \hc \psi$ of horizontally composable cells $\phi$ and $\psi$ in an augmented virtual double category, which is defined whenever $\phi$ or $\psi$ is nullary.
	
	\begin{definition} \label{weak left Kan extension}
		Consider the nullary cell $\eta$ in the composite on the right-hand side below. It is said to define $\map l{A_n}M$ as the \emph{weak left Kan extension} of $\map d{A_0}M$ along the (possibly empty) path $\ul J = (J_1, \dotsc, J_n)$ if any nullary cell $\phi$, as on the left"/hand side, factors uniquely through $\eta$ as a vertical cell $\phi'$ as shown. In that case $\eta$ is called \emph{weakly left Kan}.
		\begin{displaymath}
			\begin{tikzpicture}[textbaseline]
				\matrix(m)[math35, yshift=1.625em]{A_0 & A_1 & A_{n'} & A_n \\};
				\draw	([yshift=-3.25em]$(m-1-1)!0.5!(m-1-4)$) node (M) {$M$};
				\path[map]	(m-1-1) edge[barred] node[above] {$J_1$} (m-1-2)
														edge[transform canvas={yshift=-2pt}] node[below left] {$d$} (M)
										(m-1-3) edge[barred] node[above] {$J_n$} (m-1-4)
										(m-1-4) edge[transform canvas={yshift=-2pt}] node[below right] {$k$} (M);
				\path				($(m-1-2.south)!0.5!(m-1-3.south)$) edge[cell] node[right] {$\phi$} (M);
				\draw				($(m-1-2)!0.5!(m-1-3)$) node {$\dotsb$};
			\end{tikzpicture} \quad = \quad \begin{tikzpicture}[textbaseline]
				\matrix(m)[math35]{A_0 & A_1 & A_{n'} & A_n \\ & & & M \\};
				\path[map]	(m-1-1) edge[barred] node[above] {$J_1$} (m-1-2)
														edge[transform canvas={shift={(-2pt,-2pt)}}] node[below left] {$d$} (m-2-4)
										(m-1-3) edge[barred] node[above] {$J_n$} (m-1-4)
										(m-1-4) edge[bend right=45] node[left] {$l$} (m-2-4)
														edge[bend left=45] node[right] {$k$} (m-2-4);
				\path				($(m-1-1.south)!0.5!(m-1-4.south)$) edge[transform canvas={shift={(0.3em, 0.5em)}}, cell] node[right] {$\eta$} ($(m-2-1.north)!0.5!(m-2-4.north)$)
										(m-1-4) edge[cell, transform canvas={xshift=-0.25em}] node[right] {$\phi'$} (m-2-4);
				\draw				($(m-1-2)!0.5!(m-1-3)$) node {$\dotsb$};
			\end{tikzpicture}
		\end{displaymath}
	\end{definition}
	As usual any two nullary cells defining the same weak left Kan extension factor through each other as invertible vertical cells. In \exref{recursive left Kan extension} we will see that (weak) left Kan extensions along a path $\ul J = (J_1, \dotsc, J_n)$ can be obtained by extending along each of the $J_1$, \dots, $J_n$ recursively.
	\begin{example} \label{vertical cells defining left Kan extensions}
		A vertical cell is weakly left Kan if and only if it is invertible; in fact in that case it defines a left Kan extension in the sense of \defref{left Kan extension}.
	\end{example}
		
	\begin{example} \label{Kan extensions in quintets}
		In $Q(\mathcal C)$, the double category of quintets in a $2$"/category $\mathcal C$ (see \augexref{6.3}), the notion of weak Kan extension coincides with the usual $2$"/categorical notion of Kan extension in $\mathcal C$, as given in Section~2 of \cite{Street72}.
	\end{example}
	
	\begin{remark}
		The definition above induces a notion of weak left Kan extension for unital virtual double categories, that is virtual double categories $\K$ that admit all horizontal units (see \augsecref 4 or \defref{cartesian cells} below), as follows. Recall from \augexref{1.7} and \augsecref{10} that any such $\K$ induces an augmented virtual double category $N(\K)$ which has the same objects and morphisms as $\K$ while its nullary cells $\ul J \Rar M$ are precisely the unary cells $\ul J \Rar I_M$ of $\K$, where \mbox{$\hmap{I_M}MM$} is a chosen horizontal unit for $M$. Regarding the definition above for $N(\K)$ in terms of $\K$ we obtain a notion of weak left Kan extension for $\K$, defined by universal unary cells \mbox{$\cell\eta{\ul J}{I_M}$} whose horizontal targets are horizontal units. Likewise all definitions and results of this article can be applied to unital virtual double categories.
	\end{remark}
	
	\begin{remark}
		Recall from \augdefref{1.8} that every augmented virtual double category $\K$ has a horizontal dual $\co\K$. Horizontally dual to the definition above, a nullary cell $\cell\eps{(J_1, \dotsc, J_n)}M$ of $\K$ is called \emph{weakly right Kan}, thus defining a weak right Kan extension, whenever the corresponding cell $\cell{\co\eps}{(\co J_n, \dotsc, \co J_1)}M$ of $\co\K$ is weakly left Kan. Horizontal duals of the notions of `left Kan cell' and `pointwise left Kan cell', introduced in \defsref{left Kan extension}{pointwise left Kan extension} below, are obtained analogously. This article only concerns left Kan extensions.
	\end{remark}
	
	Remember from \augexref{1.5} that any augmented virtual double category $\K$ contains a $2$-category $V(\K)$ consisting of its objects, vertical morphisms and vertical cells. Weak left Kan extension along companions (see \augdefref{5.1} or \defref{companion and conjoint} below) in $\K$ corresponds to left Kan extension in $V(\K)$ as in the following proposition. In \propref{reducing to Kan extensions along a companion} we will see that (weak) left Kan extension along a path $\ul J$ of horizontal morphisms reduces to (weak) left Kan extension along a single companion morphism whenever the `cocartesian path of $(0,1)$"/ary cells for $\ul J$' exists. In the right"/hand side below ``cocart'' denotes the cocartesian cell that defines the companion $j_*$ of $j$; see \augsecref 5 or \defref{companion and conjoint} below.
	\begin{proposition} \label{weak left Kan extensions along companions}
		In an augmented virtual double category $\K$ consider a vertical cell $\eta$, as on the left-hand side below, and its factorisation $\eta'$ through the companion $j_*$, as shown.
		\begin{displaymath}
			\begin{tikzpicture}[textbaseline]
				\matrix(m)[math35, column sep={1.75em,between origins}]{& A & \\ & & B \\ & M & \\};
				\path[map]	(m-1-2) edge[bend left = 18] node[right] {$j$} (m-2-3)
														edge[bend right = 45] node[left] {$d$} (m-3-2)
										(m-2-3) edge[bend left = 18] node[right] {$l$} (m-3-2);
				\path[transform canvas={yshift=-1.625em}]	(m-1-2) edge[cell] node[right] {$\eta$} (m-2-2);
			\end{tikzpicture} \quad = \quad \begin{tikzpicture}[textbaseline]
    		\matrix(m)[math35, column sep={1.75em,between origins}]{& A & \\ A & & B \\ & M & \\};
    		\path[map]	(m-1-2) edge[transform canvas={xshift=2pt}] node[right] {$j$} (m-2-3)
    								(m-2-1) edge[barred] node[below, inner sep=2pt] {$j_*$} (m-2-3)
    												edge[transform canvas={xshift=-2pt}] node[left] {$d$} (m-3-2)
    								(m-2-3)	edge[transform canvas={xshift=2pt}] node[right] {$l$} (m-3-2);
    		\path				(m-1-2) edge[eq, transform canvas={xshift=-2pt}] (m-2-1);
    		\path				(m-2-2) edge[cell, transform canvas={yshift=0.1em}]	node[right, inner sep=3pt] {$\eta'$} (m-3-2);
    		\draw				([yshift=-0.5em]$(m-1-2)!0.5!(m-2-2)$) node[font=\scriptsize] {$\cocart$};
  		\end{tikzpicture}
  	\end{displaymath}
		The factorisation $\eta'$ is weakly left Kan in $\K$ if and only if $\eta$ defines $l$ as the left Kan extension of $d$ along $j$ in $V(\K)$, in the sense of Section~2 of \cite{Street72}.
	\end{proposition}
	\begin{proof}[(Sketch)]
		It is straightforward to show that, by factorising through the cocartesian cell defining $j_*$ (see \augsecref 5 or \defref{companion and conjoint} below), the universal property of $\eta$ in $\K$ is equivalent to that of $\eta'$ in $V(\K)$.
	\end{proof}
	
	\subsection{Left Kan extension}
	\defref{weak left Kan extension} strengthens to give a notion of left Kan extension in augmented virtual double categories as follows. This generalises the corresponding notion for double categories, that was given in Definition~3.10 of \cite{Koudenburg14a} under the name `pointwise left Kan extension'; see \exref{left Kan extensions when all composites exist} below.
	\begin{definition} \label{left Kan extension}
		Consider the nullary cell $\eta$ in the composite on the right-hand side below, where $\ul J = (J_1, \dotsc, J_n)$ is possibly empty. It is said to define $\map l{A_n}M$ as the \emph{left Kan extension} of $\map d{A_0}M$ along $\ul J$ if any nullary cell $\phi$ as on the left-hand side below, where $\ul H = (H_1, \dotsc, H_m)$ is any (possibly empty) path, factors uniquely through $\eta$ as a nullary cell $\phi'$, as shown. In that case $\eta$ is called \emph{left Kan}.
		\begin{displaymath}
			\begin{tikzpicture}[textbaseline]
				\matrix(m)[math35, column sep={0.7em}]
					{	A_0 & A_1 & A_{n'} & A_n & B_1 & B_{m'} & B_m \\
						\phantom C & & & M & & & \phantom C \\};
				\path[map]	(m-1-1) edge[barred] node[above] {$J_1$} (m-1-2)
														edge[transform canvas={yshift=-2pt}] node[below left] {$d$} (m-2-4)
										(m-1-3) edge[barred] node[above] {$J_n$} (m-1-4)
										(m-1-4) edge[barred] node[above] {$H_1$} (m-1-5)
										(m-1-6) edge[barred] node[above] {$H_m$} (m-1-7)
										(m-1-7) edge[transform canvas={yshift=-2pt}] node[below right] {$k$} (m-2-4);
				\draw[transform canvas={xshift=-0.5pt}]	($(m-1-2)!0.5!(m-1-3)$) node {$\dotsb$}
										($(m-1-5)!0.5!(m-1-6)$) node {$\dotsb$};
				\path				(m-1-4) edge[cell] node[right] {$\phi$} (m-2-4);
			\end{tikzpicture} \mspace{-6mu} = \mspace{-6mu} \begin{tikzpicture}[textbaseline]
				\matrix(m)[math35, column sep={0.7em}]
					{	A_0 & A_1 & A_{n'} & A_n & B_1 & B_{m'} & B_m \\
						\phantom C & & & M & & & \phantom C \\};
				\path[map]	(m-1-1) edge[barred] node[above] {$J_1$} (m-1-2)
														edge[transform canvas={yshift=-2pt}] node[below left] {$d$} (m-2-4)
										(m-1-3) edge[barred] node[above] {$J_n$} (m-1-4)
										(m-1-4) edge[barred] node[above] {$H_1$} (m-1-5)
														edge node[right] {$l$} (m-2-4)
										(m-1-6) edge[barred] node[above] {$H_m$} (m-1-7)
										(m-1-7) edge[transform canvas={yshift=-2pt}] node[below right] {$k$} (m-2-4);
				\draw[transform canvas={xshift=-0.5pt}]	($(m-1-2)!0.5!(m-1-3)$) node {$\dotsb$}
										($(m-1-5)!0.5!(m-1-6)$) node {$\dotsb$};
				\path				(m-1-3)	edge[cell, transform canvas={yshift=0.25em}] node[right] {$\eta$} (m-2-3)
										(m-1-5) edge[cell, transform canvas={yshift=0.25em}] node[right] {$\phi'$} (m-2-5);
			\end{tikzpicture}
		\end{displaymath}
	\end{definition}
	Clearly every left Kan extension is a weak left Kan extension, by restricting the universal property above to cells $\phi$ with $\ul H = (A_n)$ empty.
	
	\begin{example} \label{pointwise right extension for functors between infty-categories}
		In Section~8 of \cite{Riehl-Verity22} Riehl and Verity introduce the unital virtual equipment (see \defref{augmented virtual equipment} below) $\mathbb M\textup{od}(\K)$ of `modules in an $\infty$"/cosmos $\K$'. In $\mathbb M\textup{od}(\K)$ consider a factorisation $\eta = \eta' \of \cocart$ as in \propref{weak left Kan extensions along companions}. The vertical cell $\cell\eta d{l \of j}$ corresponds to a `$\infty$"/natural transformation' $l \of j \Rar d$ in the `homotopy $2$"/category associated to $\K$'; see Proposition~8.4.11 of \cite{Riehl-Verity22}. Using Theorem~8.4.4 and Definition~9.1.2 of the latter it is straightforward to see that this transformation defines $l$ as a `pointwise right extension', in the sense of its Theorem~9.3.3(iii), precisely if $\eta'$ is left Kan in $\mathbb M\textup{od}(\K)$, in our sense above.
	\end{example}
	
	\subsection{Weighted colimits and enriched left Kan extension}
	The notion of left Kan extension specialises to the classical notions of weighted colimit and enriched left Kan extension as follows. As noted in the introduction to \cite{Street74b} recall that the $2$"/categorical notion of pointwise left Kan extension, as introduced therein, is too strong to recover the notion of enriched Kan extension; see \exref{pointwise is stronger than enriched} below.
	\begin{example} \label{weighted colimits}
		Let $\V = (\V, \tens, I)$ be a monoidal category and $\eta$ a cell in the unital virtual equipment $\enProf\V$ of $\V$"/profunctors (\augexsref{2.4}{4.2}) that is of the form as in the composite on the right"/hand side below. Here $I$ denotes the unit $\V$"/category with single object $*$ and hom"/object $I(*, *) = I$; we identify $\V$"/functors \mbox{$\map fIM$} with objects in $M$ and $\V$"/profunctors $\hmap HII$ with $\V$"/objects.
		\begin{displaymath}
			\begin{tikzpicture}[textbaseline]
				\matrix(m)[math35, yshift=1.625em]{A_0 & A_1 & A_{n'} & I & I \\};
				\draw				([yshift=-3.25em]$(m-1-1)!0.5!(m-1-5)$) node (M) {$M$};
				\path[map]	(m-1-1) edge[barred] node[above] {$J_1$} (m-1-2)
														edge[transform canvas={yshift=-2pt}] node[below left] {$d$} (M)
										(m-1-3) edge[barred] node[above] {$J_n$} (m-1-4)
										(m-1-4) edge[barred] node[above] {$H$} (m-1-5)
										(m-1-5) edge[transform canvas={yshift=-2pt}] node[below right] {$k$} (M);
				\path				(m-1-3) edge[cell] node[right] {$\phi$} (M);
				\draw				($(m-1-2)!0.5!(m-1-3)$) node {$\dotsb$};
			\end{tikzpicture} \quad = \quad \begin{tikzpicture}[textbaseline]
				\matrix(m)[math35]{A_0 & A_1 & A_{n'} & I & I \\ & & & M & \\};
				\path[map]	(m-1-1) edge[barred] node[above] {$J_1$} (m-1-2)
														edge[transform canvas={yshift=-2pt}] node[below left] {$d$} (m-2-4)
										(m-1-3) edge[barred] node[above] {$J_n$} (m-1-4)
										(m-1-4) edge[barred] node[above] {$H$} (m-1-5)
														edge node[left] {$l$} (m-2-4)
										(m-1-5) edge[transform canvas={yshift=-2pt}] node[below right] {$k$} (m-2-4);
				\path				(m-1-2) edge[cell, transform canvas={shift={(3em,0.333em)}}] node[right] {$\eta$} (m-2-2)
										(m-1-4) edge[cell, transform canvas={shift={(1em,0.333em)}}] node[right] {$\phi'$} (m-2-4);
				\draw				($(m-1-2)!0.5!(m-1-3)$) node {$\dotsb$};
			\end{tikzpicture}
  	\end{displaymath}
  	One checks that the universal property defining $\eta$ as a left Kan cell in $\enProf\V$ reduces to the unique factorisations through $\eta$ of the cells $\phi$ of the form as on the left"/hand side above; see Proposition~2.24 of \cite{Koudenburg15a} for the horizontally dual result in the case that $\V$ has large colimits that are preserved by $\tens$ on both sides, so that $\enProf\V$ is a pseudo double category (\augexref{9.2}).
  	
  	Unpacking the reduced universal property above for a $(1,0)$"/ary cell $\cell\eta{J_1}{M}$ we recover the notion of a `couniversal $\V$"/natural pair' $(l, \eta)$ that defines $l \in M$ as the \emph{tensor product of $J_1$ with $d$} in the sense of Definition~3.5 of \cite{Lindner81}. Next consider the $(n,0)$"/ary cell $\cell\eta{\ul J}M$ above in a unital virtual equipment $\enProf{\V'}$ where $\V'$ is a closed symmetric monoidal category, so that each $\V'$"/profunctor $J_i$ can be regarded as a $\V'$"/functor \mbox{$\map{J_i}{\op A_{i'} \tens A_i}\V'$}. It is straightforward to check that $\eta$ satisfies the reduced universal property above if and only if the adjuncts of the composites
  	\begin{displaymath}
  		J_1(x_0, x_1) \tens' \dotsb \tens' J_n(x_{n'}, *) \tens' M(l, k) \xrar{\phi \tens' \id} M(dx_0, l) \tens' M(l, k) \xrar{\bar M} M(dx_0, k)
  	\end{displaymath}
  	define $M(l, k)$ as the iterated end
  	\begin{displaymath}
  		\intl_{x_0 \in A_0} \dotsb \intl_{x_{n'} \in A_{n'}} \brks{J_1(x_0, x_1) \tens' \dotsb \tens' J_n(x_{n'}, *), M(dx_0, k)}'.
  	\end{displaymath}
  	If $n = 1$, so that the end reduces to the $\V'$"/object $\inhom{\op A_0, \V'}\pars{J_1, M(d\dash, k)}$ of $\V'$"/functors $J_1 \to M(d\dash, k)$, this recovers the notion of $\eta$ defining $l$ as the \emph{$J_1$"/weighted colimit of $d$}, in the usual sense of equation~(3.5) of \cite{Kelly82} (where the colimit is said to be `indexed by $J_1$') and as originally introduced in \cite{Borceux-Kelly75}.
  	
  	In light of the previous we call paths $\hmap{\ul J}{A_0}I$ of $\V$"/profunctors \emph{$\V$"/weights} and say that the left Kan cell $\cell\eta{\ul J}M$ above defines $l \in M$ as the \emph{$\ul J$"/weighted colimit of $d$}. We also use the term $\ul J$"/weighted colimit $d$ for left Kan cells \mbox{$\ul J \Rar M$} as above in the unital virtual double category $\ensProf\V$ of small $\V$"/profunctors (\augexsref{2.8}{4.7}). Notice that the reduced universal property above, for $\eta \in \ensProf\V$, is the same whether considered in $\ensProf\V$ or in $\enProf\V$, since all $\V$"/profunctors of the form $\hmap HII$ are small. Together with \lemref{full and faithful functors reflect and preserve weakly left Kan cells} we conclude that the embedding $\ensProf\V \hookrightarrow \enProf\V$ both preserves and reflects cells $\eta$ defining weighted colimits.
  	
  	Next let $\V \subset \V'$ be a \emph{universe enlargement} in the sense of Section~3.12 of \cite{Kelly82} (see also \augexref{2.7}), that is a monoidal, limit-preserving and full embedding of $\V$ into a closed monoidal and locally large category $\V'$ that is both large complete and large cocomplete. Consider the sub-augmented virtual equipment $\enProf{(\V, \V')} \subseteq \enProf{\V'}$ of \emph{$\V$"/profunctors} $\hmap JAB$ between $\V'$"/categories, with $J(x, y) \in \V$ for all $x \in A$, $y \in B$; see \augexsref{2.7}{4.6}. Applying \lemref{full and faithful functors reflect and preserve weakly left Kan cells} we find that a cell $\cell\eta{\ul J}M$ in $\enProf{(\V, \V')}$, of the form as in the right"/hand side above, is left Kan in $\enProf{(\V, \V')}$ whenever it defines $l$ as the $\ul J$"/weighted colimit of $d$ in $\enProf{\V'}$. Notice however that the reduced universal property for $\eta$ above is in general weaker when considered in $\enProf{(\V, \V')}$ (where $H \in \V$) than when considered in $\enProf{\V'}$ (where $H \in \V'$). One checks however that the two properties coincide if the iterated end above, which is known to exist in the large complete $\V'$, is (isomorphic to) a $\V$"/object. Recalling that $\V \subset \V'$ preserves limits, to ensure the latter it suffices that $\V$ is closed symmetric monoidal and small complete, $\V \subset \V'$ is a closed symmetric monoidal functor, and all of the $A_i$ are small $\V'$"/categories.
	\end{example}
	
	\begin{example} \label{enriched left Kan extension}
		Let $\V = (\V, \tens, I)$ be a monoidal category and let $\K$ denote either the unital virtual double category $\enProf\V$ of $\V$"/profunctors or that of small $\V$"/profunctors $\ensProf\V$. Notice that in either case $\K$ has all horizontal units and restrictions on the right (\augexsref{4.2}{4.7}). Given a path $\hmap{\ul J}{A_0}{A_n}$ of (small) $\V$"/profunctors and a $\V$"/functor $\map d{A_0}M$, assume that for each $x \in A_n$ the $\bigpars{J_1, \dotsc, J_n(\id, x)}$"/weighted colimit of $d$ exists in $\K$, in the sense of the previous example. Here $J_n(\id, x)$ is the restriction (see \augdefref{4.1} or \defref{cartesian cells} below) of $J_n$ along $\map xI{A_n}$, the $\V$"/functor that picks out $x \in A_n$. Denote each of these colimits by $l_x$ and its defining cell, of the form as on the left below, by $\eta_x$. It is straightforward to show that the universal property of the $\eta_x$ ensures that the objects $l_x$ and cells $\eta_x$ uniquely combine into a $\V$"/functor $\map l{A_n}M$ and a left Kan cell $\eta$ such that the equation on the left below is satisfied in $\K$ for each $x \in A_n$; here ``cart'' denotes the cartesian cell that defines the restriction $J_n(\id, x)$.
		
		Using \cororef{left Kan extensions with unital sources} below we conclude that a cell $\eta$ in $\K$, of the form as in the right"/hand side on the left below, is left Kan precisely if for each $x \in A_n$ the composite $\eta_x$ on the left below defines $l_x = lx$ as the $\bigpars{J_1, \dotsc, J_n(\id, x)}$"/weighted colimit of $d$ in $\K$. In particular it follows that $\ensProf\V \hookrightarrow \enProf\V$ both preserves and reflects left Kan cells.
		
		Given a universe enlargement $\V \subset \V'$ (\exref{weighted colimits}) recall that the augmented virtual equipment $\enProf{(\V, \V')}$ of $\V$"/profunctors between $\V'$"/categories need not have all horizontal units (\augexref{4.6}). It follows that we cannot apply \cororef{left Kan extensions with unital sources} to $\enProf{(\V, \V')}$ in general, so that its left Kan cells may not be ``pointwise''; see \exref{left Kan extensions that are not pointwise}. Analogous to the above however any family of cells $\eta_x$ in $\enProf{(\V, \V')}$ as on the left below, defining weighted colimits $l_x$ in $\enProf{\V'}$, uniquely combine into a left Kan cell $\eta$ in $\enProf{(\V, \V')}$ that satisfies the equation on the left below for each $x \in A_n$; see also the last paragraph of the previous example. In that case applying \lemref{full and faithful cartesian cell preserving functors reflect pointwise left Kan cells} below to $\eta$ we find that it is pointwise left Kan in $\enProf{(\V, \V')}$, in the sense of \defref{pointwise left Kan extension} below.
		\begin{displaymath}
			\begin{tikzpicture}[textbaseline]
				\matrix(m)[math35, yshift=1.625em]{A_0 & A_1 & A_{n'} & I \\};
				\draw	([yshift=-3.25em]$(m-1-1)!0.5!(m-1-4)$) node (M) {$M$};
				\path[map]	(m-1-1) edge[barred] node[above] {$J_1$} (m-1-2)
														edge[transform canvas={xshift=-2pt}] node[below left] {$d$} (M)
										(m-1-3) edge[barred] node[above, xshift=-3pt] {$J_n(\id, x)$} (m-1-4)
										(m-1-4) edge[transform canvas={xshift=2pt}] node[below right] {$l_x$} (M);
				\path				($(m-1-2.south)!0.5!(m-1-3.south)$) edge[cell] node[right] {$\eta_x$} (M);
				\draw				($(m-1-2)!0.5!(m-1-3)$) node {$\dotsb$};
			\end{tikzpicture} \mspace{-15mu} = \mspace{-3mu} \begin{tikzpicture}[textbaseline]
				\matrix(m)[math35, yshift=1.625em]{A_0 & A_1 & A_{n'} & I \\ A_0 & A_1 & A_{n'} & A_n \\ };
				\draw	([yshift=-6.5em]$(m-1-1)!0.5!(m-1-4)$) node (M) {$M$};
				\path[map]	(m-1-1) edge[barred] node[above] {$J_1$} (m-1-2)
														
										(m-1-3) edge[barred] node[above, xshift=-3pt] {$J_n(\id, x)$} (m-1-4)
										(m-1-4) edge node[right] {$x$} (m-2-4)
										(m-2-4) edge[transform canvas={xshift=2pt}] node[below right] {$l$} (M)
										(m-2-1) edge[barred] node[below, inner sep=2.5pt] {$J_1$} (m-2-2)
														edge[transform canvas={xshift=-2pt}] node[below left] {$d$} (M)
										(m-2-3) edge[barred] node[below, inner sep=2.5pt] {$J_n$} (m-2-4);
				\path				(m-1-1) edge[eq] (m-2-1)
										(m-1-2) edge[eq] (m-2-2)
										(m-1-3) edge[eq, transform canvas={xshift=-2pt}] (m-2-3);
				\path				($(m-2-2.south)!0.5!(m-2-3.south)$) edge[cell] node[right] {$\eta$} (M);
				\draw				($(m-1-2)!0.5!(m-2-3)$) node {$\dotsb$}
										($(m-1-3)!0.5!(m-2-4)$) node[font=\scriptsize] {$\cart$};
			\end{tikzpicture} \mspace{8mu} \begin{tikzpicture}[textbaseline]
				\matrix(m)[math35, column sep={1.75em,between origins}]{& A & \\ & & B \\ & M & \\};
				\path[map]	(m-1-2) edge[bend left = 18] node[right] {$j$} (m-2-3)
														edge[bend right = 45] node[left] {$d$} (m-3-2)
										(m-2-3) edge[bend left = 18] node[right] {$l$} (m-3-2);
				\path[transform canvas={yshift=-1.625em}]	(m-1-2) edge[cell] node[right] {$\zeta$} (m-2-2);
			\end{tikzpicture} \mspace{-3mu} = \mspace{-3mu} \begin{tikzpicture}[textbaseline]
    		\matrix(m)[math35, column sep={1.75em,between origins}]{& A & \\ A & & B \\ & M & \\};
    		\path[map]	(m-1-2) edge[transform canvas={xshift=2pt}] node[right] {$j$} (m-2-3)
    								(m-2-1) edge[barred] node[below, inner sep=2pt] {$j_*$} (m-2-3)
    												edge[transform canvas={xshift=-2pt}] node[left] {$d$} (m-3-2)
    								(m-2-3)	edge[transform canvas={xshift=2pt}] node[right] {$l$} (m-3-2);
    		\path				(m-1-2) edge[eq, transform canvas={xshift=-2pt}] (m-2-1);
    		\path				(m-2-2) edge[cell, transform canvas={yshift=0.1em}]	node[right, inner sep=3pt] {$\zeta'$} (m-3-2);
    		\draw				([yshift=-0.5em]$(m-1-2)!0.5!(m-2-2)$) node[font=\scriptsize] {$\cocart$};
  		\end{tikzpicture}
		\end{displaymath}
		
		Next assume that $\V$ is symmetric monoidal and consider the factorisation $\zeta'$ of a vertical cell (i.e.\ a $\V$"/natural transformation) $\zeta \in \enProf\V$ as on the right above. Applying the previous to $\zeta'$, together with the observations of the previous example we find that $\zeta'$ is left Kan in $\enProf\V$ precisely if $(l, \zeta)$ is the `pointwise left Kan extension of $d$ along $j$' in the sense of Section~4 of \cite{Lindner81}. If $\V$ is moreover closed symmetric monoidal then this recovers the usual $\V$"/enriched notion of $\zeta$ `exhibiting $l$ as the left Kan extension of $d$ along $j$' in the sense of Section~4 of \cite{Kelly82}, as originally introduced in \cite{Borceux-Kelly75}.
		
		Finally assume that $A$ and $B$ are small $\V$"/categories and that $\V$ is small complete, so that the $\V$"/categories $\inhom{A, M}$ and $\inhom{B, M}$ of $\V$"/functors $A \to M$ and $B \to M$ exist; see Section~2.2 of \cite{Kelly82}. In that case the existence of the left Kan extension $\map lBM$ of $d$ along $j$ implies the following weaker condition: there exists a $\V$"/natural isomorphism $\inhom{A, M}(d, k \of j) \iso \inhom{B, M}(l, k)$ for any $\V$"/functor $\map kBM$; see Section~4.3 of \cite{Kelly82}. In particular if all left Kan extensions along $j$ exist then the $\V$"/functor \mbox{$\map{\inhom{j, M}}{\inhom{B, M}}{\inhom{A, M}}$} given by precomposition with $j$ admits a left adjoint (Theorem~4.50 of \cite{Kelly82}); the latter condition, in the case of $\V = \Set$, was originally studied by Kan in \cite{Kan58}.
	\end{example}
	
	The following straightforward lemma is useful for obtaining (weak) left Kan extensions in locally full sub"/augmented virtual double categories. For the notion of (locally) full and faithful functor between augmented virtual double categories see \augdefref{3.6}.
	\begin{lemma} \label{full and faithful functors reflect and preserve weakly left Kan cells}
		Any locally full and faithful functor $\map F\K\L$ reflects (weakly) left Kan cells, that is a cell $\eta \in \K$ is (weakly) left Kan whenever its image $F\eta$ is so in $\L$. If $F$ is full and faithful then it preserves weakly left Kan cells as well: $\eta \in \K$ is weakly left Kan if and only if $F\eta$ is so in $\L$.
	\end{lemma}
	
	\subsection{Cartesian cells and restrictions of left Kan extensions}
	In \exref{enriched left Kan extension} we used \cororef{left Kan extensions with unital sources} below. The latter uses the notion of cartesian cell (\augdefref{4.1}) which we recall here for convenience, together with the pasting lemma for cartesian cells (\auglemref{4.15}) and the notion of augmented virtual equipment (\augdefref{4.10}). We also recall the notion of full and faithful morphism (\augdefref{4.12}) and those of companion and conjoint (\augdefref{5.1}), which are related to that of cartesian cell.
	
	\begin{definition} \label{cartesian cells}
		A cell $\cell\psi{\ul J}{\ul K}$ with $\ul J$ of length $\lns{\ul J} \leq 1$, as in the right-hand side below, is called \emph{cartesian} if any cell $\chi$, as on the left-hand side, factors uniquely through $\psi$ as a cell $\phi$ as shown.
		\begin{displaymath}
			\begin{tikzpicture}[textbaseline]
				\matrix(m)[math35]{X_0 & X_1 & X_{n'} & X_n \\ A & & & B \\ C & & & D \\};
				\path[map]	(m-1-1) edge[barred] node[above] {$H_1$} (m-1-2)
														edge node[left] {$h$} (m-2-1)
										(m-1-3) edge[barred] node[above] {$H_n$} (m-1-4)
										(m-1-4) edge node[right] {$k$} (m-2-4)
										(m-2-1) edge node[left] {$f$} (m-3-1)
										(m-2-4) edge node[right] {$g$} (m-3-4)
										(m-3-1) edge[barred] node[below] {$\ul K$} (m-3-4);
				\draw				($(m-1-2)!0.5!(m-1-3)$) node {$\dotsc$};
				\path[transform canvas={yshift=-1.625em}]	($(m-1-1.south)!0.5!(m-1-4.south)$) edge[cell] node[right] {$\chi$} ($(m-2-1.north)!0.5!(m-2-4.north)$);
			\end{tikzpicture} \quad = \quad \begin{tikzpicture}[textbaseline]
				\matrix(m)[math35]{X_0 & X_1 & X_{n'} & X_n \\ A & & & B \\ C & & & D \\};
				\path[map]	(m-1-1) edge[barred] node[above] {$H_1$} (m-1-2)
														edge node[left] {$h$} (m-2-1)
										(m-1-3) edge[barred] node[above] {$H_n$} (m-1-4)
										(m-1-4) edge node[right] {$k$} (m-2-4)
										(m-2-1) edge[barred] node[below] {$\ul J$} (m-2-4)
														edge node[left] {$f$} (m-3-1)
										(m-2-4) edge node[right] {$g$} (m-3-4)
										(m-3-1) edge[barred] node[below] {$\ul K$} (m-3-4);
				\draw				($(m-1-2)!0.5!(m-1-3)$) node {$\dotsc$};
				\path				($(m-1-1.south)!0.5!(m-1-4.south)$) edge[cell] node[right] {$\phi$} ($(m-2-1.north)!0.5!(m-2-4.north)$)
										($(m-2-1.south)!0.5!(m-2-4.south)$) edge[cell, transform canvas={yshift=-2pt}] node[right] {$\psi$} ($(m-3-1.north)!0.5!(m-3-4.north)$);
			\end{tikzpicture}
		\end{displaymath}
			
		Vertically dual, provided that $\lns{\ul J} = 1$, the cell $\phi$ is called \emph{weakly cocartesian} if any cell $\chi$ factors uniquely through $\phi$ as a cell $\psi$ as shown.
	\end{definition}
	If a $(1,n)$-ary cartesian cell $\psi$ of the form above exists then its horizontal source \mbox{$\hmap JAB$} is called the \emph{restriction} of $\hmap{\ul K}CD$ along $f$ and $g$, and denoted $\ul K(f, g) \dfn J$. If $\ul K = (C \xbrar K D)$ then we call $K(f, g)$ \emph{unary}; in the case that $\ul K = C$ is an empty path we call $C(f, g)$ \emph{nullary}. Restrictions of the form $\ul K(f, \id)$ and $\ul K(\id, g)$ are called restrictions \emph{on the left} and \emph{right}. We call the nullary restriction \mbox{$\hmap{C(\id, \id)}CC$} the \emph{(horizontal) unit} of the object $C$ and denote it $I_C \dfn C(\id, \id)$; if $I_C$ exists then we call $C$ \emph{unital}.
	
	\begin{lemma}[Pasting lemma for cartesian cells] \label{pasting lemma for cartesian cells}
		If the cell $\psi$ in the composite on the right"/hand side above is cartesian then the composite $\psi \of \phi$ is cartesian if and only if $\phi$ is.
	\end{lemma}
	
	\begin{definition} \label{full and faithful morphism}
		A vertical morphism $\map fAC$ is called \emph{full and faithful} if its identity cell $\id_f$ is cartesian.
	\end{definition}
	
	\begin{definition} \label{augmented virtual equipment}
		An augmented virtual double category $\K$ is said to have \emph{restrictions on the left (resp.\ right)} if it has all unary restrictions of the form $K(f, \id)$ (resp.\ $K(\id, g)$). We call $\K$ a \emph{unital virtual double category} if it has all horizontal units (see \augsecref{10}). An \emph{augmented virtual equipment} is an augmented virtual double category that has all unary restrictions $K(f, g)$. A \emph{unital virtual equipment} is an unital virtual double category that has all restrictions $\ul K(f, g)$.
	\end{definition}
	
	\begin{definition} \label{companion and conjoint}
		Let $\map fAC$ be a vertical morphism in an augmented virtual double category. The nullary restriction $\hmap{C(f, \id)}AC$ is called the \emph{companion} of $f$ and denoted $f_*$. Likewise $\hmap{C(\id, f)}CA$ is called the \emph{conjoint} of $f$ and denoted $f^*$.
	\end{definition}
	Factorising the identity cell $\cell{\id_f}{(A)}{(C)}$ through the cartesian cell defining the companion $f_*$ we obtain a cocartesian cell, in the sense of \augdefref{7.1} (see also \defref{cocartesian path} below), as described by the following lemma which combines \augthreelemref{5.4}{5.9}{7.6} and \augcororef{8.3}. The identities below are called the \emph{companion identities}. A horizontal dual result similarly applies to the conjoint $f^*$.
	\begin{lemma} \label{companion identities lemma}
		Consider the factorisation of a vertical identity cell on the left below. The following conditions are equivalent: \textup{(a)} $\psi$ is cartesian; \textup{(b)} the identity on the right below holds; \textup{(c)} $\phi$ is weakly cocartesian; \textup{(d)} $\phi$ is cocartesian (\augdefref{7.1} or \defref{cocartesian path}). In that case $J$ is the companion of $f$.
	 	\begin{displaymath}
	  	\begin{tikzpicture}[textbaseline]
						\matrix(m)[math35]{A \\ C \\};
						\path[map]	(m-1-1) edge[bend right=45] node[left] {$f$} (m-2-1)
																edge[bend left=45] node[right] {$f$} (m-2-1);
						\path[transform canvas={xshift=-0.5em}]	(m-1-1) edge[cell] node[right] {$\id_f$} (m-2-1);
			\end{tikzpicture} \quad = \quad \begin{tikzpicture}[textbaseline]
    		\matrix(m)[math35, column sep={1.75em,between origins}]{& A & \\ A & & C \\ & C & \\};
    		\path[map]	(m-1-2) edge node[right] {$f$} (m-2-3)
    								(m-2-1) edge[barred] node[below, inner sep=2pt] {$J$} (m-2-3)
    												edge node[left] {$f$} (m-3-2);
    		\path				(m-1-2) edge[eq, transform canvas={xshift=-1pt}] (m-2-1)
    								(m-2-3) edge[eq, transform canvas={xshift=1pt}] (m-3-2);
    		\path				(m-1-2) edge[cell, transform canvas={yshift=-0.25em}] node[right, inner sep=2pt] {$\phi$} (m-2-2)
    								(m-2-2) edge[cell, transform canvas={yshift=0.1em}]	node[right, inner sep=2pt] {$\psi$} (m-3-2);
  		\end{tikzpicture} \qquad\qquad\quad \begin{tikzpicture}[textbaseline]
  			\matrix(m)[math35, column sep={1.75em,between origins}]{& A & & C \\ A & & C & \\};
  			\path[map]	(m-1-2) edge[barred] node[above] {$J$} (m-1-4)
  													edge node[above right, inner sep=0.5pt] {$f$} (m-2-3)
  									(m-2-1) edge[barred] node[below] {$J$} (m-2-3);
  			\path				(m-1-2) edge[eq, transform canvas={xshift=-1pt}] (m-2-1)
  									(m-1-4) edge[eq, transform canvas={xshift=1pt}] (m-2-3);
  			\path				(m-1-2) edge[cell, transform canvas={yshift=-0.25em}]	node[right, inner sep=2pt] {$\phi$} (m-2-2)
  									(m-1-3) edge[cell, transform canvas={yshift=0.25em}] node[right, inner sep=2pt] {$\psi$} (m-2-3);
  		\end{tikzpicture} \mspace{6mu} = \quad \begin{tikzpicture}[textbaseline]
  			\matrix(m)[math35]{A & C \\ A & C \\};
  			\path[map]	(m-1-1) edge[barred] node[above] {$J$} (m-1-2)
  									(m-2-1) edge[barred] node[below] {$J$} (m-2-2);
  			\path				(m-1-1) edge[eq] (m-2-1)
  									(m-1-2) edge[eq] (m-2-2);
  			\path[transform canvas={xshift=1.75em, xshift=-5.5pt}]	(m-1-1) edge[cell] node[right] {$\id_J$} (m-2-1);
  		\end{tikzpicture}
	  \end{displaymath}
	  
	  If $f = \id_A$ then each of the previous conditions is further equivalent to each of the following ones: \textup{(e)} $\psi$ is weakly cocartesian; \textup{(f)} $\psi$ is cocartesian; \textup{(g)} $\phi$ is cartesian. In that case $J$ is the horizontal unit of $A$.
	\end{lemma}
	
	The following is an immediate consequence of \propref{left Kan extensions preserved by restriction} below.
	\begin{corollary} \label{left Kan extensions with unital sources}
		In an augmented virtual double category that has restrictions on the right consider a cell $\eta$ as in the composite below, with $n \geq 1$ and where the object $A_n$ is unital. It is left Kan precisely if, for each vertical morphism $\map fB{A_n}$, the composite is left Kan.
		\begin{displaymath}
			\begin{tikzpicture}
				\matrix(m)[math35]{A_0 & A_1 & A_{n'} & B \\ A_0 & A_1 & A_{n'} & A_n \\};
				\draw	([yshift=-6.5em]$(m-1-1)!0.5!(m-1-4)$) node (M) {$M$};
				\path[map]	(m-1-1) edge[barred] node[above] {$J_1$} (m-1-2)
														
										(m-1-3) edge[barred] node[above, xshift=-3pt] {$J_n(\id, f)$} (m-1-4)
										(m-1-4) edge node[right] {$f$} (m-2-4)
										(m-2-4) edge[transform canvas={yshift=-2pt}] node[below right] {$l$} (M)
										(m-2-1) edge[barred] node[below, inner sep=2.5pt] {$J_1$} (m-2-2)
														edge[transform canvas={yshift=-2pt}] node[below left] {$d$} (M)
										(m-2-3) edge[barred] node[below, inner sep=2.5pt] {$J_n$} (m-2-4);
				\path				(m-1-1) edge[eq] (m-2-1)
										(m-1-2) edge[eq] (m-2-2)
										(m-1-3) edge[eq, transform canvas={xshift=-2pt}] (m-2-3);
				\path				($(m-2-2.south)!0.5!(m-2-3.south)$) edge[cell] node[right] {$\eta$} (M);
				\draw				($(m-1-2)!0.5!(m-2-3)$) node {$\dotsb$}
										($(m-1-3)!0.5!(m-2-4)$) node[font=\scriptsize] {$\cart$};
			\end{tikzpicture}
		\end{displaymath}
	\end{corollary}
	\begin{proof}
		For the `if'-part take $f = \id_{A_n}$ and use that $J_n(\id, \id) \iso J_n$. For the converse remember that $A_n$ being unital ensures that all conjoints \mbox{$\hmap{f^*}{A_n}B$} exist, by \augcororef{4.16}, and apply \propref{left Kan extensions preserved by restriction} below.
	\end{proof}
	
	\subsection{Pointwise left Kan extension}
	The previous result leads us to the following ``pointwise'' strengthening of the notion of (weak) left Kan extension.
	\begin{definition} \label{pointwise left Kan extension}
		Consider a path of horizontal morphisms $\hmap{\ul J}{A_0}{A_n}$ of length $n \geq 1$ as well as vertical morphisms $\map d{A_0}M$ and $\map fB{A_n}$. We say that the (weak) left Kan extension of $d$ along $\ul J$ (\defsref{weak left Kan extension}{left Kan extension}) \emph{restricts along $f$} if the restriction $J_n(\id, f)$ exists and, for any (weakly) left Kan cell $\eta$ of the form below, the composite below is again (weakly) left Kan. In that case we also say that $\eta$ \emph{restricts along $f$}.
		
		We call a (weakly) left Kan cell $\eta$ \emph{pointwise} if it restricts along any $\map fB{A_n}$ such that the restriction $J_n(\id, f)$ exists; in that case we say that $\eta$ defines $l$ as the \emph{pointwise} (weak) left Kan extension of $d$ along $(J_1, \dotsc, J_n)$. 
		\begin{displaymath}
			\begin{tikzpicture}
				\matrix(m)[math35]{A_0 & A_1 & A_{n'} & B \\ A_0 & A_1 & A_{n'} & A_n \\};
				\draw	([yshift=-6.5em]$(m-1-1)!0.5!(m-1-4)$) node (M) {$M$};
				\path[map]	(m-1-1) edge[barred] node[above] {$J_1$} (m-1-2)
														
										(m-1-3) edge[barred] node[above, xshift=-3pt] {$J_n(\id, f)$} (m-1-4)
										(m-1-4) edge node[right] {$f$} (m-2-4)
										(m-2-4) edge[transform canvas={yshift=-2pt}] node[below right] {$l$} (M)
										(m-2-1) edge[barred] node[below, inner sep=2.5pt] {$J_1$} (m-2-2)
														edge[transform canvas={yshift=-2pt}] node[below left] {$d$} (M)
										(m-2-3) edge[barred] node[below, inner sep=2.5pt] {$J_n$} (m-2-4);
				\path				(m-1-1) edge[eq] (m-2-1)
										(m-1-2) edge[eq] (m-2-2)
										(m-1-3) edge[eq, transform canvas={xshift=-2pt}] (m-2-3);
				\path				($(m-2-2.south)!0.5!(m-2-3.south)$) edge[cell] node[right] {$\eta$} (M);
				\draw				($(m-1-2)!0.5!(m-2-3)$) node {$\dotsb$}
										($(m-1-3)!0.5!(m-2-4)$) node[font=\scriptsize] {$\cart$};
			\end{tikzpicture}
		\end{displaymath}
	\end{definition}	
	Restrictions of pointwise (weak) left Kan extensions are again pointwise (weak) left Kan extensions as follows.
	\begin{lemma} \label{restrictions of pointwise left Kan extensions}
		Consider the composite above. If $\eta$ is pointwise (weakly) left Kan then so is the composite.
	\end{lemma}
	\begin{proof}
		A consequence of the pasting lemma for cartesian cells (\lemref{pasting lemma for cartesian cells}).
	\end{proof}
	
	\begin{remark} \label{pointwise left Kan extension in the presence of horizontal units and restrictions on the right}
		As a consequence of \cororef{left Kan extensions with unital sources} the notions of left Kan extension and pointwise left Kan extension coincide in any unital virtual double category that has all restrictions on the right (\defref{augmented virtual equipment}).
	\end{remark}
	
	\begin{remark}
		The notion of `Yoneda morphism' introduced in \defref{yoneda embedding} below formalises the classical notion of Yoneda embedding, and it is shown in \lemref{weak left Kan extensions of yoneda embeddings} that all four notions of left Kan extension along a Yoneda morphism coincide.
	\end{remark}
	
	The following is an easy consequence of \lemref{full and faithful functors reflect and preserve weakly left Kan cells}.
	\begin{lemma} \label{full and faithful cartesian cell preserving functors reflect pointwise left Kan cells}
		Any locally full and faithful functor $\map F\K\L$ that preserves cartesian cells defining restrictions on the right reflects pointwise (weakly) left Kan cells, that is a cell $\eta \in \K$ is pointwise (weakly) left Kan whenever its image $F\eta$ is so in $\L$.
	\end{lemma}
	
	\subsection{Examples}
	We will consider formal category theory in the augmented virtual double categories $\enProf\V$, $\enProf{(\V, \V')}$ and $\ensProf\V$ of (small) $\V$"/enriched profunctors (\augthreeexref{2.4}{2.7}{2.8}; their left Kan extensions were described in \exref{enriched left Kan extension}); in the unital virtual equipment $\inProf\E$ of internal profunctors in a category $\E$ (\augexsref{2.9}{4.9}); in the unital virtual equipment $\dFib{\mathcal C}$ of discrete two"/sided fibrations in a $2$"/category $\mathcal C$ (\exref{discrete two-sided fibrations} below); in the unital virtual equipment $\ModRel(\E)$ of internal modular relations in a category $\E$ (\exref{internal modular relations} below); and in the strict double category $\ClModRel$ of closed modular relations between closed-ordered closure spaces (\exref{closed modular relations} below). \tableref{relations between the notions of left Kan extension} shows the relations between the different notions of left Kan extension in each of these augmented virtual double categories. After considering left Kan extensions in $\ClModRel$, in \exref{continuous left Kan extensions}, \exref{left Kan extensions that are not pointwise} gives an example of a left Kan extension that is not pointwise.
	\begin{table}
		\centering
		{\renewcommand{\arraystretch}{1.1}
		\begin{tabular}{lll}
			$\K$ & $V(\K)$ & \textbf{Equivalence of notions of Kan extension} \\
			\hline \\[-0.8em]
			$\enProf{(\V, \V')}$ & $\enCat{\V'}$ & None in general \\
			& & \\[-0.9em]
			\hline \\[-0.8em]
			\rule[-1.2em]{0pt}{3em}\begin{tabular}{@{}l@{}} $\enProf\V$ \\ $\ensProf\V$ \\ $\ClModRel$ \end{tabular}
			& $\left.\rule[-1.2em]{0pt}{3em}\begin{tabular}{@{}l@{}} $\enCat\V$ \\ $\enCat\V$ \\ $\ClOrdCls$ \end{tabular}\right\}$
			& LK $\Leftrightarrow$ PLK \\
			& & \\[-0.9em]
			\hline \\[-0.8em]
			\rule[-1.2em]{0pt}{3em}\begin{tabular}{@{}l@{}} $\ModRel(\E)$ \\ $\inProf{\E}$ \\ $\dFib{\mathcal C}$ \end{tabular}
			& $\left.\rule[-1.2em]{0pt}{3em}\begin{tabular}{@{}l@{}} $\PreOrd(\E)$ \\ $\inCat\E$ \\ $\mathcal C$ \end{tabular}\right\}$
			& \hspace{-0.4cm} $\left\{ \rule[-1.2em]{0pt}{3em}\begin{tabular}{@{}l@{}} LK $\Leftrightarrow$ PLK $\Leftrightarrow$ PWLK \\ (PLK along $f_*$ in $\K$) = (PLK along $f$ in $V(\K)$)\end{tabular}\right.$ \\
		\end{tabular}}
		\caption{Augmented virtual double categories $\K$ grouped according to the equivalences in $\K$ of the notions of left Kan extension (LK) (\defref{left Kan extension}) and pointwise (weak) left Kan extension (P(W)LK) (\defref{pointwise left Kan extension}). The equivalence \mbox{LK $\Leftrightarrow$ PLK} follows from applying \remref{pointwise left Kan extension in the presence of horizontal units and restrictions on the right} to $\K$, while PLK $\Leftrightarrow$ PWLK follows from applying \thmref{pointwise Kan extensions in terms of pointwise weak Kan extensions} to $\K$. That pointwise left Kan extensions along companions $f_*$ in $\K$ coincide with pointwise left Kan extensions along $f$ in the vertical $2$"/category $V(\K)$, the latter in the usual $2$"/categorical sense of \cite{Street74b}, follows from applying \propref{pointwise left Kan extensions along companions} to $\K$.}
		\label{relations between the notions of left Kan extension}
	\end{table}
	\begin{example} \label{discrete two-sided fibrations}
		Let $\mathcal C$ be a finitely complete $2$"/category and consider the unital virtual equipment $\spFib{\mathcal C}$ of split two"/sided fibrations $\hmap JAB$ in $\mathcal C$ (\augexsref{2.11}{4.9}). Recall that we consider such $J$ to be internal profunctors $\hmap J{A^\2}{B^\2}$ in the underlying category $\und{\mathcal C}$, where $A^\2$ denotes the internal category that is the cotensor of $A$ with the arrow category $\2 = (0 \to 1)$; see \augexref{2.11} where $A^\2$ was denoted by $\Phi A$. The assignment $A \mapsto A^\2$ extends to a locally full embedding of $\spFib{\mathcal C}$ into the unital virtual equipment $\inProf{\und{\mathcal C}}$ of internal profunctors in $\und{\mathcal C}$ (\augexsref{2.9}{4.9}), which restricts to the identity on horizontal morphisms.
		
		A split two"/sided fibration $\hmap JAB$, with underlying span \mbox{$A \xlar{j_A} J \xrar{j_B} B$}, is called \emph{discrete} (\cite{Street80a}) if for any cell $\cell\phi u{\map vXJ}$ in $\mathcal C$ the following holds: if $j_A\of \phi$ and $j_B \of \phi$ are identity cells then so is $\phi$ (in particular $u = v$). We denote by $\dFib{\mathcal C} \subset \spFib{\mathcal C}$ the full sub"/augmented virtual double category generated by the discrete two"/sided fibrations in $\mathcal C$. In Proposition~2 of \cite{Street74b} it is shown that the assignment $A \mapsto A^\2$ extends to a locally full and faithful $2$"/functor \mbox{$\map{(\dash)^\2}{\mathcal C}{\inCat{\und{\mathcal C}} = V(\inProf{\und{\mathcal C}})}$}, with the correspondence between the cells given by the universal property of the $\2$"/cotensors. It follows that $V(\spFib{\mathcal C}) \iso \mathcal C$. Since $\dFib{\mathcal C} \subset \spFib{\mathcal C}$ is full with respect to vertical cells we likewise have \mbox{$V(\dFib{\mathcal C}) \iso \mathcal C$}.
		
		It is straightforward to check that the horizontal units $I_A = (A \leftarrow A^\2 \rightarrow A)$ of $\spFib{\mathcal C}$ are discrete two"/sided fibrations, and that any restriction $K(f, g)$ in $\spFib{\mathcal C}$ of a discrete two"/sided fibration $K$ is again discrete; for the latter recall that $K(f, g)$ is the iterated strict $2$"/pullback of $A \xrar f C \leftarrow K \rightarrow D \xlar g B$ (\augexref{4.9}) and use its $2$"/dimensional universal property. Since the full and faithful inclusion $\dFib{\mathcal C} \subset \spFib{\mathcal C}$ reflects cartesian cells (\auglemref{4.5}) we conclude that $\dFib{\mathcal C}$, like $\spFib{\mathcal C}$, is a unital virtual equipment (\defref{augmented virtual equipment}).
		
		Recall that $\spFib{\mathcal C}$ has all horizontal composites as soon as $\mathcal C$ has reflexive coequalisers preserved by pullback (see \augexref{7.5}). The horizontal composite $(J \hc H)$, in $\spFib{\mathcal C}$, of two discrete two"/sided fibrations $\hmap JAB$ and $\hmap HBE$, is not discrete in general however. Indeed in order for all horizontal composites in $\dFib{\mathcal C}$ to exist further conditions on $\mathcal C$ are needed; compare \cite{Carboni-Johnson-Street-Verity94} where discrete two"/sided fibrations in a finitely complete bicategory $\K$ are shown to form a bicategory $\mathsf{DFib}(\K)$ as soon as $\K$ is `faithfully conservational'.
	\end{example}
	
	\begin{example} \label{internal modular relations}
		Recall that a relation $J$ in a category $\E$ (\augexref{2.10}) is a span $A \xlar{j_0} J \xrar{j_1} B$ in $\E$ whose legs $j_0$ and $j_1$ are jointly monic. If $\E$ has pullbacks then relations in $\E$ form the horizontal morphisms of a unital virtual equipment $\Rel(\E)$ whose vertical morphisms are the morphisms of $\E$; see \augexsref{2.10}{5.8}. $\Rel(\E)$ is locally thin (\augexref{2.5}), that is its cells are uniquely determined by their boundaries.
		
		Consider the unital virtual double category $\ModRel(\E) \dfn (N \of \Mod)\bigpars{\Rel(\E)}$ of bimodules in $\Rel(\E)$, as defined in \augexsref{2.1}{2.2}. The objects of $\ModRel(\E)$ are \emph{internal preorders} in $\E$: they consist of objects $A$ of $\E$ equipped with a relation $\alpha = (A \xlar{\alpha_0} \alpha \xrar{\alpha_1} A)$ such that the (unique) horizontal multiplication and unit cells $\cell{\bar\alpha}{(\alpha, \alpha)}\alpha$ and $\cell{\tilde\alpha}A\alpha$ exist in $\Rel(\E)$; in \cite{Carboni-Street86} these are called the `ordered objects' of $\E$. The vertical morphisms of $\ModRel(\E)$ are the morphisms of $\E$ that preserve order, that is morphisms $\map fAC$ for which there exists a cell in $\Rel(\E)$ as on the left below, while its horizontal morphisms are \emph{internal modular relations} (`ideals' in \cite{Carboni-Street86}), that is relations $\hmap JAB$ in $\E$ such that there exist horizontal action cells $\cell\lambda{(\alpha, J)}J$ and $\cell\rho{(J, \beta)}J$ in $\Rel(\E)$. $\ModRel(\E)$ like $\Rel(\E)$ is locally thin. In particular a nullary cell of the form $\ul J \Rar (C, \gamma)$ exists in $\ModRel(\E)$ if and only if the underlying cell \mbox{$\ul J \Rar \gamma$} exists in $\Rel(\E)$. The locally thin vertical $2$"/category $\PreOrd(\E) \dfn V\bigpars{\ModRel(\E)}$ contained in $\ModRel(\E)$ coincides with the $2$"/category of ordered objects of \cite{Carboni-Street86}. Notice that $\Rel(\E)$ embeds into $\ModRel(\E)$ by equipping each object $A \in \E$ with the \emph{internal discrete ordering} $I_A = (A \xlar{\id} A \xrar{\id} A)$.
		
		$\ModRel(\E)$ is a unital virtual equipment by \augexsref{4.8}{5.8}: unary restrictions are created as in $\Span\E$ (\augexref{4.3}) while the nullary restrictions $C(f, g)$ of $\ModRel(\E)$ are created as the unary restrictions $\gamma(f, g)$ in $\Span\E$. Moreover the forgetful functor $\map U{\ModRel(\E)}{\Rel(\E)}$ creates pointwise composites (\augdefref{9.1}), so that $\ModRel(\E)$ is an equipment (\augpropref{7.8}) whenever $\Rel(\E)$ is, e.g.\ in the case that $\E$ is regular (\augexref{7.4}). Indeed consider a path $\hmap{\ul J}{A_0}{A_n}$ of internal modular relations in $\E$ and a horizontal cocartesian cell $\cell\phi{\ul J}K$ in $\Rel(\E)$, defining $K$ as the horizontal composite of $\ul J$. Then the actions $\cell\lambda{(\alpha_0, J_1)}{J_1}$ and $\cell\rho{(J_n, \alpha_n)}{J_n}$ induce actions of $\alpha_0$ and $\alpha_n$ on $K$, making $K$ into an internal modular relation and $\phi$ into a cell in $\ModRel(\E)$. That $\phi$ is pointwise cocartesian in $\ModRel(\E)$ follows from \auglemsref{9.4}{9.8}.
		\begin{displaymath}
			\begin{tikzpicture}[textbaseline]
				\matrix(m)[math35]{A & A \\ C & C \\};
				\path[map]	(m-1-1) edge[barred] node[above] {$\alpha$} (m-1-2)
														edge node[left] {$f$} (m-2-1)
										(m-1-2) edge node[right] {$f$} (m-2-2)
										(m-2-1) edge[barred] node[below] {$\gamma$} (m-2-2);
				\path[transform canvas={xshift=1.75em}]	(m-1-1) edge[cell] node[right] {$\bar f$} (m-2-1);
			\end{tikzpicture} \qquad\qquad\qquad\qquad\qquad \begin{tikzpicture}[textbaseline]
				\matrix(m)[math35, column sep=3em]{A & \alpha \\ \alpha & A \times A \\};
				\path[map]	(m-1-1) edge node[above] {$\tilde\alpha$} (m-1-2)
														edge node[left] {$\tilde\alpha$} (m-2-1)
										(m-1-2) edge node[right] {$(\alpha_1, \alpha_0)$} (m-2-2)
										(m-2-1) edge node[below] {$(\alpha_0, \alpha_1)$} (m-2-2);
			\end{tikzpicture}
		\end{displaymath}
		
		Next assume that $\E$ has all finite limits. An internal preorder $A$ in $\E$ is an \emph{internal partial order} (Examples~B2.3.8 of \cite{Johnstone02}) whenever the square on the right above is a pullback square in $\E$. Using the joint monicity of $\alpha_0$ and $\alpha_1$ one easily checks that the latter is equivalent to the following condition, which we will mostly use: for any parallel pair $\phi$, $\map\psi X\alpha$ of morphisms, $(\alpha_0, \alpha_1) \of \phi = (\alpha_1, \alpha_0) \of \psi$ implies $\phi = \psi$. Notice that the latter implies the following for a pair of parallel morphisms $f$ and $\map gXA$: if there exists a vertical isomorphism $f \iso g$ in $\ModRel(\E)$ then $f = g$.
	\end{example}
	
	\begin{example} \label{closed modular relations}
		By a \emph{closure space} $A = (A, \Cl A)$ we will mean a set $A$ equipped with a set $\Cl A$ of closed subsets of $A$, such that $A \in \Cl A$ and $\Cl A$ is closed under arbitrary intersections. As introduced by Tholen in \cite{Tholen09}, a \emph{closed"/ordered closure space} $A = (A, \Cl A, \leq)$ is a closure space $(A, \Cl A)$ equipped with a preordering $\leq$ satisfying the closedness axiom (`preservation condition' in \cite{Tholen09})
		\begin{itemize}
			\item[(C)] $V \in \Cl A \quad \Rightarrow \quad \upset V \in \Cl A$,
		\end{itemize}
		where $\upset V \dfn \set{x \in A \mid \exists v \in V \colon v \leq x}$ is the upset generated by $V$. A morphism $\map fAC$ of closed"/ordered closure spaces is an order preserving continuous map. Given a relation $\hmap JAB$ between sets, that is a subset $J \subseteq A \times B$, we will abbreviate $(x, y) \in J$ by $xJy$ and write $JS \dfn \set{y \in B \mid \exists s \in S \colon sJy}$ for the image of a subset $S \subseteq A$ under $J$. A \emph{closed modular relation} $\hmap JAB$ between closed"/ordered closure spaces is a relation that satisfies the modularity and closedness axioms
		\begin{itemize}
			\item[(M)] $x_1 \leq x_2, \quad x_2 J y_1, \quad y_1 \leq y_2 \quad \Rightarrow \quad x_1 J y_2$,
			\item[(C)] $V \in \Cl A \quad \Rightarrow \quad JV \in \Cl B$.
		\end{itemize}
		Closedness is equivalent to the reverse relation $\hmap{\dl J}BA$ being \emph{upper hemi"/continuous}, see e.g.\ Section~17.2 of \cite{Aliprantis-Border06}; see also Section~6 of \cite{Koudenburg18}.
		
		Closed modular relations and morphisms of closed"/ordered closure spaces form a locally thin (\augexref{2.5}) strict double category (\augpropref{7.8}) $\ClModRel$ in which a cell of the form below exists, and is unique, if and only if $xJy$ implies $(fx)K(gy)$ for all $x \in A$ and $y \in B$. The composite $J \hc H$ of closed modular relations $\hmap JAB$ and $\hmap HBE$ is defined as usual: $x(J \hc H)z$ if and only if there exists $y \in B$ with $xJy$ and $yHz$. The horizontal unit $\hmap{I_A}AA$ is the order relation $xI_Ay \defeq x \leq y$, which is closed because $I_A S = \upset S$ for all $S \subseteq A$.
		\begin{displaymath}
			\begin{tikzpicture}
				\matrix(m)[math35]{A & B \\ C & D \\};
				\path[map]	(m-1-1) edge[barred] node[above] {$J$} (m-1-2)
														edge node[left] {$f$} (m-2-1)
										(m-1-2) edge node[right] {$g$} (m-2-2)
										(m-2-1) edge[barred] node[below] {$K$} (m-2-2);
				\draw				($(m-1-1)!0.5!(m-2-2)$) node[rotate=-90] {$\subseteq$};
			\end{tikzpicture}
		\end{displaymath}
		Consider the locally thin strict equipment $\ModRel \dfn \enProf\2$ (\augexref{2.5}) of modular relations between preorders which, by definition, are profunctors enriched in the quantale $\2 \dfn (\bot \leq \top)$ of truth values (see Example~1.9 of \cite{Koudenburg18}). Notice that the forgetful functor $\map U\ClModRel\ModRel$ creates all restrictions $K(\id, g)$ on the right (\auglemref{4.5}) but that restrictions $K(f, \id)$ on the left do not exist in $\ClModRel$ in general. In particular the companion $f_*$ of a morphism $\map fAC$ exists if and only if $\upset fV$ is closed in $C$ for every $V \in \Cl A$.
		
		\emph{Modular closure spaces}, introduced by Tholen in \cite{Tholen09}, are closed"/ordered closure spaces $A = (A, \Cl A, \leq)$ satisfying the modularity axiom
		\begin{itemize}
			\item[(M)] $V \in \Cl A \quad \Rightarrow \quad \upset V = V$,
		\end{itemize}
		which strengthens the closedness axiom above. This notion coincides with that of \emph{modular $(P,\2)$"/categories} in the sense of Section~4 of \cite{Koudenburg18}, where $P$ is the powerset monad; see Example~3.2 of \cite{Koudenburg18}. We denote by \mbox{$\ClModRel_\textup m \subset \ClModRel$} the full sub"/double category generated by modular closure spaces.
	\end{example}
	
	\begin{example} \label{continuous left Kan extensions}
			Consider morphisms $M \xlar d A \xbrar J B$ in the locally thin strict double category $\ClModRel_\textup m$ of closed modular relations between modular closure spaces and assume that for each $y \in B$ the maximum on the right"/hand side below exists in $M$; for sufficient conditions see Theorem~8.1 of \cite{Koudenburg18}. These maxima combine to form an order preserving map $\map lBM$ given by
		\begin{displaymath}
			ly = \max_{x \in \dl Jy} dx
		\end{displaymath}
		which forms the left Kan extension of $d$ along $J$ in the locally thin strict equipment $\ModRel$ of modular relations (\exref{closed modular relations}). In fact, as a left Kan extension $l$ satisfies the \emph{left Beck"/Chevalley} condition in $\ModRel$, in the sense of Definition~2.4 of \cite{Koudenburg18}; see Example~2.7 of \cite{Koudenburg18} or \exref{left Beck-Chevalley condition in ModRel} below.
		
		Regarding $d$ and $J$ as morphisms of modular $(P, \2)$"/categories (\exref{closed modular relations}), the Beck"/Chevalley condition for $l$, closedness of $J$ and the fact that $P$ preserves composites of modular relations allows us to apply Theorem~7.9 of \cite{Koudenburg18} to $\Mod(P)$ (see Section~4 of \cite{Koudenburg18}), which asserts that $l$ is continuous and thus a morphism in $\ClModRel_\textup m$. The latter result partly generalises the ``maximum theorem'', a classical result in analysis; see e.g.\ Lemma~17.30 of \cite{Aliprantis-Border06}. Applying \lemref{full and faithful functors reflect and preserve weakly left Kan cells} to the locally full and faithful functors \mbox{$\ClModRel_\textup m \hookrightarrow \ClModRel \xrar U \ModRel$} we conclude that $l$ forms the left Kan extension of $d$ along $J$ both in $\ClModRel_\textup m$ and in $\ClModRel$. Since the latter are strict double categories that have restrictions on the right $l$ is in fact a pointwise left Kan extension by \remref{pointwise left Kan extension in the presence of horizontal units and restrictions on the right}.
	\end{example}
	
	\begin{example} \label{left Kan extensions that are not pointwise}
		In order to construct a Kan extension (\defref{left Kan extension}) that fails to be a pointwise weak Kan extension (\defref{pointwise left Kan extension}) consider the locally thin equipment $\enProf{\brks{-\infty, \infty}}$ of categories and profunctors enriched in the quantale $\brks{-\infty, \infty}$ of extended real numbers, with reversed order $\geq$ and addition $(+,0)$ as monoid structure; see Examples~1.3 and~1.8 of \cite{Koudenburg18}. As is customary (\cite{Lawvere73}) we think of $\brks{-\infty, \infty}$"/categories as generalised metric spaces with distances in $\brks{-\infty, \infty}$. Let $\K$ denote the full sub"/augmented virtual double category of $\enProf{\brks{-\infty, \infty}}$ generated by those $\brks{-\infty, \infty}$"/profunctors $\hmap JAB$ with images $J(x, y) \geq 0$ for all $x \in A$ and $y \in B$. Notice that $\K$ is closed under horizontal composition of $\brks{-\infty, \infty}$"/profunctors (Example~1.3 of \cite{Koudenburg18}) so that $\K$ has all horizontal composites by \auglemref{9.4}.
		
		Let $M$ be the generalised metric space with two points $d$ and $u$ as pictured on the left below and let $I$ be the `unit $\brks{-\infty, \infty}$"/category', consisting of a single point $*$ with $I(*, *) = 0$. We claim that \emph{any} cell $\eta$ in $\K$ of the form as in the middle below is left Kan. To see this first notice that, by \exref{left Kan extensions when all composites exist} and using that $\K$ is locally thin, it suffices to show that all cells of the form as on the right below exist in $\K$, where $\ul H$ is any path of $\brks{-\infty, \infty}$"/profunctors of length $m \leq 1$ in $\K$ and $\map k{B_m}M$ is any morphism. It is not hard to show that the latter follows from the fact that $k$ and $\ul H$ are non"/expanding (Examples~1.10 and~1.11 of \cite{Koudenburg18}).
		\begin{displaymath}
			\begin{tikzpicture}[textbaseline]
				\matrix(m)[math35, column sep=2.5em]{ & \\};
				\path[map]	(m-1-1) edge[bend left=40] node[above] {$-1$} (m-1-2)
										(m-1-2)	edge[bend left=40] node[below] {$1$} (m-1-1);
				\draw[map, transform canvas={yshift=3pt}]	(m-1-1) arc(-20:280:0.15) node[above, inner sep=10pt] {$0$};
				\draw[map, transform canvas={yshift=3pt}]	(m-1-2) arc(200:-100:0.15) node[above, inner sep=10pt] {$0$};
				\fill	(m-1-1) circle (1pt) node[below left] {$d$}
							(m-1-2) circle (1pt) node[below right, xshift=-2pt] {$u$};
			\end{tikzpicture} \qquad\qquad\qquad\qquad \begin{tikzpicture}[textbaseline]
				\matrix(m)[math35, column sep={1.75em,between origins}]{I & & M \\ & M & \\};
				\path[map]	(m-1-1) edge[barred] node[above] {$J$} (m-1-3)
														edge[transform canvas={xshift=-2pt}] node[left] {$d$} (m-2-2);
				\path[map, transform canvas={yshift=0.25em}]	(m-1-2) edge[cell] node[right] {$\eta$} (m-2-2);
				\path				(m-1-3) edge[eq, transform canvas={xshift=2pt}] (m-2-2);
			\end{tikzpicture} \qquad\qquad\qquad\qquad \begin{tikzpicture}[textbaseline]
				\matrix(m)[math35, column sep={1.75em,between origins}]{M & & B_m \\ & M & \\};
				\path[map]	(m-1-1) edge[barred] node[above] {$\ul H$} (m-1-3)
										(m-1-3)	edge[transform canvas={xshift=2pt}] node[right] {$k$} (m-2-2);
				\path[map, transform canvas={yshift=0.25em}]	(m-1-2) edge[cell] node[right] {$\phi'$} (m-2-2);
				\path				(m-1-1) edge[eq, transform canvas={xshift=-2pt}] (m-2-2);
			\end{tikzpicture}
		\end{displaymath}
		
		For a left Kan cell in $\K$ that is not pointwise weakly left Kan take \mbox{$\hmap JIM$} to be given by $J(*, d) = 1$ and $J(*, u) = 0$. Then $J$ is non"/expanding and the cell $\eta$ of the form above exists. By the previous $\eta$ is left Kan; to arrive at a contradiction let us assume that $\eta$ is pointwise weakly left Kan as well. Let $\map uIM$ denote the morphism that picks out $u$ in $M$ and consider the composition $J(\id, u) \Rar M$ of $\eta$ and the cartesian cell defining the restriction $J(\id, u)$. By the assumption and  \defref{pointwise left Kan extension} the latter defines $u$ as the weak left Kan extension of $d$ along $\hmap{J(\id, u)}II$. By the definition of $J$ the latter equals the horizontal unit of $I$ so that, by \exref{Kan extensions along horizontal units}, $\map{u \iso d}IM$ follows. But that is impossible as there exists no vertical cell $d \Rar u$ in $\K$ (nor in $\enProf{\brks{-\infty, \infty}}$).
	\end{example}
	
	\subsection{Internal left adjoints are cocontinuous}
	The two remaining results of this section describe the interaction between adjunctions and left Kan extensions. The first of these, \propref{left adjoints are cocontinuous} below, shows that left adjoints in an augmented virtual double category are cocontinuous in the sense of the definition below. Analogous results, for left adjoints in $2$-categories (see e.g.\ Proposition~2.19(1) of \cite{Weber07}) and enriched left adjoints (see e.g.\ Section 4.1 of \cite{Kelly82}), are well known. The second result, \propref{taking adjuncts preserves left Kan cells}, instead considers `external' adjunctions $\map{F \ladj G}\L\K$, between augmented virtual double categories and $\K$ and $\L$, and shows that the `adjunct' of a left Kan cell in $\L$ is again left Kan in $\K$.
	\begin{definition} \label{absolutely left Kan}
		Consider a (pointwise) (weakly) left Kan cell $\eta$ of the form below.
		\begin{displaymath}
			\begin{tikzpicture}
				\matrix(m)[math35]{A_0 & A_1 & A_{n'} & A_n \\};
				\draw	([yshift=-3.25em]$(m-1-1)!0.5!(m-1-4)$) node (M) {$M$};
				\path[map]	(m-1-1) edge[barred] node[above] {$J_1$} (m-1-2)
														edge[transform canvas={yshift=-2pt}] node[below left] {$d$} (M)
										(m-1-3) edge[barred] node[above] {$J_n$} (m-1-4)
										(m-1-4) edge[transform canvas={yshift=-2pt}] node[below right] {$l$} (M);
				\path				($(m-1-2.south)!0.5!(m-1-3.south)$) edge[cell] node[right] {$\eta$} (M);
				\draw				($(m-1-2)!0.5!(m-1-3)$) node {$\dotsb$};
			\end{tikzpicture}
		\end{displaymath}
		\begin{enumerate}[label=\textup{(\alph*)}]
			\item A morphism $\map fMN$ is said to \emph{preserve} $\eta$, and to preserve the (pointwise) (weak) left Kan extension of $d$ along $(J_1, \dotsc, J_n)$, if the composite $f \of \eta$ is again (pointwise) (weakly) left Kan.
			\item The cell $\eta$ is called \emph{absolutely} (pointwise) (weakly) left Kan if it is preserved by all morphisms $\map fMN$; in that case the (pointwise) (weak) left Kan extension of $d$ along $(J_1, \dotsc, J_n)$ is called \emph{absolute}.
			\item A morphism $\map fMN$ is called \emph{(weakly) cocontinuous} if it preserves any (weakly) left Kan cell $\zeta$ with horizontal target $M$.
		\end{enumerate}
	\end{definition}
	Notice that (weakly) cocontinuous morphisms preserve pointwise (weakly) left Kan cells as well. In \thmref{left Beck-Chevalley condition and absolute left Kan extensions} below absolutely left Kan cells are characterised in terms of a `left Beck"/Chevalley condition' (\defref{left Beck-Chevalley condition}).
	
	By a left adjoint in an augmented virtual double category $\K$ we mean a left adjoint in the vertical $2$"/category $V(\K)$ (\augexref{1.5}), in the usual $2$"/categorical sense; see e.g.\ \auglemref{5.16}.
\begin{proposition} \label{left adjoints are cocontinuous}
	Left adjoints are both cocontinuous and weakly cocontinuous.
\end{proposition}
	\begin{proof}
		Let $\map fMN$ be left adjoint to $\map gNM$, with unit and counit vertical cells $\cell\iota{\id_M}{g\of f}$ and $\cell\eps{f \of g}{\id_N}$. To show that $f$ is cocontinuous consider any left Kan cell
		\begin{displaymath}
			\begin{tikzpicture}
				\matrix(m)[math35]{A_0 & A_1 & A_{n'} & A_n \\};
				\draw	([yshift=-3.25em]$(m-1-1)!0.5!(m-1-4)$) node (M) {$M$};
				\path[map]	(m-1-1) edge[barred] node[above] {$J_1$} (m-1-2)
														edge[transform canvas={yshift=-2pt}] node[below left] {$d$} (M)
										(m-1-3) edge[barred] node[above] {$J_n$} (m-1-4)
										(m-1-4) edge[transform canvas={yshift=-2pt}] node[below right] {$l$} (M);
				\path				($(m-1-2.south)!0.5!(m-1-3.south)$) edge[cell] node[right] {$\eta$} (M);
				\draw				($(m-1-2)!0.5!(m-1-3)$) node {$\dotsb$};
				\draw				(M) node[xshift=7pt, yshift=-2pt] {;};
			\end{tikzpicture}
		\end{displaymath}
		we have to show that $f \of \eta$ is again left Kan. To this end consider the commuting diagram of assignments, between collections of cells of the forms as shown, below. The vertically drawn assignments are bijections. Indeed, the triangle identities for $\iota$ and $\eps$ imply that their inverses are given by composition with $\eps$ on the right. The bottom assignment is a bijection as well, because $\eta$ is left Kan. We conclude that the top assignment is a bijection too, showing that $f \of \eta$ is left Kan as required.
		\begin{displaymath}
			\begin{tikzpicture}
				\matrix(m)[math35]{A_n & B_1 & B_{m'} & B_m \\};
				\draw	([yshift=-3.25em]$(m-1-1)!0.5!(m-1-4)$) node (M) {$N$};
				\path[map]	(m-1-1) edge[barred] node[above] {$H_1$} (m-1-2)
														edge[transform canvas={yshift=-2pt}] node[below left] {$f \of l$} (M)
										(m-1-3) edge[barred] node[above] {$H_m$} (m-1-4)
										(m-1-4) edge[transform canvas={yshift=-2pt}] node[below right] {$k$} (M);
				\path				($(m-1-2.south)!0.5!(m-1-3.south)$) edge[cell] (M);
				\draw				($(m-1-2)!0.5!(m-1-3)$) node {$\dotsb$};
				
				\matrix(m)[math35, column sep={0.7em}, xshift=21em]
					{	A_0 & A_1 & A_{n'} & A_n & B_1 & B_{m'} & B_m \\ };
				\draw	([yshift=-3.25em]$(m-1-1)!0.5!(m-1-7)$) node (M) {$N$};
				
				\path[map]	(m-1-1) edge[barred] node[above] {$J_1$} (m-1-2)
														edge[transform canvas={yshift=-2pt}] node[below left] {$f \of d$} (M)
										(m-1-3) edge[barred] node[above] {$J_m$} (m-1-4)
										(m-1-4) edge[barred] node[above] {$H_1$} (m-1-5)
										(m-1-6) edge[barred] node[above] {$H_m$} (m-1-7)
										(m-1-7) edge[transform canvas={yshift=-2pt}] node[below right] {$k$} (M);
				\draw[transform canvas={xshift=-0.5pt}]	($(m-1-2)!0.5!(m-1-3)$) node {$\dotsb$}
										($(m-1-5)!0.5!(m-1-6)$) node {$\dotsb$};
				\path				($(m-1-1.south)!0.5!(m-1-7.south)$) edge[cell] (M);
				
				\matrix(m)[math35, yshift=-8.5em]{A_n & B_1 & B_{m'} & B_m \\};
				\draw	([yshift=-3.25em]$(m-1-1)!0.5!(m-1-4)$) node (M) {$M$};
				\path[map]	(m-1-1) edge[barred] node[above] {$H_1$} (m-1-2)
														edge[transform canvas={yshift=-2pt}] node[below left] {$l$} (M)
										(m-1-3) edge[barred] node[above] {$H_m$} (m-1-4)
										(m-1-4) edge[transform canvas={yshift=-2pt}] node[below right] {$g \of k$} (M);
				\path				($(m-1-2.south)!0.5!(m-1-3.south)$) edge[cell] (M);
				\draw				($(m-1-2)!0.5!(m-1-3)$) node {$\dotsb$};
				
				\matrix(m)[math35, column sep={0.7em}, yshift=-8.5em, xshift=21em]
					{	A_0 & A_1 & A_{n'} & A_n & B_1 & B_{m'} & B_m \\ };
				\draw	([yshift=-3.25em]$(m-1-1)!0.5!(m-1-7)$) node (M) {$M$};
				
				\path[map]	(m-1-1) edge[barred] node[above] {$J_1$} (m-1-2)
														edge[transform canvas={yshift=-2pt}] node[below left] {$d$} (M)
										(m-1-3) edge[barred] node[above] {$J_m$} (m-1-4)
										(m-1-4) edge[barred] node[above] {$H_1$} (m-1-5)
										(m-1-6) edge[barred] node[above] {$H_m$} (m-1-7)
										(m-1-7) edge[transform canvas={yshift=-2pt}] node[below right] {$g \of k$} (M);
				\draw[transform canvas={xshift=-0.5pt}]	($(m-1-2)!0.5!(m-1-3)$) node {$\dotsb$}
										($(m-1-5)!0.5!(m-1-6)$) node {$\dotsb$};
				\path				($(m-1-1.south)!0.5!(m-1-7.south)$) edge[cell] (M);
				
				\draw[font=\Large]	(-6.25em,-1.5em) node {$\lbrace$}
										(6.25em,-1.5em) node {$\rbrace$}
										(12.0em,-1.5em) node {$\lbrace$}
										(30.0em,-1.5em) node {$\rbrace$}
										(-6.25em,-10em) node {$\lbrace$}
										(6.25em,-10em) node {$\rbrace$}
										(12.0em,-10em) node {$\lbrace$}
										(30.0em,-10em) node {$\rbrace$};
				
				\path[map]	(7.75em,-1.5em) edge node[above] {$(f \of \eta) \hc \dash$} (10.5em, -1.5em)
										(6.75em,-9em) edge node[below] {$\eta \hc \dash$} (9.5em, -9em)
										(0em, -4.25em) edge node[left] {$(\iota \of l) \hc (g \of \dash)$} (0em,-5.75em)
										(19em, -4.25em) edge node[right] {$(\iota \of d) \hc (g \of \dash)$} (19em,-5.75em);
			\end{tikzpicture}
		\end{displaymath}
		Restricting the previous to the empty path $\ul H = (A_n)$ shows that $f$ is weakly cocontinuous too.
	\end{proof}
	
	\subsection{External adjunctions and left Kan extension}
	Consider an adjunction between augmented virtual double categories, that is an adjunction $\map{F \ladj G}\L\K$ in the $2$"/category $\AugVirtDblCat$ of augmented virtual double categories, their functors and the transformations between them; see \augsecref 3. \propref{taking adjuncts preserves left Kan cells} below asserts that a cell $\cell\phi{\ul J}{GM}$ is left Kan in $\K$ whenever its `adjunct' $\cell{\phi^\flat}{F\ul J}M$ is left Kan in $\L$. More generally it applies to pairs of cells that are adjunct with respect to a `locally universal morphism', as we will now define. The latter is related to (the vertical dual of) the notion of `universal 2"/cell' considered in Proposition~8.6 of \cite{Shulman08}, for functors between pseudo double categories. \propref{taking adjuncts preserves left Kan cells} and the notion of `relative universal morphism', also defined below, will be crucial in \secref{yoneda embeddings in a monoidal augmented virtual double category}.
	
	Given a functor $\map F\K\L$ between augmented virtual double categories (\augdefref{3.1}) and an object $X$ in $\L$ we define the \emph{vertical slice category} $F \vs X$ as follows. Its objects are pairs $(A, f)$ consisting of an object $A$ in $\K$ and a morphism $\map f{FA}X$ in $\L$. Its morphisms $\map{(\ul H,\phi)}{(A, f)}{(C, g)}$ are pairs consisting of a path $\hmap{\ul H}AC$ of horizontal morphisms in $\K$, of any length, and a nullary cell $\phi$ in $\L$ that is of the form below. Composition in $F \vs X$ is given by horizontal composition in $\L$; notice that invertible morphisms $(\ul H, \phi)$ of $F \vs X$ necessarily have $\ul H$ empty and $\phi$ vertical. We abbreviate $\L \vs X \dfn \id_\L \vs X$. In \secref{yoneda embeddings section} we shall consider `horizontal slice categories'; see \propref{equivalence from yoneda embedding} below.
	\begin{displaymath}
		\begin{tikzpicture}
				\matrix(m)[math35, column sep={1.75em,between origins}]{FA & & FC \\ & X & \\};
				\path[map]	(m-1-1) edge[barred] node[above] {$F\ul H$} (m-1-3)
														edge[transform canvas={xshift=-2pt}] node[left] {$f$} (m-2-2)
										(m-1-3) edge[transform canvas={xshift=2pt}] node[right] {$g$} (m-2-2);
				\path				(m-1-2) edge[cell, transform canvas={yshift=0.25em}] node[right] {$\phi$} (m-2-2);
			\end{tikzpicture}
	\end{displaymath}
	\begin{definition} \label{universal vertical morphism}
		Let $\map F\K\L$ be a functor of augmented virtual double categories (\augdefref{3.1}) and $C \in \L$ an object. We call a morphism $\map\eps{FC'}C$ \emph{locally universal from $F$ to $C$} if the functor
		\begin{displaymath}
			\map{\eps \of F\dash}{\K \vs C'}{F \vs C}
		\end{displaymath}
		is full and faithful; if it is an equivalence then we call $\eps$ \emph{universal from $F$ to $C$}.
		
		Let $\mathcal J \subseteq F \vs C$ be a full subcategory. A locally universal morphism $\map\eps{FC'}C$ is called \emph{universal relative to $\mathcal J$} if the full and faithful functor $\eps \of F\dash$ above factors through the inclusion $\mathcal J \hookrightarrow F \vs C$ as an equivalence $\K \vs C' \xrar\simeq \mathcal J$.
	\end{definition}
	
	Unpacking the definition of a locally universal morphism $\map\eps{FC'}C$, notice that any nullary cell $\phi$ as on the left"/hand side below factors through $\eps$ as the $F$"/image of a unique nullary cell $\cell{\shad\phi}{\ul J}{C'}$, as shown. The cells $\phi$ and $\shad\phi$ are said to be \emph{adjuncts} of each other. If moreover $\eps$ is universal relative to $\mathcal J$ then for any morphism $\map h{FA}C$ in $\mathcal J$ we can choose an \emph{adjunct} $\map{\shad h}A{C'}$ in $\K$ as well, such that $h \iso \eps \of F\shad h$ in $\L$.
		\begin{displaymath}
			\begin{tikzpicture}[textbaseline]
				\matrix(m)[math35, column sep={0.875em,between origins}]{FA_0 & & & & FA_n \\ & FC' \mspace{7mu} & & \mspace{7mu} FC' & \\ & & C & & \\};
				\path[map]	(m-1-1) edge[barred] node[above] {$F\ul J$} (m-1-5)
														edge[transform canvas={xshift=-1pt}] node[left] {$Ff$} (m-2-2)
										(m-1-5) edge[transform canvas={xshift=1pt}] node[right] {$Fg$} (m-2-4)
										(m-2-2) edge[transform canvas={xshift=-1pt}, ps] node[left] {$\eps$} (m-3-3)
										(m-2-4) edge[transform canvas={xshift=1pt}, ps] node[right] {$\eps$} (m-3-3);
				\path[transform canvas={yshift=-0.3em}]				(m-1-3) edge[cell] node[right] {$\phi$} (m-2-3);
			\end{tikzpicture} \qquad = \qquad \begin{tikzpicture}[textbaseline]
				\matrix(m)[math35, column sep={1.75em,between origins}]{FA_0 & & FA_n \\ & FC' & \\ & C & \\};
				\path[map]	(m-1-1) edge[barred] node[above] {$F\ul J$} (m-1-3)
														edge[transform canvas={xshift=-2pt}] node[left] {$Ff$} (m-2-2)
										(m-1-3) edge[transform canvas={xshift=2pt}] node[right] {$Fg$} (m-2-2)
										(m-2-2) edge node[right] {$\eps$} (m-3-2);
				\path[transform canvas={xshift=-0.6em, yshift=0.25em}]	(m-1-2) edge[cell] node[right, inner sep=2.5pt] {$F\shad \phi$} (m-2-2);
			\end{tikzpicture}
		\end{displaymath}
	Notice that the uniqueness of the adjuncts $\phi^\sharp$ implies that the assignment $\phi \mapsto \phi^\sharp$ is functorial with respect to horizontal composition of nullary cells, i.e.\ $(\phi \hc \psi)^\sharp = \phi^\sharp \hc \psi^\sharp$. Also notice that any morphism universal relative to $\mathcal J$ is itself contained in $\mathcal J$. It follows that, for any two morphisms $\map\eps{FC'}C$ and $\map\zeta{FC''}C$ that are universal relative to $\mathcal J$, the objects $C'$ and $C''$ are equivalent in the vertical $2$"/category $V(\K)$ (\augexref{1.5}), with the equivalence being the chosen morphism $\map{\shad\zeta}{C''}{C'}$ such that $\eps \of F\shad\zeta \iso \zeta$.
	
	\begin{example} \label{full and faithful morphisms are locally universal}
		If $\map F\K\L$ is locally full and faithful (\augdefref{3.6}) then any full and faithful morphism $\map\eps{FC'}C$ in $\L$ (\defref{full and faithful morphism}) is locally universal from $F$ to $C$.
	\end{example}
	
	\begin{example} \label{counit components are universal}
		Let $\map{F \ladj G}\L\K$ be an adjunction between augmented virtual double categories, with unit $\cell\zeta{\id_\K}{GF}$ and counit $\cell\eps{FG}{\id_\L}$ transformations (\augdefref{3.2}). For any object $C \in \L$ the vertical morphism component $\map{\eps_C}{FGC}C$ of the counit is a universal morphism from $F$ to $C$. Indeed the triangle identities for $\zeta$ and $\eps$ imply that $\map{\eps_C \of F\dash}{\K \vs GC}{F \vs C}$ has an inverse $\shad{(\dash)}$ given by $\shad h \dfn Gh \of \zeta_A$ on objects $\map h{FA}C$ and $\phi^\sharp \dfn G\phi \of \zeta_{\ul J}$ on morphisms $\cell\phi{F\ul J}C$, where $\zeta_{\ul J} \dfn (\zeta_{J_1}, \dotsc, \zeta_{J_n})$ if $\ul J = (J_1, \dotsc, J_n)$ and $\zeta_{\ul J} \dfn \zeta_{A_0}$ if $\ul J = (A_0)$. The notion of adjunction between augmented virtual double categories is used in \defref{closed monoidal augmented virtual double categories} below in order to define closed monoidal augmented virtual double categories; see also the subsequent examples.
	\end{example}
	
	\begin{proposition} \label{taking adjuncts preserves left Kan cells}
		Consider the functor $\map F\K\L$, the locally universal morphism $\map\eps{FC'}C$ and the adjuncts $\phi$ and $\shad\phi$ above. If $\cell\phi{F\ul J}C$ is (weakly) left Kan in $\L$ then so is $\cell{\phi^\sharp}{\ul J}{C'}$ in $\K$. If moreover $\phi$ restricts (\defref{pointwise left Kan extension}) along all $F$"/images $Fk$ of morphisms $\map kB{A_n}$ for which the restriction $J_n(\id, k)$ exists, and $F$ preserves such restrictions, then $\phi^\sharp$ is pointwise (weakly) left Kan.
	\end{proposition}
	\begin{proof}
		To prove that $\phi^\sharp$ is left Kan whenever $\phi$ is so we have to show that any nullary cell $\cell\psi{\ul J \conc \ul H}{C'}$ in $\K$, with vertical source $f$, factors uniquely as \mbox{$\psi = \phi^\sharp \hc \psi'$} with $\cell{\psi'}{\ul H}{C'}$, as in \defref{left Kan extension}. Since $\phi = \eps \of F\phi^\sharp$ is left Kan there exists a unique cell $\cell\theta{F\ul H}C$ in $\L$ such that $\eps \of F\psi = \phi \hc \theta$, and we claim that $\psi' \dfn \theta^\sharp$ is the unique factorisation that we seek. Indeed $\psi = (\eps \of F\psi)^\sharp = (\phi \hc \theta)^\sharp = \phi^\sharp \hc \psi'$ and, to show that $\psi'$ is unique, consider any other cell $\psi''$ with $\psi = \phi^\sharp \hc \psi''$. Applying $\eps \of F\dash$ gives $\eps \of F\psi = \phi \hc (\eps \of F\psi'')$, where we use that $F$ (like any functor of augmented virtual double categories) preserves horizontal composition of cells (see \augdefref{3.1}). By the uniqueness of factorisations through the left Kan cell $\phi$ it follows that $\eps \of F\psi'' = \theta$, and hence $\psi'' = (\eps \of F\psi'')^\sharp = \theta^\sharp = \psi'$.
		
		When restricted to the empty path $\ul H = (A_n)$ the argument above reduces to the weakly left Kan case. To prove the pointwise (weakly) left Kan case apply the above to composites of the form $\phi^\sharp \of (\id, \dotsc, \id, \cart)$, where $\cart$ is any cartesian cell that defines a right restriction $J_n(\id, k)$ of $J_n$ along some $\map kB{A_n}$, as in \defref{pointwise left Kan extension}, using the fact that $\phi \of (\id, \dotsc, \id, F\cart) = \eps \of F\bigpars{\phi^\sharp \of (\id, \dotsc, \id, \cart)}$ is (weakly) left Kan if $\phi$ restricts along $Ff$ and $F$ preserves the cartesian cell.
	\end{proof}
	
	\begin{example}
		Let $F \ladj G\colon \D \to \mathcal C$ be a strict $2$"/adjunction between $2$"/categories, $\map d{FA}M$ a morphism in $\D$ and $\map jAB$ a morphism in $\mathcal C$. The left Kan extension of $d$ along $Fj$ exists in $\D$ whenever that of the adjunct $\map{d^\sharp}A{GM}$ along $j$ does in $\mathcal C$, and in that case the extensions are adjuncts. To see this consider the induced adjunction $Q(F) \ladj Q(G)$ between the strict double categories of quintets $Q(\mathcal C)$ and $Q(\D)$; see \augpropref{6.4}. By \exref{counit components are universal} the component \mbox{$\map{\eps_M}{Q(F)Q(G)M}M$} of the counit is universal from $Q(F)$ to $M$, so that the result follows from the previous proposition and \exref{Kan extensions in quintets}.
	\end{example}
	
	\section{Pasting lemmas} \label{pasting lemmas section}
	This section describes two useful pasting lemmas for left Kan cells together with some of their consequences. The first of these, below, is the `horizontal' pasting lemma, which concerns horizontal compositions of left Kan cells and whose proof is straightforward. Applying it to the unital virtual double category $\enProf\V$ of $\V$"/profunctors (\exref{enriched left Kan extension}) recovers the classical result on iterated enriched Kan extension (see page~42 of \cite{Dubuc70} or Theorem~4.47 of \cite{Kelly82}).
	
	The `vertical' pasting lemma, \lemref{vertical pasting lemma} below, concerns (weakly) left Kan cells vertically composed with (weakly) cocartesian paths of cells; the latter in the sense of \augdefref{7.1} (see also \defref{cocartesian path} below). It also applies to weaker variants of the notion of (weakly) cocartesian path, which are introduced in \defref{cocartesian path} below.
	
	\subsection{The horizontal pasting lemma}
	\begin{lemma}[Horizontal pasting lemma for left Kan cells] \label{horizontal pasting lemma}
		Assume that the cell $\eta$ in the composite below is left Kan. Then $\eta \hc \zeta$ is (weakly) left Kan precisely if $\zeta$ is so. In that case $\eta \hc \zeta$ restricts along $\map fC{B_m}$ (\defref{pointwise left Kan extension}) precisely if $\zeta$ does so. In particular $\eta \hc \zeta$ is pointwise (weakly) left Kan precisely if $\zeta$ is so.
		\begin{displaymath}
			\begin{tikzpicture}
				\matrix(m)[math35]
					{	A_0 & A_1 & A_{n'} & A_n & B_1 & B_{m'} & B_m \\
						& & & M & & & \\};
				\path[map]	(m-1-1) edge[barred] node[above] {$J_1$} (m-1-2)
														edge[transform canvas={yshift=-2pt}] node[below left] {$d$} (m-2-4)
										(m-1-3) edge[barred] node[above] {$J_n$} (m-1-4)
										(m-1-4) edge[barred] node[above] {$H_1$} (m-1-5)
														edge node[right] {$l$} (m-2-4)
										(m-1-6) edge[barred] node[above] {$H_m$} (m-1-7)
										(m-1-7) edge[transform canvas={yshift=-2pt}] node[below right] {$k$} (m-2-4);
				\draw[transform canvas={xshift=-0.5pt}]	($(m-1-2)!0.5!(m-1-3)$) node {$\dotsb$}
										($(m-1-5)!0.5!(m-1-6)$) node {$\dotsb$};
				\path				($(m-1-1.south)!0.5!(m-1-4.south)$)	edge[cell, transform canvas={shift={(1em,0.333em)}}] node[right] {$\eta$} ($(m-2-1.north)!0.5!(m-2-4.north)$)
										($(m-1-4.south)!0.5!(m-1-7.south)$) edge[cell, transform canvas={shift={(-1em,0.333em)}}] node[right] {$\zeta$} ($(m-2-4.north)!0.5!(m-2-7.north)$);
			\end{tikzpicture}
		\end{displaymath}
	\end{lemma}
	
	\begin{example} \label{recursive left Kan extension}
		By iterating the horizontal pasting lemma we see that (pointwise) (weak) left Kan extensions along a path $\ul J = (J_1, \dotsc, J_n)$ can be obtained by extending along each of the $J_1$, \dots, $J_n$ recursively: if for each $1 \leq i < n$ the cell $\eta_i$ below is left Kan and $\eta_n$ is (pointwise) (weak) left Kan then so is their horizontal composite $\eta_1 \hc \dotsb \hc \eta_n$.
		\begin{displaymath}
			\begin{tikzpicture}
				\matrix(m)[math35, column sep={5.5em,between origins}, row sep={3.75em,between origins}]
					{ A_0 & A_1 & A_{n''} & A_{n'} & A_n \\ & & M & & \\ };
					\path[map]	(m-1-1) edge[barred] node[above] {$J_1$} (m-1-2)
															edge[transform canvas={xshift=-4pt,yshift=-2pt}] node[below left] {$d$} (m-2-3)
											(m-1-2) edge node[above, inner sep=2pt] {$l_1$} (m-2-3)
											(m-1-3) edge[barred] node[above] {$J_{n'}$} (m-1-4)
															edge[transform canvas={xshift=1pt}] node[left, inner sep=-0.5pt] {$l_{n''}$} (m-2-3)
											(m-1-4) edge[barred] node[above] {$J_n$} (m-1-5)
															edge node[below right, inner sep=0pt] {$l_{n'}$} (m-2-3)
											(m-1-5) edge[transform canvas={xshift=4pt,yshift=-2pt}] node[below right] {$l_n$} (m-2-3);
						\draw[transform canvas={xshift=-0.5pt}] ($(m-1-2)!0.5!(m-1-3)$) node {$\dotsb$};
						\path[transform canvas={yshift=0.4em}]	(m-1-2) edge[cell, shorten >= 0.75em] node[right, yshift=2pt, inner sep=2pt] {$\eta_1$} (m-2-2)
											(m-1-4) edge[cell, shorten >= 0.75em] node[right, yshift=2pt, inner sep=2pt] {$\eta_n$} (m-2-4);
						\path[transform canvas={xshift=-3.75em, yshift=0.4em}] (m-1-4) edge[cell, shorten >= 0.75em] node[right, yshift=2pt, inner sep=2pt] {$\eta_{n'}$} (m-2-4);
			\end{tikzpicture}
		\end{displaymath}
	\end{example}
	
	\begin{remark} \label{extension along paths}
		Notice that we would not have been able to state the horizontal pasting lemma if our notions of left Kan extension had been defined along a \emph{single} horizontal morphism, rather than a path of horizontal morphisms.
	\end{remark}
	
	\subsection{Right cocartesian paths}
	Throughout the remainder of this article it will become clear that the notions of weakly cocartesian path of cells and cocartesian path of cells, as defined in \augdefref{7.1}, are too strong. The vertical pasting lemma (\lemref{vertical pasting lemma} below) for instance, is most naturally stated using the weaker notions defined below. The first of these weakens the univeral property of weakly cocartesian paths by restricting it to either nullary cells or unary cells only, while the second introduces weaker, ``right"/sided'' and ``left"/sided'' variants of the notion of cocartesian path.
	\begin{definition} \label{cocartesian path}
		A path of cells $\ul\phi = (\phi_1, \dotsc, \phi_n)$, as in the right-hand side below, is called \emph{weakly nullary"/cocartesian} if any nullary cell $\chi$, as on the left-hand side and with $\ul L = (C)$ empty, factors uniquely through $\ul\phi$ as shown. Analogously $\ul\phi$ is called \emph{weakly unary"/cocartesian} if any unary cell $\chi$, as on the left"/hand side below and with $\lns{\ul L} = 1$, factors uniquely through $\ul\phi$. The path $\ul\phi$ is \emph{weakly cocartesian}, in the sense of \defref{cartesian cells}, if it is both weakly nullary"/cocartesian and weakly unary"/cocartesian.
		
		\begin{multline*}
			\begin{tikzpicture}[textbaseline, ampersand replacement=\&]
				\matrix(m)[math35, column sep={0.75em}]
					{	X_{10} \& X_{11} \& X_{1m'_1} \& X_{1m_1} \&[4em] X_{n0} \& X_{n1} \& X_{nm'_n} \& X_{nm_n} \\
						C \& \& \& \& \& \& \& D \\};
				\path[map]	(m-1-1) edge[barred] node[above] {$H_{11}$} (m-1-2)
														edge node[left] {$h \of f_0$} (m-2-1)
										(m-1-3) edge[barred] node[above, xshift=-1pt] {$H_{1m_1}$} (m-1-4)
										(m-1-5)	edge[barred] node[above] {$H_{n1}$} (m-1-6)
										(m-1-7) 	edge[barred] node[above, xshift=-2pt] {$H_{nm_n}$} (m-1-8)
										(m-1-8)	edge node[right] {$k \of f_n$} (m-2-8)
										(m-2-1) edge[barred] node[below] {$\ul L$} (m-2-8);
				\path				($(m-1-1.south)!0.5!(m-1-8.south)$) edge[cell] node[right] {$\chi$} ($(m-2-1.north)!0.5!(m-2-8.north)$);
				\draw[transform canvas={xshift=-1pt}]	($(m-1-2)!0.5!(m-1-3)$) node {$\dotsb$}
										($(m-1-6)!0.5!(m-1-7)$) node {$\dotsb$};
				\draw				($(m-1-4)!0.5!(m-1-5)$) node {$\dotsb$};
			\end{tikzpicture} \\
			= \begin{tikzpicture}[textbaseline, ampersand replacement=\&]
				\matrix(m)[math35, column sep={0.75em}]
					{	X_{10} \& X_{11} \& X_{1m'_1} \& X_{1m_1} \&[4em] X_{n0} \& X_{n1} \& X_{nm'_n} \& X_{nm_n} \\
						A_0 \& \& \& A_1 \& A_{n'} \& \& \& A_n \\
						C \& \& \& \& \& \& \& D \\};
				\path[map]	(m-1-1) edge[barred] node[above] {$H_{11}$} (m-1-2)
														edge node[left] {$f_0$} (m-2-1)
										(m-1-3) edge[barred] node[above, xshift=-1pt] {$H_{1m_1}$} (m-1-4)
										(m-1-4) edge node[right] {$f_1$} (m-2-4)
										(m-1-5)	edge[barred] node[above] {$H_{n1}$} (m-1-6)
														edge node[left] {$f_{n'}$} (m-2-5)
										(m-1-7) 	edge[barred] node[above, xshift=-2pt] {$H_{nm_n}$} (m-1-8)
										(m-1-8)	edge node[right] {$f_n$} (m-2-8)
										(m-2-1) edge[barred] node[below] {$\ul J_1$} (m-2-4)
														edge node[left] {$h$} (m-3-1)
										(m-2-5)	edge[barred] node[below] {$\ul J_n$} (m-2-8)
										(m-2-8)	edge node[right] {$k$} (m-3-8)
										(m-3-1) edge[barred] node[below] {$\ul L$} (m-3-8);
				\path				($(m-1-1.south)!0.5!(m-1-4.south)$) edge[cell] node[right] {$\phi_1$} ($(m-2-1.north)!0.5!(m-2-4.north)$)
										($(m-1-5.south)!0.5!(m-1-8.south)$) edge[cell] node[right] {$\phi_n$} ($(m-2-5.north)!0.5!(m-2-8.north)$)
										($(m-2-1.south)!0.5!(m-2-8.south)$) edge[cell] node[right] {$\chi'$} ($(m-3-1.north)!0.5!(m-3-8.north)$);
				\draw[transform canvas={xshift=-1pt}]	($(m-1-2)!0.5!(m-1-3)$) node {$\dotsb$}
										($(m-1-6)!0.5!(m-1-7)$) node {$\dotsb$};
				\draw				($(m-1-4)!0.5!(m-2-5)$) node {$\dotsb$};
			\end{tikzpicture}
		\end{multline*}
		
		A weakly nullary"/cocartesian path $(\phi_1, \dotsc, \phi_n)$ of the form below is called \emph{right nullary-cocartesian} if, for any non"/empty path $(K_1, \dotsc, K_q)$ of horizontal morphisms such that the restriction $K_1(f_n, \id)$ exists, the composite path below is weakly nullary"/cocartesian. \emph{Right unary"/cocartesian} paths are defined analogously; a path is \emph{right cocartesian} whenever it is both right nullary"/cocartesian and right unary"/cocartesian.
		\begin{displaymath}
			\begin{tikzpicture}
				\matrix(m)[math35, column sep={3em,between origins}]
					{	X_{10} & X_{11} & X_{1m'_1} &[0.5em] X_{1m_1} &[1em] X_{n0} & X_{n1} & X_{nm'_n} &[0.6em] X_{nm_n} & Y_1 & Y_2 & Y_{q'} & Y_q \\
						A_0 & & & A_1 & A_{n'} & & & A_n & Y_1 & Y_2 & Y_{q'} & Y_q \\};
				\path[map]	(m-1-1) edge[barred] node[above] {$H_{11}$} (m-1-2)
														edge node[left, inner sep=1pt] {$f_0$} (m-2-1)
										(m-1-3) edge[barred] node[above, xshift=-1pt] {$H_{1m_1}$} (m-1-4)
										(m-1-4) edge node[right, inner sep=1pt] {$f_1$} (m-2-4)
										(m-1-5)	edge[barred] node[above] {$H_{n1}$} (m-1-6)
														edge node[left, inner sep=0pt] {$f_{n'}$} (m-2-5)
										(m-1-7) edge[barred] node[above, xshift=-2pt] {$H_{nm_n}$} (m-1-8)
										(m-1-8)	edge[barred] node[above, xshift=-10pt] {$K_1(f_n, \id)$} (m-1-9)
														edge node[left, inner sep=1pt] {$f_n$} (m-2-8)
										(m-1-9) edge[barred] node[above] {$K_2$} (m-1-10)
										(m-1-11) edge[barred] node[above] {$K_q$} (m-1-12)
										(m-2-1) edge[barred] node[below] {$\ul J_1$} (m-2-4)
										(m-2-5)	edge[barred] node[below] {$\ul J_n$} (m-2-8)
										(m-2-8)	edge[barred] node[below] {$K_1$} (m-2-9)
										(m-2-9) edge[barred] node[below] {$K_2$} (m-2-10)
										(m-2-11) edge[barred] node[below] {$K_q$} (m-2-12);
				\path				(m-1-9) edge[eq] (m-2-9)
										(m-1-10) edge[eq] (m-2-10)
										(m-1-11) edge[eq] (m-2-11)
										(m-1-12) edge[eq] (m-2-12)
										($(m-1-1.south)!0.5!(m-1-4.south)$) edge[cell] node[right] {$\phi_1$} ($(m-2-1.north)!0.5!(m-2-4.north)$)
										($(m-1-5.south)!0.5!(m-1-8.south)$) edge[cell] node[right] {$\phi_n$} ($(m-2-5.north)!0.5!(m-2-8.north)$);
				\draw[transform canvas={xshift=-1pt}]	($(m-1-2)!0.5!(m-1-3)$) node {$\dotsb$}
										($(m-1-6)!0.5!(m-1-7)$) node {$\dotsb$};
				\draw				($(m-1-4)!0.5!(m-2-5)$) node {$\dotsb$}
										($(m-1-10)!0.5!(m-2-11)$) node {$\dotsb$}
										($(m-1-8)!0.5!(m-2-9)$) node[font=\scriptsize] {$\cart$};
			\end{tikzpicture}
		\end{displaymath}
		
		Horizontally dual, the weakly nullary"/cocartesian path $(\phi_1, \dotsc, \phi_n)$ above is called \emph{left nullary"/cocartesian} if, for any non"/empty path $\hmap{(K'_1, \dotsc, K'_p)}{Y'_0}{A_0}$ such that the restriction $K'_p(\id, f_0)$ exists, the composite path $(\id_{K'_1}, \dotsc, \id_{K'_{p'}}, \cart, \phi_1, \dotsc, \phi_n)$ is weakly nullary"/cocartesian, where $\cart$ defines $K'_p(\id, f_0)$. \emph{Left unary"/cocartesian paths} are defined analogously; a path is \emph{left cocartesian} whenever it is both left nullary"/cocartesian and left unary"/cocartesian.
		
		A right nullary"/cocartesian path $(\phi_1, \dotsc, \phi_n)$ is called \emph{nullary"/cocartesian} if each of the composite paths $(\phi_1, \dotsc, \phi_n, \cart, \id_{K_2}, \dotsc, \id_{K_q})$ above are left nullary"/cocartesian. \emph{Unary"/cocartesian} paths are defined analogously. A path is \emph{cocartesian} in the sense of \augdefref{7.1}\footnote{Definition~7.1 of the original version of \cite{Koudenburg20} (published 2020-02-24), defining the notion of cocartesian path $\ul \phi$, contains a typo: the paths of identity cells that $\ul\phi$ is concatenated with are allowed to be of any lengths $p$ and $q$, that is $p, q \geq 0$ (and not $p, q \geq 1$ as originally printed).} precisely if it is both nullary"/cocartesian and unary"/cocartesian.
	\end{definition}
	
	Notice that the universal property of a right (respectively weakly) nullary"/cocartesian cell does not determine its horizontal target up to isomorphism, in contrast to the universal property of (right) (respectively weakly) (unary"/)cocartesian cells. While we will mostly use the nullary variant of right cocartesian paths of cells, the unary variant will be important in our study of `exact cells' in \secref{exact cells} below. \propref{left exactness and right unary-cocartesianness} for instance shows that, under mild conditions, horizontal cells are `left exact' with respect to a `Yoneda morphism' (\defref{yoneda embedding}) if and only if they are right unary"/cocartesian. The notion of left nullary"/cocartesian paths is used in \secref{pointwise Kan extensions section}.
	
	\begin{example} \label{cocartesian paths in the presence of horizontal units or (1,0)-ary cartesian cells}
		Assume that the horizontal unit $\hmap{I_C}CC$ (\defref{cartesian cells}) of the object $C$ exists. It follows from the horizontal unit identities (see \auglemref{5.9} or \lemref{companion identities lemma}) that the factorisations for the nullary cells $\chi$ above, with empty horizontal target $C$, correspond precisely to factorisations of unary cells $\chi$ with horizontal target $I_C$. We conclude that in unital virtual double categories (\defref{augmented virtual equipment}) the notions of weakly unary"/cocartesian path and weakly cocartesian path coincide, as well as those of right (respectively left) unary"/cocartesian path and right (respectively left) cocartesian path.
		
		Similarly the notions of right (respectively left or weakly) nullary"/cocartesian path and right (respectively left or weakly) cocartesian path coincide in augmented virtual double categories $\K$ that admit, for each horizontal morphism $\hmap LCD$, a nullary cartesian cell $L \Rightarrow X$ where $X$ is any object; see also \defref{cocartesian path of (0,1)-ary cells} below. Examples of such $\K$ are the augmented virtual equipments of enriched profunctors $\enProf{\V}$ and $\enProf{(\V, \V')}$ (\augexsref{2.4}{2.7}); see \exref{cotabulations of enriched profunctors}.
	\end{example}

	\begin{example} \label{cocartesian paths are right cocartesian}
		 Clearly cocartesian paths are both left and right cocartesian. The converse however does not hold in general since cocartesian paths $\ul\phi$ are required to be preserved under two"/sided concatenations of the form \mbox{$(\id, \dotsc, \id, \cart)\conc\ul\phi\conc(\cart, \id, \dotsc, \id)$}.
	\end{example}
	
	The following is a straightforward variation on the pasting lemma for cocartesian paths (\auglemref{7.7}). Likewise \cororef{non-horizontal cocartesian cells} is a variation on \augcororef{8.5}.
	\begin{lemma}[Pasting lemma for cocartesian paths] \label{pasting lemma for cocartesian paths}
		In the configuration of cells below denote by $\ul\phi_j$ the path $\ul\phi_j \dfn (\phi_{j1}, \dotsc, \phi_{jn_j})$, for each $1 \leq j \leq n$, and assume that the path $(\phi_{11}, \dotsc, \phi_{nm_n})$ is weakly nullary"/cocartesian. Then the path $\ul\psi \dfn (\psi_1, \dotsc, \psi_n)$  is weakly nullary"/cocartesian if and only if the path of composites $\bigpars{\psi_1 \of \ul \phi_1, \dotsc, \psi_n \of \ul\phi_n}$ is so.
		
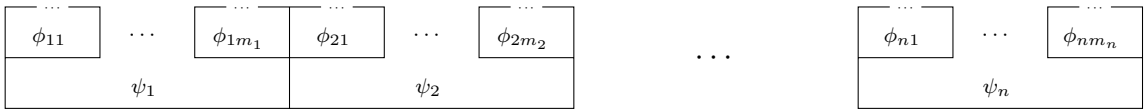
\begin{displaymath}
			\begin{tikzpicture}[scheme, x=0.42cm]
				\draw	(1,2) -- (0,2) -- (0,0) -- (18,0) -- (18,2) -- (17,2)
							(2,2) -- (3,2) -- (3,1) -- (0,1)
							(7,2) -- (6,2) -- (6,1) -- (12,1) -- (12,2) -- (11,2)
							(8,2) -- (10,2)
							(9,2) -- (9,0)
							(16,2) -- (15,2) -- (15,1) -- (18,1)
							(28,2) -- (27,2) -- (27,0) -- (36,0) -- (36,2) -- (35,2)
							(29,2) -- (30,2) -- (30,1) -- (27,1)
							(34,2) -- (33,2) -- (33,1) -- (36,1);
				\draw[shift={(0.3pt, -0.33pt)}] (1.5,2) node {\scalebox{.6}{$\dotsb$}}
							(7.5,2) node {\scalebox{.6}{$\dotsb$}}
							(10.5,2) node {\scalebox{.6}{$\dotsb$}}
							(16.5,2) node {\scalebox{.6}{$\dotsb$}}
							(28.5,2) node {\scalebox{.6}{$\dotsb$}}
							(34.5,2) node {\scalebox{.6}{$\dotsb$}};
				\draw[shift={(0.165cm, 0.3cm)}] (1,1) node {$\phi_{11}$}
							(7,1) node {$\phi_{1m_1}$}
							(10,1) node {$\phi_{21}$}
							(16,1) node {$\phi_{2m_2}$}
							(28,1) node {$\phi_{n1}$}
							(34,1) node {$\phi_{nm_n}$}
							(4,0) node {$\psi_1$}
							(13,0) node {$\psi_2$}
							(31,0) node {$\psi_n$}
							(4,1) node {$\dotsb$}
							(13,1) node {$\dotsb$}
							(31,1) node {$\dotsb$};
				\draw[font=]	(22.5,1) node {$\dotsb$};
			\end{tikzpicture}
		\end{displaymath}
		
		Next denote the vertical targets of $\phi_{nm_n}$ and $\psi_n$ by $\map{f_{nm_n}}{X_{nm_nk_{m_n}}}{A_{nm_n}}$ and $\map{h_n}{A_{nm_n}}{C_n}$. If the path $(\phi_{11}, \dotsc, \phi_{nm_n})$ is right nullary"/cocartesian then the following hold:
		\begin{enumerate}[label=\textup{(\alph*)}]
			\item	if $\ul\psi$ is right nullary"/cocartesian then so is $\bigpars{\psi_1 \of \ul \phi_1, \dotsc, \psi_n \of \ul\phi_n}$ provided that for any horizontal morphism $\hmap K{C_n}{C'}$ the following holds: if the restriction \mbox{$K(h_n \of f_{nm_n}, \id)$} exists then so does $K(h_n, \id)$;

			\item if $\bigpars{\psi_1 \of \ul \phi_1, \dotsc, \psi_n \of \ul\phi_n}$ is right nullary"/cocartesian then so is $\ul\psi$ provided that for any horizontal morphism $\hmap K{C_n}{C'}$ the following holds: if the restriction $K(h_n, \id)$ exists then so does $K(h_n \of f_{nm_n}, \id)$.
		\end{enumerate}
		
		Analogous assertions hold for right (respectively weakly) (unary"/)cocartesian paths; horizontally dual assertions hold for left (nullary"/ or unary"/)cocartesian paths.
	\end{lemma}
	
	\begin{corollary} \label{non-horizontal cocartesian cells}
		The cell $\psi$ below is right (respectively weakly) nullary"/cocartesian if and only if the composite is so.
		
		An analogous assertion holds for the composite $\psi \hc \cart$, where $\cart$ defines the companion $\hmap{g_*}{X_n}B$. Analogous assertions hold for right (respectively weakly) (unary"/)cocartesian cells and left (nullary"/ or unary"/)cocartesian cells. 
		\begin{displaymath}
			\begin{tikzpicture}
				\matrix(m)[math35, column sep={1.75em,between origins}]
					{	A & & X_0 & & X_1 & & X_{n'} & & X_n \\
						& A & & & & & & B & \\};
				\path[map]	(m-1-1) edge[barred] node[above] {$f^*$} (m-1-3)
										(m-1-3) edge[barred] node[above] {$H_1$} (m-1-5)
														edge[transform canvas={xshift=1pt}] node[right] {$f$} (m-2-2)
										(m-1-7) edge[barred] node[above] {$H_n$} (m-1-9)
										(m-1-9) edge[transform canvas={xshift=1pt}] node[right] {$g$} (m-2-8)
										(m-2-2) edge[barred] node[below] {$J$} (m-2-8);
				\path				(m-1-1) edge[transform canvas={xshift=-1pt}, eq] (m-2-2)
										(m-1-5) edge[transform canvas={xshift=0.875em}, cell] node[right] {$\psi$} (m-2-5);
				\draw[font=\scriptsize]	([yshift=0.333em]$(m-1-2)!0.5!(m-2-2)$) node {$\cart$};
				\draw				($(m-1-5)!0.5!(m-1-7)$) node {$\dotsb$};
			\end{tikzpicture}
		\end{displaymath}
	\end{corollary}
	\begin{proof}
		Apply the pasting lemma to $\psi_1 = \cart \hc \psi$ and $\ul\phi_1 = (\cocart, \id_{H_1}, \dotsc, \id_{H_n})$ where $\cocart$ denotes the cocartesian cell corresponding to the cartesian cell above (\lemref{companion identities lemma}), so that $\psi_1 \of \ul\phi_1 = \psi$.
	\end{proof}
	
	\subsection{Pointwise right cocartesian paths}
	The pointwise variants of the notions of right (respectively weakly) nullary"/ and unary"/cocartesian path are similar to the pointwise variant of cocartesian path given in \augdefref{9.1}, as follows; see \remref{right pointwise cocartesian and pointwise right cocartesian comparison} below for the differences. Recall that $n' \dfn n - 1$ for any postive integer $n$.
	\begin{definition} \label{pointwise right cocartesian path}
		Consider a right (respectively weakly) nullary"/cocartesian path $\ul\phi = (\phi_1, \dotsc, \phi_n)$ whose last cell $\phi_n$ has trivial vertical target and arity $(m_n, 1)$ with $m_n \geq 1$, as in the left"/hand side below.
		\begin{displaymath}
			\begin{tikzpicture}[textbaseline]
				\matrix(m)[math35, column sep={2.1em,between origins}]
					{ X_{n0} & & X_{n1} & & X_{n(m_n)'} &[0.25em] & Y \\
						X_{n0} & & X_{n1} & & X_{n(m_n)'} & & A_n \\
						& A_{n'} & & & & A_n & \\ };
				\path[map]	(m-1-1) edge[barred] node[above] {$H_{n1}$} (m-1-3)
										(m-1-5) edge[barred] node[above, xshift=-12pt] {$H_{nm_n}(\id, f)$} (m-1-7)
										(m-1-7) edge node[right] {$f$} (m-2-7)
										(m-2-1) edge[barred] node[below] {$H_{n1}$} (m-2-3)
														edge node[left] {$f_{n'}$} (m-3-2)
										(m-2-5) edge[barred] node[below] {$H_{nm_n}$} (m-2-7)
										(m-3-2) edge[barred] node[below] {$J_n$} (m-3-6);
				\path				(m-1-3) edge[eq] (m-2-3)
										(m-1-5) edge[eq] (m-2-5)
										(m-1-1) edge[eq] (m-2-1)
										(m-2-7) edge[eq] (m-3-6)
										(m-2-4) edge[cell, transform canvas={xshift=0.125em}] node[right] {$\phi_n$} (m-3-4);
				\draw[font=\scriptsize]	($(m-1-6)!0.5!(m-2-6)$) node {$\cart$};
				\draw				($(m-1-4)!0.5!(m-2-4)$) node {$\dotsb$};
			\end{tikzpicture} \qquad = \qquad \begin{tikzpicture}[textbaseline]
				\matrix(m)[math35, column sep={2.1em,between origins}]
					{ X_{n0} & & X_{n1} & & X_{n(m_n)'} & & Y \\
						A_{n'} & & & & & & Y \\
						& A_{n'} & & & & A_n & \\ };
				\path[map]	(m-1-1) edge[barred] node[above] {$H_{n1}$} (m-1-3)
														edge node[left] {$f_{n'}$} (m-2-1)
										(m-1-5) edge[barred] node[above, xshift=-12pt] {$H_{nm_n}(\id, f)$} (m-1-7)
										(m-2-1) edge[barred] node[below] {$J_n(\id, f)$} (m-2-7)
										(m-2-7) edge node[right] {$f$} (m-3-6)
										(m-3-2) edge[barred] node[below] {$J_n$} (m-3-6);
				\path				(m-2-1) edge[eq] (m-3-2)
										(m-1-7) edge[eq] (m-2-7)
										(m-1-4) edge[cell] node[right] {$\phi_n'$} (m-2-4);
				\draw[font=\scriptsize]	([yshift=-0.25em]$(m-2-4)!0.5!(m-3-4)$) node {$\cart$};
				\draw				(m-1-4) node[xshift=-3pt] {$\dotsb$};
			\end{tikzpicture}
		\end{displaymath}
		\begin{enumerate}[label=\textup{(\alph*)}]
			\item Let $\map fY{A_n}$ be any morphism. We say that $\ul\phi$ \emph{restricts along $f$} if both restrictions $H_{nm_n}(\id, f)$ and $J_n(\id, f)$ exist and the path $(\phi_1, \dotsc, \phi_{n'}, \phi_n')$ is again right (respectively weakly) nullary"/cocartesian, where $\phi'_n$ is the unique factorisation in the right"/hand side above.
			\item We call $\ul\phi$ \emph{pointwise} if it restricts along any morphism $\map fY{A_n}$ such that $H_{nm_n}(\id, f)$ exists.
		\end{enumerate}
		The notion of \emph{restriction} for right (respectively weakly) unary"/cocartesian and right (respectively weakly) cocartesian cells is defined analogously, as is the notion of \emph{pointwise right (respectively weakly) unary"/cocartesian} path and that of \emph{pointwise right (respectively weakly) cocartesian} path.
	\end{definition}
	Unpacking the definition of a pointwise right (respectively weakly) nullary"/cocartesian path $\ul\phi$ above we find that it requires the following for every morphism $\map fY{A_n}$: if the restriction $H_{nm_n}(\id, f)$ exists then so does $J_n(\id, f)$ and, in that case, the path $(\phi_1, \dotsc, \phi_{n'}, \phi_n')$, with $\phi_n'$ as above, is right (respectively weakly) nullary"/cocartesian. A single horizontal cell $\cell\phi{(H_1, \dotsc, H_n)}J$ is called \emph{(pointwise) right cocartesian} whenever the singleton path $(\phi)$ is (pointwise) right cocartesian; in that case we call $J$ the \emph{(pointwise) right composite} of $(H_1, \dotsc, H_n)$. If $(\phi)$ is cocartesian (\defref{cocartesian path}) then $J$ is the \emph{horizontal composite} $(H_1 \hc \dotsb \hc H_n)$ in the sense of \augdefref{7.1}; in this article we will use the notation $(H_1 \hc \dotsb \hc H_n)$ for both horizontal composites and right composites. Notice that the pasting lemma (\lemref{pasting lemma for cocartesian paths}) only allows us to combine right composites on the left, that is given a right composite $\hmap{(H_1 \hc \dotsb \hc H_n)}{A_0}{A_n}$ and a path $\hmap{\ul J}{A_n}{B_m}$, the right composite $\bigpars{(H_1 \hc \dotsb \hc H_n) \hc J_1 \hc \dotsb \hc J_m}$ exists if and only if the right composite $(H_1 \hc \dotsb \hc H_n \hc J_1 \hc \dotsb \hc J_m)$ does, and in that case they are canonically isomorphic. In particular the associator $\bigpars{(H \hc J) \hc K)} \iso \bigpars{H \hc (J \hc K)}$ of right composites need not exist in general, unless $(J \hc K)$ is a horizontal composite.
	
	\begin{remark} \label{right pointwise cocartesian and pointwise right cocartesian comparison}
		Consider a path $\ul\phi = (\phi_1, \dotsc, \phi_n)$ of the form as in the definition above, but which is not necessarily nullary"/cocartesian. In \augdefref{9.1} $\ul\phi$ is called \emph{right pointwise cocartesian} if, for every morphism $\map fY{A_n}$ for which both $H_{nm_n}(\id, f)$ and $J_n(\id, f)$ exist, the path $(\phi_1, \dotsc, \phi_{n'}, \phi_n')$, with $\phi_n'$ as above, is cocartesian. Since every cocartesian path is right cocartesian (\exref{cocartesian paths are right cocartesian}) it follows that any right pointwise cocartesian path, in the sense of \augdefref{9.1}, is pointwise right cocartesian, in the above sense, whenever, for each $\map fY{A_n}$, if the restriction $H_{nm_n}(\id, f)$ exists then so does the restriction $J_n(\id, f)$.
		
		We conclude that in augmented virtual double categories that have restrictions on the right (\defref{augmented virtual equipment}) any right pointwise cocartesian path is pointwise right cocartesian and, in particular, any pointwise composite $(H_1 \hc \dotsb \hc H_n)$ in the sense of \augdefref{9.1} is a pointwise right composite in the sense above. Under some conditions on the base category $\V$ such pointwise composites exist in the augmented virtual double categories $\enProf\V$, $\enProf{(\V, \V')}$ and $\ensProf\V$ of (small) $\V$"/profunctors, where they are given by $\V$"/coends; see \augexsref{9.2}{9.3}.
	\end{remark}
	
	Pointwise cocartesian paths are coherent in the sense of the following two lemmas. Taking into account the slight difference in the notions of pointwise right cocartesian path and right pointwise cocartesian path regarding the existence of the relevant restrictions, as described in the previous remark, the proofs of these lemmas are essentially the same as those of \auglemsref{9.5}{9.6}.
	\begin{lemma} \label{coherence of pointwise cocartesian paths}
		If the path $(\phi_1, \dotsc, \phi_n)$ is pointwise right (respectively weakly) (nullary- or unary"/)cocartesian then any path of the form $(\phi_1, \dotsc, \phi_n')$ as in \defref{pointwise right cocartesian path} is again pointwise right (respectively weakly) (nullary- or unary"/)cocartesian.
	\end{lemma}
	
	Let $(H_1 \hc \dotsb \hc H_n)$ be a pointwise right composite and $f$ a morphism such that the restriction $H_n(\id, f)$ exists. Applying the lemma above to the singleton path consisting of the pointwise right cocartesian cell defining the composite we find that the restriction $(H_1 \hc \dotsb \hc H_n)(\id, f)$ forms the pointwise right composite $\bigpars{H_1 \hc \dotsb \hc H_n(\id, f)}$.
	
	\begin{lemma}[Pasting lemma for pointwise cocartesian paths] \label{pasting lemma for pointwise cocartesian paths}
		Consider the configuration of cells of \lemref{pasting lemma for cocartesian paths}. Assume that all its cells $\psi_i$ and $\phi_{ij}$ are unary, that $\phi_{nm_n}$ has non"/empty horizontal source, and that the vertical targets of the final cells $\psi_n$ and $\phi_{nm_n}$ are both the identity morphism on the object $C_n$. Denote by $J$ the horizontal target of $\phi_{nm_n}$ and by $H$ the final morphism of its horizontal source. Assume that the path $(\phi_{11}, \dotsc, \phi_{nm_n})$ is pointwise right (respectively weakly) (nullary"/ or unary"/)cocartesian.
		
		If the path $(\psi_1, \dotsc, \psi_n)$ pointwise right (respectively weakly) (nullary"/ or unary"/)cocartesian then so is the path of composites $(\psi_1 \of \ul\phi_1, \dotsc, \psi_n \of \ul\phi_n)$. The converse holds whenever the following holds for all morphisms $\map fY{C_n}$: if the restriction $J(\id, f)$ exists then so does $H(\id, f)$.
	\end{lemma}
	
	\subsection{The vertical pasting lemma}
	We are now ready to state the vertical pasting lemma for left Kan cells.
	\begin{lemma}[Vertical pasting lemma for left Kan cells] \label{vertical pasting lemma}
		Consider the composite below. If the path $\ul\phi = (\phi_1, \dotsc, \phi_n)$ is right (respectively weakly) nullary"/cocartesian then the cell $\eta$ is (weakly) left Kan precisely if the composite $\eta \of \ul\phi$ is so.	If moreover $\ul\phi$ restricts along \mbox{$\map fY{A_n}$} then $\eta$ restricts along $f$ (\defref{pointwise left Kan extension}) precisely if the composite does so.
		
		Next assume that $\ul\phi$ is pointwise right (respectively weakly) nullary"/cocartesian. If $\eta$ is pointwise (weakly) left Kan (\defref{pointwise left Kan extension}) then so is the composite. The converse holds whenever the restrictions $H_{nm_n}(\id, f)$ exist for all $\map fY{A_n}$.
		\begin{displaymath}
			\begin{tikzpicture}
				\matrix(m)[math35]
					{	X_{10} & X_{11} & X_{1m'_1} & X_{1m_1} &[4em] X_{n0} & X_{n1} & X_{nm'_n} & A_n \\
						A_0 & & & A_1 & A_{n'} & & & A_n \\};
				\draw				([yshift=-3.25em]$(m-2-1)!0.5!(m-2-8)$) node (M) {$M$};
				\path[map]	(m-1-1) edge[barred] node[above] {$H_{11}$} (m-1-2)
														edge node[left] {$f_0$} (m-2-1)
										(m-1-3) edge[barred] node[above] {$H_{1m_1}$} (m-1-4)
										(m-1-4) edge node[right] {$f_1$} (m-2-4)
										(m-1-5)	edge[barred] node[above] {$H_{n1}$} (m-1-6)
														edge node[left] {$f_{n'}$} (m-2-5)
										(m-1-7) 	edge[barred] node[above, xshift=-2pt] {$H_{nm_n}$} (m-1-8)
										(m-2-1) edge[barred] node[below] {$\ul J_1$} (m-2-4)
														edge[transform canvas={yshift=-2pt}] node[below left] {$d$} (M)
										(m-2-5)	edge[barred] node[below] {$\ul J_n$} (m-2-8)
										(m-2-8)	edge[transform canvas={yshift=-2pt}] node[below right] {$l$} (M);
				\path				($(m-1-1.south)!0.5!(m-1-4.south)$) edge[cell] node[right] {$\phi_1$} ($(m-2-1.north)!0.5!(m-2-4.north)$)
										($(m-1-5.south)!0.5!(m-1-8.south)$) edge[cell] node[right] {$\phi_n$} ($(m-2-5.north)!0.5!(m-2-8.north)$)
										($(m-2-1.south)!0.5!(m-2-8.south)$) edge[cell] node[right] {$\eta$} (M)
										(m-1-8)	edge[eq] (m-2-8);
				\draw[transform canvas={xshift=-1pt}]	($(m-1-2)!0.5!(m-1-3)$) node {$\dotsb$}
										($(m-1-6)!0.5!(m-1-7)$) node {$\dotsb$};
				\draw				($(m-1-4)!0.5!(m-2-5)$) node {$\dotsb$};
			\end{tikzpicture}
		\end{displaymath}
	\end{lemma}
	\begin{proof}
		The main assertion follows from applying \lemref{unique factorisations factorised through cocartesian paths} below to $\eta$ and $\ul \chi = \ul \phi$, while $\ul \zeta$ ranges over all (including the empty) paths of horizontal identity cells. Indeed, the assertion of that lemma implies that the unique factorisations showing that $\eta$ is (weakly) left Kan correspond precisely to those showing that $\eta \of \ul\phi$ is so.
		
		Next assume that $\ul\phi$ restricts along a morphism $\map fY{A_n}$ so that, in particular, $\phi_n$ is of arity $(m_n, 1)$ with $m_n \geq 1$. The schematically drawn identity below, where the cartesian cells defining $H_{nm_n}(\id, f)$ and $J_n(\id, f)$ are both denoted by `c', follows from the identity of \defref{pointwise right cocartesian path}. By \defref{pointwise right cocartesian path}(a) the top row of the right"/hand side is again right (respectively weakly) nullary"/cocartesian.
		\begin{displaymath}
			\begin{tikzpicture}[scheme]
				\draw (0,2) -- (1,2) -- (1,3) -- (0,3) -- (0,0) -- (9,0) -- (9,3) -- (8,3) -- (8,2) -- (9,2) (3,2) -- (2,2) -- (2,3) -- (3,3) -- (3,1) -- (0,1) (6,2) -- (7,2) -- (7,3) -- (6,3) -- (6,1) -- (9,1);
				\draw	(1.5,2.5) node[xshift=0.75pt] {$\dotsb$}
							(1.5,1.5) node {$\phi_1$}
							(4.5,2) node[font=] {$\dotsb$}
							(4.5,0.5) node {$\eta$}
							(7.5,2.5) node[xshift=0.75pt] {$\dotsb$}
							(7.5,1.5) node {$\phi_n$}
							(8.5,2.5) node {c};
			\end{tikzpicture} \qquad = \qquad \begin{tikzpicture}[scheme]
				\draw (1,3) -- (0,3) -- (0,0) -- (9,0) -- (9,3) -- (8,3) (2,3) -- (3,3) -- (3,1) -- (0,1) (7,3) -- (6,3) -- (6,1) -- (9,1) (0,2) -- (3,2) (6,2) -- (9,2);
				\draw	(1.5,3) node[xshift=0.75pt, yshift=-0.25pt] {$\dotsb$}
							(1.5,2.5) node {$\phi_1$}
							(4.5,2.5) node[font=] {$\dotsb$}
							(4.5,1.5) node[font=] {$\dotsb$}
							(4.5,0.5) node {$\eta$}
							(7.5,3) node[xshift=0.75pt, yshift=-0.25pt] {$\dotsb$}
							(7.5,2.5) node[yshift=1pt] {$\phi_n'$}
							(7.5,1.5) node {c};
			\end{tikzpicture}
		\end{displaymath}
		Notice that, by \defref{pointwise left Kan extension}, $\eta \of \ul\phi$ restricts along $f$ as soon as the left"/hand side above is (weakly) left Kan while $\eta$ restricts along $f$ whenever the composite of the bottom two rows of the right"/hand side is so. That the latter are equivalent follows from applying the main assertion to $(\phi_1, \dotsc, \phi_n')$.
		
		Finally assume that $\ul\phi$ is pointwise right (respectively weakly) nullary"/cocartesian. If $\eta$ is pointwise (weakly) left Kan then, for each $\map fY{A_n}$ such that $H_{nm_n}(\id, f)$ exists, both $\ul\phi$ and $\eta$ restrict along $f$ so that by the previous the composite $\eta \of \ul \phi$ restricts along $f$ too. We conclude that the composite is pointwise (weakly) left Kan. The converse similarly follows from the previous provided that all restrictions $H_{nm_n}(\id, f)$ exist so that, by the assumption on $\ul\phi$, so do all restrictions $J_n(\id, f)$.
	\end{proof}
	
	\begin{example} \label{Kan extensions along composites}
		Consider a path $\hmap{(J_1, \dotsc, J_n)}{A_0}{A_n}$ of horizontal morphisms as well as a vertical morphism $\map d{A_0}M$ and assume that the horizontal composite $(J_1 \hc \dotsb \hc J_n)$ exists (\defref{pointwise right cocartesian path}). Applying the vertical pasting lemma to the cocartesian cell defining the composite we find that the (weak) left Kan extension of $d$ along $(J_1, \dotsc, J_n)$ exists precisely if that of $d$ along $(J_1 \hc \dotsb \hc J_n)$ does, and in that case they are isomorphic. If $(J_1 \hc \dotsb \hc J_n)$ is a pointwise composite (\remref{right pointwise cocartesian and pointwise right cocartesian comparison}) and all restrictions on the right exist (\defref{augmented virtual equipment}) then the analogous equivalence holds for pointwise (weak) left Kan extensions.
	\end{example}
	
	\begin{example} \label{Kan extensions along horizontal units}
		Consider parallel morphisms $d$ and $\map lAM$ whose source admits a horizontal unit $\hmap{I_A}AA$ (\defref{cartesian cells}). Applying the vertical pasting lemma to the cocartesian cell defining $I_A$ (\lemref{companion identities lemma}), combined with \exref{vertical cells defining left Kan extensions}, we find that $l$ is the (weak) left Kan extension of $d$ along $I_A$ if and only if $l \iso d$.
	\end{example}
	
	\begin{lemma} \label{unique factorisations factorised through cocartesian paths}
	Consider composable paths $\ul\chi = (\chi_1, \dotsc, \chi_n)$ and $\ul\zeta = (\zeta_1, \dotsc, \zeta_m)$ of cells $\cell{\chi_i}{\ul K_i}{\ul J_i}$ and $\cell{\zeta_j}{\ul L_j}{\ul H_j}$, and assume that their concatenation $\ul\chi \conc \ul\zeta$ is weakly nullary"/cocartesian so that the assignment of cells below, given by composition with $\ul\chi \conc \ul\zeta$, is a bijection; here $\map{f_0}{X_{10}}{A_0}$ and \mbox{$\map{g_m}{Y_{mq_m}}{B_m}$} denote the vertical source of $\chi_1$ and the vertical target of $\zeta_m$.
		\begin{displaymath}
			\begin{tikzpicture}
				\matrix(m)[math35, column sep={6em,between origins}, xshift=-10.5em]{A_0 & & B_m \\ & M & \\};
				\path[map]	(m-1-1) edge[barred] node[above] {$\ul J_1 \conc \dotsb \conc \ul J_n \conc \ul H_1 \conc \dotsb \conc \ul H_m$} (m-1-3)
														edge[transform canvas={yshift=-2pt}] node[below left] {$d$} (m-2-2)
										(m-1-3) edge[transform canvas={yshift=-2pt}] node[below right] {$k$} (m-2-2);
				\path				(m-1-2) edge[cell] node[right] {$\phi$} (m-2-2);
				
				\matrix(n)[math35, column sep={6.5em,between origins}, xshift=10.5em]{X_{10} & & Y_{mq_m} \\ & M & \\};
				\path[map]	(n-1-1) edge[barred] node[above] {$\ul K_1 \conc \dotsb \conc \ul K_n \conc \ul L_1 \conc \dotsb \conc \ul L_m$} (n-1-3)
														edge[transform canvas={yshift=-2pt}] node[below left] {$d \of f_0$} (n-2-2)
										(n-1-3) edge[transform canvas={yshift=-2pt}] node[below right] {$k \of g_m$} (n-2-2);
				\path				(n-1-2) edge[cell] node[right] {$\psi$} (n-2-2);
				
				\draw[font=\Large]	(-17.5em,0.3em) node {$\lbrace$}
										(-3.6em,0.3em) node {$\rbrace$}
										(2.7em,0.3em) node {$\lbrace$}
										(18.2em,0.3em) node {$\rbrace$};
				\path[map]	(-2.1em,0.3em) edge node[above] {$\dash \of (\ul \chi \conc \ul \zeta)$} (1.2em,0.3em);
			\end{tikzpicture}
		\end{displaymath}
		 Consider a nullary cell $\eta$ of the form below. If $\ul\zeta$ is weakly nullary"/cocartesian too then cells of the form $\phi$ above factor uniquely through $\eta$, i.e.\ each $\phi = \eta \hc \phi'$ for a unique nullary cell $\cell{\phi'}{\ul H_1 \conc \dotsb \conc \ul H_m}M$, precisely if cells of the form $\psi$ above factor uniquely through $\eta \of \ul\chi$, i.e.\ each $\psi = (\eta \of \ul \chi) \hc \psi'$ for a unique nullary cell $\cell{\psi'}{\ul L_1 \conc \dotsb \conc \ul L_m}M$.
		\begin{displaymath}
			\begin{tikzpicture}
				\matrix(m)[math35]{A_0 & A_1 & A_{n'} & A_n \\};
				\draw	([yshift=-3.25em]$(m-1-1)!0.5!(m-1-4)$) node (M) {$M$};
				\path[map]	(m-1-1) edge[barred] node[above] {$J_1$} (m-1-2)
														edge[transform canvas={yshift=-2pt}] node[below left] {$d$} (M)
										(m-1-3) edge[barred] node[above] {$J_n$} (m-1-4)
										(m-1-4) edge[transform canvas={yshift=-2pt}] node[below right] {$l$} (M);
				\path				($(m-1-2.south)!0.5!(m-1-3.south)$) edge[cell] node[right] {$\eta$} (M);
				\draw				($(m-1-2)!0.5!(m-1-3)$) node {$\dotsb$};
										
			\end{tikzpicture}
		\end{displaymath}
	\end{lemma}
	\begin{proof}
		Consider the following assignments between collections of nullary cells that are of the form as shown, where $\map{f_n}{X_{np_n}}{A_n}$ denotes the common vertical boundary of $\chi_n$ and $\zeta_1$. The diagram commutes by one of the interchange axioms (\auglemref{1.3}).
		\begin{displaymath}
			\begin{tikzpicture}
				\matrix(m)[math35]{A_n & & B_m \\ & M & \\};
				\path[map]	(m-1-1) edge[barred] node[above] {$\ul H_1 \conc \dotsb \conc \ul H_m$} (m-1-3)
														edge[transform canvas={yshift=-2pt}] node[below left] {$l$} (m-2-2)
										(m-1-3) edge[transform canvas={yshift=-2pt}] node[below right] {$k$} (m-2-2);
				\path				(m-1-2) edge[cell] node[right] {$\phi'$} (m-2-2);
			
				\matrix(m)[math35, xshift=20em]{X_{np_n} & & Y_{mq_m} \\ & M & \\};
				\path[map]	(m-1-1) edge[barred] node[above, xshift=-3pt] {$\ul L_1 \conc \dotsb \conc \ul L_m$} (m-1-3)
														edge[transform canvas={yshift=-2pt}] node[below left] {$l \of f_n$} (m-2-2)
										(m-1-3) edge[transform canvas={yshift=-2pt}] node[below right] {$k \of g_m$} (m-2-2);
				\path				(m-1-2) edge[cell] node[right] {$\psi'$} (m-2-2);
				
				\draw[font=\Large]	(-4.8em,0.3em) node {$\lbrace$}
										(4.6em,0.3em) node {$\rbrace$}
										(14.8em,0.3em) node {$\lbrace$}
										(25em,0.3em) node {$\rbrace$};
				
				\draw				(0em,-6em) node {$\set\phi$}
										(20em,-6em) node {$\set\psi$};
				
				\path[map]	(7em,0.3em) edge node[above] {$\dash \of \ul\zeta$} (12.4em,0.3em)
										(3em,-6em) edge node[below] {$\dash \of (\ul\chi \conc \ul\zeta)$} (17em,-6em)
										(0em,-3em) edge node[left] {$\eta \hc \dash$} (0em,-4.5em)
										(20em,-3em) edge node[right] {$(\eta \of \ul\chi) \hc \dash$} (20em, -4.5em);
			\end{tikzpicture}
		\end{displaymath}
		The horizontally drawn assignments are bijective because the paths $\ul\zeta$ and $\ul\chi \conc \ul\zeta$ are assumed to be weakly nullary"/cocartesian, so that the proof follows.
	\end{proof}
	
	\begin{example} \label{left Kan extensions when all composites exist}
		Consider an augmented virtual double category $\K$ that admits all horizontal composites (\defref{pointwise right cocartesian path}). Applying the lemma to the cocartesian cells $\zeta = \cocart$ defining these composites we find that the universal property of a left Kan cell $\cell\eta{\ul J}M$ in $\K$, as given in \defref{left Kan extension}, reduces to the requirement that cells of the form $\cell\phi{\ul J \conc \ul H}M$ with $\ul H$ of length $\lns{\ul H} = 0$ or $1$ factor uniquely through $\eta$.
	
		If $\K$ has horizontal units too, i.e.\ it is induced by a pseudo double category (see \augpropref{7.8}), then the universal property need only be checked for cells of the form $\cell\phi{\ul J\conc \ul H}M$ with $\lns{\ul H} = 1$. It follows that the present notion of left Kan extension in $\K$ coincides with that of `pointwise left Kan extension' in $\K$ regarded as a pseudo double category, in the sense of Definition~3.10 of \cite{Koudenburg14a}.
		
		Finally consider a `proarrow equipment' $\map{(\dash)_*}\K\M$ in the sense of Wood \cite{Wood85} (where $\M$ need not be biclosed, as is required in \cite{Wood82}). As is shown in Proposition~C.3 of \cite{Shulman08}, if $\K$ is a strict $2$"/category then $(\dash)_*$ induces a pseudo double category $\D$ that has all companions and conjoints (a `framed bicategory'). The objects and vertical morphisms of $\D$ are those of $\K$, and the horizontal morphisms are those of $\M$. In such proarrow equipments Wood's notion of `indexed limit', given in Section~2 of \cite{Wood82}, coincides with the notion of `pointwise left Kan extension' in the corresponding pseudo double category $\D$, in the sense of \cite{Koudenburg14a} (see its Section~3.5), and hence coincides with our notion of left Kan extension.
	\end{example}
	
	\subsection{Consequences of the pasting lemmas}
	The remainder of this section consists of consequences of the pasting lemmas.
	\begin{corollary} \label{restriction as left Kan extension}
		Let $\map jBA$ be a vertical morphism. The nullary cartesian cell below, that defines the conjoint of $j$, is absolutely pointwise left Kan (\defref{absolutely left Kan}).
		\begin{displaymath}
			\begin{tikzpicture}
				\matrix(m)[math35, column sep={1.75em,between origins}]{A & & B \\ & A & \\};
				\path[map]	(m-1-1) edge[barred] node[above] {$j^*$} (m-1-3)
										(m-1-3) edge[transform canvas={xshift=2pt}] node[right] {$j$} (m-2-2);
				\path				(m-1-1) edge[eq, transform canvas={xshift=-2pt}] (m-2-2);
				\draw				([yshift=0.333em]$(m-1-2)!0.5!(m-2-2)$) node[font=\scriptsize] {$\cart$};
			\end{tikzpicture}
		\end{displaymath}
	\end{corollary}
	\begin{proof}
		To see that the cartesian cell above is absolutely left Kan we have to show that $g \of \cart$ is left Kan for any $\map gAN$. This follows from applying the vertical pasting lemma to the identity $g \of \cart \of \cocart = \id_{g \of j}$ which is the conjoint identity for $j^*$ (\lemref{companion identities lemma}) composed with $g$, and where $\id_{g \of j}$ is left Kan by \exref{vertical cells defining left Kan extensions}. That the cartesian cell is in fact absolutely pointwise left Kan follows from the fact that $j^*(\id, f) \iso (j \of f)^*$ for any $\map fCB$; see \lemref{pasting lemma for cartesian cells}.
	\end{proof}
	
	Combined with the horizontal pasting lemma the previous result implies the following.
	\begin{corollary} \label{Kan extension and conjoints}
		The composite on the left below is (pointwise) (weak) left Kan precisely if the cell $\zeta$ is so. The composite on the right is pointwise left Kan as soon as the cell $\eta$ is left Kan.
		\begin{displaymath}
			\begin{tikzpicture}[baseline]
				\matrix(m)[math35, column sep={1.75em,between origins}]
					{ C & & A & & A_1 & & A_{n'} & & A_n \\
						& C & & & & & & & \\};
				\draw				([yshift=-6.5em]$(m-1-1)!0.5!(m-1-9)$)	node (M) {$M$};
				\path[map]	(m-1-1) edge[barred] node[above] {$h^*$} (m-1-3)
										(m-1-3) edge[barred] node[above] {$J_1$} (m-1-5)
														edge[transform canvas={xshift=2pt}]	node[right] {$h$} (m-2-2)
										(m-1-7) edge[barred] node[above] {$J_n$} (m-1-9)
										(m-1-9) edge[transform canvas={yshift=-3pt}] node[below right] {$l$} (M)
										(m-2-2) edge[transform canvas={yshift=-2pt}] node[below left] {$d$} (M);
				\path				(m-1-1) edge[eq, transform canvas={xshift=-2pt}]	(m-2-2)
										($(m-1-1.south)!0.5!(m-1-9.south)$) edge[cell, transform canvas={yshift=-1.75em}] node[right] {$\zeta$} ($(m-2-1.north)!0.5!(m-2-9.north)$);
				\draw				($(m-1-5)!0.5!(m-1-7)$) node {$\dotsb$}
										([yshift=0.333em]$(m-1-2)!0.5!(m-2-2)$) node[font=\scriptsize] {$\cart$};
			\end{tikzpicture} \qquad\quad \begin{tikzpicture}[baseline]
				\matrix(m)[math35, column sep={1.75em,between origins}]
					{	A_0 & & A_1 & & A_{n'} & & A_n & & B \\
						&	A_0 & & A_1 & & A_{n'} & & A_n & \\};
				\draw				([yshift=-3.25em]$(m-2-2)!0.5!(m-2-8)$) node (M) {$M$};
				\path[map]	(m-1-1) edge[barred] node[above] {$J_1$} (m-1-3)
										(m-1-5) edge[barred] node[above] {$J_n$} (m-1-7)
										(m-1-7) edge[barred] node[above] {$f^*$} (m-1-9)
										(m-1-9) edge[transform canvas={xshift=2pt}] node[right] {$f$} (m-2-8)
										(m-2-2) edge[barred] node[above] {$J_1$} (m-2-4)
														edge[transform canvas={yshift=-3pt}] node[below left] {$d$} (M)
										(m-2-6) edge[barred] node[above] {$J_n$} (m-2-8)
										(m-2-8) edge[transform canvas={yshift=-3pt}] node[below right] {$l$} (M);
				\path				(m-1-1) edge[eq, transform canvas={xshift=-2pt}] (m-2-2)
										(m-1-3) edge[eq, transform canvas={xshift=-2pt}] (m-2-4)
										(m-1-5) edge[eq, transform canvas={xshift=-2pt}] (m-2-6)
										(m-1-7) edge[eq, transform canvas={xshift=-2pt}] (m-2-8)
										($(m-2-2.south)!0.5!(m-2-8.south)$) edge[cell] node[right] {$\eta$} (M);
				\draw				($(m-1-3)!0.5!(m-2-6)$) node {$\dotsb$}
										([yshift=0.333em]$(m-1-8)!0.5!(m-2-8)$) node[font=\scriptsize] {$\cart$};
			\end{tikzpicture}
		\end{displaymath}
	\end{corollary}
	
	The following result was used in the proof of \cororef{left Kan extensions with unital sources}.
	\begin{proposition} \label{left Kan extensions preserved by restriction}
		The composite below is left Kan as soon as the cell $\eta$ is so and the conjoint $\hmap{f^*}{A_n}B$ exists.
		\begin{displaymath}
			\begin{tikzpicture}[baseline]
				\matrix(m)[math35]{A_0 & A_1 & A_{n'} & B \\ A_0 & A_1 & A_{n'} & A_n \\};
				\draw	([yshift=-6.5em]$(m-1-1)!0.5!(m-1-4)$) node (M) {$M$};
				\path[map]	(m-1-1) edge[barred] node[above] {$J_1$} (m-1-2)
														
										(m-1-3) edge[barred] node[above, xshift=-3pt] {$J_n(\id, f)$} (m-1-4)
										(m-1-4) edge node[right] {$f$} (m-2-4)
										(m-2-4) edge[transform canvas={yshift=-2pt}] node[below right] {$l$} (M)
										(m-2-1) edge[barred] node[below, inner sep=2.5pt] {$J_1$} (m-2-2)
														edge[transform canvas={yshift=-2pt}] node[below left] {$d$} (M)
										(m-2-3) edge[barred] node[below, inner sep=2.5pt] {$J_n$} (m-2-4);
				\path				(m-1-1) edge[eq] (m-2-1)
										(m-1-2) edge[eq] (m-2-2)
										(m-1-3) edge[eq, transform canvas={xshift=-2pt}] (m-2-3);
				\path				($(m-2-2.south)!0.5!(m-2-3.south)$) edge[cell] node[right] {$\eta$} (M);
				\draw				($(m-1-2)!0.5!(m-2-3)$) node {$\dotsb$}
										($(m-1-3)!0.5!(m-2-4)$) node[font=\scriptsize] {$\cart$};
			\end{tikzpicture}
		\end{displaymath}
	\end{proposition}
	\begin{proof}
		\auglemref{8.1} supplies a cocartesian horizontal cell $\cell\phi{(J_n, f^*)}{J_n(\id, f)}$ which, when composed with the composite above, gives the composite on the right of \cororef{Kan extension and conjoints}. The latter is left Kan because $\eta$ is so and the proof follows from applying the vertical pasting lemma to the path $(\id_{J_1}, \dotsc, \id_{J_{n'}}, \phi)$.
	\end{proof}
	
	The following propositions generalise well known results for Kan extensions internal in a $2$-category and enriched Kan extensions. The first of these generalises Proposition~22 of \cite{Street74b} (for internal Kan extensions) and Proposition~4.23 of \cite{Kelly82} (for enriched Kan extensions); see also Proposition~14 of \cite{Wood82} for the analogous result in proarrow equipments. Remember that a vertical morphism $\map fAC$ is called full and faithful (\defref{full and faithful morphism}) whenever its identity cell $\id_f$ is cartesian.
	\begin{proposition} \label{pointwise left Kan extension along full and faithful map}
		Consider a (weakly) left Kan cell $\eta$ as in the composite below and assume it restricts along the morphism $\map f{A_n}B$ (\defref{pointwise left Kan extension}). If $f$ is full and faithful and the restriction $B(f, f)$ exists then the composite is again (weakly) left Kan; in particular if $n = 0$ then the composite, in that case a vertical cell, is invertible.
		\begin{displaymath}
			\begin{tikzpicture}
				\matrix(m)[math35, column sep={1.75em,between origins}]
					{	& A_0 & & A_1 & & A_{n'} & & A_n & \\
						A_0 & & A_1 & & A_{n'} & & A_n & & B \\};
				\draw				([yshift=-3.25em]$(m-2-1)!0.5!(m-2-9)$) node (M) {$M$};
				\path[map]	(m-1-2) edge[barred] node[above] {$J_1$} (m-1-4)
										(m-1-6) edge[barred] node[above] {$J_n$} (m-1-8)
										(m-1-8) edge[transform canvas={xshift=2pt}]	node[right] {$f$} (m-2-9)
										(m-2-1) edge[barred] node[above] {$J_1$} (m-2-3)
														edge[transform canvas={yshift=-2pt}] node[below left] {$d$} (M)
										(m-2-5) edge[barred] node[above] {$J_n$} (m-2-7)
										(m-2-7) edge[barred] node[below, inner sep=2pt] {$f_*$} (m-2-9)
										(m-2-9) edge[transform canvas={yshift=-2pt}] node[below right] {$l$} (M);
				\draw				([xshift=-2pt]$(m-1-4)!0.5!(m-2-5)$) node {$\dotsb$}
										([yshift=-0.5em]$(m-1-8)!0.5!(m-2-8)$) node[font=\scriptsize] {$\cocart$};
				\path				($(m-2-1.south)!0.5!(m-2-9.south)$) edge[cell] node[right] {$\eta$} (M)
										(m-1-2) edge[eq, transform canvas={xshift=-2pt}] (m-2-1)
										(m-1-4) edge[eq, transform canvas={xshift=-2pt}] (m-2-3)
										(m-1-6) edge[eq, transform canvas={xshift=-2pt}] (m-2-5)
										(m-1-8)	edge[eq, transform canvas={xshift=-2pt}] (m-2-7);
			\end{tikzpicture}
		\end{displaymath}
	\end{proposition}
	If $f$ is full and faithful then $B(f, f)$ is the horizontal unit of $A_n$ by \auglemref{5.14} so that, in an augmented virtual double category with restrictions on the right, the composite above is in fact pointwise (weakly) left Kan by \cororef{left Kan extensions with unital sources}.
	\begin{proof}
		Assume that $f$ is full and faithful. Writing $\cell\chi{B(f, f)}B$ for the cartesian cell defining $B(f, f)$ consider the unique factorisation $\id_f = \chi \of \id_f'$, where the horizontal cell $\cell{\id_f'}{A_n}{B(f, f)}$ is cartesian and cocartesian by \lemref{companion identities lemma}. Factorising both sides of the latter identity through the cartesian cell $f_* \Rar B$ corresponding to the cocartesian cell above, in the sense of \lemref{companion identities lemma}, we obtain $\cocart = \chi' \of \id_f'$ where $\cell{\chi'}{B(f, f)}{f_*}$ is cartesian by the pasting lemma (\lemref{pasting lemma for cartesian cells}) and $\id_f'$ is cocartesian. The main assertion now follows from the assumption that $\eta$ restricts along $f$, so that $\eta \of (\id_{\ul J} \conc \chi')$ is (weakly) left Kan, and the application of the vertical pasting lemma (\lemref{vertical pasting lemma}) to the path $\id_{\ul J} \conc \id_f'$ in $\eta \of (\id_{\ul J} \conc \chi') \of (\id_{\ul J} \conc \id_f') = \eta \of (\id_{\ul J} \conc \cocart)$. That the composite is invertible in the case of $n = 0$ follows immediately from \exref{vertical cells defining left Kan extensions}.
	\end{proof}
	
	When applied to the unital virtual double categories $Q(\mathcal C)$, of quintets in a $2$"/category $\mathcal C$ (see \exref{Kan extensions in quintets}), or $\enProf\V$, of $\V$"/profunctors (see \exref{enriched left Kan extension}), the following result reduces to the classical description of right adjoints as left Kan extensions, see e.g.\ Example~2.17 of \cite{Weber07} and Theorem~4.81 of \cite{Kelly82} respectively.
	\begin{proposition} \label{adjunctions in terms of left Kan extensions}
		In an augmented virtual double category $\K$ consider the factorisation below.
		\begin{displaymath}
			\begin{tikzpicture}[textbaseline]
				\matrix(m)[math35, column sep={1.75em,between origins}]{& A & \\ & & C \\ & A & \\};
				\path[map]	(m-1-2) edge[bend left = 18] node[right] {$f$} (m-2-3)
										(m-2-3) edge[bend left = 18] node[right] {$g$} (m-3-2);
				\path				(m-1-2) edge[bend right = 45, eq] (m-3-2);
				\path[transform canvas={yshift=-1.625em}]	(m-1-2) edge[cell] node[right] {$\iota$} (m-2-2);
			\end{tikzpicture} = \begin{tikzpicture}[textbaseline]
    		\matrix(m)[math35, column sep={1.75em,between origins}]{& A & \\ A & & C \\ & A & \\};
    		\path[map]	(m-1-2) edge[transform canvas={xshift=2pt}] node[right] {$f$} (m-2-3)
    								(m-2-1) edge[barred] node[below, inner sep=2pt] {$f_*$} (m-2-3)
    								(m-2-3)	edge[transform canvas={xshift=2pt}] node[right] {$g$} (m-3-2);
    		\path				(m-1-2) edge[eq, transform canvas={xshift=-2pt}] (m-2-1)
    								(m-2-1) edge[eq, transform canvas={xshift=-2pt}] (m-3-2);
    		\path				(m-2-2) edge[cell, transform canvas={yshift=0.1em}]	node[right, inner sep=4pt] {$\iota'$} (m-3-2);
    		\draw				([yshift=-0.5em]$(m-1-2)!0.5!(m-2-2)$) node[font=\scriptsize] {$\cocart$};
  		\end{tikzpicture}
		\end{displaymath}
		The following are equivalent:
		\begin{enumerate}[label=\textup{(\alph*)}]
			\item $\iota$ is the unit of an adjunction $f \ladj g$ in the $2$"/category $V(\K)$ (\augexref{1.5});
			\item $\iota'$ is cartesian in $\K$ (defining $f_*$ as the conjoint of $g$);
			\item $\iota'$ is weakly left Kan in $\K$ and is preserved by $f$ (\defref{absolutely left Kan});
			\item $\iota'$ is absolutely pointwise left Kan in $\K$ (\defref{absolutely left Kan}).
		\end{enumerate}
		Under these conditions $f$ is full and faithful (\defref{full and faithful morphism}) if $\iota$ is invertible; the converse holds whenever the restriction $C(f, f)$ exists.
	\end{proposition}
	\begin{proof}
		We have (a) $\Rightarrow$ (b) by \auglemref{5.16} and (b) $\Rightarrow$ (d) by \cororef{restriction as left Kan extension} while clearly (d) $\Rightarrow$ (c). That (c) $\Rightarrow$ (a) follows from combining \propref{weak left Kan extensions along companions} with Example~2.17 of \cite{Weber07}. For the final assertion notice that the invertible unit cell $\iota$ is cartesian (\augexref{4.4}) so that, composing with the counit $\cell\eps{f \of g}{\id_C}$, the identity cell $\id_f = (f \of \iota) \hc (\eps \of f)$ is cartesian too by \auglemref{4.17}. The converse follows from \propref{pointwise left Kan extension along full and faithful map}.
	\end{proof}
	
	\begin{example} \label{left Kan extension along the companion of a left adjoint}
		Applying the first assertion of \cororef{Kan extension and conjoints} to the cell $\iota'$ of the previous proposition we obtain the following. Consider an adjoint pair \mbox{$\map{f \ladj g}{A_0}C$}, a morphism $\map dCM$ and a path $\hmap{\ul J}{A_0}{A_n}$. If the companion $\hmap{f_*}C{A_0}$ exists then the (pointwise) (weak) left Kan extension of $d \of g$ along $\ul J$ exists if and only if that of $d$ along $f_* \conc \ul J$ does, and in that case they are isomorphic.
	\end{example}
	
	\section{Pointwise Kan extension in terms of pointwise weak Kan extension} \label{pointwise Kan extensions section}
	In \thmref{pointwise Kan extensions in terms of pointwise weak Kan extensions} below we prove that the notions of pointwise weak left Kan extension and pointwise left Kan extension (\defref{pointwise left Kan extension}) coincide in augmented virtual double categories $\K$ that have restrictions on the right (\defref{augmented virtual equipment}) as well as `left nullary"/cocartesian paths of $(0,1)$"/ary cells'; the latter in the sense of the definition below, which is a strengthening of \augdefref{7.10}. Together with \remref{pointwise left Kan extension in the presence of horizontal units and restrictions on the right} and \exref{left Kan extensions when all composites exist} this result recovers Theorem~5.11 of \cite{Koudenburg14a}. In most of our examples left nullary"/cocartesian paths of $(0, 1)$"/ary cells can be obtained by ``concatenating'' cocartesian universal cells that define `tabulations', in the sense of \defref{tabulation} below; this is explained in \cororef{cocartesian tabulations imply cocartesian paths of (0,1)-ary cells}. \propref{pointwise left Kan extensions along companions} shows that if $\K$ has such `cocartesian tabulations' then pointwise left Kan extension along a companion $j_*$ in $\K$ coincides with pointwise left Kan extension along $j$ in the vertical $2$"/category $V(\K)$ (\augexref{1.5}); the latter in the classical sense of \cite{Street74b}.
	
	\subsection{Cocartesian paths of $(0,1)$"/ary cells}
	\begin{definition} \label{cocartesian path of (0,1)-ary cells}
		Let $\hmap{\ul J = (J_1, \dotsc, J_n)}{A_0}{A_n}$ be a path of horizontal morphisms. A \emph{(left, right or weakly) (nullary"/)cocartesian path of $(0,1)$"/ary cells} for $\ul J$ consists of an object $X$ together with a (left, right or weakly) (nullary"/)cocartesian path (\defref{cocartesian path}) $\ul\phi = (\phi_1, \dotsc, \phi_n)$ of $(0,1)$"/ary cells $\phi_i$ as on the left below. An augmented virtual double category $\K$ is said to \emph{have (left, right or weakly) (nullary"/)cocartesian paths of $(0,1)$"/ary cells} if every path $\ul J \in \K$ admits a (left, right or weakly) (nullary"/)cocartesian path $\ul\phi$ of $(0,1)$"/ary cells.
		\begin{displaymath}
			\begin{tikzpicture}[baseline]
  			\matrix(m)[math35, column sep={1.75em,between origins}]{& X & \\ A_{i'} & & A_i \\};
				\path[map]	(m-1-2) edge[transform canvas={xshift=-2pt}] node[left] {$f_{i'}$} (m-2-1)
														edge[transform canvas={xshift=2pt}] node[right] {$f_i$} (m-2-3)
										(m-2-1) edge[barred] node[below] {$J_i$} (m-2-3);
				\path				(m-1-2) edge[cell, transform canvas={yshift=-0.25em}] node[right, inner sep=2.5pt] {$\phi_i$} (m-2-2);
			\end{tikzpicture} \qquad \qquad \qquad \qquad \qquad \qquad \begin{tikzpicture}[baseline]
				\matrix(m)[math35, column sep={1.75em,between origins}]{A & & B \\ & C & \\};
				\path[map]	(m-1-1) edge[barred] node[above] {$J$} (m-1-3)
														edge[transform canvas={xshift=-2pt}] node[left] {$f$} (m-2-2)
										(m-1-3) edge[transform canvas={xshift=2pt}] node[right] {$g$} (m-2-2);
				\path[transform canvas={yshift=0.25em}]	(m-1-2) edge[cell] node[right] {$\psi$} (m-2-2);
			\end{tikzpicture}
		\end{displaymath}
		
		Vertically dual, a \emph{cartesian nullary cell} for a horizontal morphism $\hmap JAB$ is a cartesian cell $\psi$ as on the right above. An augmented virtual double category $\K$ is said to \emph{have cartesian nullary cells} if every horizontal morphism $J \in \K$ admits a cartesian nullary cell.
	\end{definition}
	
	Notice that the cartesian nullary cell $\psi$ on the right above defines $J$ as the nullary restriction $C(f, g)$ (\defref{cartesian cells}). Conversely nullary restrictions, including companions, conjoints and horizontal units, admit cartesian nullary cells, by definition.
	
	\begin{example}
		Given a `(weak) Yoneda morphism' $\map\yon A{\ps A}$, in the sense of \defref{yoneda embedding} below, any horizontal morphism $\hmap JAB$ admits a cartesian nullary cell $J \Rar \ps A$; this is a direct consequence of the `Yoneda axiom' that is satisfied by $\yon$.
	\end{example}
	
	\subsection{Cocartesian tabulations}
	Many augmented virtual double categories $\K$ admit a universal $(0,1)$-ary cell among all $(0,1)$"/ary cells $\cell\phi XJ$ into any fixed horizontal morphism $J$, in the sense of the following definition, which is a direct translation of the double categorical notion of tabulation that was introduced by Grandis and Par\'e in \cite{Grandis-Pare99}. \exrref{tabulations of profunctors}{tabulations of discrete two-sided fibrations} below give examples of (co)tabulations. Often this universal cell is cocartesian so that, in that case, $\K$ admits all cocartesian paths of $(0, 1)$"/ary cells that are of length $1$. In \cororef{cocartesian tabulations imply cocartesian paths of (0,1)-ary cells} below we will see that the latter can be ``concatenated'' to form cocartesian paths of $(0,1)$"/ary cells of any length, provided that $\K$ is an augmented virtual equipment (\defref{augmented virtual equipment}).
	
	\begin{definition} \label{tabulation}
		The \emph{tabulation} $\tab J$ of a horizontal morphism $\hmap JAB$ consists of an object $\tab J$ equipped with a $(0,1)$"/ary cell $\pi$ as on the left below, satisfying the following $1$-dimensional and $2$"/dimensional universal properties.
		\begin{displaymath}
			\begin{tikzpicture}[baseline]
  			\matrix(m)[math35, column sep={1.75em,between origins}]{& \tab J & \\ A & & B \\};
				\path[map]	(m-1-2) edge[transform canvas={xshift=-2pt}] node[left] {$\pi_A$} (m-2-1)
														edge[transform canvas={xshift=2pt}] node[right] {$\pi_B$} (m-2-3)
										(m-2-1) edge[barred] node[below] {$J$} (m-2-3);
				\path				(m-1-2) edge[cell, transform canvas={yshift=-0.25em}] node[right, inner sep=2.5pt] {$\pi$} (m-2-2);
			\end{tikzpicture} \qquad\qquad\qquad\qquad\qquad\qquad \begin{tikzpicture}[baseline]
  			\matrix(m)[math35, column sep={1.75em,between origins}]{& X_0 & \\ A & & B \\};
				\path[map]	(m-1-2) edge[transform canvas={xshift=-2pt}] node[left] {$\phi_A$} (m-2-1)
														edge[transform canvas={xshift=2pt}] node[right] {$\phi_B$} (m-2-3)
										(m-2-1) edge[barred] node[below] {$J$} (m-2-3);
				\path				(m-1-2) edge[cell, transform canvas={yshift=-0.25em}] node[right, inner sep=2.5pt] {$\phi$} (m-2-2);
			\end{tikzpicture}
		\end{displaymath}
		Given another $(0,1)$"/ary cell $\phi$ as on the right above, the $1$-dimensional property states that there exists a unique morphism $\map{\phi'}{X_0}{\tab J}$ such that $\pi \of \phi' = \phi$.
		
		The $2$-dimensional property is the following. Suppose we are given another $(0,1)$"/ary cell $\psi$ as in the identity below, which factors through $\pi$ as $\map{\psi'}{X_n}{\tab J}$, like $\phi$ factors as $\phi'$. Then for any pair of cells $\xi_A$ and $\xi_B$ as in the identity on the left below there exists a unique cell $\xi'$ as in the middle below such that $\pi_A \of \xi' = \xi_A$ and $\pi_B \of \xi' = \xi_B$.

		We call the tabulation $\tab J$ \emph{(left) (nullary"/)cocartesian} whenever its defining cell $\pi$ is (left) (nullary"/)cocartesian (\defref{cocartesian path}).
		\begin{displaymath}
			\begin{tikzpicture}[textbaseline]
				\matrix(m)[math35, row sep={3.75em,between origins}, column sep={2em,between origins}]
					{ X_0 & & X_n & \\ & A & & B \\};
				\path[map]	(m-1-1) edge[barred] node[above] {$\ul H$} (m-1-3)
														edge[transform canvas={xshift=-2pt}] node[left] {$\phi_A$} (m-2-2)
										(m-1-3) edge node[desc] {$\psi_A$} (m-2-2)
														edge[transform canvas={xshift=2pt}] node[right] {$\psi_B$} (m-2-4)
										(m-2-2) edge[barred] node[below] {$J$} (m-2-4);
				\path				($(m-1-1.south)!0.5!(m-1-3.south)$) edge[cell, transform canvas={shift={(-0.5em,0.8em)}}, shorten >= 6.25pt, shorten <= 6.25pt] node[right] {$\xi_A$} (m-2-2)
										(m-1-3) edge[cell, transform canvas={yshift=-0.5em}, shorten >= 6.25pt, shorten <= 6.25pt] node[right] {$\psi$} ($(m-2-2.north)!0.5!(m-2-4.north)$);
			\end{tikzpicture} \quad = \quad \begin{tikzpicture}[textbaseline]
				\matrix(m)[math35, row sep={3.75em,between origins}, column sep={2em,between origins}]
					{ & X_0 & & X_n \\ A & & B & \\};
				\path[map]	(m-1-2) edge[transform canvas={xshift=-2pt}] node[left] {$\phi_A$} (m-2-1)
														edge[barred] node[above] {$\ul H$} (m-1-4)
														edge node[desc] {$\phi_B$} (m-2-3)
										(m-1-4) edge[transform canvas={xshift=2pt}] node[right] {$\psi_B$} (m-2-3)
										(m-2-1) edge[barred] node[below] {$J$} (m-2-3);
				\path				(m-1-2) edge[cell, transform canvas={shift={(-0.3em, -0.8em)}}, shorten >= 6.25pt, shorten <= 6.25pt] node[right] {$\phi$} ($(m-2-1.north)!0.5!(m-2-3.north)$)
										($(m-1-2.south)!0.5!(m-1-4.south)$) edge[cell, transform canvas={yshift=0.5em}, shorten >= 6.25pt, shorten <= 6.25pt] node[right] {$\xi_B$} (m-2-3);
			\end{tikzpicture} \qquad \qquad \begin{tikzpicture}[textbaseline]
  			\matrix(m)[math35, column sep={1.75em,between origins}]
  				{ X_0 & & X_n \\ & \tab J \\};
  			\path[map]  (m-1-1) edge[barred] node[above] {$\ul H$} (m-1-3)
														edge[transform canvas={xshift=-2pt}] node[left] {$\phi'$} (m-2-2)
										(m-1-3)	edge[transform canvas={xshift=2pt}] node[right] {$\psi'$} (m-2-2);
				\path				($(m-1-1.south)!0.5!(m-1-3.south)$) edge[cell, transform canvas={yshift=0.25em}] node[right] {$\xi'$} (m-2-2);
			\end{tikzpicture} \qquad \begin{tikzpicture}[textbaseline]
  			\matrix(m)[math35, column sep={1.75em,between origins}]
  				{ A & & B \\ & \brks J \\};
  			\path[map]  (m-1-1) edge[barred] node[above] {$J$} (m-1-3)
														edge[transform canvas={xshift=-2pt}] node[left] {$\sigma_A$} (m-2-2)
										(m-1-3)	edge[transform canvas={xshift=2pt}] node[right] {$\sigma_B$} (m-2-2);
				\path				($(m-1-1.south)!0.5!(m-1-3.south)$) edge[cell, transform canvas={yshift=0.25em}] node[right] {$\sigma$} (m-2-2);
			\end{tikzpicture}
		\end{displaymath}
		
		Vertically dual, the \emph{cotabulation} $\brks J$ of $J$ is defined by a nullary cell $\sigma$ as on the right above, satisfying 1-dimensional and 2-dimensional universal properties that are vertical dual to those for $\tab J$. We call $\brks J$ \emph{cartesian} whenever $\sigma$ is cartesian.
	\end{definition}
	
	The following is a direct consequence of the pasting lemma for cocartesian paths of $(0,1)$"/ary cells, \lemref{pasting lemma for cocartesian paths of (0,1)-ary cells} below.
	\begin{corollary} \label{cocartesian tabulations imply cocartesian paths of (0,1)-ary cells}
		An augmented virtual double category has all left (nullary"/)cocartesian paths of $(0,1)$"/ary cells (\defref{cocartesian path of (0,1)-ary cells}) whenever it has all left (nullary"/)cocartesian tabulations and all restrictions on the right (\defref{augmented virtual equipment}). An augmented virtual equipment (\defref{augmented virtual equipment}) has all (nullary"/)cocartesian paths of $(0,1)$"/ary cells whenever it has all (nullary"/)cocartesian tabulations.
	\end{corollary}
	
	\subsection{Examples of tabulations}
	\begin{example} \label{tabulations of profunctors}
		In the unital virtual equipment $\enProf{\Set}$ (\augexref{2.4}), of $\Set$"/profunctors between locally small categories, the tabulation $\tab J$ is the well known \emph{graph} of $\hmap JAB$ as follows. It has triples $(x, u, y)$ as objects, where \mbox{$(x, y) \in A \times B$} are objects and $u \in J(x, y)$, while a morphism $(x, u, y) \to (x', u', y')$ is a pair $\map{(s, t)}{(x,y)}{(x', y')}$ in $A \times B$ such that $\lambda(s, u') = \rho(u, t)$ in $J(x, y')$, where $\lambda$ and $\rho$ denote the actions of $A$ and $B$ on $J$. The functors $\pi_A$ and $\pi_B$ are the projections while the cell $\nat \pi{\tab J}J$ maps $(x, u, y)$ to $u \in J(x, y)$. It is straightforward to check that $\pi$ satisfies the universal properties above, and that it is cocartesian. Cocartesian tabulations in the augmented virtual equipment $\enProf{(\Set, \Set')}$, of $\Set$"/profunctors between $\Set'$"/categories (\augexref{2.6}), are constructed as graphs in the same way. We conclude that $\enProf{\Set}$ and $\enProf{(\Set, \Set')}$ have all cocartesian paths of $(0,1)$"/ary cells.
	\end{example}
	
	\begin{example} \label{tabulations of 2-profunctors}
		Let $\Cat \dfn \inCat{\Set}$ (\augexref{2.9}) denote the category of small categories. In the unital virtual equipment $\enProf{\Cat}$ (\augexref{2.4}) the tabulation $\tab J$ of a $2$"/profunctor (that is a $\Cat$"/enriched profunctor) $\hmap JAB$, where $A$ and $B$ are locally small $2$"/categories, is constructed as follows. It has as underlying category $\tab J_0$ the graph of the profunctor $J_0$ underlying $J$, whose images $J_0(x, y)$ are the sets of objects of the categories $J(x, y)$, for all $x \in A$ and $y \in B$. The cells $(s, t) \Rar (s', t')$ of $\tab J$ are pairs $(\delta, \eps)$ of cells \mbox{$\cell\delta s{s'}$} in $A$ and $\cell\eps t{t'}$ in $B$ as in the diagram on the left below such that $\lambda(\delta, u') = \rho(u, \eps)$ in $J(x, y')$. Tabulations in the augmented virtual equipment $\enProf{(\Cat, \Cat')}$, of $2$"/profunctors between (possibly locally large) $2$"/categories (\augexref{2.7}), are constructed in the same way.
		\begin{displaymath}
			\begin{tikzpicture}[textbaseline]
				\matrix(m)[math35]{x & y \\ x' & y' \\};
				\path[map]	(m-1-1) edge node[above] {$u$} (m-1-2)
														edge[bend right=35] node[left] {$s$} (m-2-1)
														edge[bend left=35] node[right, inner sep=2pt] {$s'$} (m-2-1)
										(m-1-2) edge[bend right=35] node[left, inner sep=2pt] {$t$} (m-2-2)
														edge[bend left=35] node[right] {$t'$} (m-2-2)
										(m-2-1) edge node[below] {$u'$} (m-2-2);
				\path				([xshift=-0.8em]$(m-1-1)!0.5!(m-2-1)$) edge[cell] node[above] {$\delta$} ([xshift=0.9em]$(m-1-1)!0.5!(m-2-1)$)
										([xshift=-0.8em]$(m-1-2)!0.5!(m-2-2)$) edge[cell] node[above] {$\eps$} ([xshift=0.9em]$(m-1-2)!0.5!(m-2-2)$);
			\end{tikzpicture} \qquad\qquad J(*,*) = \bigpars{\begin{tikzpicture}[textbaseline, ampersand replacement=\&]
				\matrix(m)[math35, column sep=1em]{u \& v \\};
				\path[map]	(m-1-1) edge (m-1-2);
			\end{tikzpicture}} \qquad\qquad K(*,*) = \bigpars{\begin{tikzpicture}[textbaseline, ampersand replacement=\&]
				\matrix(m)[math35, column sep=1em]{u' \& v' \\};
			\end{tikzpicture}}
		\end{displaymath}

		Tabulations of $2$-profunctors fail to be cocartesian in general. As an example consider the $2$-profunctors $J$ and $\hmap K11$, where $1$ denotes the terminal $2$-category with single object $*$, whose images $J(*,*)$ and $K(*,*)$ are the `interval category' and the discrete category with two objects respectively, as shown above. The tabulation $\tab J$ is discrete with objects $(*, u, *)$ and $(*, v, *)$, so that the assignments \mbox{$(*, u, *) \mapsto u'$} and $(*, v, *) \mapsto v'$ define a cell $\cell\phi{\tab J}K$. It is easily checked that $\phi$ does not factor through $\cell\pi{\tab J}J$, showing that $\pi$ is not weakly cocartesian. As a consequence \propref{pointwise left Kan extensions along companions} below fails to hold in $\enProf{(\Cat, \Cat')}$; see \exref{pointwise is stronger than enriched}.
	\end{example}
	
	\begin{example} \label{cotabulations of enriched profunctors}
		Let $\V' = (\V' ,\tens, I)$ be a monoidal category with initial object $\emptyset$ preserved by the functors $x \tens \dash$ and $\dash \tens x$, for all $x \in \V'$. In the unital virtual equipment $\enProf{\V'}$ (\augexref{2.4}) the cotabulation $\brks J$ of a $\V'$"/profunctor \mbox{$\hmap JAB$} is the \emph{cograph} of $J$, as follows. Its collection of objects is the disjoint union $\ob\brks J \dfn \ob A \djunion \ob B$ while its hom"/objects are given by
		\begin{displaymath}
			\brks J(x, y) \dfn \begin{cases}
				A(x, y) & \text{if $x, y \in A$;} \\
				J(x, y) & \text{if $x \in A$ and $y \in B$;} \\
				B(x, y) & \text{if $x, y \in B$;} \\
				\emptyset & \text{otherwise.}
			\end{cases}
		\end{displaymath}
		Composition in $\brks J$ is induced by composition in $A$ and $B$ as well as the actions of $A$ and $B$ on $J$. Taking $\sigma_A$ and $\sigma_B$ to be the embeddings of $A$ and $B$ into $\brks J$, the universal cell $\cell\sigma J{\brks J}$ is simply given by the identities on the $\V'$"/objects $J(x, y)$. It is straightforward to check that $\sigma$ satisfies the universal properties and that it is cartesian. We conclude that $\enProf{\V'}$ has all cartesian cotabulations and that, applying \lemref{full and faithful functors reflect tabulations} below to $\enProf{(\V, \V')} \hookrightarrow \enProf{\V'}$, so does the augmented virtual equipment $\enProf{(\V, \V')}$ (\augexref{2.7}). In greater generality, in \cite{Street80b} cographs of ``$\V$"/gamuts'' are used to characterise $\V$"/profunctors; see paragraph 6.10 and Corollary~6.16 therein.
	\end{example}
	
	\begin{example} \label{tabulations of internal profunctors}
		Let $\E$ be a category with pullbacks. The unital virtual equipment $\inProf\E$ of internal profunctors in $\E$ (\augexsref{2.9}{4.9}) has cocartesian tabulations as follows, and hence has cocartesian paths of $(0,1)$"/ary cells. In Proposition~5.15 of \cite{Koudenburg14a}, where $\E$ is assumed to have coequalisers preserved by pullback so that $\inProf\E$ has all horizontal composites (\augexref{7.5}), the tabulation $\tab J$ of $\hmap JAB$ in $\inProf\E$, with underlying span $A \leftarrow J \rightarrow B$, was constructed as follows. It is the internal category that has $J$ as its object of objects and, as its object of morphisms, the pullback $\tab J$ below, where $\lambda$ and $\rho$ denote the actions of the objects of morphisms $\alpha$ and $\beta$, of the internal categories $A$ and $B$, on $J$. The source and target morphisms of $\tab J$ are the composite projections $d_0 = \bigbrks{\tab J \to J \times_B \beta \to J}$ and $d_1 = \bigbrks{\tab J \to \alpha \times_A J \to J}$. If $\E = \Set$ then the latter recovers $\tab J$ as the graph of $J$ (\exref{tabulations of profunctors}).
		\begin{displaymath}
			\begin{tikzpicture}
			\matrix(m)[math35, column sep=1em]{\tab J & J \times_B \beta \\ \alpha \times_A J & J \\};
				\path[map]	(m-1-1)	edge (m-1-2)
														edge (m-2-1)
										(m-1-2) edge node[right] {$\rho$} (m-2-2)
										(m-2-1) edge node[below] {$\lambda$} (m-2-2);
				\coordinate (hook) at ($(m-1-1)+(0.4,-0.4)$);
				\draw (hook)+(0,0.17) -- (hook) -- +(-0.17,0);
			\end{tikzpicture}
		\end{displaymath}
		The argument given in \cite{Koudenburg14a}, proving that $\tab J$ is an internal category which forms the tabulation of $J$, carries over to the general case, with $\E$ not necessarily having coequalisers, with minor adjustments. In particular the universal cell $\cell\pi {\tab J}J$, as a $(0,1)$"/ary cell, is given by the identity on $J$ and it is straightforward to show that this makes $\pi$ cocartesian in the unital virtual double category $\inProf\E$.
	\end{example}
	
	\begin{example} \label{tabulations of internal modular relations}
		Let $\E$ be a category with finite limits. The unital virtual equipment $\ModRel(\E)$ of internal modular relations in $\E$ (\exref{internal modular relations}) has cocartesian tabulations as follows, and hence has cocartesian paths of $(0,1)$"/ary cells. By applying the composite $2$"/functor $N \of \Mod$ of \augexref{2.2} to the embedding $\Rel(\E) \hookrightarrow \Span\E$ we obtain an embedding $\ModRel(\E) \hookrightarrow \Prof(\E)$. We claim that the latter creates cocartesian tabulations so that, like $\Prof(\E)$ by the previous example, $\ModRel(\E)$ has cocartesian tabulations. Since the embedding is full and faithful and preserves cartesian cells, to prove the claim it suffices, using \lemref{full and faithful functors reflect tabulations} below, to show that for any internal modular relation $\hmap JAB$ in $\ModRel(\E)$ the tabulation $\tab J$ in $\Prof(\E)$, as constructed in the previous example, is an internal preorder in $\E$ (\exref{internal modular relations}). We do so below; as an aside we remark that one can also show that if $A$ and $B$ are internal partial orders (\exref{internal modular relations}) then so is $\tab J$.
		
		We have to show that the span $J \xlar{d_0} \tab J \xrar{d_1} J$ underlying $\tab J$, with legs as described in the previous example, is a relation in $\E$, that is $d_0$ and $d_1$ are jointly monic. To do so consider any parallel pair $\phi$, $\map\psi X{\tab J}$ of morphisms in $\E$ such that $d_i \of \phi = d_i \of \psi$ for $i = 0, 1$. Writing $p^{\tab J}_\alpha$ for the composite projection \mbox{$p^{\tab J}_\alpha \dfn \bigbrks{\tab J \to \alpha \times_A J \to \alpha}$} one checks that $\alpha_0 \of p^{\tab J}_\alpha \of \phi = \alpha_0 \of p^{\tab J}_\alpha \of \psi$ follows from the fact that $\lambda$ and $\rho$ are morphisms of spans, the definition of $\tab J$ as a pullback, the definition of $d_0$, and the assumption on $\phi$ and $\psi$. Similarly $\alpha_1 \of p^{\tab J}_\alpha \of \phi = \alpha_1 \of p^{\tab J}_\alpha \of \psi$ so that $p^{\tab J}_\alpha \of \phi = p^{\tab J}_\alpha \of \psi$ follows from the joint monicity of $\alpha_0$ and $\alpha_1$. Together with $d_1 \of \phi = d_1 \of \psi$ we conclude that $p^{\tab J}_{\alpha \times_A J} \of \phi = p^{\tab J}_{\alpha \times_A J} \of \psi$ where $\map{p^{\tab J}_{\alpha \times_A J}}{\tab J}{\alpha \times_A J}$ is the projection. Similarly also $p^{\tab J}_{J \times_B \beta} \of \phi = p^{\tab J}_{J \times_B \beta} \of \psi$; together the latter imply $\phi = \psi$ as required.
	\end{example}
	
	\begin{example} \label{tabulations of discrete two-sided fibrations}
		The unital virtual equipment $\dFib{\mathcal C}$ of discrete two"/sided fibrations in a finitely complete $2$"/category $\mathcal C$ (\exref{discrete two-sided fibrations}) has cocartesian tabulations that are preserved by the embeddings $\dFib{\mathcal C} \hookrightarrow \spFib{\mathcal C} \hookrightarrow \inProf{\und{\mathcal C}}$, as follows. Given any discrete two"/sided fibration $\hmap JAB$ in $\dFib{\mathcal C}$ one can show that the square in $\mathcal C$ on the left below commutes and that it is a pullback square; compare the definition of internal discrete two"/sided fibration given in Section~8 of \cite{Street17}.
		\begin{displaymath}
			\begin{tikzpicture}[baseline]
				\matrix(m)[math35, column sep=1em]{J^\2 & J \times_B B^\2 \\ A^\2 \times_A J & J \\};
				\path[map]	(m-1-1)	edge node[above] {$(d_0, j_B^\2)$} (m-1-2)
														edge node[left] {$(j_A^\2, d_1)$} (m-2-1)
										(m-1-2) edge node[right] {$\rho$} (m-2-2)
										(m-2-1) edge node[below] {$\lambda$} (m-2-2);
				\coordinate (hook) at ($(m-1-1)+(0.4,-0.4)$);
			\end{tikzpicture} \quad \qquad \qquad \qquad \begin{tikzpicture}[baseline]
  			\matrix(m)[math35, column sep={1.75em,between origins}]{& J & \\ A & & B \\};
				\path[map]	(m-1-2) edge[transform canvas={xshift=-2pt}] node[left] {$j_A$} (m-2-1)
														edge[transform canvas={xshift=2pt}] node[right] {$j_B$} (m-2-3)
										(m-2-1) edge[barred] node[below] {$J$} (m-2-3);
				\path				(m-1-2) edge[cell, transform canvas={yshift=-0.25em}] node[right, inner sep=2.5pt] {$\pi$} (m-2-2);
			\end{tikzpicture} \qquad \qquad \qquad \begin{tikzpicture}[baseline]
  			\matrix(m)[math35, column sep={1.75em,between origins}]{& X^\2 & \\ A^\2 & & B^\2 \\};
				\path[map]	(m-1-2) edge[transform canvas={xshift=-2pt}] node[left] {$\phi_A^\2$} (m-2-1)
														edge[transform canvas={xshift=2pt}] node[right] {$\phi_B^\2$} (m-2-3)
										(m-2-1) edge[barred] node[below] {$J$} (m-2-3);
				\path				(m-1-2) edge[cell, transform canvas={yshift=-0.25em}] node[right, inner sep=2.5pt] {$\phi$} (m-2-2);
			\end{tikzpicture}
		\end{displaymath}
		
		It follows that, when regarding $J$ as an internal profunctor $\hmap J{A^\2}{B^2}$ in $\inProf{\und{\mathcal C}}$, we can take its tabulation, as described in \exref{tabulations of internal profunctors}, to have $\tab J \dfn J^\2$ as its object of morphisms. Going through the proof of Proposition~5.15 of \cite{Koudenburg14a} it is then easy to check that this choice implies that $\tab J = J^\2$ as an internal category, and that the projections $\map{\pi_A}{J^\2}{A^\2}$ and $\map{\pi_B}{J^\2}{B^\2}$ are the internal functors $j_A^\2$ and $j_B^\2$. Hence the $(0,1)$"/ary cell $\cell\pi{J^\2}J$ in $\inProf{\und{\mathcal C}}$, that defines $J^\2$ as the tabulation of $J$, forms a $(0,1)$"/ary cell $\pi$ in $\dFib{\mathcal C}$ that is of the form as in the middle above. Next consider any $(0,1)$"/ary cell $\cell\phi XJ$ in $\dFib{\mathcal C}$, i.e.\ any cell $\phi$ in $\inProf{\und{\mathcal C}}$ of the form as on the right above. One readily checks that the latter factors through $\pi$ (in $\inProf{\und{\mathcal C}}$) as the internal functor $\map{\phi^\2}{X^\2}{J^\2}$, showing that in $\dFib{\mathcal C}$ the cell $\phi$ factors through $\pi$ as the morphism $\map\phi XJ$. We conclude that the $1$"/dimensional universal property of $\pi$ (\defref{tabulation}) in $\inProf{\und{\mathcal C}}$ restricts to its sub"/unital virtual equipments $\dFib{\mathcal C}$ and $\spFib{\mathcal C}$. That the $2$"/dimensional universal property does so too follows from the fact that the embeddings \mbox{$\dFib{\mathcal C} \hookrightarrow \spFib{\mathcal C} \hookrightarrow \inProf{\und{\mathcal C}}$} are locally full and faithful; we conclude that $\pi$ defines the object $J$ as the tabulation of $\hmap JAB$ both $\dFib{\mathcal C}$ and $\spFib{\mathcal C}$. Because the embeddings preserve cartesian cells $\pi$ is cocartesian in $\dFib{\mathcal C}$ and $\spFib{\mathcal C}$ too by \auglemref{9.4}. In particular $J \iso j_A^* \hc j_{B*}$ in $\dFib{\mathcal C}$ by \lemref{horizontal morphism as composite of conjoint and companion} below; compare Proposition~4.25 of \cite{Carboni-Johnson-Street-Verity94} which shows that $E \iso p_* \of q^*$ for any `discrete fibration' $A \xlar p E \xrar q B$ in a `faithfully conservational bicategory'.
	\end{example}
	
	The following lemma was used in the examples above in obtaining (co)tabulations in full sub"/virtual double categories. The proof of its main assertion is straightforward; for the other assertions use \auglemsref{4.5}{9.4}.
	\begin{lemma} \label{full and faithful functors reflect tabulations}
		Any full and faithful functor $\map F\K\L$ (\augdefref{3.6}) reflects tabulations, that is $\cell\pi XJ$ defines $X$ as the tabulation of $J$ in $\K$ whenever $F\pi$ defines $FX$ as the tabulation of $FJ$ in $\L$. Similarly $F$ reflects (cartesian) cotabulations. If moreover $F$ preserves cartesian cells then it reflects cocartesian tabulations as well.
	\end{lemma}
	
	\subsection{Properties of cocartesian paths of $(0,1)$"/ary cells}
	Before stating the main theorem of this section we record some useful properties of cocartesian paths of $(0,1)$"/ary cells.	Recall that $n' \dfn n-1$ for any positive integer $n$.
	\begin{lemma}[Pasting lemma for cocartesian paths of $(0,1)$"/ary cells] \label{pasting lemma for cocartesian paths of (0,1)-ary cells}
		Consider composable paths of horizontal morphisms $\hmap{\ul J}{A_0}{A_n}$ and $\hmap{\ul H}{A_n}{B_m}$. Let $\ul\psi = (\psi_1, \dotsc, \psi_m)$ be a left cocartesian path of $(0,1)$"/ary cells for $\ul H$ and assume that the restriction \mbox{$\hmap{J_n(\id, g_0)}{A_{n'}}Y$} exists, where $\map{g_0}Y{A_n}$ is the vertical source of $\psi_1$. If $\ul\phi = (\phi_1, \dotsc, \phi_n)$ is a left cocartesian path of $(0,1)$"/ary cells for $\pars{J_1, \dotsc, J_{n'}, J_n(\id, g_0)}$ then the concatenation \mbox{$(\phi_1, \dotsc, \phi_{n'}, \cart \of \phi_n, \psi_1, \dotsc, \psi_m)$} is a left cocartesian path of $(0,1)$"/ary cells for $\ul J \conc \ul H$, where $\cart$ denotes the cartesian cell defining $J_n(\id, g_0)$.
		
		Denoting the vertical targets of $\phi_n$ and $\psi_n$ by $\map{f_n}XY$ and $\map{g_m}Y{B_m}$, the latter concatenation is a cocartesian path of $(0,1)$"/ary cells for $\ul J \conc \ul H$ whenever $\ul \phi$ and $\ul \psi$ are cocartesian and, for every $\hmap K{B_m}C$, if the restriction $K(g_m \of f_n, \id)$ exists then so does $K(g_m, \id)$.
		
		Analogous assertions hold for left nullary"/cocartesian paths of $(0,1)$"/ary cells; horizontally dual assertions hold for right (nullary"/)cocartesian paths of $(0,1)$"/ary cells.
	\end{lemma}
	
	 The next lemma shows that a $(0,1)$"/ary cocartesian cell $\cell{\phi_1}X{J_1}$ can be used to write $J_1$ as a composite of a conjoint followed by a companion. This formalises the classical fact (see e.g.\ Proposition~2.3.2 of \cite{Benabou73}) that any profunctor $\hmap JAB$ can be written as $J \iso \pi_A^* \hc \pi_{B*}$ where $A \xlar{\pi_A} \tab J \xrar{\pi_B} B$ are the projections of the graph $\tab J$ of $J$ (\exref{tabulations of profunctors}).
	\begin{lemma} \label{horizontal morphism as composite of conjoint and companion}
		Let $\cell{\phi_1}X{J_1}$ be a cocartesian $(0,1)$"/ary cell for $\hmap {J_1}{A_0}{A_1}$ as in \defref{cocartesian path of (0,1)-ary cells}. If the conjoint $\hmap{f_0^*}{A_0}X$ and the companion $\hmap{f_{1*}}X{A_1}$ exist then the composite $\cell{\cart \hc \phi_1 \hc \cart}{(f_0^*, f_{1*})}{J_1}$, where the cartesian cells defining $f_0^*$ and $f_{1*}$ are denoted by $\cart$, is cocartesian. In particular $J_1 \iso f_0^* \hc f_{1*}$.
	\end{lemma}
	\begin{proof}
		Denote by $\cocart$ the cocartesian cells corresponding to the cartesian cells that define $f_0^*$ and $f_{1*}$ (\lemref{companion identities lemma}). By the pasting lemma for cocartesian paths (\auglemref{7.7}) the path $\cell{(\cocart, \cocart)}X{(f_0^*, f_{1*})}$ is cocartesian. Applying the same lemma to $(\cart \hc \phi_1 \hc \cart) \of (\cocart, \cocart) = \phi_1$ shows that $\cart \hc \phi_1 \hc \cart$ is cocartesian too.
	\end{proof}
	
	Cocartesian paths of $(0,1)$"/ary cells can be used to reduce left Kan extension along a path $\ul J$ of horizontal morphisms to left Kan extension along a single companion morphism as follows.
	\begin{proposition} \label{reducing to Kan extensions along a companion}
		Consider the right (respectively weakly) nullary"/cocartesian path $\ul\phi$ of $(0,1)$"/ary cells for the path $\ul J$ in \defref{cocartesian path of (0,1)-ary cells} and assume that the companion \mbox{$\hmap{f_{n*}}X{A_n}$} exists. For any morphism $\map d{A_0}M$ the (weak) left Kan extension of $d$ along $\ul J$ exists if and only if the (weak) left Kan extension of $d \of f_0$ along $f_{n*}$ does so, and in that case they are isomorphic.
	\end{proposition}
	\begin{proof}
		Let $\cart$ and $\cocart$ denote the cartesian and cocartesian cells that define the companion $f_{n*}$, as in \lemref{companion identities lemma}. Applying the pasting lemma for nullary"/cocartesian paths (\lemref{pasting lemma for cocartesian paths}) to the identity \mbox{$(\phi_1, \dotsc, \phi_n, \cart) \of \cocart = (\phi_1, \dotsc, \phi_n)$}, which follows from the companion identities (\lemref{companion identities lemma}), we find that the path \mbox{$\cell{(\phi_1, \dotsc, \phi_n, \cart)}{f_{n*}}{\ul J}$}, which has vertical source $f_0$, is right (respectively weakly) nullary"/cocartesian because $\ul\phi$ and $\cocart$ are so. The result follows from applying the vertical pasting lemma for left Kan extensions (\lemref{vertical pasting lemma}) to the path $(\phi_1, \dotsc, \phi_n, \cart)$.
	\end{proof}
	
	\subsection{Pointwise Kan extensions in terms of pointwise weak Kan extensions}
	The following theorem is the main result of this section.
	\begin{theorem} \label{pointwise Kan extensions in terms of pointwise weak Kan extensions}
		In an augmented virtual double category that has restrictions on the right (\defref{augmented virtual equipment}) as well as left nullary"/cocartesian paths of $(0,1)$"/ary cells (\defref{cocartesian path of (0,1)-ary cells}), all pointwise weakly left Kan cells are pointwise left Kan (\defref{pointwise left Kan extension}).
	\end{theorem}
	\begin{proof}
		Consider a pointwise weakly left Kan cell $\eta$ as in the composite on the left below. We will first prove that $\eta$ is left Kan (\defref{left Kan extension}), that is any cell $\phi$ as in the middle below factors uniquely through $\eta$. To do so let $\ul\zeta = (\zeta_1, \dotsc, \zeta_m)$ be a left nullary"/cocartesian path of $(0,1)$"/ary cells for the path $\ul H = (H_1, \dotsc, H_m)$ and denote by $\map{h_0}X{A_n}$ and $\map{h_m}X{B_m}$ the vertical source of $\zeta_1$ and the vertical target of $\zeta_m$. Under precomposition with the weakly cocartesian path $(\id_{J_1}, \dotsc, \id_{J_{n'}}, \cart, \zeta_1, \dotsc, \zeta_n)$, where $\cart$ defines $J_n(\id, h_0)$, cells $\phi$ as in the middle below correspond to cells $\psi$ as on the right. Applying \lemref{unique factorisations factorised through cocartesian paths} to $\ul\chi = (\id_{J_1}, \dotsc, \id_{J_{n'}}, \cart)$ and $\ul\zeta$ we find that, under this correspondence, the cells $\phi$ factor uniquely through $\eta$ precisely if the cells $\psi$ factor uniquely through the composite on the the left below, as vertical cells $l \of h_0 \Rar k \of h_n$. The latter factorisations exist by the assumption that $\eta$ is pointwise weakly left Kan, so that the existence of the factorisations of the cells $\phi$ through $\eta$ follows as required.
		\begin{displaymath}
			\begin{tikzpicture}[baseline, yshift=1.625em]
				\matrix(m)[math35]{A_0 & A_1 & A_{n'} & X \\ A_0 & A_1 & A_{n'} & A_n \\};
				\draw	([yshift=-6.5em]$(m-1-1)!0.5!(m-1-4)$) node (M) {$M$};
				\path[map]	(m-1-1) edge[barred] node[above] {$J_1$} (m-1-2)				
										(m-1-3) edge[barred] node[above, shift={(-3pt,2pt)}] {$J_n(\id, h_0)$} (m-1-4)
										(m-1-4) edge node[right] {$h_0$} (m-2-4)
										(m-2-4) edge[transform canvas={yshift=-2pt}] node[below right] {$l$} (M)
										(m-2-1) edge[barred] node[below, inner sep=2.5pt] {$J_1$} (m-2-2)
														edge[transform canvas={yshift=-2pt}] node[below left] {$d$} (M)
										(m-2-3) edge[barred] node[below, inner sep=2.5pt] {$J_n$} (m-2-4);
				\path				(m-1-1) edge[eq] (m-2-1)
										(m-1-2) edge[eq] (m-2-2)
										(m-1-3) edge[eq] (m-2-3);
				\path				($(m-2-1.south)!0.5!(m-2-4.south)$) edge[cell] node[right] {$\eta$} (M);
				\draw				($(m-1-2)!0.5!(m-2-3)$) node {$\dotsb$}
										($(m-1-3)!0.5!(m-2-4)$) node[font=\scriptsize] {$\cart$};
			\end{tikzpicture} \qquad \begin{tikzpicture}[baseline]
				\matrix(m)[math35]{A_0 & & B_m \\ & M & \\};
				\path[map]	(m-1-1) edge[barred] node[above] {$\ul J \conc \ul H$} (m-1-3)
														edge[transform canvas={yshift=-2pt}] node[below left] {$d$} (m-2-2)
										(m-1-3) edge[transform canvas={yshift=-2pt}] node[below right] {$k$} (m-2-2);
				\path				(m-1-2) edge[cell] node[right] {$\phi$} (m-2-2);
			\end{tikzpicture} \quad \begin{tikzpicture}[baseline]
				\matrix(m)[math35, yshift=1.625em]{A_0 & A_1 & A_{n'} & X \\};
				\draw				([yshift=-3.25em]$(m-1-1)!0.5!(m-1-4)$) node (M) {$M$};
				\path[map]	(m-1-1) edge[barred] node[above] {$J_1$} (m-1-2)
														edge[transform canvas={yshift=-2pt}] node[below left] {$d$} (M)
										(m-1-3) edge[barred] node[above, shift={(-3pt,2pt)}] {$J_n(\id, h_0)$} (m-1-4)
										(m-1-4) edge[transform canvas={yshift=-2pt}] node[below right] {$k\of h_n$} (M);
				\draw				($(m-1-2)!0.5!(m-1-3)$) node {$\dotsb$};
				\path				($(m-1-1.south)!0.5!(m-1-4.south)$) edge[cell] node[right] {$\psi$} (M);
			\end{tikzpicture}
		\end{displaymath}
		
		Finally, to prove that $\eta$ is in fact pointwise left Kan, we have to show that any composite \mbox{$\eta \of (\id_{J_1}, \dotsc, \id_{J_{n'}}, \cart)$}, where $\cart$ defines the restriction $J_n(\id, f)$ along any morphism $\map fB{A_n}$, is again left Kan. Since the latter composite is again pointwise weakly left Kan by \lemref{restrictions of pointwise left Kan extensions} this follows immediately from the argument above.
	\end{proof}
	
	\subsection{Pointwise left Kan extension along a companion}
	We can use the previous theorem to extend \propref{weak left Kan extensions along companions} to the pointwise case. This generalises the corresponding result for pseudo double categories, Proposition~5.12 of \cite{Koudenburg14a}.
	\begin{proposition} \label{pointwise left Kan extensions along companions}
		Consider the following factorisation in an augmented virtual double category $\K$ that has restrictions on the right (\defref{augmented virtual equipment}) as well as left nullary"/cocartesian tabulations (\defref{tabulation}).
		\begin{displaymath}
			\begin{tikzpicture}[textbaseline]
				\matrix(m)[math35, column sep={1.75em,between origins}]{& A & \\ & & B \\ & M & \\};
				\path[map]	(m-1-2) edge[bend left = 18] node[above right] {$j$} (m-2-3)
														edge[bend right = 45] node[left] {$d$} (m-3-2)
										(m-2-3) edge[bend left = 18] node[below right] {$l$} (m-3-2);
				\path[transform canvas={yshift=-1.625em}]	(m-1-2) edge[cell] node[right] {$\eta$} (m-2-2);
			\end{tikzpicture} \quad = \quad \begin{tikzpicture}[textbaseline]
    		\matrix(m)[math35, column sep={1.75em,between origins}]{& A & \\ A & & B \\ & M & \\};
    		\path[map]	(m-1-2) edge[transform canvas={xshift=2pt}] node[right] {$j$} (m-2-3)
    								(m-2-1) edge[barred] node[below, inner sep=2pt] {$j_*$} (m-2-3)
    												edge[transform canvas={xshift=-2pt}] node[left] {$d$} (m-3-2)
    								(m-2-3)	edge[transform canvas={xshift=2pt}] node[right] {$l$} (m-3-2);
    		\path				(m-1-2) edge[eq, transform canvas={xshift=-2pt}] (m-2-1);
    		\path				(m-2-2) edge[cell, transform canvas={yshift=0.1em}]	node[right, inner sep=3pt] {$\eta'$} (m-3-2);
    		\draw				([yshift=-0.5em]$(m-1-2)!0.5!(m-2-2)$) node[font=\scriptsize] {$\cocart$};
  		\end{tikzpicture}
  	\end{displaymath}
  	
  	The cell $\eta$ defines $l$ as the pointwise left Kan extension of $d$ along $j$ in the $2$"/category $V(\K)$ (\augexref{1.5}), in the sense of Section 4 of \cite{Street74b}, precisely if its factorisation $\eta'$ is pointwise left Kan in $\K$  (\defref{pointwise left Kan extension}).
	\end{proposition}
	Street's notion of pointwise Kan extension in a $2$-category uses the well known notion of \emph{comma object}; see e.g.\ Section~1 of \cite{Street74b}. Instead of recalling the definition of comma object we record the following straightforward lemma, which relates it to that of tabulation.
	\begin{lemma}
		Consider morphisms $\map fAC$ and $\map gBC$ in an augmented virtual double category $\K$. If both the cartesian cell and the tabulation below exist then their composite defines $\tab{C(f, g)}$ as the comma object $f \slash g$ of $f$ and $g$ in $V(\K)$.
		\begin{displaymath}
			\begin{tikzpicture}
    		\matrix(m)[math35, column sep={1.75em,between origins}]{& \tab{C(f,g)} & \\ A & & B \\ & C & \\};
    		\path[map]	(m-1-2)	edge[transform canvas={xshift=-2pt}] node[left] {$\pi_A$} (m-2-1)
    												edge[transform canvas={xshift=2pt}] node[right] {$\pi_B$} (m-2-3)
    								(m-2-1) edge[barred] node[below, inner sep=2pt] {$C(f,g)$} (m-2-3)
    												edge[transform canvas={xshift=-2pt}] node[below left] {$f$} (m-3-2)
    								(m-2-3)	edge[transform canvas={xshift=2pt}] node[right] {$g$} (m-3-2);
    		\path				(m-1-2) edge[cell, transform canvas={yshift=-0.25em}]	node[right] {$\pi$} (m-2-2);
    		\draw				([yshift=0.1em]$(m-2-2)!0.5!(m-3-2)$) node[font=\scriptsize] {$\cart$};
  		\end{tikzpicture}
		\end{displaymath}
	\end{lemma}
	\begin{proof}[of \propref{pointwise left Kan extensions along companions}]
		First notice that the assumptions imply that $\K$ has left nullary"/cocartesian paths of $(0,1)$"/ary cells by \cororef{cocartesian tabulations imply cocartesian paths of (0,1)-ary cells}, so that \thmref{pointwise Kan extensions in terms of pointwise weak Kan extensions} applies. Moreover notice that all restrictions $B(j, f)$, where $\map fCB$, exist in $\K$ as the restrictions $j_*(\id, f)$ (compare \auglemref{5.11}). Next consider composites $\zeta$ as on the left below, where $\map fCB$ varies, $\cart \of \pi$ defines the comma object $j \slash f$ as in the previous lemma, and $\eta = \eta' \of \cocart$. In each $\zeta$ the composite $\cell{\cocart \hc \cart}{B(j, f)}{j_*}$ is cartesian: using the pasting lemma for cartesian cells (\lemref{pasting lemma for cartesian cells}) this follows from the fact that composing it with the cartesian cell defining $j_*$, that corresponds to $\cocart$ (\lemref{companion identities lemma}), results in the cartesian cell defining $B(j, f)$. 
		\begin{displaymath}
			\zeta \dfn \begin{tikzpicture}[textbaseline]
				\matrix(m)[math35, column sep={1.75em,between origins}]
					{ & & j \slash f & \\
						& A & & C \\
						A & & B & \\
						& M & & \\ };
				\path[map]	(m-1-3) edge[transform canvas={xshift=-2pt}] node[left] {$\pi_A$} (m-2-2)
										(m-2-2) edge[transform canvas={xshift=1pt}] node[left, yshift=1pt] {$j$} (m-3-3);
				\path[map, transform canvas={xshift=2pt}]	(m-1-3) edge node[right] {$\pi_C$} (m-2-4)
										(m-2-4) edge node[right] {$f$} (m-3-3)
										(m-3-3) edge node[right] {$l$} (m-4-2);
				\path[map]	(m-2-2) edge[barred] node[below, inner sep=2pt] {$B(j, f)$} (m-2-4)
										(m-3-1)	edge[barred] node[below, inner sep=1pt] {$j_*$} (m-3-3)
														edge[transform canvas={xshift=-2pt}] node[below left] {$d$} (m-4-2);
				\path				(m-1-3) edge[cell, transform canvas={yshift=-0.25em}]	node[right] {$\pi$} (m-2-3)
										(m-3-2) edge[cell, transform canvas={yshift=0.25em}] node[right] {$\eta'$} (m-4-2)
										(m-2-2) edge[eq, transform canvas={xshift=-2pt}] (m-3-1);
    		\draw				([yshift=0.15em, xshift=1pt]$(m-2-3)!0.5!(m-3-3)$) node[font=\scriptsize] {$\cart$}
    								([yshift=-0.55em]$(m-2-2)!0.5!(m-3-2)$) node[font=\scriptsize] {$\cocart$};
			\end{tikzpicture} \qquad\qquad \begin{tikzpicture}[textbaseline]
				\matrix(m)[math35, column sep={1.75em,between origins}, xshift=-5em]{& C & \\ B & & \\ & M & \\};
				\path[map]	(m-1-2) edge[bend right = 18] node[above left] {$f$} (m-2-1)
														edge[bend left = 45] node[right] {$k$} (m-3-2)
										(m-2-1) edge[bend right = 18] node[below left] {$l$} (m-3-2);
				\path[transform canvas={yshift=-1.625em}]	(m-1-2) edge[cell] (m-2-2);
				
				\matrix(m)[math35, column sep={1.75em,between origins}, xshift=10em, yshift=-5em]{& j \slash f & \\ A & & C \\ & M & \\};
				\path[map]	(m-1-2) edge[bend right = 18] node[above left] {$\pi_A$} (m-2-1)
														edge[bend left = 18] node[above right] {$\pi_C$} (m-2-3)
										(m-2-1) edge[bend right = 18] node[below left] {$d$} (m-3-2)
										(m-2-3) edge[bend left = 18] node[below right] {$k$} (m-3-2);
				\path[transform canvas={yshift=-1.625em}]	(m-1-2) edge[cell] (m-2-2);
				
				\matrix(m)[math35, column sep={1.75em,between origins}, xshift=10em, yshift=5em]
  				{ A & & C \\ & M \\};
  			\path[map]  (m-1-1) edge[barred] node[above] {$B(j, f)$} (m-1-3)
														edge[transform canvas={xshift=-2pt}] node[left] {$d$} (m-2-2)
										(m-1-3)	edge[transform canvas={xshift=2pt}] node[right] {$k$} (m-2-2);
				\path				(m-1-2) edge[cell, transform canvas={yshift=0.25em}] (m-2-2);
				
				\draw[font=\Large]	(7.5em,5em) node {$\lbrace$}
										(12.675em,5em) node {$\rbrace$}
										(-7.5em,0) node {$\lbrace$}
										(-1.75em,0) node {$\rbrace$}
										(7.5em,-5em) node {$\lbrace$}
										(12.5em,-5em) node {$\rbrace$};
				\path[map]	(-2.5em,3em) edge[bend left=10] node[above left] {$\eta \hc (l \of \cart) \hc \dash$} (6.5em,5em)
										(10em,1.5em) edge node[right] {$\dash \of \pi$} (10em,-0em)
										(-2.5em,-3em) edge[bend right=10] node[below left] {$\zeta \hc (\dash \of \pi_C)$} (6.5em,-5em);
			\end{tikzpicture}
		\end{displaymath}
		It follows that the top assignment in the commutative diagram on the right above, between collections of cells in $\K$ as shown, is a bijection, for every $f$, precisely if $\eta'$ is pointwise weakly left Kan in $\K$. By \thmref{pointwise Kan extensions in terms of pointwise weak Kan extensions} the latter is equivalent to $\eta'$ being pointwise left Kan. Because the cell $\pi$ is weakly nullary"/cocartesian the assignment on the right is a bijection, and we conclude that $\eta'$ is pointwise left Kan precisely if the bottom assignment is a bijection, for every $f$. Since $\cart \of \pi$ defines the comma object $j \slash f$ the latter, by definition, precisely means that $\eta$ exhibits $l$ as the pointwise left Kan extension in the $2$"/category $V(\K)$, which completes the proof.
	\end{proof}
	
	The following is Example 2.24 of \cite{Koudenburg13}.
	\begin{example} \label{pointwise is stronger than enriched}
		Recall from \exref{tabulations of 2-profunctors} that tabulations of $2$-profunctors are in general not cocartesian. As a consequence the equivalence of \propref{pointwise left Kan extensions along companions} fails to hold in the unital virtual equipment $\enProf\Cat$. For a counterexample consider the cell $\eta'$ below, where $1$ is the terminal $2$-category and the collapsing $2$-functor $l$ on the right has the `free living' cell as source and the free living parallel pair of arrows as target; $\eta'$ is uniquely determined by its boundary. Using \exref{enriched left Kan extension} it is straightforward to check that $\eta'$ is left Kan in $\enProf\Cat$, so that the corresponding $2$"/natural transformation $\cell\eta{x'}{l \of x}$ defines $l$ as the enriched left Kan extension of $x'$ along $x$, in the sense of e.g.\ Section 4.1 of \cite{Kelly82}. In contrast the pointwise left Kan extension of $x'$ along $x$ in the $2$-category $\twoCat = V(\enProf\Cat)$, of $2$-categories, $2$-functors and $2$-natural transformations, and in the sense of Section~4 of \cite{Street74b}, does not exist. To see this one checks that the composite of $\eta$ with the cell defining the comma object $x\slash y$ does not define $ly$ as a left Kan extension.
		\begin{displaymath}
			\begin{tikzpicture}
				\matrix(m)[math35, column sep=2.5em, xshift=5em]{ & \\};
				\path[map]	(m-1-1) edge[bend left=40] (m-1-2)
														edge[bend right=40] (m-1-2);
				\fill[font=\scriptsize]	(m-1-1) circle (1pt) node[left, inner sep=2pt] {$x$}
							(m-1-2) circle (1pt) node[right, inner sep=2pt] {$y$};
				\path	($(m-1-1)!0.5!(m-1-2)+(0,0.8em)$) edge[cell] ($(m-1-1)!0.5!(m-1-2)-(0,0.8em)$);
				
				\matrix(m)[math35, column sep=2.5em, yshift=-5em]{ & \\};
				\path[map]	(m-1-1) edge[bend left=40] (m-1-2)
														edge[bend right=40] (m-1-2);
				\fill[font=\scriptsize]	(m-1-1) circle (1pt) node[left, inner sep=1pt] {$x'$}
							(m-1-2) circle (1pt) node[right, inner sep=2pt] {$y'$};
				
				\draw	(-5em,0) node {$1$};
				
				\draw[font=\Large]	(2.5em,0) node {$($}
							(7.5em,0) node {$)$}
							(-2.75em,-5em) node {$($}
							(2.75em,-5em) node {$)$};
				
				\path[map]	(-3em, 0) edge[barred] node[above] {$x_*$} (0.5em, 0)
										(-4.5em, -1.25em) edge node[below left] {$x'$} (-2.75em, -3.75em)
										(4.5em, -1.25em) edge node[below right] {$\text{(collapse onto $x'$)} \nfd l$} (2.75em, -3.75em);
				
				\path				(0 em,-1.25em) edge[cell] node[right] {$\eta'$} (0 em,-3.0em);
			\end{tikzpicture}
		\end{displaymath}
	\end{example}

	\section{Yoneda morphisms}\label{yoneda embeddings section}
	Having introduced notions of left Kan extension in augmented virtual double categories we now turn to formalising the classical notion of Yoneda embedding. Informally, in \defref{yoneda embedding} below we will take a `Yoneda morphism' in an augmented virtual double category to be a vertical morphism that is `dense' (as defined below) and that satifies an axiom that formalises the classical Yoneda's lemma (see \exref{enriched yoneda embedding summary} below). These two conditions are closely related to the axioms satisfied by the formal Yoneda embeddings comprising a `Yoneda structure' on a $2$-category, as introduced by Street and Walters in \cite{Street-Walters78}. In fact given an augmented virtual double category $\K$ the main theorem of this section, \thmref{yoneda structures}, describes relations between supplying a family of (weak) Yoneda morphisms in $\K$, in our sense (\defref{yoneda embedding}), and equipping the vertical $2$"/category $V(\K)$ (\augexref{1.5}) with a Yoneda structure, in the sense of \cite{Street-Walters78}, as well as equipping $V(\K)$ with two strengthenings of the latter that were introduced in \cite{Street-Walters78} and \cite{Weber07}.
	
	An advantage of our approach to formalising Yoneda embeddings is a consequence of our viewpoint of regarding all horizontal morphisms as being `admissible' (informally these are to be thought of as ``small in size''; see the \introref): this enables us to give a relatively simple definition of a \emph{single} Yoneda morphism, which satisfies a formalisation of Yoneda's lemma ``with respect to all horizontal morphisms''. In contrast the notions of Yoneda structure of \cite{Street-Walters78} and \cite{Weber07} consist of a \emph{collection} of Yoneda embeddings satisfying a formalisation of Yoneda's lemma with respect to a specified collection of `admissible' morphisms.
	
	\subsection{Density}
	Using the notion of (weak) left Kan extension, we start by defining the notion of (weak) density as one of the three equivalent conditions below. For augmented virtual double categories $\K$ satisfying the assumptions of \propref{pointwise left Kan extensions along companions} condition~(c) below coincides with the original $2$"/categorical notion of density, given in Section~3 of \cite{Street74a}, when applied to the vertical $2$"/category $V(\K)$; in particular density in $\enProf{(\Set, \Set')}$ (\augexref{2.6}) recovers the classical notion of density for functors (see e.g.\ Section~X.6 of \cite{MacLane98}). Similarly it follows from \exref{enriched left Kan extension} that, when considered in the unital virtual equipment $\enProf\V$ of $\V$"/profunctors, condition~(c) below coincides with the classical notion of density for enriched functors; see e.g.\ Section 5.1 of \cite{Kelly82}.
	\begin{lemma} \label{density lemma}
		For a morphism $\map fAM$ the following conditions are equivalent:
		\begin{enumerate}[label=\textup{(\alph*)}]
			\item	if a cell $\eta$, as on the left below, is cartesian then it is (weakly) left Kan (\defsref{weak left Kan extension}{left Kan extension});
			\item if a cell $\eta$ as below is cartesian then it is pointwise (weakly) left Kan (\defref{pointwise left Kan extension}).
		\end{enumerate}
		If the companion $\hmap{f_*}AM$ exists then the following condition is equivalent too:
		\begin{enumerate}
			\item[\textup{(c)}]	the cartesian cell defining $f_*$ on the right below is pointwise (weakly) left Kan.
		\end{enumerate}
			\begin{displaymath}
				\begin{tikzpicture}[baseline]
				\matrix(m)[math35, column sep={1.75em,between origins}]{A & & B \\ & M & \\};
				\path[map]	(m-1-1) edge[barred] node[above] {$J$} (m-1-3)
														edge[transform canvas={xshift=-2pt}] node[left] {$f$} (m-2-2)
										(m-1-3) edge[transform canvas={xshift=2pt}] node[right] {$l$} (m-2-2);
				\path[transform canvas={yshift=0.25em}]	(m-1-2) edge[cell] node[right] {$\eta$} (m-2-2);
			\end{tikzpicture} \qquad\qquad\qquad\qquad\qquad \begin{tikzpicture}[baseline]
					\matrix(m)[math35, column sep={1.75em,between origins}]{A & & M \\ & M & \\};
					\path[map]	(m-1-1) edge[barred] node[above] {$f_*$} (m-1-3)
															edge[transform canvas={xshift=-2pt}] node[left] {$f$} (m-2-2);
					\path				(m-1-3) edge[eq, transform canvas={xshift=2pt}] (m-2-2);
					\draw				([yshift=0.333em]$(m-1-2)!0.5!(m-2-2)$) node[font=\scriptsize] {$\cart$};
				\end{tikzpicture}
	    \end{displaymath}
	\end{lemma}
	\begin{definition} \label{density definition}
		A morphism $\map fAM$ is \emph{(weakly) dense} if the equivalent conditions above are satisfied.
	\end{definition}
	Notice that the notions of weak density and density coincide in augmented virtual double categories that satisfy the assumptions of \thmref{pointwise Kan extensions in terms of pointwise weak Kan extensions}, as do the notions of weak Yoneda morphism and Yoneda morphism below.
	\begin{proof}[of \lemref{density lemma}]
		(b) $\Rightarrow$ (a) is clear. For the converse assume that (a) holds: we have to show that any composite as on the left"/hand side below is (weakly) left Kan. Since $\eta$ is cartesian the composite is cartesian too by the pasting lemma (\lemref{pasting lemma for cartesian cells}), so that it is (weakly) left Kan by (a).
		\begin{displaymath}
			\begin{tikzpicture}[textbaseline]
  			\matrix(m)[math35, column sep={1.75em,between origins}]{A & & C \\ A & & B \\ & M & \\};
  			\path[map]	(m-1-1) edge[barred] node[above] {$J(\id, g)$} (m-1-3)
  									(m-1-3) edge node[right] {$g$} (m-2-3)
  									(m-2-1) edge[barred] node[above] {$J$} (m-2-3)
  													edge[transform canvas={xshift=-2pt}] node[left] {$f$} (m-3-2)
  									(m-2-3) edge[transform canvas={xshift=2pt}] node[right] {$l$} (m-3-2);
  			\path				(m-1-1) edge[eq] (m-2-1);
  			\path[transform canvas={yshift=0.25em}]	(m-2-2) edge[cell] node[right, inner sep=3pt] {$\eta$} (m-3-2);
        \draw       ([yshift=0.25em]$(m-1-1)!0.5!(m-2-3)$) node[font=\scriptsize] {$\cart$};
  		\end{tikzpicture} \qquad\qquad\qquad\qquad\qquad \begin{tikzpicture}[textbaseline]
				\matrix(m)[math35, column sep={1.75em,between origins}]{A & & B \\ & M & \\};
				\path[map]	(m-1-1) edge[barred] node[above] {$J$} (m-1-3)
														edge[transform canvas={xshift=-2pt}] node[left] {$f$} (m-2-2)
										(m-1-3) edge[transform canvas={xshift=2pt}] node[right] {$l$} (m-2-2);
				\path[transform canvas={yshift=0.25em}]	(m-1-2) edge[cell] node[right, inner sep=3pt] {$\eta$} (m-2-2);
			\end{tikzpicture} \quad = \quad \begin{tikzpicture}[textbaseline]
  			\matrix(m)[math35, column sep={1.75em,between origins}]{A & & B \\ A & & M \\ & M & \\};
  			\path[map]	(m-1-1) edge[barred] node[above] {$J$} (m-1-3)
  									(m-1-3) edge node[right] {$l$} (m-2-3)
  									(m-2-1) edge[barred] node[below] {$f_*$} (m-2-3)
  													edge[transform canvas={xshift=-2pt}] node[left] {$f$} (m-3-2);
  			\path				(m-1-1) edge[eq] (m-2-1)
  									(m-2-3) edge[eq, transform canvas={xshift=2pt}] (m-3-2);
  			\path[transform canvas={xshift=1.75em}]
        		        (m-1-1) edge[cell] node[right] {$\eta'$} (m-2-1);
        \draw       ([yshift=0.25em]$(m-2-2)!0.5!(m-3-2)$) node[font=\scriptsize] {$\cart$};
  		\end{tikzpicture}
	  \end{displaymath}
		
		(b) $\Rightarrow$ (c) is clear. For the converse consider a cartesian cell $\eta$ as on the right above and let $\eta'$ be its factorisation as shown; $\eta'$ is cartesian because $\eta$ is so, by the pasting lemma. Assuming (c) it follows from \lemref{restrictions of pointwise left Kan extensions} that $\eta$ is pointwise (weakly) left Kan.
	\end{proof}
	
	\subsection{Yoneda morphisms}
	(Weak) Yoneda morphisms in an augmented virtual double category are (weakly) dense morphisms satisfying a `Yoneda axiom' as follows.
	\begin{definition} \label{yoneda embedding}
		A (weakly) dense morphism $\map\yon A{\ps A}$ is called a \emph{(weak) Yoneda morphism} if it satisfies the \emph{Yoneda axiom}: for every horizontal morphism $\hmap JAB$ there exists a vertical morphism $\map{\cur J}B{\ps A}$ equipped with a cartesian cell
		\begin{displaymath}
			\begin{tikzpicture}
				\matrix(m)[math35, column sep={1.75em,between origins}]{A & & B \\ & \ps A. & \\};
				\path[map]	(m-1-1) edge[barred] node[above] {$J$} (m-1-3)
														edge[transform canvas={xshift=-2pt}] node[left] {$\yon$} (m-2-2)
										(m-1-3) edge[transform canvas={xshift=2pt}] node[right] {$\cur J$} (m-2-2);
				\draw				([yshift=0.333em]$(m-1-2)!0.5!(m-2-2)$) node[font=\scriptsize] {$\cart$};
			\end{tikzpicture}
	  \end{displaymath}
	\end{definition}
	
	 Notice that the (weak) density of $\map\yon A{\ps A}$ implies that the vertical morphism $\cur J$ is unique up to vertical isomorphism. In particular, if the companion $\hmap{\yon_*}A{\ps A}$ exists then its defining cartesian cell implies that $\cur{\yon_*} \iso \id_{\ps A}$. Since density implies weak density, any Yoneda morphism is a weak Yoneda morphism. We call the target $\ps A$ of $\yon$ the \emph{object of presheaves on $A$}, or \emph{presheaf object} for short. (Weak) Yoneda morphisms $\map\yon A{\ps A}$ such that all nullary restrictions $\ps A(\yon, f)$ exist, for any $\map fB{\ps A}$, are especially pleasant to work with; in that case we will say that $\yon$ \emph{admits nullary restrictions}. Notice that, when considered in an augmented virtual double category with restrictions on the right (\defref{augmented virtual equipment}), the latter condition reduces to the existence of the companion $\yon_*$, since $\ps A(\yon, f) \iso \yon_*(\id, f)$ by the pasting lemma for cartesian cells (\lemref{pasting lemma for cartesian cells}).
	
	Before giving examples of Yoneda morphisms in \exrref{yoneda embedding for unit V-category}{yoneda embeddings for internal preorders} below we first consider full and faithfulness of Yoneda morphisms $\map\yon A{\ps A}$, that is the cartesianness of their identity cell $\id_\yon$ (\defref{full and faithful morphism}). Full and faithfulness of $\yon$ does not hold in general; it is instead related to the unitality of $A$ as described by the lemma below. This is in contrast to the situation for Yoneda structures, in the sense of Street and Walters \cite{Street-Walters78}, whose Axiom~3 ensures that its Yoneda embeddings are full and faithful, in their sense.
	
	By a \emph{(weak) Yoneda embedding} we will mean a (weak) Yoneda morphism that is full and faithful. All examples of Yoneda morphisms below are Yoneda embeddings. In \cororef{uniqueness of yoneda embeddings} below we will see that (weak) Yoneda embeddings are unique up to equivalence provided that they admit nullary restrictions.
	\begin{lemma} \label{full and faithful yoneda embedding}
		Let $\map\yon A{\ps A}$ be a weak Yoneda morphism. The object $A$ is unital (\defref{cartesian cells}) if and only if both the restriction $\ps A(\yon, \yon)$ exists and $\yon$ is full and faithful (\defref{full and faithful morphism}). In that case $I_A \iso \ps A(\yon, \yon)$ so that $\cur I_A \iso \yon$.
	\end{lemma}
	It follows that if $\yon$ admits nullary restrictions (\defref{yoneda embedding}) then the unitality of $A$ is equivalent to the full and faithfulness of $\yon$.
	\begin{proof}
		The `if'"/part follows immediately from \auglemref{5.14}. To prove the `only if'"/part assume that $A$ is unital with horizontal unit $\hmap{I_A}AA$ and consider the cartesian cell $\eps$ that defines $I_A$ as the restriction of $\ps A$ along $\yon$ and $\map{\cur I_A}A{\ps A}$, as supplied by the Yoneda axiom (\defref{yoneda embedding}). Weak density of $\yon$ (\defref{density definition}) implies that $\eps$ is weakly left Kan so that the composite $\eps \of \cocart$, where $\cocart$ denotes the weakly cocartesian cell defining $I_A$ (\lemref{companion identities lemma}), is an invertible vertical cell $\yon \iso \cur I_A$ by \exref{Kan extensions along horizontal units}. Composing $\eps$ with the inverse $\cur I_A \iso \yon$ we thus obtain a cartesian cell that defines $I_A$ as the restriction $\ps A(\yon, \yon)$. Moreover $\cocart$ is cartesian by \lemref{companion identities lemma} so that $\eps \of \cocart\colon \yon \iso \cur I_A$ is cartesian by the pasting lemma (\lemref{pasting lemma for cartesian cells}). Composing the latter with its own inverse we find that the identity cell $\id_\yon$ is cartesian, that is $\yon$ is full and faithful.
	\end{proof}
	
	\begin{remark} \label{dependence on full and faithfulness}
		We remark that several of our results concerning Yoneda morphisms do not require full and faithfulness. None of the results on exact cells (\secref{exact cells}) do. Several results on totality (\secref{totality section}) do, but not the main result, \thmref{total morphisms}, of that section. \thmref{presheaf objects as free cocompletions}, which describes the sense in which a Yoneda embedding $\map\yon A{\ps A}$ defines $\ps A$ as the `free small cocompletion' of $A$, requires $A$ to be unital.
	\end{remark}
	
	\subsection{Examples of enriched Yoneda embeddings}
	The following two examples show that the classical Yoneda embeddings $\yon_A$ for $\V$"/enriched categories $A$ satisfy \defref{yoneda embedding}. The first treates the simple case $\yon_I$ of the unit $\V$"/category $I$ (\exref{weighted colimits}). The second summarises how $\yon_I$ and the closed monoidal structure for $\V$"/categories can be used to ``generate'' Yoneda embeddings for other $\V$"/categories, by using the results of \secref{yoneda embeddings in a monoidal augmented virtual double category}. We will see in \exref{yoneda embedding for small enriched profunctors} that the Yoneda embeddings of the unital virtual double category $\ensProf\V$ of small $\V$"/profunctors (\augexref{2.8}), in the sense of \defref{yoneda embedding}, can be obtained by corestricting of the Yoneda embeddings $\yon_A$ to $\V$"/categories of small $\V$"/presheaves.
	\begin{example} \label{yoneda embedding for unit V-category}
		Let $\V$ be a monoidal category and let $\V \subset \V'$ be a universe enlargement (\exref{weighted colimits}). The Yoneda embedding $\map\yon I{\ps I}$ for the unit $\V$"/category $I$ (\exref{weighted colimits}) in the augmented virtual equipment $\enProf{(\V, \V')}$, of $\V$"/profunctors between $\V'$"/categories, is defined as follows. Since $\V'$ is closed monoidal $\V$ can be enriched in $\V'$, with hom $\V'$"/objects $\V(x, y) \dfn \brks{x, y}'$, and we take $\ps I \dfn \V$ to be the object of presheaves on $I$. As expected the Yoneda embedding $\map\yon I\V$ is the $\V'$"/functor that maps the single object $* \in I$ to $\yon(*) \dfn I$, the monoidal unit of $\V$: we will show that $\yon$ satisfies \defref{yoneda embedding}.
		
		First notice that the companion $\hmap{\yon_*}I\V$ exists in $\enProf{(\V, \V')}$ because for all $x \in \V$ the image $\yon_*(*, x) = \brks{I, x}' \iso x$ is (isomorphic to) a $\V$"/object (\augexref{4.6}). Hence to prove that $\yon$ is dense it suffices by \defref{density definition} to show that the cartesian cell defining $\yon_*$ is pointwise left Kan. By \exref{enriched left Kan extension} in turn it suffices to show that, for all $x \in \V$, the cartesian cell in the right"/hand side below defines $\map xI\V$ as the $y_*(\id, x)$"/weighted colimit of $\yon$ in $\enProf{\V'}$ which, by \exref{weighted colimits}, means that any cell $\phi$ below, with $H \in \V'$, factors uniquely as shown. It is easy to check that, under the isomorphisms $\brks{I, x}' \iso x$ and $\brks{I, y}' \iso y$, the factorisation $\map{\phi'}H{\brks{x, y}'}$ is the adjunct of $\map\phi{\brks{I, x}' \tens' H}{\brks{I, y}'}$.
		\begin{displaymath}
			\begin{tikzpicture}[textbaseline]
				\matrix(m)[math35]{I & I & I \\ & \V & \\};
				\path[map]	(m-1-1) edge[barred] node[above] {$y_*(\id, x)$} (m-1-2)
														edge[transform canvas={xshift=-2pt}] node[below left] {$\yon$} (m-2-2)
										(m-1-2) edge[barred] node[above] {$H$} (m-1-3)
										(m-1-3) edge[transform canvas={xshift=2pt}] node[below right] {$y$} (m-2-2);
				\path				(m-1-2) edge[cell] node[right] {$\phi$} (m-2-2);
			\end{tikzpicture} \quad = \quad \begin{tikzpicture}[textbaseline]
				\matrix(m)[math35]{I & I & I \\ & \V & \\};
				\path[map]	(m-1-1) edge[barred] node[above] {$y_*(\id, x)$} (m-1-2)
														edge[transform canvas={xshift=-2pt}] node[below left] {$\yon$} node[sloped, above, inner sep=6.5pt, font=\scriptsize, pos=0.55] {$\cart$} (m-2-2)
										(m-1-2) edge[barred] node[above] {$H$} (m-1-3)
														edge node[right, inner sep=2pt] {$x$} (m-2-2)
										(m-1-3) edge[transform canvas={xshift=2pt}] node[below right] {$y$} (m-2-2);
				\path[transform canvas={xshift=1.25em, yshift=0.5em}]	(m-1-2) edge[cell] node[right, inner sep=2pt] {$\phi'$} (m-2-2);
			\end{tikzpicture}
		\end{displaymath}
		
		Finally to prove that $\map\yon I\V$ satisfies the Yoneda axiom define, for any $\V$"/profunctor $\hmap JIB$, the $\V'$"/functor $\map{\cur J}B\V$ as follows. Set $\cur J(y) \dfn J(*, y)$ on objects and let the action $\map{\cur J}{B(y_1, y_2)}{\brks{J(*, y_1), J(*,y_2)}'}$ on hom"/objects be the adjunct of the action of $B$ on $J$. That $J$ is the restriction of $\V$ along $\yon$ and $\cur J$, as required, follows from the isomorphisms $\brks{I, \cur J(y)}' \iso \cur J(y) = J(*, y)$ which are natural in $y \in B$.
	\end{example}
	
	\begin{example} \label{enriched yoneda embedding summary}
		Consider a universe enlargement $\V \subset \V'$ (\exref{weighted colimits}) that is \emph{symmetric}, i.e.\ such that $\V$, $\V'$ and the embedding itself are symmetric monoidal; see Section~3.11 of \cite{Kelly82}. Writing $\brks{\op A, \V}'$ for the $\V'$"/category of $\V$"/presheaves on a $\V$"/category $A$ (see e.g.\ Section~2.4 of \cite{Kelly82}) consider the classical enriched Yoneda embedding $\map{\yon_A}A{\brks{\op A, \V}'}$ given by $\yon_Ax = A(\dash, x)$. In \exref{enriched yoneda embedding} below we will see that, using the closed monoidal structure of $\enProf{\V'}$, the Yoneda embedding $\map\yon I{\V}$ of the previous example ``generates'', for each $\V$"/category $A$, $\yon_A$ as a Yoneda embedding in $\enProf{(\V, \V')}$ and that $\yon_A$ admits nullary restrictions, both in the sense of \defref{yoneda embedding}. In particular any $\V$"/profunctor $\hmap JAB$ induces a $\V'$"/functor $\map{\cur J}B{\brks{\op A, \V}'}$ given by $\cur Jy = J(\dash, y)$, which is equipped with a cartesian cell (\augexref{4.6}) consisting of isomorphisms $J(x, y) \iso \inhom{\op A, \V}'(\yon x, \cur Jy)$ natural in $x \in A$ and $y \in B$. Restricted to $B = I$, so that $\V$"/profunctors $\hmap JAI$ can be identified with $\V$"/presheaves $\map J{\op A}\V$, these isomorphisms recover the classical enriched Yoneda lemma; see e.g.\ Section~2.4 of \cite{Kelly82}. Instantiating $\V \subset \V'$ by $\Set \subset \Set'$ recovers the original (unenriched) Yoneda lemma; see e.g.\ Proposition~I.1.4 of \cite{Grothendieck-Verdier72} or Section~III.2 of \cite{MacLane98}. Returning to general $\V \subset \V'$ and $B$, it follows from \propref{equivalence from yoneda embedding} below that the assignment $J \mapsto \cur J$ extends to an equivalence between $\V$"/profunctors $A \brar B$ and $\V'$"/functors $B \to \brks{\op A, \V}'$.
		
		In the case that the universe enlargement $\V \subset \V'$ does not admit a symmetric monoidal structure the Yoneda embedding $\map{\yon_A}A{\ps A}$ in $\enProf{(\V, \V')}$ can still be constructed, directly, for any $\V$"/category $A$. Indeed one can take $\V$"/presheaves on $A$, which are the objects of $\ps A$, to be $\V$"/profunctors of the form $\hmap pAI$, while taking the hom $\V'$"/object $\ps A(p, q)$ to be the ``end"/like'' limit of the diagram in $\V'$ that consists of all cospans of the form $\brks{px, qx}' \to \brks{A(x,y) \tens py, qx}' \leftarrow \brks{py, qy}'$, one for each pair $x, y \in A$, whose legs are induced by the actions of $A$ on $p$ and $q$. For details see Proposition~5.5 of \cite{Koudenburg15b}.
	\end{example}
	
	\begin{example}
		Consider the universe enlargement $\2 \subset \Set'$ of the category of truth values $\2 = (\bot \to \top)$, that maps $\bot$ to the empty set $\emptyset$ and $\top$ to the terminal set $1$. Applying \lemref{functors reflecting yoneda embeddings} below to the full inclusion $\ModRel \dfn \enProf\2 \hookrightarrow \enProf{(\2, \Set')}$ we find that the Yoneda embeddings in $\enProf{(\2, \Set')}$, as described in the previous example, reflect as Yoneda embeddings in the locally thin strict equipment $\ModRel$ of modular relations between preorders (\exref{closed modular relations}) as follows. The presheaf object $\ps A$ of a preorder $(A, \leq)$ is the set $\Dn A$ of \emph{downsets} $X \subseteq A$ in $A$, satisfying
		\begin{displaymath}
			x \leq y \quad \text{and} \quad y \in X \quad \Rightarrow \quad x \in X
		\end{displaymath}
		for all $x, y \in A$. The preorder on $\ps A \dfn \Dn A$ is given by inclusion and the Yoneda embedding $\map\yon A{\Dn A}$ is given by $\yon(x) = \downset x$, the downset generated by $x$. Under the equivalence of \propref{equivalence from yoneda embedding} below any modular relation $\hmap JAB$ corresponds to the order preserving morphism $\map{\cur J}B{\Dn A}$ given by $\cur J(y) = \dl Jy \dfn \set{x \in A \mid xJy}$, the preimage of $y \in B$ under $J$, as supplied by the Yoneda axiom (\defref{yoneda embedding}).
	\end{example}
	
	\begin{example} \label{closed-ordered closure space yoneda embedding}
		The Yoneda embedding for a closed"/ordered closure space $A$ (\exref{closed modular relations}) reflects along the forgetful functor $\map U\ClModRel\ModRel$ as follows. The closed subsets $\Cl A$ of $A$ induce a set $\Cl(\Dn A)$ of closed subsets of the preorder $\Dn A$ of downsets in $A$, making $\Dn A$ into a modular closure space (\exref{closed modular relations}), as follows. We take $\Cl(\Dn A)$ to be generated by (arbitrary intersections of) subsets of the form
		\begin{displaymath}
			V^+ \dfn \set{X \in \Dn A \mid X \isect V \neq \emptyset}
		\end{displaymath}
		where $V \in \Cl A$; compare the \emph{upper Vietoris topology} on the powerset $PA$ of a topological space $A$, see e.g.\ Section~1 of \cite{Clementino-Tholen97}. Since $\upset(V^+) = V^+$ for all $V \in \Cl A$, with respect to $\subseteq$ on $\Dn A$, it follows that each $W \in \Cl(\Dn A)$ is an upset, i.e.\ $\upset W = W$, showing that $\Dnp A \dfn \bigpars{\Dn A, \Cl(\Dn A), \subseteq}$ is a modular closure space. We call $\Dnp A$ the \emph{upper Vietoris space of downsets in $A$}.
		
		Next consider the Yoneda embedding $\map\yon A{\Dn A}$ in $\ModRel$ and, for each closed modular relation $\hmap JAB$ (\exref{closed modular relations}), its corresponding morphism $\hmap{\cur J}B{\Dn A}$; both as described in the previous example. Notice that $\yon$ as well as all $\cur J$ are continuous morphisms with respect to the closed subsets of $\Dnp A$: this follows from the closedness axioms satisfied by $A$ and $J$ (\exref{closed modular relations}) together with the fact that $\inv\yon(V^+) = \upset V$ and $\inv{(\cur J)}(V^+) = JV$ for all $V \in \Cl A$. Notice too that $\upset(\yon V) = V^+$ for $V \in \Cl A$, so that the companion $\hmap{\yon_*}A{\Dnp A}$ exists in $\ClModRel$. Using \lemref{functors reflecting yoneda embeddings} below we conclude that $\map\yon A{\Dnp A}$ forms a Yoneda embedding in $\ClModRel$. In particular, using \propref{equivalence from yoneda embedding} below, we obtain a correspondence between closed modular relations \mbox{$A \brar B$} and continuous maps $B \to \Dnp A$. Compare Proposition~3.1 of \cite{Clementino-Tholen97} which describes the related correspondence for (unordered) topological spaces, between closed relations $A \brar B$ and continuous maps $B \to P^+A$, where $P^+A$ is the powerset $PA$ equipped with the upper Vietoris topology. 
		
		Finally notice that, if $A$ itself is a modular closure space too then, again by the lemma below, we find that $\map\yon A{\Dnp A}$ also forms a Yoneda embedding in the full sub"/double category $\ClModRel_\textup m$ of $\ClModRel$ that is generated by all modular closure spaces.
	\end{example}
	
	Let $\map F\K\L$ be a functor and let $A \xrar f C \xbrar{\ul K} D \xlar g B$ be morphisms in $\K$, with $\lns{\ul K} \leq 1$. We say that $F$ \emph{creates} the restriction $\ul K(f, g)$ in $\K$ if, given a cartesian cell $\psi$ defining $(F\ul K)(Ff, Fg)$ in $\L$, there exists a unique cartesian cell $\phi$ in $\K$ that defines $\ul K(f, g)$ such that $F\phi = \psi$. Notice that if $(F\ul K)(Ff, Fg)$ exists then this means that $F$ preserves any cartesian cell that defines $\ul K(f, g)$. The following lemma is a straightforward consequence of \lemref{full and faithful functors reflect and preserve weakly left Kan cells}.
	\begin{lemma} \label{functors reflecting yoneda embeddings}
		A locally full and faithful functor (\augdefref{3.6}) $\map F\K\L$ reflects any (weak) Yoneda morphism $\map{\yon}A{\ps A}$ in $\K$, that is $\yon$ is a (weak) Yoneda morphism in $\K$ if $F\yon$ is so in $\L$, whenever the following conditions hold:
		\begin{enumerate}[label=\textup{(\alph*)}]
			\item $F$ preserves and creates restrictions of the form $\ps A(\yon, f)$, for any $\map fB{\ps A}$;
			\item $F$ is essentially full on morphisms of the form $\map{\cur{(FJ)}}{FB}{F\ps A}$, that is for every $\hmap JAB$ in $\K$ there exists $\map{\cur J}B{\ps A}$ in $\K$ such that $F(\cur J) \iso \cur{(FJ)}$, where $\cur{(FJ)}$ is given by the Yoneda axiom for $F\yon$ (\defref{yoneda embedding}).
		\end{enumerate}
	\end{lemma}
	
	\subsection{Generic subobjects as Yoneda embeddings}
	Let $\E$ be a category with finite limits. The next two examples show that a generic subobject in $\E$ (see e.g.\ Section~A1.6 of \cite{Johnstone02}) is the same as a Yoneda embedding $\map{\yon_1}1{\ps 1}$ in $\ModRel(E)$ (\exref{internal modular relations}) for the terminal internal preorder $1$ whose presheaf object $\ps 1$ is an internal partial order (\exref{internal modular relations}). Assuming that $\E$ is cartesian closed \exref{yoneda embeddings for internal preorders} then describes how $\yon_1$ generates the other Yoneda embeddings of $\ModRel(\E)$, using the results of \secref{yoneda embeddings in a monoidal augmented virtual double category}.
	\begin{example} \label{yoneda embedding for the terminal object in ModRel(E)}
		Consider the unital virtual equipment $\ModRel(\E)$ of internal modular relations in a category $\E$ with finite limits (\exref{internal modular relations}) and write $1 \dfn (1, I_1)$ for the terminal internal preorder in $\E$. Given any internal preorder $B = (B, \beta)$ we call internal modular relations $\hmap J1B$ \emph{modular subobjects} of $B$. They consist of monomorphisms $\mono jJB$ equipped with a right action \mbox{$\cell\rho{(J, \beta)}J$}, that is a (necessarily unique) morphism $\map\rho{J \times_B \beta}J$ over $B$. Modular subobjects of a discrete internal preorder $(B, I_B)$ are precisely the subobjects of $B$ in the usual sense; see e.g.\ Section~A1.3 of \cite{Johnstone02}. Given another modular subobject \mbox{$\hmap K1{(D, \delta)}$} and an order preserving morphism $\map g{(B, \beta)}{(D, \delta)}$ notice that a cell $J \Rar K$, as labelled (a) below, exists in $\ModRel(\E)$ if and only if there exists a morphism $J \to K$ in $\E$ that makes the square labelled (b) below commute. Recalling from \exref{internal modular relations} the way restrictions are created in $\ModRel(\E)$ we find that the cell $J \Rar K$ is cartesian if and only if the latter square is a pullback square. 
		\begin{displaymath}
			\begin{tikzpicture}[baseline]
			\matrix(m)[math35]{1 & B \\ 1 & D \\};
				\path[map]	(m-1-1) edge[barred] node[above] {$J$} (m-1-2)
										(m-1-2) edge[ps] node[right] {$g$} (m-2-2)
										(m-2-1) edge[barred] node[below] {$K$} (m-2-2);
				\path				(m-1-1) edge[eq] (m-2-1);
				\path[transform canvas={xshift=1.75em}]	(m-1-1) edge[cell] (m-2-1);
				\draw	(0,-1.6cm) node {(a)};
			\end{tikzpicture} \qquad \qquad \begin{tikzpicture}[baseline]
			\matrix(m)[math35]{J & B \\ K & D \\};
				\path[map]	(m-1-1)	edge[mono] node[above] {$j$} (m-1-2)
														edge (m-2-1)
										(m-1-2) edge node[right] {$g$} (m-2-2)
										(m-2-1) edge[mono] node[below] {$k$} (m-2-2);
				\draw	(0,-1.6cm) node {(b)};
			\end{tikzpicture} \qquad \qquad \begin{tikzpicture}[baseline]
			\matrix(m)[math35]{J & B \\ \yon_* & \ps 1 \\};
				\path[map]	(m-1-1)	edge[mono] node[above] {$j$} (m-1-2)
														edge (m-2-1)
										(m-1-2) edge[ps] node[right] {$\cur J$} (m-2-2)
										(m-2-1) edge[mono] node[below] {$\yon_{*1}$} (m-2-2);
				\coordinate (hook) at ($(m-1-1)+(0.4,-0.4)$);
				\draw (hook)+(0,0.17) -- (hook) -- +(-0.17,0);
				\draw	(0,-1.6cm) node {(c)};
			\end{tikzpicture} \qquad \qquad \begin{tikzpicture}[baseline]
			\matrix(m)[math35]{J & B \\ \yon_* & \ps 1 \\};
				\path[map]	(m-1-1)	edge[mono] node[above] {$j$} (m-1-2)
														edge (m-2-1)
										(m-1-2) edge[ps] node[right] {$g$} (m-2-2)
										(m-2-1) edge[mono] node[below] {$\yon_{*1}$} (m-2-2);
				\draw	(0,-1.6cm) node {(d)};
			\end{tikzpicture}
		\end{displaymath}
		
		Next consider an order preserving morphism of the form $\map\yon 1{(\ps 1, \omega)}$. Its companion $\yon_* = (1 \xlar ! \yon_* \xrar{\yon_{*1}} \ps 1)$ is a modular subobject of $\ps 1$; recall that $\yon_*$ is the pullback of $1 \xrar \yon \ps 1 \xlar{\omega_0} \omega$ and that $\yon_{*1}$ is the composite $\yon_* \to \omega \xrar{\omega_1} \ps 1$. Using the previous we can translate the Yoneda and density axioms for $\yon$ (\defsref{density definition}{yoneda embedding}) in terms of (pullback) squares in $\E$, by factoring the (cartesian) cells considered in these axioms through $\yon_*$ and by appying the pasting lemma for cartesian cells (\lemref{pasting lemma for cartesian cells}). We find that $\yon$ is a weak Yoneda embedding (and hence a Yoneda embedding, by \defref{density definition} and \exref{tabulations of internal modular relations}) if and only if for each modular subobject $\mono jJB$ there exists an order preserving morphism $\map{\cur J}B{\ps 1}$ equipped with a pullback square in $\E$ as labelled (c) above, that is universal as follows: for every commuting square as labelled (d) above, with $\map gB{\ps 1}$ order preserving, there exists a (unique) vertical cell $\cur J \Rar g$ in $\ModRel(\E)$. Using the weak density of $\yon$ this implies that $\cur J$ is unique up to isomorphism among order preserving morphisms $\map gB\ps 1$ for which there exists a pullback square of the form as on the right above.
		
		In particular if $\ps 1$ is an internal partial order (\exref{internal modular relations}) then $\cur J$ is uniquely determined by $\hmap J1B$ so that, restricting the previous to discrete $B = (B, I_B)$ and thus to ordinary subobjects $J \rightarrowtail B$ in $\E$, we find that $\mono{\yon_{*1}}{\yon_*}{\ps 1}$ is a \emph{generic subobject} in $\E$, defining $\ps 1$ as the \emph{subobject classifier} for $\E$; see e.g.\ Section A1.6 of \cite{Johnstone02}. As discussed there it follows in this case that $\yon_* \iso 1$; since the cocartesian cell defining $\hmap{\yon_*}1{\ps 1}$ necessarily consists of this isomorphism we conclude that $\yon = \bigbrks{1 \iso \yon_* \xrar{\yon_{*1}} \ps 1}$ is itself a generic subobject in~$\E$, whose companion in $\ModRel(\E)$ can be taken to be $\yon_* \iso (1 \xlar{\id} 1 \xrar\yon \ps 1)$.
	\end{example}
	
	\begin{example} \label{generic subobjects are yoneda embeddings}
		In the previous example we saw that a (weak) Yoneda embedding $\map\yon 1{\ps 1}$ in $\ModRel(\E)$ is a generic subobject in $\E$ whenever $\ps 1$ is an internal partial order. Here we will show the converse: a generic subobject $\map\gso 1\Omega$ in $\E$ lifts as a Yoneda embedding $\map\gso{(1, I_1)}{(\Omega, \omega)}$ in $\ModRel(\E)$. To start, first recall from e.g.\ Lemma~A1.6.3(a) of \cite{Johnstone02} that $\Omega$ admits the structure of an internal Heyting semilattice in $\E$, which equips $\Omega$ with an internal partial ordering \mbox{$\omega = (\Omega \xlar{\omega_0} \omega \xrar{\omega_1} \Omega)$}; see also Examples~B2.3.8(a) therein. That $\gso$ is an order preserving morphism \mbox{$\map\gso{(1, I_1)}{(\Omega, \omega)}$} (\exref{internal modular relations}) follows immediately from the definition of $\omega$; see the proof of Lemma~A1.6.3(i) of \cite{Johnstone02}, where the ordering is denoted $\Omega_1$. The latter contains another assertion that we will use: any morphism $\map{(f, g)}A{\Omega \times \Omega}$ in $\E$, where $f$ and $\map gA\Omega$ classify subobjects $f^*(\gso)$ and $g^*(\gso)$ of $A$, factors through $\map{(\omega_0, \omega_1)}\omega{\Omega \times \Omega}$ if and only if $f^*(\gso) \leq g^*(\gso)$, that is there exists a morphism $f^*(\gso) \to g^*(\gso)$ in the slice category $\E \slash A$ of $\E$ over $A$. If $f$ and $g$ are order preserving morphisms $(A, \alpha) \to (\Omega, \omega)$ in $\ModRel(\E)$ then the former means that there exists a vertical cell $f \Rar g$ in $\ModRel(\E)$. For instance notice that \mbox{$\id_\Omega^*(\top) = \top \leq \id_\Omega = (\top \of \term)^*(\top)$} so that there exists a vertical cell $\cell\phi{\id_\Omega}{\top \of \term}$ in $\ModRel(\E)$.
		
		We can use the cell $\phi$ to prove that the companion $\gso_*$ is the modular relation \mbox{$(1 \xlar{\id} 1 \xrar{\gso} \Omega)$} in $\E$, as follows. Analogous to the construction of $\yon_*$ in the previous example recall that $\gso_*$ is the pullback of $1 \xrar\gso \Omega \xlar{\omega_0} \omega$, with non"/trivial leg $\gso_{*1}$ the composite $\gso_* \to \omega \xrar{\omega_1} \Omega$. Hence it suffices to show that the commuting square on the left below is a pullback. To do so consider any morphism $\map fX\omega$ with $\omega_0 \of f = \gso \of \term$; we have to show that $f = \tilde\omega \of \gso \of \term$. Regarding $\omega_1 \of f$ as an order preserving morphism \mbox{$\map{\omega_1 \of f}{(X, I_X)}{(\Omega, \omega)}$}, the assumption on $f$ implies that there exists a vertical cell $\gso \of \term \Rar \omega_1 \of f$ in $\ModRel(\E)$. Together with the vertical cell $\cell{\phi \of \omega_1 \of f}{\omega_1 \of f}{\gso \of \term}$ we conclude that $\omega_1 \of f \iso \gso \of \term$ so that, because $\Omega$ is an internal partial order, $\omega_1 \of f = \gso \of \term$ follows; see \exref{internal modular relations}. But this means that $(\omega_0, \omega_1) \of f = (\omega_1, \omega_0) \of \tilde\omega \of \gso \of \term$ from which, as required, $f = \tilde \omega \of \gso \of \term$ follows, again by using that $\Omega$ is an internal partial order.

		We can now show that $\gso$ is weakly dense in $\ModRel(\E)$ so that, since $\ModRel(\E)$ has cocartesian tabulations (\exref{tabulations of internal modular relations}), $\gso$ is in fact dense; see \defref{density definition}. Weak density of $\gso$ means that any nullary cartesian cell in $\ModRel(\E)$ as in the middle below is weakly left Kan (\defref{weak left Kan extension}), that is any nullary cell as in the right below induces a vertical cell $l \Rar k$ in $\ModRel(\E)$. As explained in the previous example, the factorisations of these nullary cells through $\gso_* = (1 \xlar{\id} 1 \xrar\gso \Omega)$ correspond to commutative squares in $\E$: the cartesian cell corresponds to the pullback square $l \of j = \gso \of \term$ while the other cell corresponds to the commuting square $k \of j = \gso \of \term$. The pullback square implies $l^*(\gso) = j$ while the other square factors through the pullback that defines the subobject $k^*(\gso)$. We conclude that $l^*(\gso) = j \leq k^*(\gso)$ so that the existence of the cell $l \Rar k$ follows as required. 
				\begin{displaymath}
			\begin{tikzpicture}[baseline]
				\matrix(m)[math35]{1 & \omega \\ 1 & \Omega \\};
				\path[map]	(m-1-1)	edge node[above] {$\tilde\omega \of \gso$} (m-1-2)
														edge node[left] {$\id$} (m-2-1)
										(m-1-2) edge[ps] node[right] {$\omega_0$} (m-2-2)
										(m-2-1) edge[mono] node[below] {$\gso$} (m-2-2);
			\end{tikzpicture} \qquad \qquad \qquad \qquad \begin{tikzpicture}[baseline]
				\matrix(m)[math35, column sep={1.75em,between origins}]{1 & & B \\ & \Omega & \\};
				\path[map]	(m-1-1) edge[barred] node[above] {$J$} (m-1-3)
														edge[transform canvas={xshift=-2pt}] node[left] {$\top$} (m-2-2)
										(m-1-3) edge[transform canvas={xshift=2pt}] node[right] {$l$} (m-2-2);
				\draw				([yshift=0.333em]$(m-1-2)!0.5!(m-2-2)$) node[font=\scriptsize] {$\cart$};
			\end{tikzpicture} \qquad \qquad \qquad \qquad \begin{tikzpicture}[baseline]
				\matrix(m)[math35, column sep={1.75em,between origins}]{1 & & B \\ & \Omega & \\};
				\path[map]	(m-1-1) edge[barred] node[above] {$J$} (m-1-3)
														edge[transform canvas={xshift=-2pt}] node[left] {$\gso$} (m-2-2)
										(m-1-3) edge[transform canvas={xshift=2pt}] node[right] {$k$} (m-2-2);
				\path[transform canvas={yshift=0.25em}]	(m-1-2) edge[cell] (m-2-2);
			\end{tikzpicture}
		\end{displaymath}
		
		It remains to show that $\top$ satisfies the Yoneda axiom in $\ModRel(\E)$ (\defref{yoneda embedding}). To do so let $\hmap J{(1, I_1)}{(B, \beta)}$ be a modular subobject and consider the morphism $\map lB\Omega$ that classifies $J \rightarrowtail B$. It suffices to show that $l$ is an order preserving morphism $(B, \beta) \to (\Omega, \omega)$: as explained in the previous example the pullback square defining $l$ then corresponds to the cartesian cell in the middle above in $\ModRel(\E)$. That $l$ is order preserving means that $\map{(l \of \beta_0, l \of \beta_1)}\beta{\Omega \times \Omega}$ factors through $(\omega_0, \omega_1)$. Composing the pullback square that defines $l$ with the action $\map\rho{J \times_B \beta}J$ we obtain the commuting square $\bigbrks{J \times_B \beta \to \beta \xrar{\beta_1} B \xrar l \Omega} = \gso \of \term$. The latter factors through the pullback square that defines $(l \of \beta_1)^*(\gso)$ so that $(l \of \beta_0)^*(\gso) = \bigbrks{J \times_B \beta \to \beta} \leq (l \of \beta_1)^*(\gso)$; hence $(l \of \beta_0, l \of \beta_1)$ factors through $(\omega_0, \omega_1)$ as required.
	\end{example}
	
	\begin{example} \label{yoneda embeddings for internal preorders}
		Assume that $\E$ is a cartesian closed category with finite limits. A Yoneda embedding $\map\yon 1{\ps 1}$ in $\ModRel(\E)$ (\exref{yoneda embedding for the terminal object in ModRel(E)}) induces Yoneda embeddings for all internal preorders $A = (A, \alpha)$ in $\E$ as follows. Writing $\dl A$ for the `horizontal dual' $\dl A \dfn (A, \dl\alpha)$ of $A$, with reversed internal order $\dl\alpha = (A \xlar{\alpha_1} \alpha \xrar{\alpha_0} A)$, we will see in \exref{horizontal dual of internal preorders} below that the internal order $\alpha$ can be considered as a modular subobject of $\dl A \times A$ which, under the correspondence of \exref{yoneda embedding for the terminal object in ModRel(E)}, induces an order preserving morphism $\dl A \times A \to \ps 1$. The cartesian closed structure on $\E$ induces a cartesian closed structure on the locally thin $2$"/category $\PreOrd(\E) = V\bigpars{\ModRel(\E)}$ (\exsref{internal modular relations}{ModRel(E) is closed cartesian monoidal}) of internal preorders, under which the latter corresponds to an order preserving morphism \mbox{$\map{\yon_A}A{\inhom{\dl A, \ps 1}}$}. In \exref{yoneda embeddings for internal preorders summary} we will see that $\yon_A$ forms a Yoneda embedding in $\ModRel(\E)$. In particular, using \propref{equivalence from yoneda embedding} below, internal modular relations $A \brar B$ correspond to order preserving morphisms $B \to \inhom{\dl A, \ps 1}$.
		
		Next assume that $\ps 1$ is an internal partial order so that $\yon$ defines $\ps 1$ as the subobject classifier $\Omega \dfn \ps 1$ for $\E$ (\exref{yoneda embedding for the terminal object in ModRel(E)}). In this case the latter correspondence recovers the natural equivalence considered on page~283 of \cite{Carboni-Street86}. Moreover, using that $I_A \iso \inhom{\dl A, \ps 1}(\yon, \yon)$ by \lemref{full and faithful yoneda embedding} we find that our $\map{\yon_A}A{\inhom{\dl A, \ps 1}}$ coincides with the Yoneda embedding $\map{y_A}A{\mathcal PA}$ considered there. In our terms the `membership ideal', considered on the same page, is the companion $\hmap{\in_A \dfn \yon_{A*}}A{\inhom{\dl A, \ps 1}}$. Finally restrict to a discrete internal preorder $A = (A, I_A)$ in the previous, so that $\dl A = A$ and, as we will see in \exref{ModRel(E) is closed cartesian monoidal}, $\inhom{A, \ps 1}$ is an internal partial order with the exponential ${\ps 1}^A$ as underlying $\E$"/object. Using arguments similar to the ones used in \exref{yoneda embedding for the terminal object in ModRel(E)} it is straightforward to see that, in this case, $\yon_A$ being a Yoneda embedding in $\ModRel(\E)$ implies that the internal relation underlying $\in_A = \yon_{A*}$ defines ${\ps 1}^A$ as the \emph{power object} of $A$ in $\E$ in the classical sense; see e.g.\ Section~A2.1 of \cite{Johnstone02}.
	\end{example}
	
	\subsection{Properties of Yoneda morphisms}
	The results below record some basic properties of Yoneda morphisms. The first of these follows immediately from the density of Yoneda morphisms (\defref{density definition}) and the vertical pasting lemma (\lemref{vertical pasting lemma}). 
	\begin{corollary} \label{left Kan extensions of yoneda embeddings along paths}
		Let $\map\yon C{\ps C}$ be a (weak) Yoneda morphism. Consider the composite below where $\cur K$ is supplied by the Yoneda axiom (\defref{yoneda embedding}). If the cell $\phi$ is right (respectively weakly) nullary"/cocartesian (\defref{cocartesian path}) then the composite is (weakly) left Kan (\defsref{weak left Kan extension}{left Kan extension}).
		
		If moreover $\phi$ restricts along $\map gXD$ (\defref{pointwise right cocartesian path}) then so does the (weakly) left Kan composite (\defref{pointwise left Kan extension}). If $\phi$ is pointwise right (respectively weakly) nullary"/cocartesian (\defref{pointwise right cocartesian path}) then the composite is pointwise (weakly) left Kan (\defref{pointwise left Kan extension}).
		\begin{displaymath}
			\begin{tikzpicture}
				\matrix(m)[math35, column sep={1.75em,between origins}]
					{ A_0 & & A_1 & \cdots & A_{n'} & & D \\
						& C & & & & D & \\
						& & & \ps C & & & \\ };
				\path[map]	(m-1-1) edge[barred] node[above] {$J_1$} (m-1-3)
														edge[transform canvas={xshift=-2pt}] node[left] {$f$} (m-2-2)
										(m-1-5) edge[barred] node[above] {$J_n$} (m-1-7)
										(m-2-2) edge[barred] node[below, inner sep=1.5pt] {$K$} (m-2-6)
														edge node[below left] {$\yon$} (m-3-4)
										(m-2-6) edge node[below right] {$\cur K$} (m-3-4);
				\path				(m-1-7) edge[transform canvas={xshift=1pt}, eq] (m-2-6)
										(m-1-4) edge[cell] node[right] {$\phi$} (m-2-4);
				\draw				([yshift=0.25em]$(m-2-4)!0.5!(m-3-4)$) node[font=\scriptsize] {$\cart$};
			\end{tikzpicture}
		\end{displaymath}
	\end{corollary}
	
	The lemma below is a variation of Lemma~3.2 of \cite{Weber07}. It implies that the notions of weak left Kan extension and pointwise weak left Kan extension coincide when extending along a weak Yoneda morphism, and that all four notions of left Kan extension coincide when extending along a Yoneda morphism.
	\begin{lemma} \label{weak left Kan extensions of yoneda embeddings}
		Let $\map\yon A{\ps A}$ be a (weak) Yoneda morphism. The following are equivalent for the cell $\eta$ below: \textup{(a)} $\eta$ is cartesian; \textup{(b)} $\eta$ is pointwise (weakly) left Kan (\defref{pointwise left Kan extension}); \textup{(c)} $\eta$ is weakly left Kan (\defref{weak left Kan extension}).
		\begin{displaymath}
			\begin{tikzpicture}
				\matrix(m)[math35, column sep={1.75em,between origins}]{A & & B \\ & \ps A & \\};
				\path[map]	(m-1-1) edge[barred] node[above] {$J$} (m-1-3)
														edge[transform canvas={xshift=-2pt}] node[left] {$\yon$} (m-2-2)
										(m-1-3) edge[transform canvas={xshift=2pt}] node[right] {$l$} (m-2-2);
				\path[transform canvas={yshift=0.25em}]	(m-1-2) edge[cell] node[right, inner sep=3pt] {$\eta$} (m-2-2);
			\end{tikzpicture}
	  \end{displaymath}
	\end{lemma}
	\begin{proof}
		(a) $\Rar$ (b) follows from the (weak) density of $\yon$ (\defref{density definition}) and (b)~$\Rar$~(c) is clear. To prove (c) $\Rar$ (a) consider the morphism $\map{\cur J}B{\ps A}$ supplied by the Yoneda axiom (\defref{yoneda embedding}), which comes equipped with a cartesian cell that defines $J$ as the nullary restriction of $\ps A$ along $\yon$ and $\cur J$. By (weak) density of $\yon$ this cartesian cell is weakly left Kan so that, assuming (c) and using the uniqueness of left Kan extensions, $\eta$ factors through it as a vertical isomorphism $\cur J \iso l$. Since cartesian cells are preserved by horizontal composition with vertical isomorphisms we conclude that the cell $\eta$ is cartesian too, as required.
	\end{proof}
	
	The following is a variation of Corollary~3.5(2) of \cite{Weber07}.
	\begin{proposition}
		Let $\map\yon A{\ps A}$ be a weak Yoneda embedding. If the composite below is invertible then $\map fAC$ is full and faithful (\defref{full and faithful morphism}). The converse holds whenever the restriction $C(f, f)$ exists.
		\begin{displaymath}
			\begin{tikzpicture}
    		\matrix(m)[math35, column sep={1.75em,between origins}]{& A & \\ A & & C \\ & \ps A & \\};
 	  		\path[map]	(m-1-2) edge[transform canvas={xshift=2pt}] node[right] {$f$} (m-2-3)
 	  								(m-2-1) edge[barred] node[below, inner sep=2pt] {$f_*$} (m-2-3)
 	  												edge[transform canvas={xshift=-2pt}] node[left] {$\yon$} (m-3-2)
 	  								(m-2-3)	edge[transform canvas={xshift=2pt}] node[right] {$\cur{f_*}$} (m-3-2);
 	  		\path				(m-1-2) edge[eq, transform canvas={xshift=-2pt}] (m-2-1);
 	  		\draw				([yshift=-0.5em]$(m-1-2)!0.5!(m-2-2)$) node[font=\scriptsize] {$\cocart$}
 	  								([yshift=0.25em]$(m-2-2)!0.5!(m-3-2)$) node[font=\scriptsize] {$\cart$};
 			\end{tikzpicture}
		\end{displaymath}
	\end{proposition}
	\begin{proof}
		Assume that the composite is invertible. Since the identity cell $\id_\yon$ is cartesian and because cartesian cells are preserved under horizontal composition with invertible vertical cells, it follows that the composite is cartesian as well. Hence the cocartesian cell $\cocart$ is cartesian by the pasting lemma for cartesian cells (\lemref{pasting lemma for cartesian cells}). Since $\id_f$ factors through the latter as a cartesian cell $f_* \Rar C$ (see \lemref{companion identities lemma}), the same pasting lemma implies that $\id_f$ is cartesian too, so that $f$ is full and faithful. The converse follows from the fact that the cartesian cell defining $\cur{f_*}$ is weakly left Kan, by the weak density of $\yon$ (\defref{density definition}), and \propref{pointwise left Kan extension along full and faithful map}.
	\end{proof}
	
	The following lemma is important to our study of Yoneda morphisms: its corollary \propref{left Kan extensions along a yoneda embedding in terms of left y-exact cells}, for instance, is used in \thmref{presheaf objects as free cocompletions} to give a condition that ensures the `cocompleteness' (\defref{cocompletion}) of presheaf objects. Its restriction to empty paths $\ul H = (B)$ is analogous to Proposition~7 of \cite{Street-Walters78} for Yoneda structures.
	\begin{lemma} \label{bijection between cells induced by Yoneda morphisms}
		Let $\map\yon A{\ps A}$ be a Yoneda morphism. The equality below determines a bijection between cells $\phi$ and $\psi$ of the forms as shown.
		\begin{displaymath}
			\begin{tikzpicture}[textbaseline]
				\matrix(m)[math35, column sep={1.75em,between origins}]
					{	A & & B & & B_1 & \cdots & B_{n'} & & B_n \\
						& & & A & & D & & & \\
						& & & & \ps A & & & & \\};
				\path[map]	(m-1-1) edge[barred] node[above] {$J$} (m-1-3)
										(m-1-3) edge[barred] node[above] {$H_1$} (m-1-5)
										(m-1-7) edge[barred] node[above] {$H_n$} (m-1-9)
										(m-1-9) edge[transform canvas={xshift=2pt}] node[below right] {$s$} (m-2-6)
										(m-2-4) edge[barred] node[below, inner sep=2pt] {$K$} (m-2-6)
														edge[transform canvas={xshift=-2pt}] node[left] {$\yon$} (m-3-5)
										(m-2-6) edge[transform canvas={xshift=2pt}] node[right] {$\cur K$} (m-3-5);
				\path				(m-1-1) edge[transform canvas={xshift=-1pt}, eq] (m-2-4)
										(m-1-5) edge[cell] node[right] {$\phi$} (m-2-5);
				\draw				([yshift=0.25em]$(m-2-5)!0.5!(m-3-5)$) node[font=\scriptsize] {$\cart$};
			\end{tikzpicture} \quad = \quad \begin{tikzpicture}[textbaseline]
				\matrix(m)[math35, column sep={1.75em,between origins}]
					{	A & & B & & B_1 & \cdots & B_{n'} & & B_n \\
						& & & & & D & & & \\
						& & \ps A & & & & & & \\};
				\path[map]	(m-1-1) edge[barred] node[above] {$J$} (m-1-3)
														edge[transform canvas={xshift=-2pt}] node[left] {$\yon$} (m-3-3)
										(m-1-3) edge[barred] node[above] {$H_1$} (m-1-5)
														edge[ps] node[right] {$\cur J$} (m-3-3)
										(m-1-7) edge[barred] node[above] {$H_n$} (m-1-9)
										(m-1-9) edge[transform canvas={xshift=2pt}] node[below right] {$s$} (m-2-6)
										(m-2-6) edge[transform canvas={xshift=2pt}] node[below right] {$\cur K$} (m-3-3);
				\path				(m-1-6) edge[transform canvas={xshift=-1.5em,yshift=-0.5em}, cell] node[right] {$\psi$} (m-2-6);
				\draw				([yshift=1.15em,xshift=0.5em]$(m-1-2)!0.5!(m-3-2)$) node[font=\scriptsize] {$\cart$};
			\end{tikzpicture}
		\end{displaymath}
		
		If $\yon$ is merely a weak Yoneda morphism then the equality above determines a bijection between cells $\phi$ and $\psi$ with $\ul H = (B)$ empty. Under the latter restriction $\phi$ is cartesian if and only if the corresponding $\psi$ (in this case a vertical cell) is invertible.
	\end{lemma}
	\begin{proof}
		The bijection is given by the assignment $\phi \mapsto \psi$ obtained by factorising the left"/hand side above through the cartesian cell that defines $\cur J$ (\defref{yoneda embedding}), which is (weakly) left Kan by density of $\yon$ (\defref{density definition}), and the assignment $\psi \mapsto \phi$ obtained by factorising the right"/hand side through the cartesian cell that defines $\cur K$. Restricting to empty paths $\ul H = (B)$, notice that $\psi$ is invertible if and only if the right"/hand side, and hence both sides, are weakly left Kan. By the previous lemma this is equivalent to the left"/hand side being cartesian which, by the pasting lemma for cartesian cells (\lemref{pasting lemma for cartesian cells}), is in turn equivalent to $\phi$ being cartesian.
	\end{proof}
		
	\subsection{Equivalence of morphisms $A \brar B$ and morphisms $B \to \ps A$}
	Let $\map\yon A{\ps A}$ be a weak Yoneda morphism. The proposition below uses the previous lemma to describe the functoriality of the assignment $J \mapsto \cur J$ (\defref{yoneda embedding}), showing also that the resulting functor is an equivalence if and only if $\yon$ admits nullary restrictions (\defref{yoneda embedding}).
	
	In stating the proposition we use the following two notions of slice category in an augmented virtual double category $\K$. Given an object $A$ of $\K$ the \emph{horizontal slice category} $A \hs \K$ has as objects horizontal morphisms $\hmap JAB$ and as morphisms $J \to K$ cells $\phi$ of the form as on the left below. Fixing a target $B$ in $\K$ we denote by $H(\K)(A, B) \subseteq A \hs \K$ the subcategory generated by morphisms $J \to K$ with vertical target $s = \id_B$. Next recall that $\K$ contains a $2$"/category $V(\K)$ of vertical morphisms (\augexref{1.5}). Given an object $P$ of $\K$ we denote by $V(\K) \slash P$ the \emph{lax slice category} consisting of vertical morphisms $\map gAP$ as objects and cells $\psi$ of the form as on the right below as morphisms $g \to h$; see e.g.\ Section I,2.5 of \cite{Gray74} for the more general notion of ``lax comma $2$"/category'' (therein called ``$2$"/comma category''). Notice that $V(\K) \slash P$ contains the hom"/categories $V(\K)(A, P)$ as subcategories.
	\begin{displaymath}
		\begin{tikzpicture}
			\matrix(m)[math35]{A & B \\ A & D \\};
				\path[map]	(m-1-1) edge[barred] node[above] {$J$} (m-1-2)
										(m-1-2) edge node[right] {$s$} (m-2-2)
										(m-2-1) edge[barred] node[below] {$K$} (m-2-2);
				\path				(m-1-1) edge[eq] (m-2-1)
										(m-1-1) edge[cell, transform canvas={xshift=1.75em}] node[right] {$\phi$} (m-2-1);
		\end{tikzpicture} \qquad\qquad\qquad\qquad\qquad\qquad\qquad \begin{tikzpicture}
			\matrix(m)[math35, column sep={1.75em,between origins}]{A & & B \\ & P & \\};
				\path[map]	(m-1-1) edge node[above] {$s$} (m-1-3)
														edge[transform canvas={xshift=-2pt}] node[left] {$g$} (m-2-2)
										(m-1-3) edge[transform canvas={xshift=2pt}] node[right] {$h$} (m-2-2);
				\path				($(m-1-1)!0.5!(m-2-2)$) edge[transform canvas={}, cell, shorten >= 4.5pt, shorten <= 4.5pt] node[below, xshift=2pt] {$\psi$} (m-1-3);
		\end{tikzpicture}
	\end{displaymath}
	\begin{proposition} \label{equivalence from yoneda embedding}
		Let $\map\yon A{\ps A}$ be a weak Yoneda morphism in an augmented virtual double category $\K$. Choosing for each $\hmap JAB$ a morphism $\map{\cur J}B{\ps A}$ as in the Yoneda axiom (\defref{yoneda embedding}) induces a full and faithful functor
		\begin{displaymath}
			\map{\cur{(\dash)}}{A \hs \K}{V(\K) \slash \ps A}
		\end{displaymath}
		that maps a cell $\cell\phi JK$ as on the left above to the vertical cell \mbox{$\cell{\cur\phi}{\cur J}{\cur K \of s}$} that corresponds to $\phi$ under the bijection of \lemref{bijection between cells induced by Yoneda morphisms} (with $\ul H = (B)$ empty). Fixing the target $B$ restricts $\cur{(\dash)}$ to a full and faithful functor
		\begin{displaymath}
			\map{\cur{(\dash)}}{H(\K)(A, B)}{V(\K)(B, \ps A)}.
		\end{displaymath}
		
		Moreover $\cur{(\dash)}$ is an equivalence of categories $A \hs \K \simeq V(\K) \slash \ps A$ if and only if $\yon$ admits nullary restrictions (\defref{yoneda embedding}).
	\end{proposition}
	\begin{proof}
		To show that $\cur{(\dash)}$ is functorial consider composable morphisms in $A \hs \K$, that is cells $\cell\phi JK$ and $\cell\psi KL$ both with vertical source $\id_A$; let us denote their vertical targets by $\map sBD$ and $\map tDE$ respectively. Using the bijection of \lemref{bijection between cells induced by Yoneda morphisms} and the interchange axioms (\auglemref{1.3}) we obtain the following equality, where the cartesian cells define the morphisms $\cur J$, $\cur K$ and $\cur L$.
		\begin{multline*}
			\cart_{\cur J} \hc \cur{(\psi \of \phi)} = \cart_{\cur L} \of \psi \of \phi = (\cart_{\cur K} \hc \cur\psi) \of \phi\\
				= (\cart_{\cur K} \of \phi) \hc (\cur\psi \of s) = \cart_{\cur J} \hc \cur\phi \hc (\cur\psi \of s)
		\end{multline*}
		By weak density of $\yon$ (\defref{density definition}) the cartesian cell $\cart_{\cur J}$ is weakly left Kan so that, using the uniqueness of factorisations through left Kan cells, the equality above implies $\cur{(\psi \of \phi)} = \cur\phi \hc (\cur\psi \of s)$ as required. That the functor $\cur{(\dash)}$ is full and faithful is because the correspondence of \lemref{bijection between cells induced by Yoneda morphisms} is a bijection.
		
		To show the final assertion it suffices to prove that $\yon$ admits nullary restrictions if and only if $\cur{(\dash)}$ is essentially surjective (see e.g.\ Theorem~1 of Section~IV.4 of \cite{MacLane98}). For the `only if'"/part consider any morphism $\map gB{\ps A}$ and assume that the restriction $\hmap{\ps A(\yon, g)}AB$ exists; we will show that $g \iso \cur{\ps A(\yon, g)}$ (\defref{yoneda embedding}). Consider the cartesian cells that define $\ps A(\yon ,g)$ and $\cur{\ps A(\yon, g)}$. By weak density of $\yon$ (\defref{density definition}) both define the weak left Kan extension of $\yon$ along $\ps A(\yon, g)$ so that the required isomorphism exists by the uniqueness of Kan extensions. To show the converse assume that $\cur{(\dash)}$ is essentially surjective: for every $\map gB{\ps A}$ there exists $\hmap JAB$ with $\cur J \iso g$. Composing the latter isomorphism with the cartesian cell defining $\cur J$ (\defref{yoneda embedding}) we obtain a cartesian cell that defines $J$ as the nullary restriction $\ps A(\yon ,g)$.
	\end{proof}
	
	\begin{remark}
		Consider a Yoneda structure on a $2$"/category $\mathcal C$, in the sense of \cite{Street-Walters78}, consisting of a right ideal $\mathcal A$ of `admissible' morphisms in $\mathcal C$ and a `Yoneda embedding' $\map{yA}A{\mathcal PA}$ for each admissible object $A$ (i.e.\ with $\id_A \in \mathcal A$). Similar to the above result such a structure induces full and faithful functors \mbox{$\mathcal C_{\mathcal A}(A, B) \to \mathcal C(B, \mathcal PA)$}, given by \mbox{$f \mapsto B(f, 1)$} (see Axiom~1 and Proposition~7 of \cite{Street-Walters78}), where \mbox{$\mathcal C_{\mathcal A}(A, B) \subseteq \mathcal C(A, B)$} denotes the full subcategory of admissible morphisms. These functors however are not essentially surjective in general, as can be easily seen by taking $\mathcal C = \inCat{\Set'}$ the $2$"/category of large categories and $A = B = 1$ the terminal category.
	\end{remark}
	
	The following result is a partial converse to the previous proposition.
	\begin{proposition} \label{density and the yoneda axiom in terms of nullary restrictions}
		Let $\map fAP$ be a morphism in an augmented virtual double category $\K$ and let $B \in \K$ be an object. Choosing nullary restrictions $\hmap{P(f, g)}AB$ for all $\map gBP$ induces a functor
		\begin{displaymath}
			\map{P(f, \dash)}{V(\K)(B, P)}{H(\K)(A,B)}
		\end{displaymath}
		which maps a vertical cell $\cell\phi gh$ to the unique factorisation \mbox{$\cell{P(f, \phi)}{P(f, g)}{P(f, h)}$} in the right"/hand side below.
		\begin{displaymath}
			\begin{tikzpicture}[textbaseline]
  			\matrix(m)[math35]{A & B \\ & P \\};
  			\path[map]	(m-1-1) edge[barred] node[above] {$P(f, g)$} (m-1-2)
  													edge[transform canvas={xshift=-3pt}] node[below left] {$f$} (m-2-2)
  									(m-1-2) edge[bend right=30] node[left, inner sep=0.5pt] {$g$} (m-2-2)
  													edge[bend left=30] node[right] {$h$} (m-2-2);
  			\path				(m-1-2) edge[cell, transform canvas={xshift=-1pt}] node[right, inner sep=1.5pt] {$\phi$} (m-2-2);
  			\draw				($(m-1-1)!0.5!(m-2-2)$) node[yshift=0.75em, xshift=2pt, font=\scriptsize] {$\cart$};
  		\end{tikzpicture} \quad = \quad \begin{tikzpicture}[textbaseline]
  			\matrix(m)[math35, column sep={1.75em,between origins}]{A & & B \\ A & & B \\ & P & \\};
  			\path[map]	(m-1-1) edge[barred] node[above] {$P(f, g)$} (m-1-3)
  									(m-2-1) edge[barred] node[below, inner sep=2pt] {$P(f, h)$} (m-2-3)
  													edge[transform canvas={xshift=-2pt}] node[left] {$f$} (m-3-2)
  									(m-2-3) edge[transform canvas={xshift=2pt}] node[below right] {$h$} (m-3-2);
  			\path				(m-1-1) edge[eq] (m-2-1)
  									(m-1-3) edge[eq] (m-2-3)
  									(m-1-2) edge[cell, transform canvas={xshift=-1.2em}] node[right,inner sep=2pt] {$P(f, \phi)$} (m-2-2);
  			\draw				($(m-2-2)!0.5!(m-3-2)$) node[font=\scriptsize] {$\cart$};
  		\end{tikzpicture}
		\end{displaymath}
		The following hold:
		\begin{enumerate}[label=\textup{(\alph*)}]
			\item $f$ is weakly dense (\defref{density definition}) if and only if the functors $P(f, \dash)$ are full and faithful for each object $B \in \K$;
			\item $f$ satisfies the Yoneda axiom (\defref{yoneda embedding}) if and only if the functors $P(f, \dash)$ are essentially surjective for each $B \in \K$;
			\item $f$ is a weak Yoneda morphism if and only if the functors $P(f, \dash)$ are equivalences of categories $V(\K)(B, P) \simeq H(\K)(A, B)$ for each $B \in \K$.
		\end{enumerate}
		
		Moreover if $f$ is a weak Yoneda morphism then the functors $P(f, \dash)$ extend to a pseudo"/inverse to the equivalence $\cur{(\dash)}\colon A \hs \K \simeq V(\K) \slash P$ of \propref{equivalence from yoneda embedding}, that is induced by $f$.
	\end{proposition}
	\begin{proof}
		To see part (b) notice that the Yoneda axiom asserts that, for every horizontal morphism $\hmap JAB$, there exists a vertical morphism $\map{\cur J}BP$ such that \mbox{$J \iso P(f, \cur J)$}. That part (c) follows from parts (a) and (b) is well"/known (see e.g.\ Theorem~1 of Section~IV.4 of \cite{MacLane98}). To prove part (a) consider the assignments below between cells of $\K$ of the form as shown, given by composition with the cartesian cells that define the chosen nullary restrictions $P(f, g)$ and $P(f, h)$ respectively. By definition (\defref{cartesian cells}) the assignment on the right is a bijection.
		\begin{displaymath}
			\begin{tikzpicture}[baseline]
				\matrix(m)[math35, xshift=-14em]{B \\ P \\};
				\path[map]	(m-1-1) edge[bend right=45] node[left] {$g$} (m-2-1)
														edge[bend left=45] node[right] {$h$} (m-2-1);
				\path	(m-1-1) edge[cell] (m-2-1);
			
				\matrix(m)[math35, column sep={1.75em,between origins}]{A & & B \\ & P & \\};
				\path[map]	(m-1-1) edge[barred] node[above] {$P(f, g)$} (m-1-3)
														edge[transform canvas={xshift=-2pt}] node[left] {$f$} (m-2-2)
										(m-1-3) edge[transform canvas={xshift=2pt}] node[right] {$h$} (m-2-2);
				\path				(m-1-2) edge[transform canvas={yshift=0.333em}, cell] (m-2-2);
			
				\matrix(m)[math35, xshift=15em]{A & B \\ A & B \\};
				\path[map]	(m-1-1) edge[barred] node[above] {$P(f, g)$} (m-1-2)
										(m-2-1) edge[barred] node[below] {$P(f, h)$} (m-2-2);
				\path				(m-1-1) edge[eq] (m-2-1)
										(m-1-2) edge[eq] (m-2-2)
										(m-1-1) edge[transform canvas={xshift=1.75em}, cell] (m-2-1);
				
				\draw[font=\Large]	(-16.3em, 0) node {$\lbrace$}
										(-11.7em, 0) node {$\rbrace$}
										(-2.5em, 0) node {$\lbrace$}
										(2.5em, 0) node {$\rbrace$}
										(12.5em, 0) node {$\lbrace$}
										(17.4em, 0) node {$\rbrace$};
				
				\path[map]	(-9.7em, 0) edge node[above] {$\cart_{P(f, g)} \hc \dash$} (-4.5em,0)
										(10.5em, 0) edge node[above] {$\cart_{P(f, h)} \of \dash$} (4.5em, 0);
			\end{tikzpicture}
		\end{displaymath}
		By definition the image under $P(f, \dash)$ of a vertical cell $\cell\phi gh$ as on the left above is the horizontal cell $\cell{P(f, \phi)}{P(f, g)}{P(f, h)}$ that corresponds to $\cart_{P(f, g)} \hc \phi$ under the bijection on the right. Hence $P(f, \dash)$ is full and faithful precisely if the assignment $\cart_{P(f, g)} \hc \dash$ on the left is a bijection for each $\map gBP$. By \defref{weak left Kan extension} the latter means that any nullary cartesian cell with $f$ as vertical source is weakly left Kan, that is $f$ is weakly dense (\defref{density definition}) as required.
		
		For the final assertion assume that $\map fAP$ is a weak Yoneda morphism, thus inducing the an equivalence of categories $\cur{(\dash)}\colon A \hs \K \simeq V(\K) \slash P$ as in \propref{equivalence from yoneda embedding}. As explained in proof of the latter $\cur{(\dash)}$ is essentially surjective because $g \iso \cur{P(f, g)}$ for each $\map gAP$. It follows that the assignment \mbox{$g \mapsto P(f, g)$} uniquely extends to a pseudo"/inverse to $\cur{(\dash)}$, with the latter isomorphims forming the unit (see e.g.\ Theorem~1 of Section~IV.4 of \cite{MacLane98}). That the action of this pseudo"/inverse on vertical cells coincides with that of the functors $P(f, \dash)$ of the statement is straightforward to check.
	\end{proof}
	
	\begin{remark}
		Given a morphism $\map yC{\bar C}$ in a `bicategory $\mathcal C$ equipped with proarrows', in the sense of \cite{Wood82}, Melli\`es and Tabareau consider in \cite{Mellies-Tabareau08} functors analogous to the functors $P(f, \dash)$ of \propref{density and the yoneda axiom in terms of nullary restrictions}; they define $y$ to be a `Yoneda situation' if both $y$ and each of these functors are full and faithful.
	\end{remark}
	
	\subsection{Yoneda embeddings from $2$"/toposes}
	Weber shows in \cite{Weber07} that every `$2$"/topos' $\mathcal C$ admits a `good Yoneda structure', the construction of which he attributes to Street (\cite{Street74a} and \cite{Street80a}). Given a finitely complete cartesian closed category $\E$, in Section~7 of \cite{Street17} this result is used to obtain a good Yoneda structure on the $2$"/category $\inCat \E$ of categories internal to $\E$. Similarly to Weber's result, \exref{yoneda embeddings in 2-topoi} below shows that a $2$"/topos structure on a $2$"/category $\mathcal C$ induces a collection of Yoneda embeddings, in our sense, in a certain full sub"/augmented virtual equipment of the unital virtual equipment $\dFib{\mathcal C}$ of discrete two"/sided fibrations in $\mathcal C$ (\exref{discrete two-sided fibrations}). We will use the following lemma, which generalises the functors $P(f, \dash)$ described in \propref{density and the yoneda axiom in terms of nullary restrictions}, to functors $K(\id, \dash)$ given by restriction of a fixed horizontal morphism $K$ on the right.
		
		\begin{lemma} \label{restricting on the right is pseudofunctorial}
		Let $\hmap KCD$ be a morphism and $B$ an object in an augmented virtual double category $\K$. Choosing a restriction $\hmap{K(\id, g)}CB$ for each morphism $\map gBD$ induces a functor $\map{K(\id, \dash)}{V(\K)(B, D)}{H(\K)(C, B)}$ which maps a cell $\cell\phi gh$ to the unique factorisation $\cell{K(\id,\phi)}{K(\id, g)}{K(\id, h)}$ in the right"/hand side below:
		\begin{displaymath}
			\begin{tikzpicture}[textbaseline]
				\matrix(m)[math35]{C & B \\ C & D \\};
				\path[map]	(m-1-1) edge[barred] node[above] {$K(\id, g)$} (m-1-2)
										(m-1-2) edge[bend right=30] node[below left, inner sep=1pt] {$g$} (m-2-2)
														edge[bend left=30] node[right] {$h$} (m-2-2)
										(m-2-1) edge[barred] node[below] {$K$} (m-2-2);
				\path[transform canvas={xshift=-1.5pt}]	(m-1-1) (m-1-2) edge[cell] node[right, inner sep=2.5pt] {$\phi$} (m-2-2);
				\path				(m-1-1) edge[eq] (m-2-1);
				\draw[font=\scriptsize]	([xshift=-3pt]$(m-1-1)!0.5!(m-2-2)$) node {$\cart$};
			\end{tikzpicture} = \begin{tikzpicture}[textbaseline]
				\matrix(m)[math35]{C & B \\ C & B \\ C & D \\};
				\path[map]	(m-1-1) edge[barred] node[above] {$K(\id, g)$} (m-1-2)
										(m-2-1) edge[barred] node[below, inner sep=2pt] {$K(\id, h)$} (m-2-2)
										(m-2-2) edge node[right] {$h$} (m-3-2)
										(m-3-1) edge[barred] node[below] {$K$} (m-3-2);
				\path				(m-1-1) edge[eq] (m-2-1)
										(m-1-2) edge[eq] (m-2-2)
										(m-2-1) edge[eq] (m-3-1)
										(m-1-1) edge[cell, transform canvas={xshift=0.475em}] node[right] {$K(\id,\phi)$} (m-2-1);
				\draw[font=\scriptsize]	($(m-2-1)!0.5!(m-3-2)$) node {$\cart$};
			\end{tikzpicture}
		\end{displaymath}
	\end{lemma}
	
	\begin{example} \label{yoneda embeddings in 2-topoi}
		Let $\mathcal C$ be a finitely complete $2$"/category with terminal object denoted by $1$. A \emph{discrete opfibration} (Section~2 of \cite{Weber07}) is, in our terms, a discrete two"/sided fibration $\hmap J1B$ in $\dFib{\mathcal C}$ (\exref{discrete two-sided fibrations}). In Definition~4.1 of \cite{Weber07} a discrete opfibration $\hmap \tau 1\Omega$ is called \emph{classifying} if, in our terms,  for each $B \in \mathcal C$ the functor \mbox{$\map{\tau(\id, \dash)}{\mathcal C(B, \Omega)}{H(\dFib{\mathcal C})(1, B)}$} of the previous lemma, given by pulling back $\tau$, is full and faithful. Definition~4.10 of \cite{Weber07} defines a \emph{$2$"/topos} $(\mathcal C, (\dash)^\circ, \tau)$ to consist of a finitely complete cartesian closed $2$"/category $\mathcal C$ equipped with a `duality involution' $\map{(\dash)^\circ}{\co{\mathcal C}}{\mathcal C}$ (Definition~2.14 of \cite{Weber07}) and a classifying discrete opfibration $\hmap\tau 1\Omega$.
		
		Given an object $A$ in a $2$"/topos $\mathcal C$, Section~5 of \cite{Weber07} sets $\ps A \dfn \brks{A^\circ, \Omega}$ and considers, for each $B \in \mathcal C$, the composite functor
		\begin{displaymath}
			\mathcal C(B, \ps A) \iso \mathcal C(A^\circ \times B, \Omega) \xrar{\tau(\id, \dash)} H(\dFib{\mathcal C})(1, A^\circ \times B) \simeq H(\dFib{\mathcal C})(A, B),
		\end{displaymath}
		where the isomorphism is given by the cartesian closed structure and the equivalence is given by the involution structure; this composite is full and faithful and pseudonatural in $A$ and $B$. A discrete two"/sided fibration $\hmap JAB$ is then called an \emph{attribute} if it is contained in the essential image of this composite. We write $\mathsf{Attr}(\mathcal C) \subseteq \dFib{\mathcal C}$ for the full sub"/augmented virtual double category generated by the attributes. The pseudonaturality of the composite above implies that $\mathsf{Attr}(\mathcal C)$ is closed under taking restrictions so that it, like $\dFib{\mathcal C}$ (\exref{discrete two-sided fibrations}), is an augmented virtual equipment. Moreover by \lemref{full and faithful functors reflect tabulations} $\mathsf{Attr}(\mathcal C)$ has all cocartesian tabulations, that are created as in $\dFib{\mathcal C}$ (\exref{tabulations of discrete two-sided fibrations}).
		
		 A morphism $\map fAC$ of $\mathcal C$ is defined to be \emph{admissible} in Section~5 of \cite{Weber07} whenever, in our terms, the companion $\hmap{f_*}AC$ of $f$ (in $\dFib{\mathcal C}$) is an attribute; i.e.\ $f$ admits a companion in $\mathsf{Attr}(\mathcal C)$. In particular an object $A$ is admissible if and only if its horizontal unit $\hmap{I_A}AA$ (\defref{cartesian cells}) is an attribute, that is $A$ admits a horizontal unit in $\mathsf{Attr}(\mathcal C)$. Consider an admissible object $A$. By definition there exists a morphism $\map{\yon_A}A{\ps A}$ whose image, under the composite above, is isomorphic to $I_A$ and we will show that $\yon_A$ forms a Yoneda embedding (\defref{yoneda embedding}) in the augmented virtual equipment $\mathsf{Attr}(\mathcal C)$ of attributes in $\mathcal C$. 
		 
		 To prove the Yoneda axiom for $\yon_A$ consider any attribute $\hmap JAB$ and let \mbox{$\map{\cur J}B{\ps A}$} be any morphism whose image under the composite above is isomorphic to $J$. By Proposition~5.2 of \cite{Weber07} we have $J \iso \yon_A / \cur J$, the \emph{comma object} of $\yon_A$ and $\cur J$ in $\mathcal C$; see e.g.\ Section~1 of \cite{Street74b}. In our terms, by Proposition~1 of the latter, $\yon_A / \cur J \iso \yon_{A*} \hc (\ps A)^\2 \hc J^{\lambda*}$ as spans in $\Span{\und{\mathcal C}}$ (\augexref{2.9}), so that $J \iso \ps A(\yon_A, \cur J)$ in $\dFib{\mathcal C}$ (see \exref{discrete two-sided fibrations} and \augexref{4.9}) and hence in $\mathsf{Attr}(\mathcal C)$, which proves the Yoneda axiom for $\yon_A$. It remains to show that $\yon_A$ is dense in $\mathsf{Attr}(\mathcal C)$. We use that $\yon_A$ is admissible, which is a consequence of Proposition~5.2 of \cite{Weber07}. It follows that the companion $\yon_{A*}$ exists in $\mathsf{Attr}(\mathcal C)$ so that, by \defref{density definition}, it suffices to show that the cartesian cell $\cart$ defining $\yon_{A*}$ is pointwise left Kan in $\mathsf{Attr}(\mathcal C)$. By applying \propref{pointwise left Kan extensions along companions} to the companion identity $\id_{\yon_A} = \cart \of \cocart$ (\lemref{companion identities lemma}), we may equivalently show that the identity cell $\id_{\yon_A}$ defines $\id_{\ps A}$ as the pointwise left Kan extension of $\yon_A$ along $\yon_A$ in the vertical $2$"/category $V\bigpars{\mathsf{Attr}(\mathcal C)} \iso \mathcal C$. Since $\id_{\yon_A}$ trivially defines $\yon_A$ as an absolute left lifting of $\yon_A$ along $\id_{\ps A}$ the latter follows from Theorem~5.3(2) of \cite{Weber07}.
	\end{example}
	
	\subsection{Lifting Yoneda morphisms along universal morphisms}
	Given a functor $\map F\K\L$ and an object $P \in \L$ consider a universal morphism $\map\eps{FP'}P$ from $F$ to $P$ (\defref{universal vertical morphism}). Given $A \in \K$ and a (weak) Yoneda morphism $\map{\yon_{FA}}{FA}P$ in $\L$, the following theorem gives conditions ensuring that $\yon_{FA}$ induces a (weak) Yoneda embedding $\map{\yon_A}A{P'}$ in $\K$. In \exref{yoneda embedding for small enriched profunctors} below we use this to obtain Yoneda embeddings in the unital virtual double category $\ensProf\V$ of small $\V$"/profunctors (\augexref{2.8}) from those in $\enProf{(\V, \V')}$ (\exref{enriched yoneda embedding summary}). A related result is Theorem~4.1 of \cite{Hermida01}, which allows one to transfer a Yoneda structure (\cite{Street-Walters78}) on a $2$"/category $\L$ along a biadjunction $\map{F \ladj G}\L\K$ with full and faithful unit.
	
	To state the theorem consider the vertical slice category $F \vs P$ (\defref{universal vertical morphism}) and let $\cur{(F \vs P)} \subseteq F \vs P$ denote the full subcategory generated by all objects \mbox{$(B, \map f{FB}P)$} such that $f \iso \cur{(FJ)}$ for some $\hmap JAB$ in $\K$, where $\cur{(FJ)}$ is supplied by the Yoneda axiom for $\yon_{FA}$ (\defref{yoneda embedding}). In other words \mbox{$(B, f) \in \cur{(F\vs P)}$} if and only if the restriction $P(\yon_A, f)$ exists in $\L$, and it is contained in the essential image of $F$. Assume that \mbox{$\map\eps{FP'}P$} is universal from $F$ to $P$ relative to $\cur{(F \vs P)}$ (\defref{universal vertical morphism}) and that $A$ is unital, with horizontal unit $\hmap{I_A}AA$. It follows that $FA \in \L$ is unital, with horizontal unit $FI_A$ (see \augcororef{5.5}), and that $FI_A \iso P(\yon_{FA}, \yon_{FA})$ by \lemref{full and faithful yoneda embedding}. Comparing the cartesian cell defining the latter restriction with the cartesian cell defining $\cur{(FI_A)}$ (\defref{yoneda embedding}), using that both cartesian cells are (weakly) left Kan by the (weak) density of $\yon_{FA}$ (\defref{density definition}), we find that $\cur{(FI_A)} \iso \yon_{FA}$. We conclude that \mbox{$\yon_{FA} \in \cur{(F \vs P)}$} so that, by universality of $\eps$, there exists a morphism \mbox{$\map{\yon_A \dfn \shad{(\yon_{FA})}}A{P'}$} in $\K$ such that $\yon_{FA} \iso \eps \of F\yon_A$.
	\begin{theorem} \label{yoneda embedding from a vertical universal morphism}
		Let the functor $\map F\K\L$, the unital object $A \in \K$, the (weak) Yoneda morphism $\map{\yon_{FA}}{FA}P$ in $\L$ and the universal morphism \mbox{$\map\eps{FP'}P$} relative to $\cur{(F \vs P)}$ be as above. The morphism $\map{\yon_A}A{P'}$, as obtained above, is a (weak) Yoneda embedding in $\K$ as long as the functor \mbox{$\map{\eps \of F\dash}{\K \vs P'}{F \vs P}$} (\defref{universal vertical morphism}) preserves and reflects all cartesian cells in $\K \vs P'$ that define restrictions of the form $P'(\yon_A, g)$, where $\map gB{P'}$ is any morphism.
	\end{theorem}
	\begin{proof}
		Denote the invertible vertical cell $\yon_{FA} \Rar \eps \of F\yon_A$ by $\sigma$. To show that $\yon_A$ is (weakly) dense we have to show that any cartesian cell $\eta$ in $\K$ as on the left below is (weakly) left Kan (\defref{density definition}). By assumption $\eps \of F\eta$ is cartesian in $\L$ so that $\sigma \hc (\eps \of F\eta)$ is (weakly) left Kan by the (weak) density of $\yon_{FA}$. We conclude that $\eps \of F\eta$ is (weakly) left Kan so that $\eta$, being the adjunct of $\eps \of F\eta$, is (weakly) left Kan by \propref{taking adjuncts preserves left Kan cells}.
		\begin{displaymath}
			\begin{tikzpicture}[textbaseline]
				\matrix(m)[math35, column sep={1.75em,between origins}]{A & & B \\ & P' & \\};
				\path[map]	(m-1-1) edge[barred] node[above] {$J$} (m-1-3)
														edge[transform canvas={xshift=-2pt}] node[left] {$\yon_A$} (m-2-2)
										(m-1-3) edge[transform canvas={xshift=2pt}] node[right] {$l$} (m-2-2);
				\path[transform canvas={yshift=0.25em}]	(m-1-2) edge[cell] node[right] {$\eta$} (m-2-2);
			\end{tikzpicture} \qquad\qquad\qquad\qquad \begin{tikzpicture}[baseline]
					\matrix(m)[math35, column sep={2em,between origins}]{& FA & & FB & \\ FP' & & & & FP' \\ & & P & & \\};
					\path[map]	(m-1-2) edge[barred] node[above] {$FJ$} (m-1-4)
															edge[bend right=20] node[above left] {$F\yon_A$} (m-2-1)
															edge[transform canvas={xshift=1pt}] node[left, yshift=7pt, xshift=0pt] {$\yon_{FA}$} (m-3-3)
											(m-1-4) edge[transform canvas={xshift=-1pt}] node[right, yshift=7pt, xshift=0pt] {$\cur{(FJ)}$} (m-3-3)
															edge[bend left=20] node[above right] {$F\cur J$} (m-2-5)
											(m-2-1) edge[bend right=35] node[below left] {$\eps$} (m-3-3)
											(m-2-5) edge[bend left=35] node[below right] {$\eps$} (m-3-3);
					\draw[font=\scriptsize]				([yshift=1.5em]$(m-1-3)!0.5!(m-3-3)$) node {$\cart$};
					\draw				([yshift=1em]$(m-2-3)!0.5!(m-3-1)$) node[rotate=25] {$\iso$}
											([yshift=1em]$(m-2-3)!0.5!(m-3-5)$) node[rotate=-25] {$\iso$};
			\end{tikzpicture} \quad = \quad \begin{tikzpicture}[textbaseline]
				\matrix(m)[math35, column sep={1.75em,between origins}]{FA & & FB \\ & FP' & \\ & P & \\};
				\path[map]	(m-1-1) edge[barred] node[above] {$FJ$} (m-1-3)
														edge[transform canvas={xshift=-2pt}] node[left] {$F\yon_A$} (m-2-2)
										(m-1-3) edge[transform canvas={xshift=2pt}] node[right] {$F\cur J$} (m-2-2)
										(m-2-2) edge node[right] {$\eps$} (m-3-2);
				\path[transform canvas={xshift=-0.5em, yshift=0.25em}]	(m-1-2) edge[cell] node[right, inner sep=2.5pt] {$F\phi$} (m-2-2);
			\end{tikzpicture}
		\end{displaymath}
		
		To prove the Yoneda axiom (\defref{yoneda embedding}) for $\map{\yon_A}A{P'}$ we have to supply, for every $\hmap JAB$ in $\K$, a morphism $\map{\cur J}B{P'}$ and a cartesian cell that defines $J$ as the restriction $P'(\yon_A, \cur J)$. By the Yoneda axiom for $\yon_{FA}$ there exists a morphism $\map{\cur{(FJ)}}{FB}P$ such that $FJ \iso P\bigpars{\yon_{FA}, \cur{(FJ)}}$, and we take $\cur J \dfn \shad{\bigpars{\cur{(FJ)}}}$ to be its adjunct, satisfying $\cur{(FJ)} \iso \eps \of F\cur J$, which exists by the universality of $\eps$. Consider the composite cartesian cell on the left"/hand side above, of the cartesian cell that defines $\cur{(FJ)}$ and the invertible vertical cells that equip the adjuncts $\yon_A$ and $\cur J$. Its adjunct $\cell\phi J{P'}$, which exists by the local universality of $\eps$ (\defref{universal vertical morphism}), satisfies the identity above. By assumption the cartesianness of the left"/hand side above implies the cartesianness of $\phi$, which thus defines $J$ as the restriction $P'(\yon_A, \cur J)$ as required.
	\end{proof}
	
	\begin{example} \label{yoneda embedding for small enriched profunctors}
		Let $\V \subset \V'$ be a symmetric universe enlargement (\exref{enriched yoneda embedding summary}), and assume that $\V$ is small complete. By applying the previous theorem to the full embedding $\emb F{\ensProf\V}{\enProf{(\V, \V')}}$ we will see that the Yoneda embeddings $\map{\yon_A}A{\inhom{\op A, \V}'}$ of $\enProf{(\V, \V')}$ (\exref{enriched yoneda embedding summary}), for $\V$"/categories $A$, can be corestricted to $\V$"/categories of `small $\V$"/presheaves on $A$', as introduced by Lindner in \cite{Lindner74}, to form Yoneda embeddings in the unital virtual double category $\ensProf\V$ of small $\V$"/profunctors (\augexref{2.8}).
		
		We denote by $\inhom{\op A, \V}_\textup s \subseteq \inhom{\op A, \V}'$ the full sub"/$\V'$"/category of $\V$"/presheaves $\map p{\op A}\V$ that are \emph{small} in the sense of \cite{Lindner74} and \cite{Day-Lack07}: these are precisely the $\V$"/pre\-sheaves $\op A \to \V$ that correspond to small $\V$"/profunctors $A \brar I$ in the sense of \augexref{2.8}. It is straightforward to check that $\V$ being small complete implies that $\inhom{\op A, \V}_\textup s$ is a $\V$"/category; see also Corollary~2.3 of \cite{Lindner74}. In fact, using that any small $\V$"/presheaf $\hmap pAI$ is ``generated'' by its restriction to a small sub"/$\V$"/category $A_* \subseteq A$, in the sense of \augexref{2.8}, one checks that the hom $\V'$"/objects $\inhom{\op A, \V}'(p, q)$ are computed by the small $\V$"/ends $\int_{x \in A_*}\inhom{px, qx}$.
		
		Regarding $\inhom{\op A, \V}_\textup s$ as an object of $\ensProf\V$ let $\emb\eps{F\inhom{\op A, \V}_\textup s}{\inhom{\op A, \V}'}$ denote the embedding in $\enProf{(\V, \V')}$. Because $\eps$ is full and faithful in $\enProf{(\V, \V)'}$ and $F$ is a full and faithful functor, $\eps$ is locally universal by \exref{full and faithful morphisms are locally universal}. To show that $\eps$ is universal relative to $\cur{(F\vs\inhom{\op A, \V}')}$ (\defref{universal vertical morphism}) let $\map f{FB}{\inhom{\op A, \V}'}$ be any $\V'$"/functor, with $B$ a $\V$"/category. Using the correspondence of $\V$"/profunctors $A \brar B$ and $\V'$"/functors $B \to \inhom{\op A, \V}'$ in $\enProf{(\V, \V')}$ (\exref{enriched yoneda embedding summary}) one checks that \mbox{$(B, f) \in \cur{\bigpars{F \vs \inhom{\op A, \V}'}}$} precisely if $f$, regarded as a $\V$"/profunctor $A \brar B$, is small in the sense of \augexref{2.8}, that is $\map{f(y)}{\op A}\V$ is a small $\V$"/presheaf for every $y \in B$ or, equivalently, the $\V'$"/functor $\op A \to \inhom{B, \V}'$ corresponding to $f$ is \emph{pointwise small} in the sense of \cite{Day-Lack07}. Hence \mbox{$(B, f) \in \cur{\bigpars{F \vs \inhom{\op A, \V}'}}$} if and only if there exists a $\V$"/functor $\map{f'}B{\inhom{\op A, \V}_\textup s}$ such that $f = \eps \of Ff'$ and we conclude that $\map{\eps \of F\dash}{\ensProf\V \vs \inhom{\op A, \V}_\textup s}{F \vs \inhom{\op A, \V}'}$ factors as an equivalence through \mbox{$\cur{\bigpars{F \vs \inhom{\op A, \V}'}} \hookrightarrow F\vs \inhom{\op A, \V}'$}, so that $\eps$ is universal relative to $\cur{\bigpars{F \vs \inhom{\op A, \V}'}}$.
		
		 Moreover since $F$ preserves and reflects cartesian cells (use \augexref{4.7}), so does $\map{\eps \of F\dash}{\ensProf\V \vs \inhom{\op A, \V}_\textup s}{F \vs \inhom{\op A, \V}'}$, by combining \augexref{4.11} and the pasting lemma (\lemref{pasting lemma for cartesian cells}). Thus all hypotheses of the theorem are satisfied. We conclude that the Yoneda embedding $\map{\yon_A}A{\inhom{\op A, \V}'}$ in $\enProf{(\V, \V')}$ factors through $\eps$ as a $\V$"/functor $\map{\yon_A}A{\inhom{\op A, \V}_\textup s}$ that is a Yoneda embedding in $\ensProf\V$. Finally notice that the companion $\hmap{\yon_{A*}}A{\inhom{\op A, \V}_\textup s}$ is a small $\V$"/profunctor so that, because $\ensProf\V$ has restrictions on the right (\augexref{4.7}), $\yon_A$ admits nullary restrictions in the sense of \defref{yoneda embedding}. To see this consider, for any $p \in \inhom{A, \V}_\textup s$, the small sub"/$\V$"/category $A_p \dfn A_* \subseteq A$ as in the above. The cascade of isomorphisms below, for each $x \in A$, shows that $\yon_*$ is small. The first and last isomorphisms here are induced by the `strong Yoneda lemma' (see e.g.\ Formula~2.31 of \cite{Kelly82}) while the middle isomorphism exhibits the smallness of $p$.
		\begin{displaymath}
			\int^{x' \in A_p} A(x, x') \tens \yon_*(x', p) \iso \int^{x' \in A_p} A(x, x') \tens px' \iso px \iso \yon_*(x, p)
		\end{displaymath} 
	\end{example}
	
	\subsection{Yoneda morphisms compared to Yoneda structures and powers}
	We next compare our notion of (weak) Yoneda embedding to the notions of \emph{Yoneda structure} and \emph{good Yoneda structure} of \cite{Street-Walters78} and \cite{Weber07} respectively, in the theorem below, as well as to the notion of \emph{power} of \cite{Lambert22}, in \propref{powers from weak yoneda morphisms} below. Recall that part of a (good) Yoneda structure is a \emph{right ideal} $\mathcal A$ of \emph{admissible} morphisms in $V(\K)$: any composite $g \of f$ in $V(\K)$ is admissible as soon as $g$ is so. An object $A \in \K$ is called \emph{admissible} whenever its identity morphism $\id_A$ is admissible.
	
	\begin{theorem} \label{yoneda structures}
		Let $\K$ be an augmented virtual double category and let $\mathcal A$ be a right ideal of admissible morphisms in $V(\K)$. Assume that for every admissible \mbox{$\map fAC$} the companion $\hmap{f_*}AC$ exists and consider given an admissible morphism \mbox{$\map{\yon_A}A{\ps A}$} for each admissible object $A$.
		
		The implications \textup{(gys)} $\Rightarrow$ \textup{(ys*)} $\Rightarrow$ \textup{(ys)} and \textup{(ye)} $\Rightarrow$ \textup{(wye)} hold among the conditions below. If $\K$ has restrictions on the right (\defref{augmented virtual equipment}) then \textup{(wye)}~$\Rightarrow$~\textup{(ys)} holds too. If $\K$ has weakly cocartesian paths of $(0,1)$"/ary cells (\defref{cocartesian path of (0,1)-ary cells}) then \textup{(wye)}~$\Rightarrow$~\textup{(ys*)}. If $\K$ has left nullary"/cocartesian tabulations (\defref{tabulation}) and restrictions on the right then \textup{(ye)} $\Leftrightarrow$ \textup{(wye)} $\Rightarrow$ \textup{(gys)}. If $\K$ has left cocartesian tabulations, restrictions on the right and, for each $\hmap JAB$ with $A$ admissible, a cartesian nullary cell (\defref{cocartesian path of (0,1)-ary cells}) below exists with $f$ admissible, then \textup{(gys)}~$\Rightarrow$~\textup{(ye)} holds too.
		\begin{displaymath}
			\begin{tikzpicture}
				\matrix(m)[math35, column sep={1.75em,between origins}]{A & & B \\ & C & \\};
				\path[map]	(m-1-1) edge[barred] node[above] {$J$} (m-1-3)
														edge[transform canvas={xshift=-2pt}] node[left] {$f$} (m-2-2)
										(m-1-3) edge[transform canvas={xshift=2pt}] node[right] {$g$} (m-2-2);
				\draw				([yshift=0.333em]$(m-1-2)!0.5!(m-2-2)$) node[font=\scriptsize] {$\cart$};
			\end{tikzpicture}
		\end{displaymath}
		\begin{enumerate}
			\item[\textup{(ys)}] The morphisms $\yon_A$ form a Yoneda structure on $V(\K)$ in the sense of \cite{Street-Walters78};
			\item[\textup{(ys*)}] the morphisms $\yon_A$ form a Yoneda structure on $V(\K)$ that satisfies Axiom~3* of \cite{Street-Walters78};
			\item[\textup{(gys)}] the morphisms $\yon_A$ form a good Yoneda structure on $V(\K)$ in the sense of \cite{Weber07};
			\item[\textup{(wye)}] each morphism $\yon_A$ is a weak Yoneda embedding (\defref{yoneda embedding});
			\item[\textup{(ye)}] each morphism $\yon_A$ is a Yoneda embedding (\defref{yoneda embedding}).
		\end{enumerate}
	\end{theorem}
	\begin{proof}
		That (gys) $\Rightarrow$ (ys*) and (ye) $\Rightarrow$ (wye) is easily checked. Proposition~11 of \cite{Street-Walters78} proves (ys*) $\Rightarrow$ (ys). Suppose that $\K$ has restrictions on the right. To show (wye) $\Rightarrow$ (ys) we have to show that the family of weak Yoneda embeddings $\yon_A$ satisfies Axioms~1, 2 and 3 of \cite{Street-Walters78}. Axioms~1 and 2 ask, for each $\map fAB$ with both $A$ and $f$ admissible, a morphism $\map{B(f, 1)}B{\ps A}$ equipped with a cell $\cell{\chi^f}{\yon_A}{B(f, 1) \of f}$ that simultaneously defines, in $V(\K)$, $B(f, 1)$ as the left Kan extension of $\yon_A$ along $f$ and $f$ as the absolute left lifting of $\yon_A$ along $B(f, 1)$. By assumption $f_*$ exists and we can take $B(f, 1) \dfn \cur{f_*}$ as supplied by the Yoneda axiom (\defref{yoneda embedding}); let $\chi^f$ be the composite
		\begin{displaymath}
			\chi^f \dfn \quad \begin{tikzpicture}[textbaseline]
    		\matrix(m)[math35, column sep={1.75em,between origins}]{& A & \\ A & & B. \\ & \ps A & \\};
    		\path[map]	(m-1-2) edge[transform canvas={xshift=2pt}] node[right] {$f$} (m-2-3)
    								(m-2-1) edge[barred] node[below, inner sep=2pt] {$f_*$} (m-2-3)
    												edge[transform canvas={xshift=-2pt}] node[left] {$\yon_A$} (m-3-2)
    								(m-2-3)	edge[transform canvas={xshift=2pt}] node[right] {$B(f, 1)$} (m-3-2);
    		\path				(m-1-2) edge[eq, transform canvas={xshift=-2pt}] (m-2-1);
    		\draw				([yshift=-0.5em]$(m-1-2)!0.5!(m-2-2)$) node[font=\scriptsize] {$\cocart$}
    								([yshift=0.25em]$(m-2-2)!0.5!(m-3-2)$) node[font=\scriptsize] {$\cart$};
  		\end{tikzpicture} 
		\end{displaymath}
		It follows from weak density of $\yon_A$ (\defref{density definition}) and \propref{weak left Kan extensions along companions} that $\chi^f$ defines $B(f, 1)$ as the left Kan extension of $\yon_A$ along $f$ in $V(\K)$; that it defines $f$ as the absolute left lifting of $\yon_A$ along $B(f, 1)$ too follows from \auglemref{5.17}. Notice that Axioms~1 and 2 hold without requiring that $\K$ has restrictions on the right. 
		
		The first part of Axiom~3 requires that the identity cell $\cell{\id_{\yon_A}}{\yon_A}{\id_{\ps A} \of \yon_A}$ defines $\id_A$ as the left Kan extension of $\yon_A$ along $\yon_A$. By assumption $\yon_A$ is admissible so that $\yon_{A*}$ exists and, by weak density of $\yon_A$ (\defref{density definition}), it follows that the defining cell $\cell{\cart}{\yon_{A*}}{\ps A}$ is weakly left Kan so that $\id_{\yon_A} = \cart \of \cocart$ (\lemref{companion identities lemma}) defines $\id_{\ps A}$ as the required left Kan extension by \propref{weak left Kan extensions along companions}. Next consider a composable pair of morphisms $\map fAB$ and $\map gBC$ with $A$, $B$ and $g$ admissible; notice that admissibility of $B$ implies that of $f$. The second part of Axiom~3 requires that the composite $\chi^{\yon_B \of f} \hc (\ps B(\yon_B \of f, 1) \of \chi^g \of f)$, which by definition of $\chi^{\yon_B \of f}$ and $\chi^g$ is the left"/hand side of the equality below, defines $\ps B(\yon_B \of f,1) \of C(g, 1)$ as the left Kan extension of $\yon_A$ along $g \of f$ in $V(\K)$.
		
		In the middle composite below the cell $\cart'$ denotes the factorisation of the cartesian cell that defines $C(g,1)$ through the cartesian cell that defines $\yon_{B*}$ while the cell $\phi$ is the composite of the cocartesian cell defining $(\yon_B \of f)_*$ with the cartesian cells defining $f_*$ and $\yon_{B*}$. Notice that the first identity below follows from the companion identity for $f_*$ (\lemref{companion identities lemma}) and the definition of $\cart'$, and that the latter is again cartesian by the pasting lemma for cartesian cells (\lemref{pasting lemma for cartesian cells}). Moreover it follows easily from the companion identities that, together with the cell $\cell\psi{(\yon_B \of f)_*}{\yon_{B*}}$ obtained by composing the cartesian cell defining $(\yon_B \of f)_*$ with the cocartesian cell defining $\yon_{B*}$, the cell $\phi$ satisfies both identities of \auglemref{8.1}, so that $\phi$ is pointwise right cocartesian by \auglemref{9.7} and \remref{right pointwise cocartesian and pointwise right cocartesian comparison}. By weak density of $\yon_A$ (\defref{density definition}) the cartesian cell $\cart$ defining $\ps B(\yon_B \of f, 1)$ in the middle composite below is pointwise weakly left Kan, so that by the vertical pasting lemma (\lemref{vertical pasting lemma}) the composite $\cart \of \phi$ is pointwise weakly left Kan too. By \defref{pointwise left Kan extension} it follows that the composite $\eta \dfn \cart \of \phi \of (\id, \cart')$ of the bottom three rows of the middle composite below is weakly left Kan. Finally consider, as in the second identity below, the unique factorisation $\eta'$ of $\eta$ through the cocartesian cell $(f_*, g_*) \Rar (g \of f)_*$ that is defined analogously to the cocartesian cell $\phi$. By the vertical pasting lemma $\eta'$ is weakly left Kan. By definition of $\cell\cocart{(f_*, g_*)}{(g \of f)_*}$ the composite of the top three rows in the right"/hand side below is the cocartesian cell that defines the companion $(g \of f)_*$. Applying \propref{weak left Kan extensions along companions} we conclude that the right"/hand side, and hence the left"/hand side, defines $\ps B(\yon_B \of f, 1) \of C(g, 1)$ as the left Kan extension of $\yon_A$ along $g \of f$ in $V(\K)$. This completes the proof of (wye) $\Rightarrow$ (ys).
		
		\begin{displaymath}
			\begin{tikzpicture}[textbaseline]
				\matrix(m)[math35, column sep={1.75em,between origins}]
					{ & & & A & & \\ & & A & & B & \\ & & & B & & C \\ A & & & & \ps B & \\ & & \ps A & & & & \\};
				\path[map]	(m-1-4) edge[transform canvas={xshift=2pt}] node[right] {$f$} (m-2-5)
										(m-2-3) edge node[left] {$f$} (m-3-4)
										(m-2-5) edge[transform canvas={xshift=2pt}] node[right] {$g$} (m-3-6)
										(m-3-4) edge[barred] node[below, inner sep=2pt] {$g_*$} (m-3-6)
														edge[transform canvas={xshift=-2pt}] node[left, inner sep=1pt] {$\yon_B$} (m-4-5)
										(m-3-6) edge[transform canvas={xshift=2pt}] node[right] {$C(g, 1)$} (m-4-5)
										(m-4-1) edge[barred] node[above, inner sep=1pt] {$(\yon_B \of f)_*$} (m-4-5)
														edge[transform canvas={xshift=-2pt}] node[below left] {$\yon_A$} (m-5-3)
										(m-4-5) edge[transform canvas={xshift=2pt}] node[below right, inner sep=1.5pt] {$\ps B(\yon_B \of f, 1)$} (m-5-3);
				\path				(m-1-4) edge[eq, transform canvas={xshift=-1pt}] (m-2-3)
										(m-2-3) edge[eq, transform canvas={xshift=-1pt}] (m-4-1)
										(m-2-5) edge[eq, transform canvas={xshift=-2pt}] (m-3-4);
				\draw[font=\scriptsize]	([yshift=-0.5em]$(m-2-3)!0.5!(m-4-3)$) node {$\cocart$}
										([yshift=-0.5em]$(m-2-5)!0.5!(m-3-5)$) node {$\cocart$}
										([yshift=0.25em]$(m-3-5)!0.5!(m-4-5)$) node {$\cart$}
										([yshift=0.25em]$(m-4-3)!0.5!(m-5-3)$) node {$\cart$};
			\end{tikzpicture} = \quad \begin{tikzpicture}[textbaseline]
    		\matrix(m)[math35, column sep={1.75em,between origins}]
    			{ & & A & & \\ & A & & B & \\ A & & B & & C \\ A & & B & & \ps B \\ A & & & & \ps B \\ & & \ps A & & \\};
    		\path[map]	(m-1-3) edge[transform canvas={xshift=2pt}] node[right] {$f$} (m-2-4)
    								(m-2-2) edge[barred] node[below] {$f_*$} (m-2-4)
    								(m-2-4) edge[transform canvas={xshift=2pt}] node[right] {$g$} (m-3-5)
    								(m-3-1) edge[barred] node[below] {$f_*$} (m-3-3)
    								(m-3-3) edge[barred] node[below] {$g_*$} (m-3-5)
    								(m-3-5) edge[ps] node[right] {$C(g,1)$} (m-4-5)
    								(m-4-1) edge[barred] node[below] {$f_*$} (m-4-3)
    								(m-4-3) edge[barred] node[below] {$\yon_{B*}$} (m-4-5)
    								(m-5-1) edge[barred] node[below, inner sep=1pt] {$(\yon_B \of f)_*$} (m-5-5)
    												edge[transform canvas={xshift=-2pt}] node[below left] {$\yon_A$} (m-6-3)
    								(m-5-5) edge[transform canvas={xshift=2pt}] node[below right, inner sep=1.5pt] {$\ps B(\yon_B \of f,1)$} (m-6-3);
    		\path				(m-1-3) edge[eq, transform canvas={xshift=-2pt}] (m-2-2)
    								(m-2-2) edge[eq, transform canvas={xshift=-2pt}] (m-3-1)
    								(m-2-4) edge[eq, transform canvas={xshift=-2pt}] (m-3-3)
    								(m-3-1) edge[eq] (m-4-1)
    								(m-3-3) edge[eq] (m-4-3)
    								(m-4-1) edge[eq] (m-5-1)
    								(m-4-5) edge[eq, ps] (m-5-5)
    								(m-4-3) edge[cell] node[right] {$\phi$} (m-5-3);
    		\draw[font=\scriptsize]	([yshift=-0.5em]$(m-1-3)!0.5!(m-2-3)$) node {$\cocart$}
    								([yshift=-0.5em]$(m-2-4)!0.5!(m-3-4)$) node {$\cocart$}
    								($(m-3-3)!0.5!(m-4-5)$) node {$\cart'$}
    								($(m-5-3)!0.5!(m-6-3)$) node {$\cart$};    								
    	\end{tikzpicture} \quad = \quad \begin{tikzpicture}[textbaseline]
    		\matrix(m)[math35, column sep={1.75em,between origins}]
    			{ & & A & & \\ & A & & B & \\ A & & B & & C \\ & A & & C & \\ & & & & \\ & & \ps A & & \\};
    		\draw				($(m-4-4)!0.5!(m-6-3)$) node (psB) {$\ps B$};
    		\path[map]	(m-1-3) edge[transform canvas={xshift=2pt}] node[right] {$f$} (m-2-4)
    								(m-2-2) edge[barred] node[below] {$f_*$} (m-2-4)
    								(m-2-4) edge[transform canvas={xshift=2pt}] node[right] {$g$} (m-3-5)
    								(m-3-1) edge[barred] node[below] {$f_*$} (m-3-3)
    								(m-3-3) edge[barred] node[below] {$g_*$} (m-3-5)
    								(m-4-4) edge[transform canvas={xshift=1pt}] node[right] {$C(g,1)$} (psB)
    								(m-4-2) edge[barred] node[below, inner sep=2pt, xshift=1pt] {$(g \of f)_*$} (m-4-4)
    												edge[transform canvas={xshift=-1pt},ps] node[left] {$\yon_A$} (m-6-3)
    								(psB) edge[transform canvas={xshift=1pt},ps] node[right, inner sep=1.5pt] {$\ps B(\yon_B \of f,1)$} (m-6-3);
    		\path				(m-1-3) edge[eq, transform canvas={xshift=-2pt}] (m-2-2)
    								(m-2-2) edge[eq, transform canvas={xshift=-2pt}] (m-3-1)
    								(m-2-4) edge[eq, transform canvas={xshift=-2pt}] (m-3-3)
    								(m-3-1) edge[eq] (m-4-2)
    								(m-3-5) edge[eq] (m-4-4);
    		\path[transform canvas={yshift=-0.666em}]		(m-4-3) edge[cell] node[right, inner sep=2pt] {$\eta'$} (m-5-3);
    		\draw[font=\scriptsize]	([yshift=-0.5em]$(m-1-3)!0.5!(m-2-3)$) node {$\cocart$}
    								([yshift=-0.5em]$(m-2-4)!0.5!(m-3-4)$) node {$\cocart$}
    								($(m-3-3)!0.5!(m-4-3)$) node {$\cocart$};    								
    	\end{tikzpicture}
		\end{displaymath}
				
		Next, assume that $\K$ has weakly cocartesian paths of $(0,1)$"/ary cells (\defref{cocartesian path of (0,1)-ary cells}). To show (wye) $\Rightarrow$ (ys*) we have to prove that the weak Yoneda embeddings $\yon_A$ satisfy Axioms~1, 2 and 3* of \cite{Street-Walters78} in $V(\K)$. That they satisfy Axioms~1 and 2 follows from the argument given above. Axiom~3* is equivalent to the assertion that any cell $\psi$ as on the left below, with $A$ and $f$ admissible, defines $g$ as a left Kan extension of $\yon_A$ along $f$ in $V(\K)$ as soon as it defines $f$ as an absolute left lifting of $\yon_A$ along $g$. Considering the factorisation $\psi'$ of $\psi$ through $f_*$ as shown, this is a consequence of the assumption on $\K$ as follows: $\psi$ defining $f$ as an absolute left lifting in $V(\K)$ means that $\psi'$ is cartesian in $\K$, by \auglemref{5.17} and \augpropref{7.12}, which implies that $\psi'$ is weakly left Kan by weak density of $\yon_A$, so that $\psi$ defines $g$ as a left Kan extension in $V(\K)$ by \propref{weak left Kan extensions along companions}.
		\begin{displaymath}
			\begin{tikzpicture}[textbaseline]
				\matrix(m)[math35, column sep={1.75em,between origins}]{& A & \\ & & B \\ & \ps A & \\};
				\path[map]	(m-1-2) edge[bend left = 18] node[above right] {$f$} (m-2-3)
														edge[bend right = 45] node[left] {$\yon_A$} (m-3-2)
										(m-2-3) edge[bend left = 18] node[below right] {$g$} (m-3-2);
				\path[transform canvas={yshift=-1.625em}]	(m-1-2) edge[cell] node[right] {$\psi$} (m-2-2);
			\end{tikzpicture} \quad = \quad \begin{tikzpicture}[textbaseline]
    		\matrix(m)[math35, column sep={1.75em,between origins}]{& A & \\ A & & B \\ & \ps A & \\};
    		\path[map]	(m-1-2) edge[transform canvas={xshift=2pt}] node[right] {$f$} (m-2-3)
    								(m-2-1) edge[barred] node[below, inner sep=2pt] {$f_*$} (m-2-3)
    												edge[transform canvas={xshift=-2pt}] node[left] {$\yon_A$} (m-3-2)
    								(m-2-3)	edge[transform canvas={xshift=2pt}] node[right] {$g$} (m-3-2);
    		\path				(m-1-2) edge[eq, transform canvas={xshift=-2pt}] (m-2-1);
    		\path				(m-2-2) edge[cell, transform canvas={yshift=0.1em}]	node[right, inner sep=3pt] {$\psi'$} (m-3-2);
    		\draw				([yshift=-0.5em]$(m-1-2)!0.5!(m-2-2)$) node[font=\scriptsize] {$\cocart$};
  		\end{tikzpicture} \qquad\qquad\qquad\qquad\qquad\qquad \begin{tikzpicture}[textbaseline]
  			\matrix(m)[math35, column sep={1.75em,between origins}]{A & & B \\ A & & C \\ & \ps A & \\};
  			\path[map]	(m-1-1) edge[barred] node[above] {$J$} (m-1-3)
  									(m-1-3) edge node[right] {$g$} (m-2-3)
  									(m-2-1) edge[barred] node[below, inner sep=2pt] {$f_*$} (m-2-3)
  													edge[transform canvas={xshift=-2pt}] node[left] {$\yon_A$} (m-3-2)
  									(m-2-3) edge[transform canvas={xshift=2pt}] node[below right] {$C(f, 1)$} (m-3-2);
  			\path				(m-1-1) edge[eq] (m-2-1)
  									(m-2-2) edge[cell, transform canvas={yshift=0.2em}] node[right] {$\chi'$} (m-3-2);
  			\draw				($(m-1-2)!0.5!(m-2-2)$) node[font=\scriptsize] {$\cart'$};
  		\end{tikzpicture}
		\end{displaymath}
		
		Next we will show that (ye) $\Leftrightarrow$ (wye) $\Rightarrow$ (gys) whenever $\K$ has left nullary"/cocartesian tabulations and restrictions on the right. By \cororef{cocartesian tabulations imply cocartesian paths of (0,1)-ary cells} $\K$ has left nullary"/cocartesian paths of $(0,1)$"/ary cells so that by \thmref{pointwise Kan extensions in terms of pointwise weak Kan extensions} the notions of pointwise weak Kan extension and pointwise Kan extension coincide in $\K$. As remarked after \defref{density definition} this implies (ye) $\Leftrightarrow$ (wye). Assuming (ye) we will show that the Yoneda embeddings $\yon_A$, together with the cells $\chi^f$ defined above, satisfy the two axioms of a good Yoneda structure on $V(\K)$, in the sense of Definition~3.1 of \cite{Weber07}, which proves (gys). The first axiom of a good Yoneda structure coincides with Axiom~2 of a Yoneda structure, which we proved previously. The second axiom strengthens Axiom~3* as follows: any cell $\psi$ as on the left above, with $A$ and $f$ admissible, defines $g$ as a pointwise left Kan extension in $V(\K)$ as soon as it defines $f$ as an absolute left lifting. To see this notice that, as before, the assumption on $\psi$ implies that its factorisation $\psi'$ is cartesian in $\K$, so that $\psi'$ is pointwise left Kan by density of $\yon_A$ (\defref{density definition}) and hence $\psi$ defines $g$ as a pointwise left Kan extension in $V(\K)$ by \propref{pointwise left Kan extensions along companions}.
		
		For the converse (gys) $\Rightarrow$ (ye) assume that, besides left cocartesian tabulations and restrictions on the right, $\K$ has cartesian nullary cells for all $\hmap JAB$ with $A$ admissible, in the sense described in the statement. By \cororef{cocartesian tabulations imply cocartesian paths of (0,1)-ary cells} $\K$ has all left cocartesian paths of $(0,1)$"/ary cells. That each $\yon_A$ is dense follows easily from Proposition~3.4(1) of \cite{Weber07}, which shows that the identity cell $\id_{\yon_A}$ defines $\id_{\ps A}$ as the pointwise left Kan extension of $\yon_A$ along $\yon_A$ in $V(\K)$: applying \propref{pointwise left Kan extensions along companions} to the companion identity $\id_{\yon_A} = \cart \of \cocart$ for $\yon_{A*}$ it follows that $\cart$ is pointwise left Kan in $\K$ so that $\yon_A$ is dense by \defref{density definition}. It remains to show that $\yon_A$ satisfies the Yoneda axiom (\defref{yoneda embedding}), that is for each $\hmap JAB$ there exists a morphism $\map{\cur J}B{\ps A}$ such that $J$ is the restriction $\ps A(\yon_A, \cur J)$. Consider a cartesian nullary cell $\cell\cart JC$ as in the statement, with admissible vertical source $\map fAC$ and vertical target $\map gBC$, and set $\cur J \dfn C(f, 1) \of g$, where $\map{C(f, 1)}C{\ps A}$ is supplied by the good Yoneda structure (Definition~3.1 of \cite{Weber07}). To show that $J$ is the restriction $\ps A(\yon_A, \cur J)$ consider the composite cell on the right above, where $\cart'$ is the factorisation of $\cell\cart JC$ through $f_*$ and $\chi'$ is the factorisation of the vertical cell $\cell{\chi^f}{\yon_A}{C(f, 1) \of f}$, that defines $C(f, 1)$, through $f_*$. By the pasting lemma (\lemref{pasting lemma for cartesian cells}) $\cart'$ is cartesian and, since $\chi^f$ defines $f$ as an absolute left lifting, $\chi'$ is cartesian by \auglemref{5.17} and \augpropref{7.12}. Applying the pasting lemma again we find that $\chi' \of \cart'$ is cartesian, thus defining $J$ as $\ps A(\yon_A, \cur J)$. This completes the proof.
	\end{proof}
	
	\begin{example} \label{no good yoneda structure on (Cat, Cat')-Prof}
		Applying the theorem to the Yoneda embeddings of the augmented virtual equipment $\enProf{(\V, \V')}$ (\exref{enriched yoneda embedding summary}) we recover the Yoneda structure on the $2$"/category $\enCat{\V'}$ of (large) $\V'$"/categories, as described in Section~7 of \cite{Street-Walters78}, whose admissible morphisms are the $\V'$"/functors $\map fAC$ for which all hom objects $C(fx, z)$ are (isomorphic to) $\V$"/objects (\augexref{5.6}).
		
		Instantiating $\V \subset \V'$ by $\Cat \subset \Cat'$ (\augexref{2.7}) recall that the augmented virtual equipment $\enProf{(\Cat, \Cat')}$ of $2$"/profunctors does not have cocartesian tabulations (see \exref{tabulations of 2-profunctors}), so that we cannot apply the previous theorem to obtain a good Yoneda structure on the $2$"/category $\enCat{\Cat'} = V\bigpars{\enProf{(\Cat, \Cat')}}$ of (locally large) $2$"/categories, in the sense of \cite{Weber07}. In fact it is shown in Remark~9 of \cite{Walker18} that the Yoneda embedding $\map\yon 1\inCat{\Set}$ for the terminal $2$"/category~$1$ (\exref{yoneda embedding for unit V-category}), with presheaf object $\inCat{\Set}$ the $2$"/category of small categories (\augexref{2.9}), cannot be part of a Yoneda structure on the $2$"/category $\enCat{\Cat'}$ that satisfies Axiom~3* of \cite{Street-Walters78}.
	\end{example}
	
	\begin{example} \label{yoneda structure from weak yoneda embeddings}
		In an augmented virtual equipment $\K$ (\defref{augmented virtual equipment}) call a morphism $\map fAC$ admissible if its companion $\hmap{f_*}AC$ exists. Notice that the class $\mathcal A$ of admissible morphisms in $V(\K)$ is a right ideal: for a composite $g \of f$ with $g$ admissible the companion $(g \of f)_* \iso g_*(f, \id)$ (\auglemref{5.11}) exists as well because $\K$ has restrictions on the left, so that $g \of f$ is admissible too. An object $A \in \K$ is admissible precisely if it is unital, that is its horizontal unit $I_A \iso (\id_A)_*$ (\defref{cartesian cells}) exists. Using the assumption that $\K$ has restrictions on the right the previous theorem shows that choosing an admissible weak Yoneda embedding $\map{\yon_A}A{\ps A}$ in $\K$, for each unital object $A$, induces a Yoneda structure on $V(\K)$, in the sense of \cite{Street-Walters78}.
	\end{example}
	
	\begin{example} \label{yoneda embeddings in Q(C)}
		Consider a right ideal $\mathcal A$ of admissible morphisms in a $2$"/category $\mathcal C$. Recall from \augexsref{6.3}{7.9} the strict double category $Q(\mathcal C)$ of quintets in $\mathcal C$, whose vertical and horizontal morphisms both are the morphisms of $\mathcal C$. Let \mbox{$Q_{\mathcal A}(\mathcal C) \subseteq Q(\mathcal C)$} be the full sub"/augmented virtual double category generated by those horizontal morphisms $\hmap jAB$ that are admissible. Supplying a family of admissible weak Yoneda embeddings $\map\yon A{\ps A}$ in $Q_{\mathcal A}(\mathcal C)$, one for each admissible object $A$, is equivalent to equipping $\mathcal C = V\bigpars{Q_{\mathcal A}(\mathcal C)}$ with a Yoneda structure that satisfies Axiom~3* of \cite{Street-Walters78}. Indeed notice that all horizontal morphisms of $Q_{\mathcal A}(\mathcal C)$, like those of $Q(\mathcal C)$, are companions, so that they admit cocartesian $(0,1)$"/ary cells as well as restrictions on the left. It follows from the pasting lemma for cocartesian paths of $(0,1)$"/ary cells (\lemref{pasting lemma for cocartesian paths of (0,1)-ary cells}) that $Q_{\mathcal A}(\mathcal C)$ has right cocartesian paths of $(0,1)$"/ary cells so that (wye)~$\Rightarrow$~(ys*) holds in the previous theorem. Using \auglemref{5.17} and \augpropref{7.12} it is straightforward to check that the converse (ys*) $\Rightarrow$ (wye) holds as well.
	\end{example}

	\begin{example}
		Let $\mathcal A$ be a right ideal of admissible morphisms in a finitely complete $2$"/category $\mathcal C$ and consider the unital virtual equipment $\dFib{\mathcal C}$ of discrete two"/sided fibrations in $\mathcal C$ (\exref{discrete two-sided fibrations}). Write $\mathsf{dFib}_{\mathcal A}(\mathcal C) \subseteq \dFib{\mathcal C}$ for the full sub"/augmented virtual double category generated by those discrete two"/sided fibrations \mbox{$\hmap JAB$} that admit a cartesian nullary cell $\cell\psi JC$ (\defref{cocartesian path of (0,1)-ary cells}) whose vertical source $A \to C$ is admissible. Using the pasting lemma for cartesian cells (\lemref{pasting lemma for cartesian cells}) and the fact that $\mathcal A$ is a right ideal notice that $\mathsf{dFib}_{\mathcal A}(\mathcal C)$ is closed under taking restrictions so that it is an augmented virtual equipment by \auglemref{4.5}. By \lemref{full and faithful functors reflect tabulations} $\mathsf{dFib}_{\mathcal A}(\mathcal C)$ has all cocartesian tabulations too, which are created as in $\dFib{\mathcal C}$ (\exref{tabulations of discrete two-sided fibrations}), so that by \cororef{cocartesian tabulations imply cocartesian paths of (0,1)-ary cells} $\mathsf{dFib}_{\mathcal A}(\mathcal C)$ has left cocartesian paths of $(0,1)$"/ary cells. Applying the previous theorem we find firstly that equipping $\mathcal C = V\bigpars{\mathsf{dFib}_{\mathcal A}(\mathcal C)}$ with a good Yoneda structure, in the sense of \cite{Weber07}, is equivalent to supplying a family of admissible Yoneda embeddings $\map{\yon_A}A\ps A$ in $\mathsf{dFib}_{\mathcal A}(\mathcal C)$, in the sense of \defref{yoneda embedding}, one for each admissible object $A$. Secondly the theorem shows that a family of admissible weak Yoneda embeddings $\yon_A$ in $\mathsf{dFib}_{\mathcal A}(\mathcal C)$, one for each admissible object $A$, induces a Yoneda structure on $\mathcal C$ that satisfies Axiom~3* of \cite{Street-Walters78}.
	\end{example}
	
	Finally consider an equipment $\K$ (\augpropref{7.8}), that is $\K$ is a pseudo double category that has all restrictions. The horizontal dual of Definition~11.4 of \cite{Lambert22} defines $\K$ to have \emph{powers} whenever each object $A \in \K$ admits a horizontal morphism $\hmap{\in_A}A{PA}$ satisfying the following property: for each $\hmap JAB$ there is a unique $\map fB{PA}$ equipped with a cartesian cell of the form below. Having powers is one of several conditions required for $\K$ to be equivalent to the equipment $\Rel(\E)$ of relations (\augthreeexref{2.10}{5.8}{7.4}) internal in some topos $\E$; this is proved in Theorem~11.5 of \cite{Lambert22}.
	\begin{displaymath}
		\begin{tikzpicture}
			\matrix(m)[math35]{A & B \\ A & PA \\};
			\path[map]	(m-1-1) edge[barred] node[above] {$J$} (m-1-2)
									(m-1-2) edge node[right] {$f$} (m-2-2)
									(m-2-1) edge[barred] node[below] {$\in_A$} (m-2-2);
			\path	(m-1-1) edge[eq] (m-2-1);
			\draw[font=\scriptsize] ($(m-1-1)!0.5!(m-2-2)$) node {$\cart$};
		\end{tikzpicture}
	\end{displaymath}
	Let us call the vertical $2$"/category $V(\K)$ (\augexref{1.5}) \emph{locally skeletal} whenever any isomorphic pair $f \iso g$ of morphisms in $V(\K)$ is equal: $f = g$. Notice that in that case the morphisms $\map{\cur J}B{\ps A}$ supplied by the Yoneda axiom (\defref{yoneda embedding}) are uniquely determined by the horizontal morphisms $\hmap JAB$ that induce them. The following result is a straightforward consequence of the pasting lemma for cartesian cells (\lemref{pasting lemma for cartesian cells}).
	\begin{proposition} \label{powers from weak yoneda morphisms}
		Let $\K$ be an equipment with $V(\K)$ locally skeletal. If $\K$ has weak Yoneda morphisms $\map{\yon_A}A{\ps A}$ (\defref{yoneda embedding}) for all objects $A$ then it has all powers $PA \dfn \ps A$, defined as such by the companions $\hmap{\in_A \dfn \yon_{A*}}A{\ps A}$.
	\end{proposition}	

	\section{Exact cells} \label{exact cells}
	In this sections we use the formal notions of left Kan extension and Yoneda morphism to consider in augmented virtual double categories the classical notions of exact squares, as studied by Guitart \cite{Guitart80}. In \secref{totality section} we likewise consider total morphisms and objects, originally introduced by Street and Walters \cite{Street-Walters78}, and in \secref{cocompleteness section} we consider cocomplete objects.
	
	\subsection{Left exact paths of cells}
	The following definition generalises Guitart's notion of `carr\'e exact' \cite{Guitart80} to various notions of exactness for paths of cells in augmented virtual double categories. In \exref{left Beck-Chevalley condition in (Set,Set')-Prof} we will see that, when considered in the augmented virtual double category $\enProf{(\Set, \Set')}$ (\augexref{2.6}), the notion of pointwise left exactness below coincides with the classical notion of exactness. Analogous to the classical situation \thmref{left exactness and Beck-Chevalley} below characterises exactness in terms of a `Beck"/Chevalley condition', in the sense of \defref{left Beck-Chevalley condition}. The latter condition is used in \thmref{left Beck-Chevalley condition and absolute left Kan extensions} to characterise absolute left Kan extensions (\defref{absolutely left Kan}).
	\begin{definition} \label{left exact}
		Consider a path $\ul \phi = (\phi_1, \dotsc, \phi_n)$ as in the composite below, with $n \geq 1$, and let $\map d{C_0}M$ be any vertical morphism. The path $\ul \phi$ is called \emph{(weakly) left $d$-exact} if for any (weakly) left Kan cell $\eta$ (\defsref{weak left Kan extension}{left Kan extension}) as below, with $d$ as vertical source, the composite $\eta \of \ul\phi$ is again (weakly) left Kan. If $\ul \phi$ is (weakly) left $d$-exact for all morphisms $\map d{C_0}M$, where $M$ varies, then it is called \emph{(weakly) left exact}.
		\begin{displaymath}
			\begin{tikzpicture}
				\matrix(m)[math35]
					{	A_{10} & A_{11} & A_{1m'_1} & A_{1m_1} &[4em] A_{n0} & A_{n1} & A_{nm'_n} & A_{nm_n} \\
						C_0 & & & C_1 & C_{n'} & & & C_n \\};
				\draw				([yshift=-3.25em]$(m-2-1)!0.5!(m-2-8)$) node (M) {$M$};
				\path[map]	(m-1-1) edge[barred] node[above] {$J_{11}$} (m-1-2)
														edge node[left] {$f_0$} (m-2-1)
										(m-1-3) edge[barred] node[above] {$J_{1m_1}$} (m-1-4)
										(m-1-4) edge node[right] {$f_1$} (m-2-4)
										(m-1-5)	edge[barred] node[above] {$J_{n1}$} (m-1-6)
														edge node[left] {$f_{n'}$} (m-2-5)
										(m-1-7) edge[barred] node[above] {$J_{nm_n}$} (m-1-8)
										(m-1-8) edge node[right] {$f_n$} (m-2-8)
										(m-2-1) edge[barred] node[below] {$\ul K_1$} (m-2-4)
														edge[transform canvas={yshift=-2pt}] node[below left] {$d$} (M)
										(m-2-5)	edge[barred] node[below] {$\ul K_n$} (m-2-8)
										(m-2-8)	edge[transform canvas={yshift=-2pt}] node[below right] {$l$} (M);
				\path				($(m-1-1.south)!0.5!(m-1-4.south)$) edge[cell] node[right] {$\phi_1$} ($(m-2-1.north)!0.5!(m-2-4.north)$)
										($(m-1-5.south)!0.5!(m-1-8.south)$) edge[cell] node[right] {$\phi_n$} ($(m-2-5.north)!0.5!(m-2-8.north)$)
										($(m-2-1.south)!0.5!(m-2-8.south)$) edge[cell] node[right] {$\eta$} (M);
				\draw[transform canvas={xshift=-1pt}]	($(m-1-2)!0.5!(m-1-3)$) node {$\dotsb$}
										($(m-1-6)!0.5!(m-1-7)$) node {$\dotsb$};
				\draw				($(m-1-4)!0.5!(m-2-5)$) node {$\dotsb$};
			\end{tikzpicture}
		\end{displaymath}
				
		Analogously $\ul \phi$ is called \emph{pointwise} (weakly) left $d$-exact if for any pointwise (weakly) left Kan cell $\eta$ (\defref{pointwise left Kan extension}) as above the composite above is pointwise (weakly) left Kan too. A path that is pointwise (weakly) left $d$-exact for all $\map d{C_0}M$ is called \emph{pointwise} (weakly) left exact.
	\end{definition}
	
	\begin{example} \label{cocartesian paths are exact}
		Consider a (pointwise) right (respectively weakly) nullary"/cocartesian path $\ul\phi = (\phi_1, \dotsc, \phi_n)$ (\defsref{cocartesian path}{pointwise right cocartesian path}) such that the vertical target of $\phi_n$ is trivial. By the vertical pasting lemma (\lemref{vertical pasting lemma}) $\ul\phi$ is (pointwise) (weakly) left exact.
	\end{example}
	
	\subsection{Cells that are left exact with respect to Yoneda morphisms}
	Given a Yoneda morphism $\map\yon A{\ps A}$ the following two results concern left $\yon$"/exact cells with a trivial vertical target. The first of these is used in \exref{presheaf objects are total} below to show that, under mild conditions, presheaf objects $\ps A$ are `total', in the sense of \defref{totality} below. In \thmref{presheaf objects as free cocompletions} it is also used, to give a condition that ensures the `cocompleteness' (\defref{cocompletion}) of presheaf objects. The second result below is used in \thmref{left exactness and Beck-Chevalley}.
	\begin{proposition} \label{left Kan extensions along a yoneda embedding in terms of left y-exact cells}
		Let $\map\yon A{\ps A}$ be a Yoneda morphism (\defref{yoneda embedding}) and let $\phi$ and $\psi$ be cells that correspond under the bijection of \lemref{bijection between cells induced by Yoneda morphisms}. The cell $\phi$ is (pointwise) (weakly) left $\yon$"/exact if and only if the cell $\psi$ is (pointwise) (weakly) left Kan.
		
		In particular consider a morphism $\map dB{\ps A}$ such that the restriction $\ps A(\yon, d)$ exists as well as a path $\hmap{\ul H}B{B_n}$. If there exists a (pointwise) (weakly) horizontal left $\yon$-exact cell $\phi$ of the form below then the (pointwise) (weak) left Kan extension $\map l{B_n}{\ps A}$ of $d$ along $\ul H$ exists. The converse holds whenever the restriction $\ps A(\yon, l)$ exists.
		\begin{displaymath}
			\begin{tikzpicture}
				\matrix(m)[math35]{ A & B & B_1 & B_{n'} & B_n \\ & A & & B_n & \\ };
				\path[map]	(m-1-1) edge[barred] node[above] {$\ps A(\yon, d)$} (m-1-2)
										(m-1-2) edge[barred] node[above] {$H_1$} (m-1-3)
										(m-1-4) edge[barred] node[above] {$H_n$} (m-1-5)
										(m-2-2) edge[barred] node[below] {$K$} (m-2-4);
				\path				(m-1-1) edge[eq] (m-2-2)
										(m-1-5) edge[eq] (m-2-4)
										(m-1-3) edge[cell] node[right] {$\phi$} (m-2-3);
				\draw				($(m-1-3)!0.5!(m-1-4)$) node {$\dotsb$};
			\end{tikzpicture}
		\end{displaymath}
	\end{proposition}
	\begin{proof}
		We only treat the left $y$"/exact and left Kan case; the proofs for the (pointwise) (weakly) cases are the same. Recall from \lemref{bijection between cells induced by Yoneda morphisms} that the unary cell \mbox{$\cell\phi{J \conc \ul H}K$} and the nullary cell $\cell\psi{\ul H}{\ps A}$ correspond via the equality
		\begin{displaymath}
			\cart_{\cur K} \of \phi = \cart_{\cur J} \hc \psi,
		\end{displaymath}
		where the cartesian cells define $\cur K$ and $\cur J$ (\defref{yoneda embedding}). By density of $\yon$ (\defref{density definition}) these cartesian cells define $\cur K$ and $\cur J$ as pointwise left Kan extensions of $\yon$. Hence, by definition, $\phi$ is left $\yon$"/exact if and only if the left"/hand side above, and hence both sides, is left Kan which, by the horizontal pasting lemma (\lemref{horizontal pasting lemma}), is equivalent to $\psi$ being left Kan.
		
		For the second assertion take $J \dfn \ps A(\yon, d)$ so that $\cur J \iso d$ by uniqueness of $\cur J$ (\defref{yoneda embedding}). By the main assertion any horizontal left $\yon$"/exact cell \mbox{$\cell\phi{\ps A(\yon, d) \conc \ul H}K$} corresponds to a nullary cell $\cell\psi{\ul H}{\ps A}$ that defines $\cur K$ as the left Kan extension of $d$ along $\ul H$. For the converse set $K \dfn \ps A(\yon, l)$ so that $\cur K \iso l$. By the main assertion any nullary cell $\cell\psi{\ul H}{\ps A}$ that defines $l$ as the left Kan extension of $\ul H$ along $d$ corresponds to a horizontal left $\yon$"/exact cell $\cell\phi{\ps A(\yon, d) \conc \ul H}{\ps A(\yon, l)}$.
	\end{proof}
	
	\begin{proposition} \label{left exactness and right unary-cocartesianness}
		Let $\map\yon C{\ps C}$ be a (weak) Yoneda morphism (\defref{yoneda embedding}) such that the restrictions $L(h, \id)$ exist for all morphisms $\hmap LXY$ and \mbox{$\map hCX$}. Consider the cell $\phi$ below. Assuming that either $f = \id_C$ or the conjoint $\hmap{f^*}C{A_0}$ exists, $\phi$ is (weakly) $\yon$"/left exact if and only if $\phi$ is right (respectively weakly) unary"/cocartesian (\defref{cocartesian path}). If moreover the restrictions $K(\id, g)$ exist for all $\map gXD$ then the analogous assertion for the pointwise (weak) case (\defref{pointwise right cocartesian path}) holds too.
		\begin{displaymath}
			\begin{tikzpicture}
				\matrix(m)[math35]{A_0 & A_1 & A_{n'} & D \\ C & & & D \\};
				\path[map]	(m-1-1) edge[barred] node[above] {$J_1$} (m-1-2)
														edge node[left] {$f$} (m-2-1)
										(m-1-3) edge[barred] node[above] {$J_n$} (m-1-4)
										(m-2-1) edge[barred] node[below] {$K$} (m-2-4);
				\path[transform canvas={xshift=1.75em}]	(m-1-2) edge[cell] node[right] {$\phi$} (m-2-2);
				\path				(m-1-4) edge[eq] (m-2-4);
				\draw				($(m-1-2)!0.5!(m-1-3)$) node[xshift=-1.5pt] {$\dotsb$};
			\end{tikzpicture}
		\end{displaymath}
	\end{proposition}
	Notice that together with \exref{cocartesian paths are exact} this proposition implies that, under the conditions above, if the cell $\phi$ above is (pointwise) right (respectively weakly) nullary"/cocartesian then it is (pointwise) right (respectively weakly) cocartesian.
	\begin{proof}
		We will first prove the non-weak case in which $f = \id_C$. Afterwards we generalise to the case where $f^*$ exists as well as consider the pointwise case. For the weak case take $\ul H$ to be empty in the following. Given any path \mbox{$\hmap{\ul H}D{B_m}$} of horizontal morphisms consider the assignments below of cells of the forms as shown, that are given by composition with the cartesian cells defining respectively the restriction $L(h, \id)$ and the morphism $\cur{L(h, \id)}$ (\defref{yoneda embedding}). By definition $\phi$ is right unary"/cocartesian if all unary cells $\ul J \conc \ul H \Rightarrow L$ on the left below factor uniquely through the path $\phi \conc \id_{\ul H}$, where $\id_{\ul H} \dfn (\id_{H_1}, \dotsc, \id_{H_q})$.
		\begin{displaymath}
			\begin{tikzpicture}
				\matrix(m)[math35, xshift=-13em]{C & B_m \\ X & Y \\};
				\path[map]	(m-1-1) edge[barred] node[above] {$\ul J \conc \ul H$} (m-1-2)
														edge node[left] {$h$} (m-2-1)
										(m-1-2) edge node[right] {$k$} (m-2-2)
										(m-2-1) edge[barred] node[below] {$L$} (m-2-2);
				\path[transform canvas={xshift=1.75em}]	(m-1-1) edge[cell] (m-2-1);
				
				\matrix(m)[math35]{C & B_m \\ C & Y \\};
				\path[map]	(m-1-1) edge[barred] node[above] {$\ul J \conc \ul H$} (m-1-2)
										(m-1-2) edge node[right] {$k$} (m-2-2)
										(m-2-1) edge[barred] node[below] {$L(h, \id)$} (m-2-2);
				\path				(m-1-1) edge[eq] (m-2-1);
				\path[transform canvas={xshift=1.75em}]	(m-1-1) edge[cell] (m-2-1);
				
				\matrix(m)[math35, column sep={0.875em,between origins}, xshift=13em]
					{C & & & & B_m \\ & & & \mspace{6mu} Y & \\ & & \ps C & & \\};
				\path[map]	(m-1-1) edge[barred] node[above] {$\ul J \conc \ul H$} (m-1-5)
														edge[transform canvas={xshift=-1pt}, ps] node[left] {$\yon$} (m-3-3)
										(m-1-5) edge[transform canvas={xshift=1pt}] node[right] {$k$} (m-2-4)
										(m-2-4) edge[transform canvas={xshift=1pt}, ps] node[right, inner sep=1pt] {$\cur{L(h, \id)}$} (m-3-3);
				\path[transform canvas={yshift=-0.7em}]	(m-1-3) edge[cell] (m-2-3);
					
				\draw[font=\Large]	(-16.0em, 0) node {$\lbrace$}
										(-10.2em,0) node {$\rbrace$}
										(-2.9em,0) node {$\lbrace$}
										(2.8em, 0) node {$\rbrace$}
										(10.3em,0) node {$\lbrace$}
										(16.4em,0) node {$\rbrace$};
				\path[map]	(-4.9em,0) edge node[above] {$\cart \of \dash$} (-8.2em,0)
										(4.8em,0) edge node[above] {$\cart \of \dash$} (8.3em,0);
			\end{tikzpicture}
		\end{displaymath}
		By the uniqueness of factorisations through cartesian cells the assignments above are bijections, so that the latter is equivalent to all nullary cells $\ul J \conc \ul H \Rightarrow \ps C$ on the right above factoring uniquely through $\phi \conc \id_{\ul H}$. Abbreviating $M \dfn L(h, \id)$, spelled out this means that $\phi$ is right unary"/cocartesian if and only if the second assignment below is a bijection. The first assignment below is given by horizontal composition with the cartesian cell $\cart$ defining $\cur K$; it is a bijection because $\cart$ is pointwise left Kan, by the density of $\yon$ (\defref{density definition}). We conclude that $\phi$ is right unary"/cocartesian if and only if the composite assignment below, given by horizontal composition with $\cart \of \phi$, is a bijection. But the latter means that $\cart \of \phi$ again left Kan which, since $\cart$ has $\yon$ as vertical source, by definition is equivalent to $\phi$ being left $\yon$"/exact.
		\begin{displaymath}
			\begin{tikzpicture}
				\matrix(m)[math35, column sep={0.875em,between origins}, xshift=-14em]
					{D & & & & B_m \\ & & & \mspace{6mu} Y & \\ & & \ps C & & \\};
				\path[map]	(m-1-1) edge[barred] node[above] {$\ul H$} (m-1-5)
														edge[transform canvas={xshift=-1pt}, ps] node[left] {$\cur K$} (m-3-3)
										(m-1-5) edge[transform canvas={xshift=1pt}] node[right] {$k$} (m-2-4)
										(m-2-4) edge[transform canvas={xshift=1pt}, ps] node[right] {$\cur M$} (m-3-3);
				\path[transform canvas={yshift=-0.7em}]				(m-1-3) edge[cell] (m-2-3);
				
				\matrix(m)[math35, column sep={0.875em,between origins}]
					{C & & & & B_m \\ & & & \mspace{6mu} Y & \\ & & \ps C & & \\};
				\path[map]	(m-1-1) edge[barred] node[above] {$K \conc \ul H$} (m-1-5)
														edge[transform canvas={xshift=-1pt}, ps] node[left] {$\yon$} (m-3-3)
										(m-1-5) edge[transform canvas={xshift=1pt}] node[right] {$k$} (m-2-4)
										(m-2-4) edge[transform canvas={xshift=1pt}, ps] node[right] {$\cur M$} (m-3-3);
				\path[transform canvas={yshift=-0.7em}]	(m-1-3) edge[cell] (m-2-3);
				
				\matrix(m)[math35, column sep={0.875em,between origins}, xshift=14em]
					{C & & & & B_m \\ & & & \mspace{6mu} Y & \\ & & \ps C & & \\};
				\path[map]	(m-1-1) edge[barred] node[above] {$\ul J \conc \ul H$} (m-1-5)
														edge[transform canvas={xshift=-1pt}, ps] node[left] {$\yon$} (m-3-3)
										(m-1-5) edge[transform canvas={xshift=1pt}] node[right] {$k$} (m-2-4)
										(m-2-4) edge[transform canvas={xshift=1pt}, ps] node[right] {$\cur M$} (m-3-3);
				\path[transform canvas={yshift=-0.7em}]	(m-1-3) edge[cell] (m-2-3);
					
				\draw[font=\Large]	(-17.0em, 0) node {$\lbrace$}
										(-11.6em,0) node {$\rbrace$}
										(-2.6em,0) node {$\lbrace$}
										(2.5em, 0) node {$\rbrace$}
										(11.4em,0) node {$\lbrace$}
										(16.5em,0) node {$\rbrace$};
				\path[map]	(-9.6em,0) edge node[above] {$\cart \hc \dash$} (-4.6em,0)
										(4.5em,0) edge node[above] {$\dash \of (\phi \conc \id_{\ul H})$} (9.4em,0);
			\end{tikzpicture}
		\end{displaymath}
		This completes the proof of the non"/pointwise case with $f = \id_C$. In the general case, where the conjoint $f^*$ exists, apply the previous to the composite horizontal cell $\cell{\cart \hc \phi}{f^* \conc \ul J}K$ where $\cart$ defines $f^*$. The proof follows by noticing that $\phi$ is (weakly) left $\yon$"/exact if and only if $\cart \hc \phi$ is so, by \cororef{Kan extension and conjoints}, and likewise that $\phi$ is right (respectively weakly) unary"/cocartesian if and only if $\cart \hc \phi$ is so, by \cororef{non-horizontal cocartesian cells}.
		
		Finally assume that the restrictions $K(\id, g)$ exist for all $\map gXD$. Restricting to those $g$ for which the restriction $J_n(\id, g)$ holds too, consider the unique factorisations $\phi'_g$ in
		\begin{displaymath}
			\phi \of (\id_{J_1}, \dotsc, \id_{J_{n'}}, \cart_{J_n(\id, g)}) = \cart_{K(\id, g)} \of \phi'_g
		\end{displaymath}
		as in \defref{pointwise right cocartesian path}, where the cartesian cells define the restrictions $J_n(\id, g)$ and $K(\id, g)$ respectively. Still denoting the cartesian cell defining $\cur K$ by $\cart$, recall that it is pointwise (weakly) left Kan. It follows that $\phi$ is pointwise (weakly) $\yon$"/exact precisely if for each $g$ the left"/hand side above, and thus either side, postcomposed with $\cart$ is (weakly) left Kan. Since the composites $\cart \of \cart_{K(\id, g)}$ are cartesian by the pasting lemma (\lemref{pasting lemma for cartesian cells}), and thus (weakly) left Kan by (weak) density of $\yon$ (\defref{density definition}), the latter is equivalent to the factorisations $\phi'_g$ being (weakly) $\yon$"/exact. By the main assertion this is equivalent to $\phi'_g$ being right (respectively weakly) unary"/cocartesian, for each $g$ such that both $J_n(\id, g)$ and $K(\id, g)$ exist. But that means precisely that $\phi$ is pointwise right (respectively weakly) unary cocartesian. This completes the proof.
	\end{proof}
	
	Combining \propsref{cocartesian paths are exact}{left exactness and right unary-cocartesianness} we obtain the following.
	\begin{corollary} \label{left Kan extension, exact cells, cocartesian cells}
		In an augmented virtual double category $\K$ consider a horizontal morphism $\hmap JAB$ as well as a path $\hmap{\ul H}B{B_n}$ of length $n \geq 1$, and assume that the Yoneda morphism $\map{\yon_A}A{\ps A}$ exists (\defref{yoneda embedding}). Among the following conditions the implications \textup{(a)}~$\Rar$~\textup{(b)}~$\Rar$~\textup{(c)} and \textup{(b)} $\Rar$ \textup{(d)} $\Rar$ \textup{(e)} hold. If $\K$ is unital (\defref{augmented virtual equipment}) then \textup{(b)}~$\Leftrightarrow$~\textup{(c)}. If $\K$ has restrictions on the left (\defref{augmented virtual equipment}) then \textup{(c)}~$\Leftrightarrow$~\textup{(d)}. If $\yon_A$ admits nullary restrictions (\defref{yoneda embedding}) then \textup{(d)} $\Leftrightarrow$ \textup{(e)}. The same implications hold for the pointwise variants of the conditions (\defsref{pointwise left Kan extension}{pointwise right cocartesian path} and \augdefref{9.1}), except for \textup{(a)} $\Rightarrow$ \textup{(b)} and \textup{(c)}~$\Leftrightarrow$~\textup{(d)} which moreover require the restrictions $K(\id, g)$ to exist for all $\map gX{B_n}$.
		\begin{enumerate}[label=\textup{(\alph*)}]
			\item The horizontal composite $(J \hc H_1 \hc \dotsb \hc H_n)$ exists (\defref{pointwise right cocartesian path});
			\item a right (respectively weakly) cocartesian cell of the form below exists (\defref{cocartesian path});
			\item a right (respectively weakly) unary"/cocartesian cell of the form below exists;
			\item a (weakly) left $\yon_A$"/exact cell of the form below exists;
			\item the (weak) left Kan extension of $\map{\cur J}B{\ps A}$ (\defref{yoneda embedding}) along $\ul H$ exists (\defsref{weak left Kan extension}{left Kan extension}).
		\end{enumerate}
		
		Finally restrict to paths $\ul H = (H_1)$ of length $n = 1$ and assume that the Yoneda morphism $\map{\yon_B}B{\ps B}$ exists. If the pointwise (weak) left Kan extension of \mbox{$\map{\cur J}B{\ps A}$} along $\yon_{B*}$ exists then the pointwise variant of condition \textup{(e)} is satisfied.
		\begin{displaymath}
			\begin{tikzpicture}
				\matrix(m)[math35]{ A & B & B_1 & B_{n'} & B_n \\ & A & & B_n & \\ };
				\path[map]	(m-1-1) edge[barred] node[above] {$J$} (m-1-2)
										(m-1-2) edge[barred] node[above] {$H_1$} (m-1-3)
										(m-1-4) edge[barred] node[above] {$H_n$} (m-1-5)
										(m-2-2) edge[barred] node[below] {$K$} (m-2-4);
				\path				(m-1-1) edge[eq] (m-2-2)
										(m-1-5) edge[eq] (m-2-4)
										(m-1-3) edge[cell] node[right] {$\phi$} (m-2-3);
				\draw				($(m-1-3)!0.5!(m-1-4)$) node {$\dotsb$};
			\end{tikzpicture}		\end{displaymath}
	\end{corollary}
	Notice that if $\yon_A$ admits nullary restrictions then condition (e) above also applies to (pointwise) (weak) left Kan extensions along $\ul H$ of any morphism $\map dB{\ps A}$, by taking $J \dfn \ps A(\yon_A, d)$ so that, by uniqueness of $\cur J$ (\defref{yoneda embedding}), $\cur J \iso d$.
	\begin{proof}
		(a) $\Rar$ (b) follows from \exref{cocartesian paths are right cocartesian}, \defref{cocartesian path} and \remref{right pointwise cocartesian and pointwise right cocartesian comparison}. (b)~$\Rar$~(c) and (c) $\Rar$ (b) follow from \defsref{cocartesian path}{pointwise right cocartesian path}, and \exref{cocartesian paths in the presence of horizontal units or (1,0)-ary cartesian cells}. (b) $\Rar$ (d) follows from \exref{cocartesian paths are exact}. \propref{left exactness and right unary-cocartesianness} shows that (c)~$\Leftrightarrow$~(d). Applying \propref{left Kan extensions along a yoneda embedding in terms of left y-exact cells} to $\map{d \dfn \cur J}B{\ps A}$, so that $\ps A(\yon_A, \cur J) \iso J$ by the Yoneda axiom (\defref{yoneda embedding}), shows that (d) $\Rar$ (e) and (e) $\Rar$ (d). For the final assertion write $\cell\eta {\yon_{B*}}{\ps A}$ for the nullary cell that defines the pointwise (weak) left Kan extension of $\cur J$ along $\yon_{B*}$, and write $\cell{\cart'}{H_1}{\yon_{B*}}$ for the factorisation of the cartesian cell defining $\map{\cur H_1}{B_1}{\ps B}$ (\defref{yoneda embedding}) through the cartesian cell defining the companion $\yon_{B*}$. Since $\cart'$ is cartesian by the pasting lemma (\lemref{pasting lemma for cartesian cells}), the composite $\cell{\eta \of \cart'}{H_1}{\ps A}$ defines the pointwise (weak) left Kan extension of $\cur J$ along $H_1$, by \defref{pointwise left Kan extension}.
	\end{proof}
	
	In an augmented virtual equipment (\defref{augmented virtual equipment}) let us for a moment call an object $A$ `admissible' if it admits a Yoneda morphism $\map{\yon_A}A{\ps A}$ that has nullary restrictions (\defref{yoneda embedding}). Fixing an object $B$ it follows from the previous corollary that the following conditions are equivalent where, for condition (c) only, we assume that $B$ itself is admissible.
	\begin{enumerate}[label=\textup{(\alph*)}]
		\item For any path $A \xbrar J B \xbrar H E$ with $A$ admissible a horizontal pointwise right unary"/cocartesian cell (\defref{pointwise right cocartesian path}) with horizontal source $(J, H)$ exists;
		\item for any admissible object $A$ and any morphisms $\ps A \xlar d B \xbrar H E$ the pointwise left Kan extension (\defref{pointwise left Kan extension}) of $d$ along $H$ exists;
		\item for any admissible object $A$ and any morphism $\map dB\ps A$ the pointwise left Kan extension of $d$ along $\yon_{B*}$ exists.
	\end{enumerate}
	The next example shows that for a locally small category $B$ condition (a) implies essential smallness. This is why we think of the conditions above as ``restricting the size of the object $B$''.
	
	\begin{example}
		Consider a locally small category $B$ in the augmented virtual equipment $\enProf{(\Set, \Set')}$ (\augexref{2.6}) and assume that for any path $1 \xbrar J B \xbrar H 1$ of $\Set$"/profunctors a weakly unary"/cocartesian horizontal cell (\defref{cocartesian path}) with horizontal source $(J, H)$ exists. Using an argument similar to the one used in \augexref{4.18}, together with the fact that the embedding $\Set \hookrightarrow \Set'$ preserves colimits, this assumption implies that the coend $\int^{y \in B} J(*, y) \times H(y, *)$ is small. We can use the latter to show that the category $\ps B$ of presheaves $\map p{\op B}\Set$ is locally small too, so that $B$ is essentially small by \cite{Freyd-Street95}. To do so let $p$ and $q$ be any presheaves on $B$ and consider the composite below, where the injection is induced by the unit of the double"/powerset monad $2^{(2^{\dash})}$ on $\Set$ (see e.g. Example~1 of \cite{Tholen09}). Applying the assumption to $J = 2^q$ and $H = p$ we find that the right"/hand side below is small, so that the hom"/set $\ps B(p, q)$ on the left"/hand side is small too, as required. 
		\begin{displaymath}
			\ps B(p, q) = \int\limits_{y \in B} (qy)^{py} \hookrightarrow \int\limits_{y \in B} \bigpars{2^{\pars{2^{qy}}}}^{py} \iso \int\limits_{y \in B} 2^{2^{qy} \times py} \iso 2^{\int\limits^{y \in B} 2^{qy} \times py}
		\end{displaymath}
	\end{example}
	
	\subsection{The left Beck"/Chevalley condition}
	In \exref{cocartesian paths are exact}, we saw that nullary"/cocartesian paths with trivial vertical target are exact. The following definition, of the `left Beck"/Chevalley condition', allows us to use this to obtain exact paths with any vertical target (\cororef{left Beck-Chevalley condition implies exactness}). In \exref{left Beck-Chevalley condition in (Set,Set')-Prof} we will see that this condition recovers the classical notion of exactness for transformations of functors, as used by Guitart in \cite{Guitart80}.
	\begin{definition} \label{left Beck-Chevalley condition}
		Consider a non"/empty path $\ul \phi = (\phi_1, \dotsc, \phi_n)$ of cells. Denote the (possibly empty) horizontal target of each $\phi_i$ by $\hmap{\ul K_i}{C_{i'}}{C_i}$ and assume that $\phi_n$ is of the form as on the left"/hand side below. We say that $\ul \phi$ satisfies the \emph{(weak) left Beck-Chevalley condition} if the restriction $\ul K_n(\id, f_n)$ exists and the path $\ul \phi' \dfn (\phi_1, \dotsc, \phi_{n'}, \phi_n')$ is right (respectively weakly) nullary"/cocartesian (\defref{cocartesian path}), where $\phi_n'$ is the unique factorisation below. Here if $\ul K_n = (C_{n'})$ is empty then $\ul K_{n}(\id, f_n) = f_n^*$ is the conjoint of $f_n$.
		\begin{displaymath}
			\begin{tikzpicture}[textbaseline]
				\matrix(m)[math35]{A_{n0} & A_{n1} & A_{nm_n'} & A_{nm_n} \\ C_{n'} & & & C_n \\};
				\path[map]	(m-1-1) edge[barred] node[above] {$J_{n1}$} (m-1-2)
														edge node[left] {$f_{n'}$} (m-2-1)
										(m-1-3) edge[barred] node[above] {$J_{nm_n}$} (m-1-4)
										(m-2-1) edge[barred] node[below] {$\ul K_n$} (m-2-4)
										(m-1-4) edge node[right] {$f_n$} (m-2-4);
				\path[transform canvas={xshift=2em}]	(m-1-2) edge[cell] node[right] {$\phi_n$} (m-2-2);
				\draw				($(m-1-2)!0.5!(m-1-3)$) node[xshift=-1.5pt] {$\dotsb$};
			\end{tikzpicture} \quad = \quad \begin{tikzpicture}[textbaseline]
				\matrix(m)[math35]{A_{n0} & A_{n1} & A_{nm_n'} & A_{nm_n} \\ C_{n'} & & & A_{nm_n} \\ C_{n'} & & & C_n \\};
				\path[map]	(m-1-1) edge[barred] node[above] {$J_{n1}$} (m-1-2)
														edge node[left] {$f_{n'}$} (m-2-1)
										(m-1-3) edge[barred] node[above] {$J_{nm_n}$} (m-1-4)
										(m-2-1) edge[barred] node[below] {$\ul K_n(\id, f_n)$} (m-2-4)
										(m-2-4) edge node[right] {$f_n$} (m-3-4)
										(m-3-1) edge[barred] node[below] {$\ul K_n$} (m-3-4);
				\draw				([xshift=-1pt]$(m-1-2)!0.5!(m-1-3)$) node {$\dotsc$}
										([yshift=-0.333em]$(m-2-1)!0.5!(m-3-4)$) node[font=\scriptsize] {$\cart$};
				\path				(m-1-4) edge[eq] (m-2-4)
										(m-2-1) edge[eq] (m-3-1);
				\path[transform canvas={xshift=2em}]	(m-1-2) edge[cell] node[right] {$\phi_n'$} (m-2-2);
			\end{tikzpicture}
		\end{displaymath}
		Likewise we say that $\ul \phi$ satisfies the \emph{pointwise (weak) left Beck"/Chevalley condition} if $\ul\phi'$ is pointwise right (respectively weakly) nullary"/cocartesian (\defref{pointwise right cocartesian path}). A single cell $\phi$ is said to satisfy the (pointwise) (weak) left Beck"/Chevalley condition whenever the single path $(\phi)$ does so.
	\end{definition}
	Notice that if the vertical target $f_n = \id_{C_n}$ of $\phi_n$ is trivial then the left Beck"/Chevalley conditions above reduce to $\ul\phi$ being (pointwise) right (respectively weakly) nullary"/cocartesian.
	
	\begin{example} \label{left Beck-Chevalley condition in (Set,Set')-Prof}
		Let $\V \subset \V'$ be a universe enlargement and consider the cell $\phi$ below in the augmented virtual equipment $\enProf{(\V, \V')}$ of $\V$"/profunctors between $\V'$"/categories (\augexref{2.7}). Assume that the conjoint $\hmap{f^*}CA$ exists (\augexref{4.6}).
		\begin{displaymath}
			\begin{tikzpicture}
				\matrix(m)[math35]{A & B \\ C & D \\};
				\path[map]	(m-1-1) edge[barred] node[above] {$J$} (m-1-2)
														edge node[left] {$f$} (m-2-1)
										(m-1-2) edge node[right] {$g$} (m-2-2)
										(m-2-1) edge[barred] node[below] {$K$} (m-2-2);
				\path[transform canvas={xshift=1.75em}]	(m-1-1) edge[cell] node[right] {$\phi$} (m-2-1);
			\end{tikzpicture}
		\end{displaymath}
		Combining \cororef{non-horizontal cocartesian cells}, \augexref{9.2} and \remref{right pointwise cocartesian and pointwise right cocartesian comparison} we find that $\phi$ satisfies the pointwise left Beck"/Chevalley condition whenever the morphisms
		\begin{displaymath}
			C(z, fx) \tens J(x, y) \xrar{\id \tens \phi} C(z, fx) K(fx, gy) \xrar\lambda K(z, gy)
		\end{displaymath}
		define each $K(z, gy)$ as the coend $\int^{x \in A} C(z, fx) \tens J(x, y)$ in $\V'$ (or, equivalently, in $\V$, provided that $\V \subset \V'$ preserves large colimits). If the latter coend, which exists in $\V'$, is known to be (isomorphic to) a $\V$"/object then by \thmref{left exactness and Beck-Chevalley} below the afore"/mentioned sufficient condition is equivalent to $\phi$ satisfying the pointwise left Beck"/Chevalley condition, as well as to $\phi$ being pointwise left exact (\defref{left exact}).
		
		In particular, specialising to $\Set \subset \Set'$ and assuming that $A$ is small, if $J = j_*$ and $K = k_*$ for functors $\map jAB$ and $\map kCD$, so that $\phi$ corresponds to a transformation $k \of f \Rar g \of j$ of functors, we find that our notion of pointwise left exactness coincides with the notion of `carr\'e exact' given in Definition~1.1 of \cite{Guitart80}; see condition BC' of Theorem~1.2 therein.
	\end{example}
	
	Paths satisfying the left Beck-Chevalley condition can be concatenated as follows.
	\begin{lemma} \label{concatenation of paths satisfying the Beck-Chevalley condition}
		Consider the path $\ul \phi$ of \defref{left Beck-Chevalley condition} and assume that it satisfies the left Beck"/Chevalley condition. A concatenation $\ul\phi \conc \ul\psi$ of paths satisfies the (pointwise) (weak) left Beck"/Chevalley condition if and only if the path \mbox{$\id_{\ul K'} \conc \cart \conc \ul\psi$} does so, where $\ul K' \dfn \ul K_1 \conc \dotsb \conc \ul K_{n'}$ and $\cart$ defines the restriction $\ul K_n(\id, f_n)$.
	\end{lemma}
	\begin{proof}[(Sketch)]
		Factorise $\phi_n = \cart \of \phi_n'$ and $\psi_m = \cart \of \psi_m'$ as in \defref{left Beck-Chevalley condition} and apply the pasting lemma for cocartesian paths (\lemref{pasting lemma for cocartesian paths}) to $(\phi_1, \dotsc, \phi_n', \id_{\ul H_1}, \dotsc, \id_{\ul H_m})$ and $\id_{\ul K'} \conc (\cart, \psi_1, \dotsc, \psi_m')$, where each $\ul H_j$ denotes the horizontal source of $\psi_j$ and $\cart$ defines $\ul K_n(\id, f_n)$.
	\end{proof}
	
	The following result, which is a straightforward consequence of the vertical pasting lemma for left Kan extensions (\lemref{vertical pasting lemma}), shows that the left Beck"/Chevalley conditions imply left exactness. Recall from \defref{pointwise left Kan extension} the notion of a left Kan extension restricting along a vertical morphism.
	\begin{corollary} \label{left Beck-Chevalley condition implies exactness}
		Consider the path $\ul\phi$ of \defref{left Beck-Chevalley condition} and a morphism \mbox{$\map d{C_0}M$}. The path $\ul \phi$ is (weakly) left $d$"/exact (\defref{left exact}) if $\ul \phi$ satisfies the (weak) left Beck-Chevalley condition and the (weak) left Kan extension of $d$ along \mbox{$\ul K_1 \conc \dotsb \conc \ul K_n$} restricts along $f_n$ (and $\lns{\ul K_n} \geq 1$).
		
		Likewise $\ul\phi$ is pointwise (weakly) left exact if $\ul\phi$ satisfies the pointwise (weak) left Beck"/Chevalley condition (and $\lns{\ul K_n} \geq 1$).
	\end{corollary}
	\begin{proof}
		Write $\ul K \dfn \ul K_1 \conc \dotsb \conc \ul K_n$ and consider the factorisation $\phi_n = \cart \of \phi_n'$ and the path $\ul\phi' = (\phi_1, \dotsc, \phi_{n'}, \phi_n')$ of \defref{left Beck-Chevalley condition}, where $\cart$ defines the restriction $\ul K_n(\id, f_n)$. We first consider the case where $\ul K_n$ is non"/empty. Given a (weakly) left Kan cell $\cell\eta{\ul K} M$ that defines the (weak) left Kan extension of $d$ along $\ul K$ we have to show that \mbox{$\eta \of \ul \phi = \bigpars{\eta \of (\id, \dotsc, \id, \cart)} \of \ul\phi'$} is again (weakly) left Kan, assuming that $\ul\phi'$ is right (respectively weakly) nullary"/cocartesian and that $\eta$ restricts along $f_n$. By \defref{pointwise left Kan extension} the latter means that the composite $\eta \of (\id, \dotsc, \id, \cart)$ is (weakly) left Kan, so that the result follows by applying the vertical pasting lemma (\lemref{vertical pasting lemma}) to the nullary"/cocartesian path $\ul\phi'$.
		
		For the second assertion assume that $\cell\eta{\ul K}M$ is pointwise (weakly) left Kan (\defref{pointwise left Kan extension}). By \lemref{restrictions of pointwise left Kan extensions} the composite $\eta \of (\id, \dotsc, \id, \cart)$ is pointwise (weakly) left Kan as well so that $\eta \of \ul\phi = \bigpars{\eta \of (\id, \dotsc, \id, \cart)} \of \ul \phi'$ is so too by applying the vertical pasting lemma to the path $\ul\phi'$, which is pointwise right (respectively weakly) nullary"/cocartesian by assumption.
		
		Finally consider the non"/weak case where $\ul K_n = (C_{n'})$ is empty, so that $\cart$ in \mbox{$\phi_n = \cart \of \phi_n'$} defines the conjoint $\hmap{f_n^*}{C_{n'}}{A_{nm_n}}$ of $f_n$. The same arguments apply except for the claim that $\eta \of (\id, \dotsc, \id, \cart)$ is (pointwise) left Kan, which in this case follows from the assumption that $\eta$ is left Kan by \cororef{Kan extension and conjoints}.
	\end{proof}
	Combining \propref{left exactness and right unary-cocartesianness} and \cororef{left Beck-Chevalley condition implies exactness} the following theorem relates the notions of nullary- and unary"/cocartesianness, the notion of left exactness and the left Beck"/Chevalley condition.
	\begin{theorem} \label{left exactness and Beck-Chevalley}
		Consider the factorisation below. Assume that the (weak) Yoneda morphism $\map\yon C{\ps C}$ exists and that either $f = \id_C$ or the conjoint $f^*$ exists.
		\begin{displaymath}
			\begin{tikzpicture}[textbaseline]
				\matrix(m)[math35]{A_0 & A_1 & A_{n'} & A_n \\ C & & & D \\};
				\path[map]	(m-1-1) edge[barred] node[above] {$J_1$} (m-1-2)
														edge node[left] {$f$} (m-2-1)
										(m-1-3) edge[barred] node[above] {$J_n$} (m-1-4)
										(m-2-1) edge[barred] node[below] {$K$} (m-2-4)
										(m-1-4) edge node[right] {$g$} (m-2-4);
				\path[transform canvas={xshift=2em}]	(m-1-2) edge[cell] node[right] {$\phi$} (m-2-2);
				\draw				($(m-1-2)!0.5!(m-1-3)$) node[xshift=-1.5pt] {$\dotsb$};
			\end{tikzpicture} \quad = \quad \begin{tikzpicture}[textbaseline]
				\matrix(m)[math35]{A_0 & A_1 & A_{n'} & A_n \\ C & & & A_n \\ C & & & D \\};
				\path[map]	(m-1-1) edge[barred] node[above] {$J_1$} (m-1-2)
														edge node[left] {$f$} (m-2-1)
										(m-1-3) edge[barred] node[above] {$J_n$} (m-1-4)
										(m-2-1) edge[barred] node[below] {$K(\id, g)$} (m-2-4)
										(m-2-4) edge node[right] {$g$} (m-3-4)
										(m-3-1) edge[barred] node[below] {$K$} (m-3-4);
				\draw				([xshift=-1pt]$(m-1-2)!0.5!(m-1-3)$) node {$\dotsc$}
										([yshift=-0.333em]$(m-2-1)!0.5!(m-3-4)$) node[font=\scriptsize] {$\cart$};
				\path				(m-1-4) edge[eq] (m-2-4)
										(m-2-1) edge[eq] (m-3-1);
				\path[transform canvas={xshift=2em}]	(m-1-2) edge[cell] node[right] {$\phi'$} (m-2-2);
			\end{tikzpicture}
		\end{displaymath}
		Among the following conditions the implications \textup{(a)} $\Leftrightarrow$ \textup{(b)} $\Rightarrow$ \textup{(c)} $\Rightarrow$ \textup{(d)} $\Leftrightarrow$ \textup{(e)} hold. All conditions are equivalent whenever a right (respectively weakly) cocartesian cell with horizontal source $\ul J$ and vertical morphisms $f$ and $\id_{A_n}$ is known to exist, or when they are considered in a unital virtual double category (\defref{augmented virtual equipment}).
		\begin{enumerate}[label=\textup{(\alph*)}]
			\item $\phi'$ is right (respectively weakly) nullary"/cocartesian (\defref{cocartesian path});
			\item $\phi$ satisfies the (weak) left Beck"/Chevalley condition (\defref{left Beck-Chevalley condition});
			\item $\phi$ is (weakly) left $d$"/exact (\defref{left exact}) for any morphism $\map dCM$ whose (weak) left Kan extension along $K$ restricts along $g$ (\defref{pointwise left Kan extension});
			\item $\phi$ is (weakly) left $\yon$"/exact;
			\item $\phi'$ is right (respectively weakly) unary"/cocartesian.
		\end{enumerate}
		
		If moreover the restrictions $K(\id, g \of h)$ exist for all $\map hX{A_n}$ then the same implications hold among the pointwise variants (\defref{pointwise right cocartesian path}) of the conditions above, with the pointwise variant of \textup{(c)} being
		\begin{enumerate}
			\item[\textup{(c')}] $\phi$ is pointwise (weakly) left exact.
		\end{enumerate}
	\end{theorem}
	\begin{proof}
		(a) $\Leftrightarrow$ (b) by \defref{left Beck-Chevalley condition}. (b) $\Rightarrow$ (c) follows from \cororef{left Beck-Chevalley condition implies exactness}. Let $\cell{\cart_{\cur K}}K{\ps C}$ denote the cartesian cell that defines $\map{\cur K}D{\ps C}$ (\defref{yoneda embedding}). By (weak) density of $\yon$ (\defref{density definition}) it defines $\cur K$ as the pointwise (weak) left Kan extension of $\yon$ along $K$, proving (c) $\Rightarrow$ (d). Next consider the equality $\cart_{\cur K} \of \phi = \cart_{\cur K} \of \cart \of \phi'$ where both $\cart_{\cur K}$ and $\cart_{\cur K} \of \cart$ define pointwise left Kan extensions of $\yon$, the latter by \lemref{restrictions of pointwise left Kan extensions}. It follows that condition (d) is equivalent to $\phi'$ being (pointwise) (weakly) left $\yon$"/exact which, by \propref{left exactness and right unary-cocartesianness}, is equivalent to (e). Finally assume either that all horizontal units exist or that a (pointwise) right (respectively weakly) cocartesian cell with horizontal source $\ul J$ and vertical morphisms $f$ and $\id_{A_n}$ exists. In both cases (e)~$\Rightarrow$~(a) follows, either by \exref{cocartesian paths in the presence of horizontal units or (1,0)-ary cartesian cells} or by using that weakly unary"/cocartesian cells uniquely determine their horizontal targets up to isomorphism.
	\end{proof}
	
	\subsection{The left Beck"/Chevalley condition and absolute left Kan extension}
	The second theorem of this section describes the relation between absolutely left Kan cells (\defref{absolutely left Kan}) and the left Beck"/Chevalley condition for nullary cells. Specialising (c) $\Leftrightarrow$ (d) to $\enProf\Set$ (\augexref{2.4}) recovers Example~1.14(9) of \cite{Guitart80}; specialising (a) $\Leftrightarrow$ (b) to a locally thin equipment (see \augpropref{7.8} and \augexref{2.5}) recovers Theorem~2.6 of \cite{Koudenburg18}. When considered in a pseudo double category (\augpropref{7.8}) condition (f) below is equivalent to the nullary cell $\eta$ exhibiting $l$ as an `absolute left Kan extension' in the sense of Section~2.2 of \cite{Grandis-Pare08}. The latter notion is more general however as it allows for unary cells $\eta$ as well.
	\begin{theorem} \label{left Beck-Chevalley condition and absolute left Kan extensions}
		Consider the cell $\eta$ on the left"/hand side below.
		\begin{displaymath}
			\begin{tikzpicture}[textbaseline]
				\matrix(m)[math35, column sep={1.75em,between origins}]{A_0 & & A_n \\ & M & \\};
				\path[map]	(m-1-1) edge[barred] node[above] {$\ul J$} (m-1-3)
														edge[transform canvas={xshift=-2pt}] node[left] {$d$} (m-2-2)
										(m-1-3) edge[transform canvas={xshift=2pt}] node[right] {$l$} (m-2-2);
				\path[transform canvas={yshift=0.25em}]	(m-1-2) edge[cell] node[right, inner sep=3pt] {$\eta$} (m-2-2);
			\end{tikzpicture} \quad = \quad \begin{tikzpicture}[textbaseline]
  			\matrix(m)[math35, column sep={1.75em,between origins}]{A_0 & & A_n \\ M & & A_n \\ & M & \\};
  			\path[map]	(m-1-1) edge[barred] node[above] {$\ul J$} (m-1-3)
  													edge node[left] {$d$} (m-2-1)
  									(m-2-1) edge[barred] node[below] {$l^*$} (m-2-3)
  									(m-2-3)	edge[transform canvas={xshift=2pt}] node[right] {$l$} (m-3-2);
  			\path				(m-1-3) edge[eq] (m-2-3)
  									(m-2-1) edge[eq, transform canvas={xshift=-2pt}] (m-3-2);
  			\path[transform canvas={xshift=1.75em}]
        		        (m-1-1) edge[cell] node[right] {$\eta'$} (m-2-1);
        \draw       ([yshift=0.25em]$(m-2-2)!0.5!(m-3-2)$) node[font=\scriptsize] {$\cart$};
  		\end{tikzpicture}
		\end{displaymath}
		Among the conditions below the implication \textup{(a)} $\Rightarrow$ \textup{(d)} holds. If the conjoint $l^*$ exists then the implications \textup{(a)} $\Leftrightarrow$ \textup{(b)} $\Rightarrow$ \textup{(c)} $\Leftrightarrow$ \textup{(d)} hold. If moreover all cartesian nullary cells (\defref{cocartesian path of (0,1)-ary cells}) exist then all four conditions are equivalent.
		\begin{enumerate}[label=\textup{(\alph*)}]
			\item Any cell $\phi$ as on the left"/hand side below (with $\ul H$ empty) factors uniquely as shown;
			\item the factorisation $\eta'$ in the right"/hand side above is right (respectively weakly) cocartesian (\defref{cocartesian path});
			\item $\eta$ satisfies the (weak) left Beck"/Chevalley condition (\defref{left Beck-Chevalley condition});
			\item $\eta$ is absolutely (weakly) left Kan (\defref{absolutely left Kan}).
		\end{enumerate}
		\begin{displaymath}
			\begin{tikzpicture}[textbaseline]
				\matrix(m)[math35, column sep={0.875em,between origins}]
					{	A_0 & & & & A_n & & & & B_m \\
					 	& M & & & & & & & \\
					 	& & N & & & & K & & \\ };
				\path[map]	(m-1-1) edge[barred] node[above] {$\ul J$} (m-1-5)
														edge node[left] {$d$} (m-2-2)
										(m-1-5) edge[barred] node[above] {$\ul H$} (m-1-9)
										(m-1-9) edge node[right] {$h$} (m-3-7)
										(m-2-2) edge node[left] {$f$} (m-3-3)
										(m-3-3) edge[barred] node[below] {$\ul L$} (m-3-7);
				\path[transform canvas={yshift=-1.625em}]	(m-1-5) edge[cell] node[right] {$\phi$} (m-2-5);
			\end{tikzpicture} \quad = \quad \begin{tikzpicture}[textbaseline]
				\matrix(m)[math35, column sep={0.875em,between origins}]
					{	A_0 & & & & A_n & & & & B_m \\
					 	& & M & & & & & & \\
					 	& & & N & & & & K & \\ };
				\path[map]	(m-1-1) edge[barred] node[above] {$\ul J$} (m-1-5)
														edge[transform canvas={xshift=-2pt}] node[left] {$d$} (m-2-3)
										(m-1-5) edge[barred] node[above] {$\ul H$} (m-1-9)
														edge[transform canvas={xshift=2pt}] node[right] {$l$} (m-2-3)
										(m-1-9) edge[transform canvas={xshift=1pt}] node[right] {$h$} (m-3-8)
										(m-2-3) edge[transform canvas={xshift=-1pt}] node[left] {$f$} (m-3-4)
										(m-3-4) edge[barred] node[below] {$\ul L$} (m-3-8);
				\path[transform canvas={yshift=0.25em}]	(m-1-3) edge[cell] node[right, inner sep=3pt] {$\eta$} (m-2-3);
				\path[transform canvas={yshift=-1.625em}]	(m-1-6) edge[cell] node[right] {$\phi'$} (m-2-6);
			\end{tikzpicture}
		\end{displaymath}
		If the (weak) Yoneda morphism $\map\yon M{\ps M}$ exists then, among the conditions above and those below, the implications \textup{(a)} $\Rightarrow$ \textup{(d)} $\Rightarrow$ \textup{(e)} $\Rightarrow$ \textup{(f)} $\Rightarrow$ \textup{(g)} hold. If all companions exist then \textup{(g)} $\Rightarrow$ \textup{(d)}.
		\begin{enumerate}[label=\textup{(\alph*)}, resume]
			\item $\eta$ is a (weakly) left Kan cell that is preserved by $\yon$ (\defref{absolutely left Kan});
			\item every cell $\phi$ above with $f = \id_M$ (and with $\ul H$ empty) factors as shown;
			\item $\eta$ is a (weakly) left Kan cell that is preserved by any $\map gMN$ whose companion $g_*$ exists.
		\end{enumerate}
		
		If the conjoint $(l \of k)^*$ exists for all $\map kX{A_n}$ then the same implications hold for the analogous conditions in the pointwise (weak) case where, in \textup{(a)} and \textup{(f)}, the unique factorisations are through composites of the form $f \of \eta \of (\id_{J_1}, \dotsc, \id_{J_{n'}}, \cart)$, with $\cart$ defining any restriction of the form $J_n(\id, k)$. 
	\end{theorem}
	\begin{proof}
		For the weak case take $\ul H = (A_n)$ empty in the following. The implication (a)~$\Rightarrow$~(d) simply follows from the fact that restricting the unique factorisation above to cells $\phi$ that are nullary, that is $\ul L = (N)$ empty, gives the universal property of $\eta$ asserted by (d). (b) $\Rightarrow$ (c) immediately follows from \defsref{left Beck-Chevalley condition}{cocartesian path} while  (c)~$\Rightarrow$~(b) follows from \exref{cocartesian paths in the presence of horizontal units or (1,0)-ary cartesian cells} provided that all cartesian nullary cells exist. To see that (a)~$\Leftrightarrow$~(b) consider the assignments below, of cells of the form as shown; here $\cart$ denotes the cartesian cell defining the conjoint $l^*$. It follows from the conjoint identities (\lemref{companion identities lemma}) that the left assignment below is a bijection. Hence the assignment on the right below is a bijection, that is (b) holds, precisely if the composite assignment $(f \of \eta) \hc \dash$ below is so, that is (a) holds. This proves (a)~$\Leftrightarrow$~(b). That (c) $\Leftrightarrow$ (d) is similar: restricting the assignments below to nullary cells, that is $\ul L = (N)$ empty, it follows that the assignment on the right is a bijection, that is $\eta'$ is right (respectively weakly) nullary"/cocartesian or, equivalently, (c) holds, if and only if the composite is a bijection, that is (d) holds.
		\begin{displaymath}
			\begin{tikzpicture}
				\matrix(m)[math35]{M & B_m \\ N & K \\};
				\path[map]	(m-1-1) edge[barred] node[above] {$l^* \conc \ul H$} (m-1-2)
														edge node[left] {$f$} (m-2-1)
										(m-1-2) edge node[right] {$h$} (m-2-2)
										(m-2-1) edge[barred] node[below] {$\ul L$} (m-2-2);
				\path[transform canvas={xshift=1.75em}] (m-1-1) edge[cell] (m-2-1);
				
				\matrix(m)[math35, xshift=-15em]{A_n & B_m \\ M & \\ N & K \\};
				\path[map]	(m-1-1) edge[barred] node[above] {$\ul H$} (m-1-2)
														edge node[left] {$l$} (m-2-1)
										(m-1-2) edge node[right] {$h$} (m-3-2)
										(m-2-1) edge node[left] {$f$} (m-3-1)
										(m-3-1) edge[barred] node[below] {$\ul L$} (m-3-2);
				\path[transform canvas={xshift=1.75em,yshift=-1.625em}] (m-1-1) edge[cell] (m-2-1);

				\matrix(m)[math35, xshift=15em]{A_0 & B_m \\ M & \\ N & K \\};
				\path[map]	(m-1-1) edge[barred] node[above] {$\ul J \conc \ul H$} (m-1-2)
														edge node[left] {$d$} (m-2-1)
										(m-1-2) edge node[right] {$h$} (m-3-2)
										(m-2-1) edge node[left] {$f$} (m-3-1)
										(m-3-1) edge[barred] node[below] {$\ul L$} (m-3-2);
				\path[transform canvas={xshift=1.75em,yshift=-1.625em}] (m-1-1) edge[cell] (m-2-1);
				
				\draw[font=\Large]	(-17.9em, 0) node {$\lbrace$}
										(-12.2em,0) node {$\rbrace$}
										(-3.0em,0) node {$\lbrace$}
										(2.8em, 0) node {$\rbrace$}
										(12.1em,0) node {$\lbrace$}
										(17.8em,0) node {$\rbrace$};
				\path[map]	(-10.2em,0) edge node[above] {$(f \of \cart) \hc \dash$} (-5em,0)
										(4.8em,0) edge node[above] {$\dash \of (\eta' \conc \id_{\ul H})$} (10.1em,0);
			\end{tikzpicture}
		\end{displaymath}
		
		Next assume that the (weak) Yoneda morphism $\map\yon M{\ps M}$ exists. Clearly (d)~$\Rightarrow$~(e). To prove (e) $\Rightarrow$ (f) first notice that the required unique factorisation for unary cells $\cell\phi{\ul J \conc \ul H}L$, as in the statement and where $f = \id_M$, reduces to that of the corresponding nullary cells $\cell{\cart \of \phi}{\ul J \conc \ul H}{\ps M}$, with $f = \yon$ and where $\cart$ defines $\map{\cur L}K{\ps M}$ (\defref{yoneda embedding}); here we use the uniqueness of factorisations through $\cart$. Thus it suffices to prove that nullary cells of the form either $\ul J \conc \ul H \Rightarrow M$, with $f = \id_M$, or $\ul J \conc \ul H \Rightarrow \ps M$, with $f = \yon$, factorise uniquely as in the statement. But that is precisely the assertion that $\eta$ and $\yon \of \eta$ are (weakly) left Kan, which is condition~(e). To see that (f) $\Rightarrow$ (g) first notice that the universal property of $\eta$ being (weakly) left Kan is the unique factorisation of nullary cells $\ul J \conc \ul H \Rightarrow M$ as in the statement, with $f = \id_M$, which is asserted by (f). Next consider any $\map gMN$ whose companion $g_*$ exists. The unique factorisations that exhibit $g \of \eta$ as being left Kan factor precisely, through the cartesian cell $g_* \Rightarrow N$ that defines $g_*$, as the unique factorisations of unary cells $\ul J \conc \ul H \Rightarrow g_*$ as in the statement, with $f = \id_M$: these too exist by (f). This proves (f) $\Rightarrow$ (g). Clearly (g)~$\Rightarrow$~(d) if all companions exist.
		
		Notice that the pointwise (weak) variants of (a)~$\Rightarrow$~(d)~$\Rightarrow$~(e)~$\Rightarrow$~(f)~$\Rightarrow$~(g) follow from applying these implications to the composites $\eta \of (\id_{J_1}, \dotsc, \id_{J_{n'}}, \cart_{J_n(\id, k)})$, where $\cart_{J_n(\id, k)}$ defines any restriction of the form $J_n(\id, k)$. The pointwise (weak) variants of (b) $\Rightarrow$ (c) and (c) $\Rightarrow$ (b) follow from \defsref{left Beck-Chevalley condition}{pointwise right cocartesian path}. For the pointwise (weak) variants of (a) $\Leftrightarrow$ (b) and (c) $\Leftrightarrow$ (d) consider, for each \mbox{$\map kX{A_n}$} such that the restriction $J_n(\id, k)$ exists and instead of the two assignments above, the analogous assignments \mbox{$(f \of \cart \of \cart_{(l \of k)^*}) \hc \dash$}, where \mbox{$\cell{\cart_{(l \of k)^*}}{(l \of k)^*}{l^*}$} defines \mbox{$(l \of k)^*$} as a restriction of $l^*$ (see \auglemref{5.11}), and $\dash \of (\eta'' \conc \id_{\ul H})$, where $\eta''$ is the unique factorisation in
		\begin{displaymath}
			\eta' \of (\id_{J_1}, \dotsc, \id_{J_{n'}}, \cart_{J_n(\id, k)}) = \cart_{(l \of k)^*} \of \eta'',
		\end{displaymath}
		as in \defref{pointwise right cocartesian path}. Notice that the composite of these two assignments is
		\begin{align*}
			(f \of \cart \of \cart_{(l \of k)^*} \of \eta'') \hc \dash &= \bigpars{f \of \cart \of \eta' \of (\id_{J_1}, \dotsc, \id_{J_{n'}}, \cart_{J_n(\id, k)})} \hc \dash \\
				&= \bigpars{f \of \eta \of (\id_{J_1}, \dotsc, \id_{J_{n'}}, \cart_{J_n(\id, k)})} \hc \dash.
		\end{align*}
		Using arguments similar to those used above one then concludes that, for each morphism $\map kX{A_n}$ such that the restriction $J_n(\id, k)$ exists, the cell $\eta''$ is right (respectively weakly) nullary"/cocartesian if and only if unique factorisations through the composite \mbox{$f \of \eta \of (\id_{J_1}, \dotsc, \id_{J_{n'}}, \cart_{J_n(\id, k)})$} exist, analogous to the factorisations of condition (a). From this the pointwise (weak) analogue of (a) $\Leftrightarrow$ (b) follows and, as before, that of (c)~$\Leftrightarrow$~(d) follows by restricting the factorisations to nullary cells.
	\end{proof}
	
	\begin{example} \label{left Beck-Chevalley condition in ModRel}
		In the locally thin strict equipment $\ModRel$ of modular relations between preorders (see Example~1.9 of \cite{Koudenburg18}) consider an order preserving map $\map dAM$ and a modular relation $\hmap JAB$. The left Kan extension \mbox{$\map lBM$} of $d$ along $J$ in $\ModRel$, if it exists, is given by the suprema
		\begin{displaymath}
			ly = \sup_{x \in \rev Jy}dx \quad\qquad \text{where} \quad\qquad \rev Jy = \set{x \in A \mid xJy};
		\end{displaymath}
		see Example~2.3 of \cite{Koudenburg18}. Example~2.7 therein describes the left Beck"/Chevalley condition for the left Kan cell defining $l$: it requires that for every $y \in B$ there is $x \in \rev Jy$ with $dx = ly$, that is the suprema above are attained as maxima.
	\end{example}
	
	As an immediate corollary of the previous theorem we find that an absolute left Kan extension $l$ is preserved by any functor of augmented virtual double categories (\augdefref{3.1}) that preserves right cocartesian cells, provided that the conjoint of $l$ exists, as follows.
	\begin{corollary}
		Let $\map F\K\L$ be a functor of augmented double categories. In $\K$ consider the factorisation $\eta = \cart \of \eta'$ of \thmref{left Beck-Chevalley condition and absolute left Kan extensions} and assume that $\eta$ satisfies the (pointwise) (weak) left Beck"/Chevalley condition (\defref{left Beck-Chevalley condition}). The following are equivalent:
		\begin{enumerate}[label=\textup{(\alph*)}]
			\item $F$ preserves $\eta$, that is $F\eta$ satisfies the (pointwise) (weak) left Beck"/Chevalley condition;
			\item $F$ preserves $\eta'$, that is $F\eta'$ is (pointwise) right (respectively weakly) cocartesian.
		\end{enumerate}
	\end{corollary}
	\begin{proof}
		This follows from \defref{left Beck-Chevalley condition} and the fact that the image \mbox{$F\eta = F\cart \of F\eta'$} of $\eta = \cart \of \eta'$ under $F$ is again of the form as that of \thmref{left Beck-Chevalley condition and absolute left Kan extensions}. This is because functors of augmented virtual double categories preserve conjoints (\augcororef{5.5}) and vertical identity cells.
	\end{proof}

	\section{Totality} \label{totality section}
	In this section we study the notions of total morphism and total object in augmented virtual double categories, which are introduced in the definition below. It follows from \exsref{weighted colimits}{enriched left Kan extension} that, when applied to the unital virtual equipment $\enProf\V$ of profunctors enriched in a closed symmetric monoidal category $\V$ (\augexref{2.4}), the notion of total object below coincides with that of total $\V$"/category as considered in \cite{Day-Street86} and \cite{Kelly86}. In contrast to Street and Walters' original $2$"/categorical notion of total morphism given in Section~6 of \cite{Street-Walters78}, which assumes a Yoneda structure, the notion of (weak) total morphism below does not require the existence of (weak) Yoneda morphisms. 
	
	In order to make precise the relation between our notion of weakly total object below and that of total object in the sense of Street and Walters, consider an augmented virtual equipment $\K$ (\defref{augmented virtual equipment}) that has weak Yoneda embeddings $\map{\yon_A}A{\ps A}$ admitting nullary restrictions (\defref{yoneda embedding}) for all unital objects $A$. We will see in \exref{weak totality in an augmented virtual equipment with yoneda embeddings} below that a unital object $A$ is weakly total in the vertical $2$"/category $V(\K)$ (\augexref{1.5}) in the sense of \cite{Street-Walters78}, that is the weak Yoneda embedding $\yon_A$ admits a left adjoint, if and only if $A$ is weakly total in $\K$ in our sense.
	\pagebreak
	\subsection{Total morphisms and objects}
	\begin{definition} \label{totality}
		A morphism $\map fMN$ is called \emph{(weakly) total} if, for any horizontal morphism $\hmap JMB$, the pointwise (weak) left Kan extension of $f$ along $J$ (\defref{pointwise left Kan extension}) exists. An object $M$ is called \emph{(weakly) total} if its identity morphism $\id_M$ is (weakly) total. 
	\end{definition}
	
	\begin{example} \label{presheaf objects are total}
		In an augmented virtual equipment $\K$ (\defref{augmented virtual equipment}) any presheaf object $\ps M$, defined by a Yoneda morphism $\map\yon M{\ps M}$ (\defref{yoneda embedding}), is total provided that the companion $\yon_*$ exists. To see this let $\hmap J{\ps M}B$ be any horizontal morphism. Since the restriction $J(\yon, \id)$ exists in $\K$ so does the pointwise horizontal composite \mbox{$(\yon_* \hc J)$}, by \auglemsref{8.1}{9.7}. By \remref{right pointwise cocartesian and pointwise right cocartesian comparison} the cell defining $(\yon_* \hc J)$ is pointwise right cocartesian so that, by \exref{cocartesian paths are exact}, it is pointwise left exact. Hence the existence of the pointwise left Kan extension of $\id_{\ps M}$ along $J$ follows from \propref{left Kan extensions along a yoneda embedding in terms of left y-exact cells}.
		
		Notice that besides assuming that $\yon_*$ exists we do not need to require any size restriction on $\ps M$. This is in contrast to Corollary~14 of \cite{Street-Walters78} which, in order to prove that a presheaf object $\mathcal PA$, given by a Yoneda structure on a $2$"/category, is total, in the sense of \cite{Street-Walters78}, requires $\mathcal PA$ to be admissible.
	\end{example}
	
	The following result is analogous to the first assertion of Proposition~27 of \cite{Street-Walters78} for Yoneda structures.
	\begin{proposition} \label{totality and adjunctions}
		Consider an adjunction $f \ladj \map gMN$ with $g$ full and faithful (\defref{full and faithful morphism}). If $g$ is (weakly) total then so is $M$.
	\end{proposition}
	\begin{proof}
		The counit $\cell\eps{f \of g}{\id_M}$ of the adjunction is invertible because $g$ is full and faithful, by \auglemref{4.14} and (b) $\Rightarrow$ (c) of Proposition~10 of \cite{Street-Walters78}. We have to show that the pointwise (weak) left Kan extension of any $\hmap JMB$ along $\id_M$ exists. If $g$ is (weakly) total then the pointwise (weak) left Kan extension of $g$ along $J$ exists; let $\cell\eta JN$ denote its defining nullary cell. Using the fact that the left adjoint $f$ is cocontinuous (\lemref{left adjoints are cocontinuous}), $f \of \eta$ defines the pointwise (weak) left Kan extension of $f \of g$ along $J$ so that, composing on the left with the inverse $\cell{\inv\eps}{\id_M}{f \of g}$, we obtain a nullary cell defining the pointwise (weak) left Kan extension of $\id_M$ along $J$, as required.
	\end{proof}
	
	The following formalises Corollary~6.2 of \cite{Kelly86}, which is recovered by specialising to the unital virtual equipment $\enProf\V$ (\augexref{2.4}).
	\begin{proposition}
		Consider an adjunction $f \ladj \map gMN$ with $g$ full and faithful. Assume that the companion $\hmap{f_*}NM$ exists as well as all restrictions of the form $J(f, \id)$, for all $\hmap JMB$. If $N$ is (weakly) total then so are $g$ and $M$. 
	\end{proposition}
	\begin{proof}
		The counit $\cell\eps{f \of g}{\id_M}$ of the adjunction is invertible because $g$ is full and faithful, by \auglemref{4.14} and (b) $\Rightarrow$ (c) of Proposition~10 of \cite{Street-Walters78}. Notice that (weak) totality of $M$ follows from that of $g$ by the previous proposition. To show that $g$ is total we have to show that the pointwise (weak) left Kan extension of $g$ along any $\hmap JMB$ exists. To this end consider the composite below where $\eta$ defines $\map lBN$ as the pointwise (weak) left Kan extension of $\id_N$ along $J(f, \id)$, which exists because $N$ is (weakly) total, and where $\cocart$ defines $J(f, \id)$ as the pointwise horizontal composite of $(f_*, J)$; see \auglemsref{8.1}{9.7}.
		\begin{displaymath}
			\begin{tikzpicture}
				\matrix(m)[math35, column sep={1.75em,between origins}]
					{ N & & M & & B \\
						& N & & B & \\
						& & N & & \\  };
				\path[map]	(m-1-1) edge[barred] node[above] {$f_*$} (m-1-3)
										(m-1-3) edge[barred] node[above] {$J$} (m-1-5)
										(m-2-2) edge[barred] node[above, inner sep=2pt] {$J(f, \id)$} (m-2-4)
										(m-2-4) edge[transform canvas={xshift=2pt}] node[right] {$l$} (m-3-3);
				\path				(m-1-1) edge[eq, transform canvas={xshift=-1pt}] (m-2-2)
										(m-1-5) edge[eq, transform canvas={xshift=1pt}] (m-2-4)
										(m-2-2) edge[eq, transform canvas={xshift=-1pt}] (m-3-3)
										(m-2-3) edge[cell, transform canvas={yshift=0.25em}] node[right] {$\eta$} (m-3-3);
				\draw[font=\scriptsize]	($(m-1-3)!0.5!(m-2-3)$) node[yshift=0.25em] {$\cocart$};
			\end{tikzpicture}
		\end{displaymath}
		 Notice that for each morphism $\map hXB$, if the restriction $J(\id, h)$ exists then so does $J(f, \id)(\id, h) \iso J(\id, h)(f, \id)$ (\lemref{pasting lemma for cartesian cells}), by assumption. Hence $\cocart$ is pointwise right cocartesian by \remref{right pointwise cocartesian and pointwise right cocartesian comparison}, so that the composite above defines $l$ as the pointwise (weak) left Kan extension of $\id_N$ along $(f_*, J)$ by the vertical pasting lemma (\lemref{vertical pasting lemma}). By \exref{left Kan extension along the companion of a left adjoint} $l$ is the pointwise (weak) left Kan extension of $g$ along $J$, as required.
	\end{proof}
	
	The following `adjoint functor theorem' is an immediate corollary of \propref{adjunctions in terms of left Kan extensions}. The analogous result for Yoneda structures is mentioned on page~372 of \cite{Street-Walters78}.
	\begin{corollary}
		Consider a morphism $\map fMN$ whose companion $f_*$ exists and whose source $M$ is weakly total, so that the weak left Kan extension $\map lNM$ of $\id_M$ along $f_*$ exists. If $l$ is preserved by $f$ (\defref{absolutely left Kan}) then $f \ladj l$.
	\end{corollary}
	
	\subsection{Totality in the presence of Yoneda morphisms}
	The following theorem is the main result of this section. Its implication (c)~$\Rightarrow$~(d) is analogous to Lemma~3.18 of \cite{Weber07} for Yoneda structures. The latter requires $\ps M$ to be admissible; here only the existence of the companion $\hmap{\yon_*}M\ps M$ is assumed.
	\begin{theorem} \label{total morphisms}
		For a morphism $\map fMN$ the following are equivalent:
		\begin{enumerate}[label=\textup{(\alph*)}]
			\item $f$ is (weakly) total;
			\item if a pointwise (weakly) left $f$"/exact cell (\defref{left exact}) of the form below exists then so does the pointwise (weak) left Kan extension of $f \of d$ along $J$ (\defref{pointwise left Kan extension}).
				\begin{displaymath}
					\begin{tikzpicture}
						\matrix(m)[math35]{A & B \\ M & B \\};
						\path[map]	(m-1-1) edge[barred] node[above] {$J$} (m-1-2)
																edge node[left] {$d$} (m-2-1)
												(m-2-1) edge[barred] node[below] {$K$} (m-2-2);
						\path				(m-1-2) edge[eq] (m-2-2);
						\path[transform canvas={xshift=1.75em}]	(m-1-1) edge[cell] node[right] {$\phi$} (m-2-1);
					\end{tikzpicture}
				\end{displaymath}
		\end{enumerate}
		If the (weak) Yoneda morphism $\map\yon M{\ps M}$ (\defref{yoneda embedding}) as well as the companions $\yon_*$ and $f_*$ exist then the conditions above and those below are all equivalent.
		\begin{enumerate}[label=\textup{(\alph*)}, resume]
			\item The weak left Kan extension of $f$ along $\yon_*$ exists and restricts along \mbox{$\map{\cur f_*}N{\ps M}$} (\defthreeref{weak left Kan extension}{pointwise left Kan extension}{yoneda embedding});
			\item $\cur f_*$ admits a left adjoint $\map z{\ps M}N$.
		\end{enumerate}
		Moreover in that case the weak left Kan extension of condition \textup{(c)} and the left adjoint $z$ of condition \textup{(d)} coincide.
	\end{theorem}
	\begin{proof}
		(a) $\Rightarrow$ (b) follows from applying the \defref{left exact} to the composites as on the left below, where the cell $\eta$ defines $\map lBN$ as the pointwise left Kan extension of $f$ along $K$. (b) $\Rightarrow$ (a) simply follows from applying (b) to the identity cell $\id_J$. (a) $\Rightarrow$ (c) is clear, but notice that the restriction $\yon_*(\id, \cur f_*)$ exists and is equal to $f_*$. Indeed this follows from the fact that the unique factorisation $\cell{\cart'}{f_*}{y_*}$, of the cartesian cell defining $\cur f_*$ (\defref{yoneda embedding}) through that defining $y_*$, is cartesian by the pasting lemma (\lemref{pasting lemma for cartesian cells}).
		
		(d) $\Rightarrow$ (a). Consider the composite in the middle below, of the cartesian cell that defines $\cur f_*$ and the counit $\eps$ of $z \ladj \cur f_*$. It is cartesian by \auglemref{4.17}, showing that $(z \of \yon)_* \iso f_*$. Since taking companions extends to an equivalence $(\dash)_*$ of augmented virtual double categories (\augthmref{6.5}) we conclude that $z \of \yon \iso f$. The pointwise (weak) left Kan extension of $f$ along any $\hmap JMB$ can now be constructed as $z \of \cur J$, with the composite on the right below as the defining pointwise (weakly) left Kan cell. Here we use the fact that the left adjoint $z$ is cocontinuous; see \propref{left adjoints are cocontinuous}.
		\begin{displaymath}
			\begin{tikzpicture}[textbaseline]
  			\matrix(m)[math35, column sep={1.75em,between origins}]{A & & B \\ M & & B \\ & N & \\};
  			\path[map]	(m-1-1) edge[barred] node[above] {$J$} (m-1-3)
  													edge node[left] {$d$} (m-2-1)
  									(m-2-1) edge[barred] node[below, inner sep=1.5pt] {$K$} (m-2-3)
  													edge[transform canvas={xshift=-2pt}] node[left] {$f$} (m-3-2)
  									(m-2-3)	edge[transform canvas={xshift=2pt}] node[right] {$l$} (m-3-2);
  			\path				(m-1-3) edge[eq] (m-2-3)
  									(m-2-2) edge[cell] node[right] {$\eta$} (m-3-2)
  									(m-1-2) edge[cell] node[right] {$\phi$} (m-2-2);
  		\end{tikzpicture} \qquad\qquad\qquad\qquad \begin{tikzpicture}[textbaseline]
				\matrix(m)[math35, column sep={1.75em,between origins}]{M & & N \\ & \ps M & \\ & & N\\};
				\path[map]	(m-1-1) edge[barred] node[above] {$f_*$} (m-1-3)
														edge[ps, transform canvas={xshift=-1pt}] node[left] {$\yon$} (m-2-2)
										(m-1-3) edge[ps, transform canvas={xshift=1pt}] node[right, inner sep=1pt] {$\cur f_*$} (m-2-2)
										(m-2-2) edge[bend right = 18] node[left] {$z$} (m-3-3);
				\path				(m-1-3) edge[eq, bend left = 42] (m-3-3)
														edge[cell, transform canvas={yshift=-1.625em}] node[right] {$\eps$} (m-2-3);
				\draw				([yshift=0.333em]$(m-1-2)!0.5!(m-2-2)$) node[font=\scriptsize] {$\cart$};
			\end{tikzpicture}	\qquad\qquad\qquad\qquad \begin{tikzpicture}[textbaseline]
				\matrix(m)[math35, column sep={1.75em,between origins}]{M & & B \\ & \ps M & \\ N & & \\};
				\path[map]	(m-1-1) edge[barred] node[above] {$J$} (m-1-3)
														edge[ps, transform canvas={xshift=-1pt}] node[left] {$\yon$} (m-2-2)
														edge[bend right = 45] node[left] {$f$} (m-3-1)
										(m-1-3) edge[ps, transform canvas={xshift=1pt}] node[right] {$\cur J$} (m-2-2)
										(m-2-2) edge[bend left = 18] node[right] {$z$} (m-3-1);
				\draw				([yshift=0.333em]$(m-1-2)!0.5!(m-2-2)$) node[font=\scriptsize] {$\cart$}
										(m-2-1) node {$\iso$};
			\end{tikzpicture}
		\end{displaymath}
		
		(c) $\Rightarrow$ (d). Let $\map z{\ps M}N$ denote the weak left Kan extension of $f$ along $\yon_*$ and let $\eta$ denote its defining cell, as in the left-hand side of the identity on the left below; we claim that $z$ forms the left adjoint of $\cur f_*$.
		\begin{displaymath}
			\begin{tikzpicture}[textbaseline]
				\matrix(m)[math35, column sep={1.75em,between origins}]
					{ & M & & \ps M \\
						M & & N & \\
						& \ps M & & \\ };
				\path[map]	(m-1-2) edge[transform canvas={xshift=1pt}] node[left, yshift=1pt] {$f$} (m-2-3)
														edge[barred] node[above] {$\yon_*$} (m-1-4)
										(m-1-4) edge[transform canvas={xshift=2pt}] node[right] {$z$} (m-2-3)
										(m-2-3) edge[transform canvas={xshift=2pt}, ps] node[right] {$\cur f_*$} (m-3-2)
										(m-2-1)	edge[barred] node[below, inner sep=1.5pt] {$f_*$} (m-2-3)
														edge[transform canvas={xshift=-2pt}, ps] node[left] {$\yon$} (m-3-2);
				\path				(m-1-3) edge[cell, transform canvas={yshift=0.25em}] node[right] {$\eta$} (m-2-3)
										(m-1-2) edge[eq, transform canvas={xshift=-2pt}] (m-2-1);
    		\draw				([yshift=0.25em]$(m-2-2)!0.5!(m-3-2)$) node[font=\scriptsize] {$\cart$}
    								([yshift=-0.55em]$(m-1-2)!0.5!(m-2-2)$) node[font=\scriptsize] {$\cocart$};
			\end{tikzpicture} \mspace{9mu} = \mspace{9mu} \begin{tikzpicture}[textbaseline]
				\matrix(m)[math35, column sep={1.75em,between origins}]{M & & \ps M & \\ & \ps M & & N \\ & & \ps M & \\ };
				\path[map]	(m-1-1) edge[barred] node[above] {$\yon_*$} (m-1-3)
														edge[ps, transform canvas={xshift=-1pt}] node[left] {$\yon$} (m-2-2)
										(m-1-3) edge[bend left=18] node[above right] {$z$} (m-2-4)
										(m-2-4) edge[bend left=18] node[below right] {$\cur f_*$} (m-3-3);
				\path				(m-1-3) edge[eq, ps, transform canvas={xshift=1pt}] (m-2-2)
										(m-2-2) edge[eq, bend right=18] (m-3-3)
										(m-1-3) edge[cell, transform canvas={yshift=-1.625em}] node[right] {$\iota$} (m-2-3);
				\draw				([yshift=0.25em]$(m-1-2)!0.5!(m-2-2)$) node[font=\scriptsize] {$\cart$};
			\end{tikzpicture} \qquad\qquad \begin{tikzpicture}[textbaseline]
					\matrix(m)[math35, column sep={1.75em,between origins}]{M & & N \\ & N & \\};
					\path[map]	(m-1-1) edge[barred] node[above] {$f_*$} (m-1-3)
															edge[transform canvas={xshift=-2pt}] node[left] {$f$} (m-2-2);
					\path				(m-1-3) edge[eq, transform canvas={xshift=2pt}] (m-2-2);
					\draw				([yshift=0.333em]$(m-1-2)!0.5!(m-2-2)$) node[font=\scriptsize] {$\cart$};
				\end{tikzpicture} \mspace{9mu} = \mspace{9mu} \begin{tikzpicture}[textbaseline]
  			\matrix(m)[math35]{M & N \\ M & \ps M \\ & N \\};
  			\path[map]	(m-1-1) edge[barred] node[above] {$f_*$} (m-1-2)
  									(m-1-2)	edge[ps] node[right, inner sep=1pt] {$\cur f_*$} (m-2-2)
  									(m-2-1) edge[barred] node[below, inner sep=1.5pt] {$y_*$} (m-2-2)
  													edge[transform canvas={xshift=-2pt}] node[below left, inner sep=1pt] {$f$} (m-3-2)
  									(m-2-2)	edge node[right, inner sep=1pt] {$z$} (m-3-2);
  			\path				(m-1-1) edge[eq] (m-2-1)
  									(m-1-2.east) edge[eq, bend left=65] (m-3-2.east)
  									(m-2-1) edge[cell, transform canvas={xshift=2.5em,yshift=0.25em}] node[right, inner sep=2pt] {$\eta$} (m-3-1)
  									(m-1-2) edge[cell, transform canvas={xshift=1.25em,yshift=-1.625em}] node[right] {$\eps$} (m-2-2);
				\path				($(m-1-1)!0.5!(m-2-2)$) node[font=\scriptsize] {$\cart'$};
  		\end{tikzpicture}
		\end{displaymath}
		As the unit cell $\cell\iota{\id_{\ps M}}{\cur f_* \of z}$ we take the unique factorisation as in the identity on the left above; recall that the cartesian cell defining $\yon_*$ is weakly left Kan because $\yon$ is weakly dense (\defref{density definition}). The counit $\cell\eps{z \of \cur f_*}{\id_N}$ we take to the be unique factorisation in the identity on the right above, where $\cart'$ is the cartesian cell defining $f_*$ as the restriction $\yon_*(\id, \cur f_*)$, as described previously, so that $\eta \of \cart'$ is weakly left Kan because $\eta$ restricts along $\cur f_*$ by assumption.
		
		To prove the triangle identity $(\iota \of \cur f_*) \hc (\cur f_* \of \eps) = \id_{\cur f_*}$ consider the equality of composite cells that are drawn schematically below where, for ease of drawing, vertical and nullary cells are drawn as rectangles. Here the cartesian cells defining $\cur f_*$, $\yon_*$ and $f_*$ are denoted by `c', the cocartesian cells corresponding to the latter two cartesian cells by `cc', and the cartesian cell $\cell{\cart'}{f_*}{\yon_*}$ by `$\textup c'$'. The identities follow from the definitions of $\cart'$, $\iota$ and $\eps$, as well as the horizontal companion identity for $f_*$ (\lemref{companion identities lemma}). Notice that the left and right"/hand side below are the two sides of the triangle identity composed with the cartesian cell that defines $\cur f_*$. Since the latter is weakly left Kan, by the weak density of $\yon$, and because factorisations through weakly left Kan cells are unique, the triangle identity itself follows.
		\begin{displaymath}
			\begin{tikzpicture}[scheme]
  			\draw	(2,2) -- (0,2) -- (0,3) -- (1,3) -- (1,0) -- (2,0) -- (2,3) -- (3,3) -- (3,1) -- (2,1);
  			\draw	(0.5,2.5) node {c}
  						(1.5,1) node {$\iota$}
  						(2.5,2) node {$\eps$};
  		\end{tikzpicture} \quad = \quad \begin{tikzpicture}[scheme]
  			\draw	(1,0) -- (1,3) -- (0,3) -- (0,0) -- (2,0) -- (2,3) -- (3,3) -- (3,1) -- (2,1) (0,2) -- (2,2);
  			\draw	(0.5,2.5) node {$\textup c'$}
  						(0.5,1) node {c}
  						(1.5,1) node {$\iota$}
  						(2.5,2) node {$\eps$};
  		\end{tikzpicture} \quad = \quad \begin{tikzpicture}[scheme]
  			\draw	(2,2) -- (0,2) -- (0,0) -- (1,0) -- (1,3) -- (3,3) -- (3,1) -- (0,1) (2,1) -- (2,3);
  			\draw	(0.5,1.5) node {cc}
  						(0.5,0.5) node {c}
  						(1.5,2.5) node {$\textup c'$}
  						(1.5,1.5) node {$\eta$}
  						(2.5,2) node {$\eps$};
  		\end{tikzpicture} \quad = \quad \begin{tikzpicture}[scheme, yshift=0.8em]
  			\draw	(1,2) -- (1,0) -- (0,0) -- (0,2) -- (2,2) -- (2,1) -- (0,1);
  			\draw	(0.5,1.5) node {cc}
  						(0.5,0.5) node {c}
  						(1.5,1.5) node {c};
  		\end{tikzpicture} \quad = \quad \begin{tikzpicture}[scheme, yshift=1.6em]
  			\draw	(1,0) -- (1,1) -- (0,1) -- (0,0) -- (1,0);
  			\draw	(0.5,0.5) node {c};
  		\end{tikzpicture}
		\end{displaymath}
		The schematically drawn equality below likewise shows that the sides of the second triangle identity $(z \of \iota) \hc (\eps \of z) = \id_z$ coincide after composition on the left with $\eta$. Since $\eta$ is weakly left Kan the triangle identity itself follows.
		\begin{displaymath}
			\begin{tikzpicture}[scheme]
  			\draw	(2,1) -- (0,1) -- (0,0) -- (1,0) -- (1,3) -- (2,3) -- (2,0) -- (3,0) -- (3,2) -- (2,2);
  			\draw	(0.5,0.5) node {$\eta$}
  						(1.5,2) node {$\iota$}
  						(2.5,1) node {$\eps$};
  		\end{tikzpicture} \mspace{13mu} = \mspace{13mu} \begin{tikzpicture}[scheme]
  			\draw	(1,3) -- (1,0) -- (0,0) -- (0,3) -- (3,3) -- (3,0) -- (4,0) -- (4,2) -- (3,2) (0,1) -- (3,1) (2,3) -- (2,1);
  			\draw	(0.5,2) node {cc}
  						(0.5,0.5) node {$\eta$}
  						(1.5,2) node {c}
  						(2.5,2) node {$\iota$}
  						(3.5,1) node {$\eps$};
  		\end{tikzpicture} \mspace{13mu} = \mspace{13mu} \begin{tikzpicture}[scheme]
  			\draw	(3,2) -- (0,2) -- (0,0) -- (1,0) -- (1,3) -- (3,3) -- (3,0) -- (2,0) -- (2,3) (0,1) -- (2,1);
  			\draw	(0.5,1.5) node {cc}
  						(0.5,0.5) node {$\eta$}
  						(1.5,2.5) node {cc}
  						(1.5,1.5) node {c}
  						(2.5,2.5) node {$\eta$}
  						(2.5,1) node {$\eps$};
  		\end{tikzpicture} \mspace{13mu} = \mspace{13mu} \begin{tikzpicture}[scheme]
  			\draw	(1,3) -- (1,0) -- (0,0) -- (0,3) -- (1,3) -- (2,3) -- (2,0) -- (1,0) (0,1) -- (1,1) (0,2) -- (2,2);
  			\draw	(0.5,2.5) node {cc}
  						(0.5,1.5) node {$\textup c'$}
  						(0.5,0.5) node {$\eta$}
  						(1.5,2.5) node {$\eta$}
  						(1.5,1) node {$\eps$};
  		\end{tikzpicture} \mspace{13mu} = \mspace{13mu} \begin{tikzpicture}[scheme, yshift=0.8em]
  			\draw	(1,2) -- (1,0) -- (0,0) -- (0,2) -- (2,2) -- (2,1) -- (0,1);
  			\draw	(0.5,1.5) node {cc}
  						(0.5,0.5) node {c}
  						(1.5,1.5) node {$\eta$};
  		\end{tikzpicture} \mspace{13mu} = \mspace{13mu} \begin{tikzpicture}[scheme, yshift=1.6em]
  			\draw	(1,0) -- (1,1) -- (0,1) -- (0,0) -- (1,0);
  			\draw	(0.5,0.5) node {$\eta$};
  		\end{tikzpicture}
		\end{displaymath}
		The third identity above follows from the claim that $\cocart_{\yon_*} \hc \cart_{\cur f_*} = \cart'$, where the subscripts in the left"/hand side denote the companions defined by the cocartesian and cartesian cells. Indeed notice that the two sides of the latter identity coincide after composing them with the cartesian cell corresponding to $\cocart_{\yon_*}$, by the vertical companion identity for $\yon_*$ and the definition of $\cart'$. The claim then follows from the uniqueness of factorisations through cartesian cells. The other identities above follow from the horizontal companion identity for $\yon_*$, the definition of $\iota$, that of $\eps$, and the vertical companion identity for $f_*$. This completes the proof.
	\end{proof}
	
	\begin{example} \label{weak totality in an augmented virtual equipment with yoneda embeddings}
		In an augmented virtual equipment $\K$ (\defref{augmented virtual equipment}) consider chosen weak Yoneda embeddings $\map{\yon_A}A{\ps A}$ that admit nullary restrictions (\defref{yoneda embedding}), one for each unital object $A$. By \exref{yoneda structure from weak yoneda embeddings} they induce a Yoneda structure on $V(\K)$ whose admissible morphisms are those morphisms admitting companions. Fixing a unital object $M$ notice that $\cur{(\id_M)_*} \iso \cur{I_M} \iso \yon_M$ by \propref{equivalence from yoneda embedding} and \lemref{full and faithful yoneda embedding}. Together with (a) $\Leftrightarrow$ (d) above it follows that $M$ is total in the sense of Section~6 of \cite{Street-Walters78}, that is $\yon_M$ admits a left adjoint, if and only if $M$ is weakly total in our sense of \defref{totality}.
		
		Similarly an admissible morphism $\map fMN$, with $M$ unital, is total in the sense of \cite{Street-Walters78} if and only if it is weakly total in our sense such that the left adjoint $\map z{\ps M}N$ of (d) above admits a companion.
	\end{example}
	Applying the previous theorem to an identity morphism gives the following corollary. Specialising (a) $\Leftrightarrow$ (b) below to a $\V$"/category $M$ in the augmented virtual equipment $\enProf{(\V, \V')}$ (\augexref{2.7}) recovers Theorem~5.2 of \cite{Kelly86}, while (a)~$\Leftrightarrow$~(c) is similar to its Theorem~5.3.
	\begin{corollary} \label{total objects}
		Consider the following conditions for a (weak) Yoneda embedding \mbox{$\map\yon M{\ps M}$} (\defref{yoneda embedding}). The implication \textup{(b)} $\Rightarrow$ \textup{(a)} holds. If the companion $\yon_*$ and the horizontal unit $I_M$ exist then all conditions are equivalent. In either case the counit $\cell\eps{z \of \yon}{\id_M}$ of the adjunction of condition \textup{(b)} is invertible.
		\begin{enumerate}[label=\textup{(\alph*)}]
			\item $M$ is (weakly) total;
			\item $\yon$ admits a left adjoint $\map z{\ps M}M$;
			\item if a pointwise (weakly) left $\id_M$"/exact (\defref{left exact}) of the form below exists then so does the pointwise (weak) left Kan extension of $d$ along $J$ (\defref{pointwise left Kan extension}).
		\end{enumerate}
		\begin{displaymath}
					\begin{tikzpicture}
						\matrix(m)[math35]{A & B \\ M & B \\};
						\path[map]	(m-1-1) edge[barred] node[above] {$J$} (m-1-2)
																edge node[left] {$d$} (m-2-1)
												(m-2-1) edge[barred] node[below] {$K$} (m-2-2);
						\path				(m-1-2) edge[eq] (m-2-2);
						\path[transform canvas={xshift=1.75em}]	(m-1-1) edge[cell] node[right] {$\phi$} (m-2-1);
					\end{tikzpicture}
				\end{displaymath}
	\end{corollary}
	\begin{proof}
		The proof of (b) $\Rightarrow$ (a) is similar to that of the implication (d)~$\Rightarrow$~(a) of the previous theorem, as follows. Write $\map z{\ps M}M$ for the left adjoint and \mbox{$\cell\eps{z \of \yon}{\id_M}$} for the counit of the adjunction. That $\yon$ is full and faithful implies that $\eps$ is invertible, by \auglemref{4.14} and (b) $\Rightarrow$ (c) of Proposition~10 of \cite{Street-Walters78}. Using the fact that $z$ is (weakly) cocontinuous (\propref{left adjoints are cocontinuous}) we can now show that $M$ is (weakly) total: for any $\hmap JMB$ the pointwise (weak) left Kan extension of $\id_M$ along $J$ is $z \of \cur J$, defined as such by the composite of the inverse of $\eps$ with the cartesian cell that defines $\cur J$ (\defref{yoneda embedding}), the latter of which is pointwise (weakly) left Kan by (weak) density of $\yon$ (\defref{density definition}).
		
		If the companion $\yon_*$ and the horizontal unit $I_M$ exist then the equivalence of the conditions above follows from applying the previous theorem to $f = \id_M$: remember that $(\id_M)_* \iso I_M$ is the horizontal unit of $M$ so that $\cur{(\id_M)_*} \iso \cur{I_M} \iso \yon$ by \propref{equivalence from yoneda embedding} and \lemref{full and faithful yoneda embedding}.
	\end{proof}
	
	The following is analogous to the first assertion of Proposition~25 of \cite{Street-Walters78} for Yoneda structures.
	\begin{corollary}
		Let $\map fMN$ be a morphism and let $\map\yon N{\ps N}$ be a (weak) Yoneda embedding such that the companion $\yon_*$ exists and $N$ is unital. If $N$ and $\yon \of f$ are (weakly) total then so is $f$.
	\end{corollary}
	\cororef{left adjoint to ps f} below gives conditions ensuring that $\yon \of f$ is (weakly) total.
	\begin{proof}
		By the previous corollary $\yon$ has a left adjoint $\map z{\ps N}N$ with invertible counit $\cell\eps{z \of \yon}{\id_N}$. To prove that $f$ is (weakly) total we have to show that the pointwise (weak) left Kan extension of $f$ along any $\hmap JMB$ exists. If $\yon \of f$ is (weakly) total then the pointwise (weakly) left Kan extension of $\yon \of f$ along $J$ exists; let $\cell\eta J{\ps N}$ denote its defining nullary cell. Using that the left adjoint $z$ is cocontinuous (\propref{left adjoints are cocontinuous}), it follows that the composite $(\inv\eps \of f) \hc (z \of \eta)$ defines the pointwise (weakly) left Kan extension of $f$ along $J$, as required.
	\end{proof}
	
	\subsection{Formal restriction of presheaves}
	Given a morphism $\map fAC$, the following definition introduces the induced morphism $\map{\ps f}{\ps C}{\ps A}$ that formalises the restriction of presheaves along $f$; this is analogous to the definition of `$\map{\mathcal Pf}{\mathcal PC}{\mathcal PA}$' of Section~2 of \cite{Street-Walters78} and that of `$\map{\textup{res}_f}{\ps C}{\ps A}$' of Section~3 of \cite{Weber07}, for (good) Yoneda structures.
	\begin{definition} \label{restriction}
		Consider weak Yoneda morphisms $\map{\yon_A}A\ps A$ and $\map{\yon_C}C\ps C$ and let $\map fAC$ be a morphism such that the companion $\hmap{(\yon_C \of f)_*}A\ps C$ exists. We set $\ps f \dfn \cur{(\yon_C \of f)_*}$, as defined by the cartesian cell below (\defref{yoneda embedding}).
		\begin{displaymath}
			\begin{tikzpicture}
					\matrix(m)[math35, column sep={2em,between origins}]{A & & \ps C \\ & \ps A & \\};
					\path[map]	(m-1-1) edge[barred] node[above] {$(\yon_C \of f)_*$} (m-1-3)
															edge[transform canvas={xshift=-2pt}] node[left] {$\yon_A$} (m-2-2)
											(m-1-3) edge[transform canvas={xshift=2pt}] node[right] {$\ps f$} (m-2-2);
					\draw				([yshift=0.333em]$(m-1-2)!0.5!(m-2-2)$) node[font=\scriptsize] {$\cart$};
				\end{tikzpicture} 
		\end{displaymath}
	\end{definition}
	
	Recall that $(\yon_C \of f)_* \iso \yon_{C*}(f, \id)$ by \auglemref{5.11}, so that in an augmented virtual double category with restrictions on the left (\defref{augmented virtual equipment}) all companions $(\yon_C \of f)_*$ exist, for any $\map fAC$, as soon as the companion $\yon_{C*}$ exists. Recall that the cell above defines $\ps f$ uniquely up to isomorphism (\defref{yoneda embedding}).
	\begin{example} \label{restriction of enriched presheaves}
		As in \exref{enriched yoneda embedding summary} consider $\V$"/categories $A$ and $C$ as well as their $\V'$"/enriched presheaf categories $\brks{\op A, \V}'$ and $\brks{\op C, \V}'$. For any $\V$-functor \mbox{$\map fAC$} the $\V'$-functor $\map{\ps f}{\brks{\op C, \V}'}{\brks{\op A, \V}'}$ exists in $\enProf{(\V, \V')}$ and can be taken to be given by precomposition with $f$, that is $\ps f(p) = p \of \op f$ for every $\V$-presheaf \mbox{$p \in \brks{\op C, \V}'$}.
	\end{example}
	
	Proposition~12 of \cite{Street-Walters78}, for Yoneda structures, is analogous to the second assertion below.
	\begin{proposition} \label{restriction and curry}
		Let $\map{\yon_A}A{\ps A}$ be a weak Yoneda morphism and $\map{\yon_C}C{\ps C}$ a (weak) Yoneda morphism. Consider morphisms $\map fAC$ and $\hmap KCD$ as well as $\map{\cur K}D{\ps C}$, as defined by the cartesian cell below (\defref{yoneda embedding}). If the companions $\yon_{C*}$, $f_*$ and $(\yon_C \of f)_*$ exist, as well as the restriction $\ps C(\yon_C \of f, \cur K)$, then $\ps f \of \cur K \iso \cur{K(f, \id)}$.
		
		If moreover $\yon_A$ is a Yoneda morphism and $C$ is unital (\defref{cartesian cells}) then $\map{\ps f}{\ps C}{\ps A}$ preserves the (weak) left Kan cell below (\defsref{density definition}{absolutely left Kan}). If the restrictions $\ps C(\yon_C, g)$ and $\ps C(\yon_C \of f, g)$ exist, for all $\map gX{\ps C}$, then $\ps f$ preserves the pointwise (weak) left Kan cell below.
		\begin{displaymath}
			\begin{tikzpicture}
				\matrix(m)[math35, column sep={1.75em,between origins}]{C & & D \\ & \ps C & \\};
				\path[map]	(m-1-1) edge[barred] node[above] {$K$} (m-1-3)
														edge[transform canvas={xshift=-2pt}] node[left] {$\yon_C$} (m-2-2)
										(m-1-3) edge[transform canvas={xshift=2pt}] node[right] {$\cur K$} (m-2-2);
				\draw				([yshift=0.333em]$(m-1-2)!0.5!(m-2-2)$) node[font=\scriptsize] {$\cart$};
			\end{tikzpicture}
	  \end{displaymath}
	\end{proposition}
	Applying the first assertion above to the companion $K = h_*$ of a morphism $\map hCE$ we obtain $\ps f \of \cur{h_*} \iso \cur{(h \of f)_*}$ (use \auglemref{5.11}). Analogous isomorphisms are required to exist in a Yoneda structure; see Axiom~3 of \cite{Street-Walters78}. If $C$ is unital, so that $\cur I_C \iso \yon_C$ by \lemref{full and faithful yoneda embedding}, then taking $K = I_C$ gives  $\ps f \of \yon_C \iso \cur{f_*}$ (use \augcororef{4.16}).
	\begin{proof}
		Write $\cell{\cart'}K{\yon_{C*}}$ for the factorisation of the cartesian cell defining $\cur K$ above through the cartesian cell defining $\yon_{C*}$, as in the left"/hand side below; $\cart'$ is cartesian by the pasting lemma (\lemref{pasting lemma for cartesian cells}). Next notice that $(\yon_C \of f)_* \iso \yon_{C*}(f, \id)$ (\auglemref{5.11}) so that, by (\auglemsref{8.1}{9.7}), there exists a pointwise horizontal cocartesian cell $\cell\cocart{(f_*, \yon_{C*})}{(\yon_C \of f)_*}$. Since the restrictions $\yon_{C*}(\id, \cur K) \iso K$ and \mbox{$(\yon_C \of f)_*(\id, \cur K) \iso \ps C(\yon_C \of f, \cur K)$} (\lemref{pasting lemma for cartesian cells}) exist, it follows from \augdefref{9.1} that $\cocart$ is a right cocartesian cell that restricts along $\cur K$ (\defref{pointwise right cocartesian path}). Hence by \cororef{left Kan extensions of yoneda embeddings along paths} there exists a weakly left Kan cell $\eta$ as in the composites below that restricts along $\cur K$, so that the composite on the left"/hand side is weakly left Kan too. Similarly, by factorising the latter composite through the cocartesian cell $(f_*, K) \Rar K(f, \id)$ (\auglemref{8.1}) we obtain a nullary cell $K(f, \id) \Rar \ps A$ that defines $\ps f \of \cur K$ as the weak left Kan extension of $\yon_A$ along $K(f, \id)$, by the vertical pasting lemma (\lemref{vertical pasting lemma}). Since $\cur{K(f, \id)}$ too is the weak left Kan extension of $\yon_A$ along $K(f, \id)$ (\defsref{density definition}{yoneda embedding}), the assertion $\ps f \of \cur K \iso \cur{K(f, \id)}$ follows from the uniqueness of Kan extensions.
		\begin{displaymath}
			\begin{tikzpicture}[textbaseline]
				\matrix(m)[math35, column sep={2em,between origins}]
					{ A & & C & & D \\
						A & & C & & \ps C \\
						& & \ps A & & \\ };
				\path[map]	(m-1-1) edge[barred] node[above] {$f_*$} (m-1-3)
										(m-1-3) edge[barred] node[above] {$K$} (m-1-5)
										(m-1-5) edge[ps] node[right] {$\cur K$} (m-2-5)
										(m-2-1) edge[barred] node[above] {$f_*$} (m-2-3)
														edge[transform canvas={xshift=-2pt}] node[below left] {$\yon_A$} (m-3-3)
										(m-2-3) edge[barred] node[below, inner sep=2.5pt] {$\yon_{C*}$} (m-2-5)
										(m-2-5)	edge[transform canvas={xshift=2pt}] node[below right] {$\ps f$} (m-3-3);
				\path				(m-1-1) edge[eq] (m-2-1)
										(m-1-3) edge[eq] (m-2-3)
										(m-2-3) edge[cell] node[right] {$\eta$} (m-3-3);
				\draw[font=\scriptsize]	($(m-1-3)!0.5!(m-2-5)$) node {$\cart'$};
			\end{tikzpicture} \quad = \quad \begin{tikzpicture}[textbaseline]
				\matrix(m)[math35, column sep={2em,between origins}]
					{ & A & & C & & D \\
						A & & C & & \ps C & \\
						& & \ps A & & & \\ };
								\path[map]	(m-1-2) edge[barred] node[above] {$f_*$} (m-1-4)
										(m-1-4) edge[barred] node[above] {$K$} (m-1-6)
														edge node[left, inner sep=1pt, yshift=1pt] {$\yon_C$} (m-2-5)
										(m-1-6) edge[transform canvas={xshift=2pt}] node[right] {$\cur K$} (m-2-5)
										(m-2-1) edge[barred] node[above] {$f_*$} (m-2-3)
														edge[transform canvas={xshift=-2pt}] node[below left] {$\yon_A$} (m-3-3)
										(m-2-3) edge[barred] node[below, inner sep=2.5pt] {$\yon_{C*}$} (m-2-5)
										(m-2-5)	edge[transform canvas={xshift=2pt}] node[below right] {$\ps f$} (m-3-3);
				\path				(m-1-2) edge[eq, transform canvas={xshift=-1pt}] (m-2-1)
										(m-1-4) edge[eq, transform canvas={xshift=-1pt}] (m-2-3)
										(m-2-3) edge[cell] node[right] {$\eta$} (m-3-3);
				\draw[font=\scriptsize]	($(m-1-4)!0.7!(m-2-4)$) node {$\cocart$}
																([yshift=0.333em]$(m-1-5)!0.5!(m-2-5)$) node {$\cart$};
			\end{tikzpicture}
		\end{displaymath}
		
		Next assume that $\yon_A$ is a Yoneda morphism so that, by the same argument as above, the cell $\eta$  and the left"/hand side above are left Kan. Consider the right"/hand side above, obtained by substituting $\cart' = \cocart \hc \cart$, which follows immediately from the definition of $\cart'$ and the vertical companion identity for $\yon_{C*}$ (\lemref{companion identities lemma}). To prove that $\ps f \of \cart$ is again left Kan it suffices by the horizontal pasting lemma (\lemref{horizontal pasting lemma}) to show that the composite $f_* \Rar \ps A$ of the first column of the right"/hand side above is left Kan. Assuming that $C$ is unital, so that $\yon_C$ is full and faithful and the restriction $\ps C(\yon_C, \yon_C)$ exists by \lemref{full and faithful yoneda embedding}, this follows from \propref{pointwise left Kan extension along full and faithful map}. Finally if for all $\map gX{\ps C}$ the restrictions $\ps C(\yon_C, g)$ and $\ps C(\yon_C \of f, g)$ exist then $\cell\cocart{(f_*, \yon_{C*})}{(\yon_C \of f)_*}$ is pointwise right cocartesian by \remref{right pointwise cocartesian and pointwise right cocartesian comparison}, so that $\eta$ and the left"/hand side above are pointwise left Kan by \cororef{left Kan extensions of yoneda embeddings along paths} and \lemref{restrictions of pointwise left Kan extensions}. Applying the horizontal pasting lemma we conclude that $\ps f \of \cart$ is pointwise left Kan as well. This completes the proof.
	\end{proof}
	
	The previous proposition can be used to describe the uniqueness of (weak) Yoneda embeddings, as follows. In an augmented virtual double category $\K$ consider weak Yoneda embeddings $\map\yon AP$ and $\map{\yon'}A{P'}$ that admit nullary restrictions (\defref{yoneda embedding}), for the same object $A$. Given any horizontal morphism $\hmap JAB$ we denote by $\map{\cur J}BP$ and $\map{J^{\lambda'}}B{P'}$ the morphisms associated to $J$ by the Yoneda axiom (\defref{yoneda embedding}) for $\yon$ and $\yon'$ respectively. Consider the morphisms $\map{(\yon'_*)^\lambda}{P'}P$ and \mbox{$\map{(\yon_*)^{\lambda'}}P{P'}$}.
	\begin{corollary} \label{uniqueness of yoneda embeddings}
		The morphisms $(\yon'_*)^\lambda$ and $(\yon_*)^{\lambda'}$ defined above form an equivalence $P \simeq P'$ in $V(\K)$. Moreover the pointwise weakly left Kan cartesian cells that define the morphisms $\map{J^{\lambda'}}B{P'}$, as defined above, are preserved by postcomposition with $\map{(\yon'_*)^\lambda}{P'}P$ (\defref{absolutely left Kan}). Finally $\yon$ is dense (\defref{density definition}) if and only if $\yon'$ is so.
	\end{corollary}	
	\begin{proof}
		Using weak density of $\yon$ and $\yon'$ (\defref{density definition}) the cartesian cells defining $(\yon'_*)^\lambda$ and $(\yon_*)^{\lambda'}$ are pointwise weakly left Kan so that, applying \propref{pointwise left Kan extension along full and faithful map}, we obtain isomorphisms \mbox{$(\yon'_*)^\lambda \of \yon' \iso \yon$} and \mbox{$(\yon_*)^{\lambda'} \of \yon \iso \yon'$}. Next notice that $\map{\ps{\id_A} = (\yon'_*)^\lambda}{P'}P$ by \defref{restriction} so that $(\yon'_*)^\lambda \of (\yon_*)^{\lambda'} = \ps{\id_A} \of (\yon_*)^{\lambda'} \iso (\yon_*)^\lambda \iso \id_P$ where the first isomorphism follows from the proposition above; for the second one see \defref{yoneda embedding}. By symmetry $(\yon_*)^{\lambda'} \of (\yon'_*)^{\lambda} \iso \id_{P'}$ too, and we conclude that $(\yon'_*)^\lambda$ and $(\yon_*)^{\lambda'}$ form an equivalence $P \simeq P'$ in $V(\K)$.
		
		To prove the second assertion, for any $\hmap JAB$ we denote by $\cell{\eta_J}J{P'}$ the pointwise weakly left Kan cartesian cell that defines $\map{J^{\lambda'}}B{P'}$ (\defsref{density definition}{yoneda embedding}). Using that both $(\yon'_*)^\lambda$ and $(\yon_*)^{\lambda'}$ are left adjoints (see e.g.\ Proposition~1.5.7 of \cite{Leinster04}) and \propref{left adjoints are cocontinuous} it follows that $(\yon'_*)^\lambda \of \eta_J$ is pointwise weakly left Kan. Composing $(\yon'_*)^\lambda \of \eta_J$ with $\yon \iso (\yon'_*)^\lambda \of \yon'$ and using \lemref{weak left Kan extensions of yoneda embeddings} we find that $(\yon'_*)^\lambda \of \eta_J$ is cartesian too, proving the second assertion. Finally notice that if $\yon'$ is dense (\defref{density definition}) then the latter composite defines $(\yon'_*)^\lambda \of J^{\lambda'}$ as the pointwise left Kan extension of $\yon$ along $J$, by combining \propref{left adjoints are cocontinuous}, \lemref{vertical cells defining left Kan extensions} the horizontal pasting lemma (\lemref{horizontal pasting lemma}). Using the uniqueness of left Kan extensions we conclude that $\yon$ is dense too.
	\end{proof}
	
	\subsection{Adjoints of $\ps f$}
	The remaining two results of this section ensure the existence of left and right adjoints to the induced morphism $\map{\ps f}{\ps C}{\ps A}$ of \defref{restriction}. In both cases the existence of the adjoint is a consequence of a related morphism being (weakly) total; in particular both results are corollaries of \thmref{total morphisms}. The first of these, below, ensures the existence of the right adjoint. It is analogous to Proposition~13 of \cite{Street-Walters78} for Yoneda structures.
	\begin{corollary} \label{right adjoint to ps f}
		Let $\map{\yon_A}A{\ps A}$ be a Yoneda morphism, $\map{\yon_C}C{\ps C}$ a (weak) Yoneda morphism and $\map fAC$ a morphism. If the companions $\yon_{C*}$, $f_*$, $(\yon_C \of f)_*$ and $(\cur f_*)_*$ exist as well as the restriction $\ps C\bigpars{\yon_C \of f, \cur{(\cur f_*)}_*}$ then $\map{\cur f_*}C{\ps A}$ is (weakly) total and $\ps f \ladj \map{\cur{(\cur f_*)}_*}{\ps A}{\ps C}$.
	\end{corollary}
	\begin{proof}
		Using (c) $\Rightarrow$ (a) of \thmref{total morphisms} applied to $\cur f_*$ it suffices to construct the weak left Kan extension of $\cur f_*$ along $\yon_{C*}$ so that it restricts along $\map{\cur{(\cur f_*)}_*}{\ps A}{\ps C}$. To do so apply the proof of the first assertion of \propref{restriction and curry} to $\map fAC$ and $\hmap{K \dfn (\cur f_*)_*}C{\ps A}$, thus obtaining a weakly left Kan cell $\eta$ as on the left"/hand side below that restricts along $\cur K = \cur{(\cur f_*)}_*$.
		\begin{displaymath}
			\begin{tikzpicture}[textbaseline]
				\matrix(m)[math35, column sep={1.8em,between origins}]
					{ A & & C & & \ps C \\
						& & \ps A & & \\};
				\path[map]	(m-1-1) edge[barred] node[above] {$f_*$} (m-1-3)
														edge[transform canvas={xshift=-2pt}] node[below left] {$\yon_A$} (m-2-3)
										(m-1-3) edge[barred] node[above] {$\yon_{C*}$} (m-1-5)
										(m-1-5) edge[transform canvas={xshift=2pt}] node[below right, inner sep=1pt] {$\ps f$} (m-2-3);
				\path				(m-1-3) edge[cell] node[right] {$\eta$} (m-2-3);
			\end{tikzpicture} \quad = \quad \begin{tikzpicture}[textbaseline]
				\matrix(m)[math35, column sep={4em,between origins}]{A & C & \ps C \\ & \ps A & \\};
				\path[map]	(m-1-1) edge[barred] node[above] {$f_*$} (m-1-2)
														edge[transform canvas={xshift=-2pt}] node[below left] {$\yon_A$} node[sloped, above, inner sep=6.5pt, font=\scriptsize, pos=0.5] {$\cart$} (m-2-2)
										(m-1-2) edge[barred] node[above] {$\yon_{C*}$} (m-1-3)
														edge[ps] node[left, inner sep=0.3pt, yshift=2pt] {$\cur f_*$} (m-2-2)
										(m-1-3) edge[transform canvas={xshift=1pt}] node[below right, inner sep=1pt] {$\ps f$} (m-2-2);
				\path[transform canvas={xshift=1.1em, yshift=3pt}]	(m-1-2) edge[cell] node[right] {$\zeta$} (m-2-2);
			\end{tikzpicture}
		\end{displaymath}
		Next consider the unique factorisation $\zeta$ in the right"/hand side above, through the cartesian cell that defines $\cur f_*$, which is left Kan by the density of $\yon_A$ (\defref{density definition}). Applying the horizontal pasting lemma (\lemref{horizontal pasting lemma}) we find that $\zeta$ defines $\ps f$ as the weak left Kan extension of $\cur f_*$ along $\yon_{C*}$, which restricts along $\cur{(\cur{f_*})}_*$ as required.
	\end{proof}

	The following example is a variation of Corollary~14 of \cite{Street-Walters78}.	
	\begin{example}
		Let $\map{\yon_M}M{\ps M}$ and $\map{\yon_{\ps M}}{\ps M}{\psps M}$ be Yoneda morphisms. Assume that $\ps M$ is unital and that the companions $\yon_{M*}$, $\yon_{\ps M*}$ and $(\yon_{\ps M} \of \yon_M)_*$ exist, as well as the restriction $\psps M(\yon_{\ps M} \of \yon_M, \yon_{\ps M})$. Using that $\cur{\yon_{M*}} \iso \id_{\ps M}$ (\defref{yoneda embedding}), so that \mbox{$\cur{(\cur{\yon_{M*}})}_* \iso \cur{I_{\ps M}} \iso \yon_{\ps M}$} (\lemref{full and faithful yoneda embedding}), it follows from the previous corollary that $\ps{\yon_M} \ladj \yon_{\ps M}$. As an alternative to \exref{presheaf objects are total}, by \cororef{total objects} it now follows that $\ps M$ is total.
	\end{example}
	
	The following result ensures the existence of a left adjoint to $\map{\ps f}{\ps C}{\ps A}$; its implication (b) $\Rightarrow$ (a) is analogous to Theorem~3.20(2) of \cite{Weber07} for good Yoneda structures.
	\begin{corollary} \label{left adjoint to ps f}
		Let $\map{\yon_A}A{\ps A}$ be a (weak) Yoneda morphism and let \mbox{$\map{\yon_C}C{\ps C}$} be a weak Yoneda morphism. Given a morphism $\map fAC$ assume that the pointwise weak left Kan extension $\map{f^\sharp}{\ps A}{\ps C}$ of \mbox{$\yon_C \of f$} along $\yon_{A*}$ exists (\defref{pointwise left Kan extension}). Consider the conditions below. If \textup{(b)} holds then $\yon_C \of f$ is (weakly) total and \mbox{$f^\sharp \ladj \map{\ps f}{\ps C}{\ps A}$}, so that \textup{(a)} holds too.
		
		If $\yon_A$ admits nullary restrictions (\defref{yoneda embedding}) and is full and faithful (\defref{full and faithful morphism}) then \textup{(a)} $\Rightarrow$ \textup{(b)}; if moreover all restrictions on the right exist (\defref{augmented virtual equipment}) then \textup{(b)} $\Rightarrow$ \textup{(c)} as well. If the horizontal composites $(J \hc \yon_{C*})$ (\defref{pointwise right cocartesian path}) exist for all $\hmap JAC$ then \textup{(c)} $\Rightarrow$ \textup{(b)}.
		\begin{enumerate}[label=\textup{(\alph*)}]
			\item $f^\sharp$ admits a right adjoint;
			\item the companion $\hmap{(\yon_C \of f)_*}A{\ps C}$ exists;
			\item the companion $\hmap{f_*}AC$ exists.
		\end{enumerate}
	\end{corollary}
	\begin{proof}
		Assuming that the companion $(\yon_C \of f)_*$ exists it follows from (c) $\Rightarrow$ (d) of \thmref{total morphisms} that $f^\sharp$ has $\ps f \dfn \cur{(\yon_C \of f)_*}$ as a right adjoint, while $\yon_C \of f$ is (weakly) total by (c)~$\Rightarrow$~(a) of the same theorem. This proves the first assertion.
		
		To prove (a) $\Rightarrow$ (b) denote the right adjoint to $f^\sharp$ by $\map r{\ps C}{\ps A}$; we claim that the restriction $\hmap{\ps A(\yon_A, r)}A{\ps C}$, which exists because $y_A$ admits nullary restrictions, forms the companion of $y_C \of f$. To see this consider the composite below where $\eps$ is the counit of $f^\sharp \ladj r$; it is cartesian by \auglemref{4.17}. Because $f^\sharp$ is the pointwise weak Kan extension along the companion of the Yoneda embedding $\yon_A$ we have $\yon_C \of f \iso f^\sharp \of \yon_A$ by \propref{pointwise left Kan extension along full and faithful map}. Composing the composite with this isomorphism we obtain a cartesian cell that defines $\ps A(\yon_A, r)$ as the companion of $\yon_C \of f$.
		\begin{displaymath}
			\begin{tikzpicture}
				\matrix(m)[math35, column sep={1.75em,between origins}]{A & & \ps C & \\ & \ps A & & \\ & & \ps C & \\ };
				\path[map]	(m-1-1) edge[barred] node[above] {$\ps A(\yon_A, r)$} (m-1-3)
														edge[ps, transform canvas={xshift=-1pt}] node[left] {$\yon_A$} (m-2-2)
										(m-1-3) edge[ps, transform canvas={xshift=1pt}] node[right] {$r$} (m-2-2)
										(m-2-2) edge[bend right=18] node[below left] {$f^\sharp$} (m-3-3);
				\path				(m-1-3) edge[eq, ps, bend left=45] (m-3-3)
										(m-1-3) edge[cell, transform canvas={yshift=-1.625em}] node[right] {$\eps$} (m-2-3);
				\draw				([yshift=0.25em]$(m-1-2)!0.5!(m-2-2)$) node[font=\scriptsize] {$\cart$};
			\end{tikzpicture}
		\end{displaymath}
		
		If all restrictions on the right exist then the existence of $(\yon_C \of f)_*$ implies that of $(\yon_C \of f)_*(\id, \yon_C) \iso \ps C(\yon_C \of f, \yon_C) \iso f_*$, where the isomorphisms follow from \lemref{pasting lemma for cartesian cells} and \auglemref{5.12}; this proves (b) $\Rightarrow$ (c). Using \auglemsref{8.1}{5.11} the converse follows from $(f_* \hc \yon_{C*}) \iso \yon_{C*}(f, \id) \iso (\yon_C \of f)_*$.
	\end{proof}
	
	\begin{example}
		Let $\V$ be a closed symmetric monoidal category that is small complete and small cocomplete. In the pseudo double category $\ensProf\V$ of small $\V$"/profunctors (\augexsref{2.8}{9.3}) consider a $\V$"/functor $\map fAC$ as well as the Yoneda embeddings $\map{\yon_A}A{\brks{\op A, \V}_\textup s}$ and $\map{\yon_C}C{\brks{\op C, \V}_\textup s}$, as in \exref{yoneda embedding for small enriched profunctors}. The $\V$"/functor \mbox{$\map{f^\sharp}{\brks{\op A, \V}_\textup s}{\brks{\op C, \V}_\textup s}$} of the corollary exists and is given by left Kan extending small $\V$"/presheaves on $A$ along $f$: to see this notice that, using \exref{cocartesian paths are exact} and \propref{left Kan extensions along a yoneda embedding in terms of left y-exact cells}, $f^\sharp$ corresponds to the pointwise horizontal composite $\bigpars{\inhom{\op C, \V}_\textup s(\yon_C, \yon_C \of f) \hc \yon_{A*}} \iso (f^* \hc \yon_{A*})$ (\auglemref{5.12}). All assumptions of the corollary are satisfied and we conclude that $f^\sharp$ has a right adjoint if and only if the companion $\hmap{f_*}AC$ exists in $\ensProf\V$, that is $f_*$ is a small $\V$"/profunctor (\augexref{2.8}). This recovers Proposition~3.3 of \cite{Day-Lack07}.
	\end{example}
	
	\section{Cocompleteness} \label{cocompleteness section}
	In this section we study the sense in which a Yoneda embedding $\map\yon M{\ps M}$ (\defref{yoneda embedding}) exhibits the presheaf object $\ps M$ as the free `small' cocompletion of $M$. As described at the end of the \overref\ and as is clear from \exrref{enriched free cocompletion}{left diagrams for lifted algebraic yoneda embeddings} below, the right notion of smallness here depends on the augmented virtual double category under consideration; it is defined in terms of `left diagrams' as follows. Recall that every augmented virtual double category $\K$ contains a vertical $2$"/category $V(\K)$ (\augexref{1.5}).
	
	\subsection{Cocompleteness, cocontinuity and free cocompletion}
	\begin{definition} \label{cocompletion}
		Let $\K$ be an augmented virtual double category. By a \emph{left diagram} in $\K$ we mean a span of the form $M \xlar d A \xbrar J B$. A collection $\mathcal S$ of left diagrams is called an \emph{ideal} if $(f \of d, J) \in \catvar S$ for all $(d, J) \in \catvar S$ and $f \in \K$ composable with $d$; given such an ideal and an object $N \in \K$ we write $\mathcal S(N) \subset \mathcal S$ for the subcollection of spans of the form $N \xlar d A \xbrar J B$. We make the following definitions.
		\begin{enumerate}[label=-]
			\item An object $N$ is called \emph{$\mathcal S$-cocomplete} if for any $(d, J) \in \mathcal S(N)$ the pointwise left Kan extension of $d$ along $J$ exists (\defref{pointwise left Kan extension});
			\item a morphism $\map fMN$ is called \emph{$\mathcal S$-cocontinuous} if, for any $(d, J) \in \mathcal S(M)$ and any pointwise left Kan cell $\eta$ that defines the pointwise left Kan extension of $d$ along $J$, the composite $f \of \eta$ is again pointwise left Kan;
			\item a morphism $\map wM{\bar M}$ is said to define $\bar M$ as the \emph{free $\mathcal S$-cocompletion} of $M$ if $\bar M$ is $\mathcal S$"/cocomplete and, for any $\mathcal S$-cocomplete $N$, the composite
			\begin{displaymath}
				V_\textup{$\mathcal S$-cocts}(\K)(\bar M, N) \hookrightarrow V(\K)(\bar M, N) \xrar{V(\K)(w, N)} V(\K)(M, N)
			\end{displaymath}
			is an equivalence, where $V_\textup{$\mathcal S$-cocts}(\K) \subseteq V(\K)$ denotes the locally full sub-$2$-category of $\mathcal S$"/cocontinuous morphisms.
		\end{enumerate}
	\end{definition}
	
	The following result describes the relation between cocompleteness and totality; it is a direct consequence of the definitions involved.
	\begin{corollary} \label{totality and cocompleteness}
		Let $\mathcal S$ be an ideal of left diagrams and $M$ an object. If $(\id_M, J) \in \mathcal S$ for all $\hmap JMB$ then $M$ is total (\defref{totality}) whenever it is $\mathcal S$"/cocomplete. The converse holds if, for each $(d, J) \in \mathcal S(M)$, there exists a pointwise left $\id_M$"/exact cell (\defref{left exact})
		\begin{displaymath}
			\begin{tikzpicture}
				\matrix(m)[math35]{A & B \\ M & B. \\};
				\path[map]	(m-1-1) edge[barred] node[above] {$J$} (m-1-2)
														edge node[left] {$d$} (m-2-1)
										(m-2-1) edge[barred] node[below] {$K$} (m-2-2);
				\path				(m-1-2) edge[eq] (m-2-2);
				\path[transform canvas={xshift=1.75em}]	(m-1-1) edge[cell] node[right] {$\phi$} (m-2-1);
			\end{tikzpicture}
		\end{displaymath}
	\end{corollary}
	
	The proof of the following result is similar to that of \propref{totality and adjunctions}.
	\begin{proposition}
		Let $\mathcal S$ be an ideal of left diagrams and $f \ladj \map gMN$ an adjunction with $g$ full and faithful (\defref{full and faithful morphism}). If $N$ is $\mathcal S$"/cocomplete then so is $M$.
	\end{proposition}
	
	\subsection{Presheaf objects as free cocompletions}
	The following theorem is the main result of this section. Using the notion of left exactness (\defref{left exact}) it gives conditions ensuring that a Yoneda embedding $\map\yon M{\ps M}$ defines $\ps M$ as the free $\mathcal S$"/cocompletion of $M$. Recall that, for condition~(e) below to be satisfied, it suffices that the cell $\phi$ defines $K$ as the pointwise right composite $\bigpars{\ps M(\yon, d) \hc J}$ (\defref{pointwise right cocartesian path}); see \exref{cocartesian paths are exact}.
	\begin{theorem} \label{presheaf objects as free cocompletions}
		Let $\map\yon M{\ps M}$ be a Yoneda embedding in an augmented virtual double category $\K$ and let $\mathcal S$ be an ideal of left diagrams in $\K$. Assume that $\yon$ admits nullary restrictions (\defref{yoneda embedding}). The presheaf object $\ps M$ is $\mathcal S$"/cocomplete if and only if
		\begin{enumerate}
			\item[\textup{(e)}] for every $(d, J) \in \mathcal S(\ps M)$ there exists a pointwise left $\yon$-exact cell
				\begin{displaymath}
					\begin{tikzpicture}
						\matrix(m)[math35]{ M & A & B \\ M & & B \\ };
						\path[map]	(m-1-1) edge[barred] node[above] {$\ps M(\yon, d)$} (m-1-2)
												(m-1-2) edge[barred] node[above] {$J$} (m-1-3)
												(m-2-1) edge[barred] node[below] {$K$} (m-2-3);
						\path				(m-1-1) edge[eq] (m-2-1)
												(m-1-3) edge[eq] (m-2-3)
												(m-1-2) edge[cell] node[right] {$\phi$} (m-2-2);
					\end{tikzpicture}
				\end{displaymath}
		\end{enumerate}
		and, in that case, $\yon$ defines $\ps M$ as the free $\mathcal S$"/cocompletion of $M$ if moreover 
		\begin{enumerate}
			\item[\textup{(y)}] $(f, \yon_*) \in \mathcal S$ for all $\map fMN$.
		\end{enumerate}
	\end{theorem}
	It follows from \lemref{left exactness and right unary-cocartesianness} that if $\K$ is an augmented virtual equipment (\defref{augmented virtual equipment}) then pointwise left $\yon$"/exactness of the cell $\phi$ in condition (e) is equivalent to pointwise right unary"/cocartesianness (\defref{pointwise right cocartesian path}). We think of condition (y) as ``restricting the size of the object $M$'', as explained after \cororef{left Kan extension, exact cells, cocartesian cells}.
	\begin{proof}
		That condition (e) is equivalent to $\ps M$ being $\mathcal S$"/cocomplete follows from \propref{left Kan extensions along a yoneda embedding in terms of left y-exact cells}. That, for any $\mathcal S$-cocomplete $N$, precomposition with $\yon$ induces an equivalence $V_\textup{$\mathcal S$-cocts}(\K)(\ps M, N) \simeq V(\K)(M, N)$ is shown in \lemref{cocompletion equivalence} below.
	\end{proof}
	
	Before proving the main lemma used in its proof, in the remark below we compare \thmref{presheaf objects as free cocompletions} to an analogous result for Yoneda structures and, in \exrref{enriched free cocompletion}{left diagrams for lifted algebraic yoneda embeddings}, describe some of its applications.
	\begin{remark} \label{differences to Weber's result}
		The above theorem is similar to Theorem~3.20(1) of \cite{Weber07} which, given a small object $C$ in a $2$"/category $\mathcal C$ equipped with a good Yoneda structure, asserts the existence of equivalences $\mathcal C_\textup{cocts}(\ps C, X) \simeq \mathcal C(C, X)$ given by precomposition with the Yoneda embedding $\map{y_C}C{\ps C}$, one for each admissible and cocomplete object $X$. The latter result differs from the theorem above in the following ways.
		\begin{enumerate}[label=-]
			\item It assumes that $C$ is small, that is both $C$ and $\ps C$ are admissible; the size restrictions that we require are that $\map\yon M{\ps M}$ admits nullary restrictions and that it is full and faithful, which imply that $M$ is unital by \lemref{full and faithful yoneda embedding}.
			\item It proves the existence of the equivalences only for cocomplete objects $X$ that are admissible. In contrast our notion of free $\mathcal S$"/cocompletion (\defref{cocompletion}) does not restrict the size of the $\mathcal S$"/cocomplete objects $N$.
			\item It assumes that the presheaf object $\ps C$ itself is cocomplete. In contrast, the theorem above asserts that $\mathcal S$"/cocompleteness of $\ps M$ is equivalent to its condition (e). In some cases the latter is trivially satisfied, e.g.\ when $\K$ is a pseudo double category with all restrictions on the right (\defref{augmented virtual equipment}); see \exref{all left diagrams} below.
		\end{enumerate}
	\end{remark}
	
	\subsection{Examples of free cocompletions}	
	\begin{example} \label{enriched free cocompletion}
		In the augmented virtual equipment $\enProf{(\V, \V')}$ of $\V$"/profunctors between $\V'$"/categories (\augexref{2.7}) consider the ideal of left diagrams
		\begin{displaymath}
			\mathcal S = \set{(d, J) \mid \textup{$\hmap JAB$ is a $\V$-profunctor with $A$ a small $\V$-category}}.
		\end{displaymath}
		Assuming that $\V \subset \V'$ is symmetric in the sense of \exref{enriched yoneda embedding summary}, consider the Yoneda embedding $\map\yon M{\brks{\op M, \V}'}$ for a $\V$"/category $M$ as described there. If moreover $\V$ is small cocomplete then the pointwise composites $\ps M(\yon, d) \hc J$ exist in $\enProf{(\V, \V')}$ for all $(d, J) \in \mathcal S(\ps M)$ by \augexref{9.2}, so that condition (e) above is satisfied by \remref{right pointwise cocartesian and pointwise right cocartesian comparison}. Condition (y) is satisfied if moreover $M$ is small, so that in that case $\yon$ defines the $\V'$"/category $\brks{\op M, \V}'$ of $\V$"/presheaves on $M$ as the free $\mathcal S$"/cocompletion of $M$ in $\enProf{(\V, \V')}$. 
		
		Next assume that $\V$ is closed symmetric monoidal and both small complete and small cocomplete, and that the embedding $\V \subset \V'$ is a closed symmetric monoidal functor (\augexref{2.7}). From \exsref{weighted colimits}{enriched left Kan extension} it follows that for $\V$-categories $\mathcal S$"/cocompleteness coincides with the classical notion of small cocompleteness, the latter in the sense of e.g.\ of Section 3.2 of \cite{Kelly82}. The theorem in this case recovers the fact that, for a small $\V$-category $M$, the $\V$-category of $\V$"/presheaves on $M$ forms the free small cocompletion of $M$; see Theorem 4.51 of \cite{Kelly82}.
	\end{example}
	
	\begin{example} \label{all left diagrams}
		In any augmented virtual double category $\K$ consider the ideal $\mathcal A$ of all left diagrams below. By \cororef{totality and cocompleteness} $\mathcal A$"/cocompleteness implies totality.
		\begin{displaymath}
			\mathcal A = \set{(d, J) \mid \textup{$(d, J)$ is any left diagram}}
		\end{displaymath}
		If $\K$ is a pseudo double category (\augpropref{7.8}) that has restrictions on the right (\defref{augmented virtual equipment}) then any Yoneda embedding $\map\yon M{\ps M}$ in $\K$ satisfies the conditions of the theorem above, and thus defines $\ps M$ as the free $\mathcal A$"/cocompletion of $M$, as long as its companion $\yon_*$ exists. Indeed condition (e) follows from the fact that $\K$ has all pointwise right composites (\remref{right pointwise cocartesian and pointwise right cocartesian comparison}). Moreover by combining \cororef{totality and cocompleteness} and \cororef{total objects} we find that the following are equivalent: $M$ is $\mathcal A$"/cocomplete; $M$ is total (\defref{totality}); $\map\yon M{\ps M}$ admits a left adjoint. Indeed we can use \augcororef{8.5}, \auglemref{9.8} and \remref{right pointwise cocartesian and pointwise right cocartesian comparison} to construct pointwise right cocartesian cells $\phi$ of the form as in \cororef{totality and cocompleteness}.
	\end{example}
	
	\begin{example} \label{small enriched free cocompletion}
		In the unital virtual double category $\ensProf\V$ of small $\V$"/profunctors (\augexref{2.8}) consider the ideal $\mathcal A$ of all left diagrams as in the previous example. We claim that if $\V$ is closed symmetric monoidal then $\mathcal A$"/cocompleteness coincides with the classical notion of small cocompleteness for $\V$"/categories, in the sense of e.g.\ of Section~3.2 of \cite{Kelly82}, as we will show below. If moreover $\V$ is small complete then for any (possibly large) $\V$"/category $M$ the Yoneda embedding $\map\yon M{\brks{\op M, \V}_\textup s}$ exists in $\ensProf\V$, with $\brks{\op M, \V}_\textup s$ the $\V$"/category of small $\V$"/presheaves on $M$; see \exref{yoneda embedding for small enriched profunctors}. If $\V$ is also small cocomplete, so that $\ensProf\V$ is a pseudo double category (\augexref{9.3}), then by the previous example $\yon$ defines $\brks{\op M, \V}_\textup s$ as the free $\mathcal A$"/cocompletion of $M$; this recovers Theorem~2.11 of \cite{Lindner74}. As in the previous example we find that the following are equivalent: $M$ is small cocomplete, in the classical sense; $M$ is $\mathcal A$"/cocomplete (\defref{cocompletion}); $M$ is total in $\ensProf\V$ (\defref{totality}); the Yoneda embedding $\map\yon M{\inhom{\op M, \V}_\textup s}$ admits a left adjoint.
		
		 To prove the claim first recall from \exref{weighted colimits} that small $\V$"/weighted colimits can be equivalently defined as left Kan extensions in $\ensProf\V$ along $\V$"/profunctors of the form $\hmap JAI$ with $A$ small: this shows that $\mathcal A$"/cocompleteness implies small cocompleteness. For the converse consider any left diagram \mbox{$M \xlar d A \xbrar J B$} with $J$ small and $M$ small cocomplete: we will show that the left Kan extension of $d$ along $J$ exists in $\enProf\V$ and hence in $\ensProf\V$ (\exref{enriched left Kan extension}), so that it is pointwise by \remref{pointwise left Kan extension in the presence of horizontal units and restrictions on the right}. Smallness of $J$ (\augexref{2.8}) can be rephrased as follows: for every $y \in B$ there exists a small sub"/$\V$"/category $\iota_y \colon A_y \subseteq A$ and a cocartesian cell \mbox{$\cell\cocart{\bigpars{\iota_y^*, J(\iota_y, y)}}{J(\id, y)}$} in $\enProf\V$. For each $y \in B$ let $ly \in M$ denote the $J(\iota_y, y)$"/weighted colimit of $d \of \iota_y$ (\exref{weighted colimits}), which exists by assumption. We assert that each $ly$ coincides with the $J(\id, y)$"/weighted colimit of $d$, so that together the $ly$ combine into a $\V$"/functor $\map lBM$ that forms the left Kan extension of $d$ along $J$, as described in \exref{enriched left Kan extension}. To prove this assertion denote by $\cell{\eta_y}{J(\iota_y, y)}M$ the nullary cell in $\enProf\V$ that defines $ly$ and consider the composite $\cell{(d \of \cart) \hc \eta_y}{\bigpars{\iota_y^*, J(\iota_y, y)}}M$, where $\cart$ defines $\iota_y^*$. The latter is left Kan by \cororef{Kan extension and conjoints} so that, using the vertical pasting lemma (\lemref{vertical pasting lemma}), by factorising it through $\cell\cocart{\bigpars{\iota_y^*, J(\iota_y, y)}}{J(\id, y)}$ we obtain a left Kan cell that defines $ly$ as the $J(\id, y)$"/weighted colimit of $d$. This completes the proof of the claim.
	\end{example}
	
	\begin{example}
		Let $\E$ be a cartesian closed regular category, so that $\ModRel(\E)$ is an equipment (\augpropref{7.8}) by \exref{internal modular relations}. \exref{all left diagrams} applies to the Yoneda embedding $\map\yon M{\inhom{\dl M, \ps 1}}$ of an internal preorder $M$, as constructed in \exref{yoneda embeddings for internal preorders}, so that it defines $\inhom{\dl M, \ps 1}$ as the $\mathcal A$"/cocompletion of $M$.
	\end{example}
	
	\begin{example} \label{cocompleteness for closed-odered closure spaces}
		\exref{all left diagrams} applies to the Yoneda embedding $\map\yon M{\Dnp M}$ of a closed"/ordered closure space $M$ (\exref{closed-ordered closure space yoneda embedding}) in the locally thin strict double category $\ClModRel$ of closed modular relations, so that $\yon$ defines $\Dnp M$ as the free $\mathcal A$"/cocompletion of $M$ therein. Likewise if $M$ is a modular closure space (\exref{closed modular relations}) then $\yon$ defines $\Dnp M$ as the free $\mathcal A$"/completion in the full sub"/double category $\ClModRel_\textup m \subset \ClModRel$ generated by modular closure spaces.
		
		We claim that any $\mathcal A$"/cocomplete closed"/ordered closure space $N$ has all suprema. Indeed consider the singleton closed"/ordered closure space $* \dfn \set *$ with closed subsets $\Cl * \dfn \set{\emptyset, \set *}$ and notice that downsets $X \subseteq N$ correspond precisely to closed modular relations $\hmap XN*$. In fact it is straightforward to see that $l \in N$ is a supremum of $X$ if and only if the morphism $\map l*N$, that picks out $l$, forms the pointwise left Kan extension of $\hmap XN*$ along $\id_N$ in $\ClModRel$, so that the claim follows. Having all suprema is however unlikely to be sufficient for $\mathcal A$"/cocompleteness in general, as follows. Consider the full sub"/augmented virtual double category $\CptClModRel_\textup m \subset \ClModRel_\textup m$ generated by those closed modular relations $\hmap JAB$ for which, for each $y \in B$, the preimage $\dl Jy$ is both compact and up"/directed in $A$ in the sense of Section~8 of \cite{Koudenburg18}. Combining Theorem~8.1 therein with \exref{continuous left Kan extensions} it follows that a modular closure space $N$ is $\mathcal A$"/cocomplete in $\CptClModRel_\textup m$ whenever $N$ has all suprema and is \emph{normalised}, that is $\overline{\set x} = \upset x$ for all $x \in N$, where $\overline{\set x}$ denotes the closure of the singleton set; see Section~4 of \cite{Koudenburg18}.
	\end{example}
	
	\begin{example} \label{left composable horizontal morphisms}
		To give a useful ideal $\mathcal C$ of left diagrams in a general augmented virtual double category $\K$ let us call a horizontal morphism $\hmap JAB$ \emph{left composable} when the pointwise right composite of $(H, J)$ exists for any $\hmap HCA$ (\defref{pointwise right cocartesian path}); we set
		\begin{displaymath}
			\mathcal C = \set{(d, J) \mid \textup{$J$ is left composable}}.
		\end{displaymath}
		Using \exref{cocartesian paths are exact} a Yoneda embedding $\map\yon M{\ps M}$ in $\K$ satisfies the conditions of the theorem, so that $\ps M$ forms the free $\mathcal C$-cocompletion of $M$, as soon as $\yon$ admits nullary restrictions and its companion $\yon_*$ is left composable. Notice that if $\K$ has restrictions on the right (\defref{augmented virtual equipment}) then the latter is equivalent to the existence of the pointwise right composites of all $\hmap HAM$ and $\hmap JMB$, which follows from \lemref{coherence of pointwise cocartesian paths} and the fact that $J \iso \yon_*(\id, \cur J)$ (\defref{yoneda embedding}).
		
		Consider a `monoidal augmented virtual double category' $(\K, \tens, I)$ in the sense of \defref{monoidal augmented virtual double category} below. In \thmref{free cocompletion of the monoidal unit} we will apply the previous to obtain conditions ensuring that the Yoneda embedding $\map\yon I{\ps I}$ for the monoidal unit defines $\ps I$ as the free $\mathcal C$"/cocompletion of $I$.
	\end{example}
	
	\begin{example} \label{left diagrams for lifted algebraic yoneda embeddings}
		Let $T$ be a monad on an augmented virtual double category $\K$, in the sense of Section~6 of \cite{Koudenburg15b}. As explained there `colax $T$"/algebras', `lax vertical $T$"/morphisms' and `horizontal $T$"/morphisms' in $\K$ form an augmented virtual double category $\wAlg\clx\lax T$ that comes equipped with a forgetful functor $\map U{\wAlg\clx\lax T}\K$. Given a colax $T$"/algebra $A$ and a Yoneda embedding $\map\yon{UA}{\ps A}$ in $\K$ one of the main results of \cite{Koudenburg15b}, its Theorem~8.1, gives conditions ensuring that $\yon$ can be lifted along $U$ as an ``algebraic'' Yoneda embedding in $\wAlg\clx\lax T$. This formalises equipping the category $\brks{\op A, \Set}$ of presheaves on a monoidal category $(A, \tens)$ with the monoidal structure given by \emph{Day"/convolution} induced by $\tens$, as introduced in \cite{Day70}.
		
		Next consider any ideal $\mathcal S$ of left diagrams in $\K$ and assume that $\map\yon{UA}{\ps A}$ satisfies the conditions of the theorem above, so that $\ps A$ is the free $\mathcal S$"/cocompletion of $UA$ in $\K$. As described in Section~7.4 of \cite{Koudenburg15b} the ideal $\mathcal S$ induces an ideal $\mathcal S_{(\clx, \lax)}$ of left diagrams in $\wAlg\clx\lax T$, and its Theorem~8.5 gives conditions on $\mathcal S$ ensuring that the lift of $\yon$ in $\wAlg\clx\lax T$ satisfies the theorem above with respect to $\mathcal S_{(\clx, \lax)}$, so that it defines the free $\mathcal S_{(\clx, \lax)}$"/cocompletion of $A$ in $\wAlg\clx\lax T$. Theorems~8.1 and~8.5 of \cite{Koudenburg15b} also treat the lifting of Yoneda embeddings along the forgetful functor for the augmented virtual double category $\lbcwAlg\clx T$ of colax $T$"/algebras, `pseudo vertical $T$"/morphisms' and horizontal $T$"/morphisms satisfying a `left Beck"/Chevalley condition'.
	\end{example}
	
	\subsection{Lemmas used in the proof of \thmref{presheaf objects as free cocompletions}}
	In the proof of \lemref{cocompletion equivalence} below the following lemma is used, which is a weakening of the implication (c) $\Rightarrow$ (d) of \thmref{total morphisms}. Given a Yoneda morphism $\map\yon M{\ps M}$ the latter asserts that a left Kan extension of a morphism $\map fMN$ along the companion $\yon_*$ is a left adjoint but, unlike the result below, it requires the companion $f_*$ to exist. The result below allows us to prove in \thmref{presheaf objects as free cocompletions} that $\ps M$ is the free $\mathcal S$"/cocompletion of $M$ among all $\mathcal S$"/cocomplete objects $N$ (see \defref{cocompletion}), without having to restrict to those $\mathcal S$"/cocomplete objects that are unital.
	\begin{lemma} \label{pointwise left Kan extensions along yoneda embeddings}
		Let $\map\yon M{\ps M}$ be a Yoneda morphism that admits all nullary restrictions (\defref{yoneda embedding}) and let $\mathcal S$ be an ideal of left diagrams that satisfies condition \textup{(e)} of \thmref{presheaf objects as free cocompletions}. Any pointwise left Kan extension along the companion $\yon_*$ is $\mathcal S$-cocontinuous.
	\end{lemma}
	\begin{proof}
		Suppose that the cell $\zeta$ in the composite below defines $k$ as the pointwise left Kan extension of $e$ along $\yon_*$. We have to show that, for any left diagram \mbox{$(d, J) \in \mathcal S(\ps M)$} and any pointwise left Kan cell $\eta$ as in the composite $\theta$ below, the composite $k \of \eta$ is again pointwise left Kan.
		\begin{displaymath}
			\theta \dfn \quad \begin{tikzpicture}[textbaseline]
				\matrix(m)[math35, column sep={1.75em,between origins}]
					{ M & & A & & B \\
						& M & & \ps M & \\
						& & N & & \\ };
				\path[map]	(m-1-1) edge[barred] node[above] {$\ps M(\yon, d)$} (m-1-3)
										(m-1-3) edge[barred] node[above] {$J$} (m-1-5)
														edge[ps] node[left] {$d$} (m-2-4)
										(m-1-5) edge[ps, transform canvas={xshift=2pt}] node[right] {$l$} (m-2-4)
										(m-2-2) edge[barred] node[below, inner sep=2pt] {$\yon_*$} (m-2-4)
														edge[transform canvas={xshift=-2pt}] node[left] {$e$} (m-3-3)
										(m-2-4) edge[transform canvas={xshift=2pt}] node[right] {$k$} (m-3-3);
				\path				(m-1-1) edge[eq, transform canvas={xshift=-1pt}] (m-2-2)
										(m-2-3) edge[cell, transform canvas={yshift=0.25em}] node[right] {$\zeta$} (m-3-3)
										(m-1-4) edge[cell, transform canvas={yshift=0.25em}] node[right] {$\eta$} (m-2-4);
				\draw				($(m-1-1)!0.5!(m-2-4)$) node[font=\scriptsize] {$\cart'$};
			\end{tikzpicture}
		\end{displaymath}
		To see this consider the pointwise left $\yon$"/exact cell $\cell\phi{\bigpars{M(\yon, d), J}}K$ that exists by condition (e) of \thmref{presheaf objects as free cocompletions}. By the uniqueness of Kan extensions we may without loss of generality assume that $l = \map{\cur K}B{\ps M}$ (\defref{yoneda embedding}) and that $\eta$ corresponds to $\phi$ as in \propref{left Kan extensions along a yoneda embedding in terms of left y-exact cells}; that is
		\begin{displaymath}
			\cart_{\cur K} \of \phi = \cart_{\ps M(\yon, d)} \hc \eta
		\end{displaymath}
		as in \lemref{bijection between cells induced by Yoneda morphisms}, where the cartesian cells define $\cur K$ and $\ps M(\yon, d)$ respectively. We may take the top row of $\theta$ above to be the factorisation of the right"/hand side immediately above through the cartesian cell that defines $\yon_*$ so that, by the equation above, \mbox{$\theta = \zeta \of \cart'_{\cur K} \of \phi$} where $\cart'_{\cur K}$ is the factorisation of $\cart_{\cur K}$ through $\yon_*$. The factorisations $\cart'$ and $\cart'_{\cur K}$ are cartesian by the pasting lemma (\lemref{pasting lemma for cartesian cells}) so that, using \lemref{restrictions of pointwise left Kan extensions} and the assumption that $\phi$ is pointwise left $\yon$-exact, we conclude that both $\theta$ and $\zeta \of \cart'$ are pointwise left Kan. By the horizontal pasting lemma (\lemref{horizontal pasting lemma}) it follows that the second column $k \of \eta$ of $\theta$ is also pointwise left Kan, as required.
	\end{proof}
	
	\begin{lemma} \label{cocompletion equivalence}
		Let $\map\yon M{\ps M}$ and $\mathcal S$ be as in \thmref{presheaf objects as free cocompletions}. For any object $N$ the top leg of the diagram
		\begin{displaymath}
			\begin{tikzpicture}
				\matrix(m)[math35, column sep=3em, row sep=2em]
					{	V(\K)(\ps M, N) & V(\K)(M, N) \\
						V_\textup{$\mathcal S$-cocts}(\K)(\ps M, N) & V(\K)(M, N)' \\ };
				\path[map]	(m-1-1) edge node[above] {$V(\K)(\yon, N)$} (m-1-2)
										(m-2-1) edge[rinj] (m-1-1)
														edge[dashed] node[below] {$\eq$} (m-2-2)
										(m-2-2) edge[rinj] (m-1-2);
			\end{tikzpicture}
		\end{displaymath}
			factors through the full subcategory $V(\K)(M, N)'$ of $V(\K)(M, N)$, that is generated by all $\map gMN$ whose pointwise left Kan extension along $\yon_*$ exists, as an equivalence as shown.
		
		In particular if $N$ is $\mathcal S$-cocomplete then the factorisation above reduces to an equivalence $V_\textup{$\mathcal S$-cocts}(\K)(\ps M, N) \eq V(\K)(M, N)$.
	\end{lemma}
	\begin{proof}
		Firstly, for the final assertion, simply notice that condition (y) of \thmref{presheaf objects as free cocompletions} ensures that $V(\K)(M, N)' = V(K)(M, N)$ for $\mathcal S$-cocomplete $N$. Next to see that the top leg of the diagram above factors through $V(\K)(M, N)'$ consider any $\mathcal S$"/cocontinuous $\map f{\ps M}N$; we have to show that the pointwise left Kan extension of $f \of \yon$ along $\yon_*$ exists. Since $(\yon, \yon_*) \in \mathcal S$ by the same condition it clearly does: it is $f$ itself, defined by the composite on the left below, where $\cart$ is pointwise left Kan by the density of $\yon$ (\defref{density definition}).
		
		To prove that the factorisation is an equivalence we will show that it is essentially surjective and full and faithful. For the former consider any $g \in V(\K)(M, N)'$, so that the pointwise left Kan extension $\map l{\ps M}N$ of $g$ along $\yon_*$ exists; we denote its defining cell by $\eta$, as in the middle below. By \lemref{pointwise left Kan extensions along yoneda embeddings} $l$ is $\mathcal S$-cocontinuous while, by precomposing $\eta$ with the weakly cocartesian cell defining $\yon_*$, we obtain a vertical isomorphism $g \iso l \of \yon$, as follows from the assumption that $\yon$ is full and faithful and \propref{pointwise left Kan extension along full and faithful map}. This shows essential surjectivity.
		
		To prove full and faithfulness, consider any vertical cell $\cell\phi{f \of \yon}{g \of \yon}$ with $f$ and $g$ $\mathcal S$-cocontinuous. We have to show that there exists a unique vertical cell $\cell{\phi'}fg$ such that $\phi = \phi' \of \yon$.
		\begin{displaymath}
			\begin{tikzpicture}[textbaseline]
				\matrix(m)[math35, column sep={1.75em,between origins}]{M & & \ps M \\ & \ps M & \\ & N & \\};
				\path[map]	(m-1-1) edge[barred] node[above] {$\yon_*$} (m-1-3)
														edge[transform canvas={xshift=-2pt}] node[left] {$\yon$} (m-2-2)
										(m-2-2) edge node[left] {$f$} (m-3-2);
				\path				(m-1-3) edge[ps, eq, transform canvas={xshift=2pt}] (m-2-2);
				\draw				([yshift=0.333em]$(m-1-2)!0.5!(m-2-2)$) node[font=\scriptsize] {$\cart$};
			\end{tikzpicture} \qquad\qquad \begin{tikzpicture}[textbaseline]
				\matrix(m)[math35, column sep={1.75em,between origins}]{M & & \ps M \\ & N & \\};
				\path[map]	(m-1-1) edge[barred] node[above] {$\yon_*$} (m-1-3)
														edge[transform canvas={xshift=-2pt}] node[left] {$g$} (m-2-2)
										(m-1-3) edge[transform canvas={xshift=2pt}] node[right] {$l$} (m-2-2);
				\path[transform canvas={yshift=0.25em}]	(m-1-2) edge[cell] node[right, inner sep=3pt] {$\eta$} (m-2-2);
			\end{tikzpicture} \qquad\qquad\qquad \begin{tikzpicture}[textbaseline]
				\matrix(m)[math35, column sep={1.75em,between origins}]{& M & & \ps M \\ \ps M & & \ps M \\ & N & \\};
				\path[map]	(m-1-2) edge[barred] node[above] {$\yon_*$} (m-1-4)
														edge[ps, transform canvas={xshift=-1pt}] node[left] {$\yon$} (m-2-3)
														edge[ps, bend right = 18] node[left] {$\yon$} (m-2-1)
										(m-2-1) edge[bend right = 18] node[left] {$f$} (m-3-2)
										(m-2-3) edge[bend left = 18] node[right] {$g$} (m-3-2);
				\path				(m-1-2) edge[cell, transform canvas={yshift=-1.625em}] node[right] {$\phi$} (m-2-2)
										(m-1-4) edge[ps, eq, transform canvas={xshift=2pt}] (m-2-3);
				\draw				([yshift=0.333em]$(m-1-3)!0.5!(m-2-3)$) node[font=\scriptsize] {$\cart$};
			\end{tikzpicture} \quad = \quad \begin{tikzpicture}[textbaseline]
				\matrix(m)[math35, column sep={1.75em,between origins}]{M & & \ps M \\ & \ps M & \\ & N & \\};
				\path[map]	(m-1-1) edge[barred] node[above] {$\yon_*$} (m-1-3)
														edge[ps, transform canvas={xshift=-2pt}] node[left] {$\yon$} (m-2-2)
										(m-2-2) edge[bend right=45] node[left] {$f$} (m-3-2)
														edge[bend left=45] node[right] {$g$} (m-3-2);
				\path				(m-1-3) edge[ps, eq, transform canvas={xshift=2pt}] (m-2-2)
										(m-2-2) edge[cell] node[right, inner sep=2.5pt] {$\phi'$} (m-3-2);
				\draw				([yshift=0.333em]$(m-1-2)!0.5!(m-2-2)$) node[font=\scriptsize] {$\cart$};
			\end{tikzpicture}
		\end{displaymath}
		Since $\yon$ is dense, the cartesian cell in the right"/hand side of the equation above is pointwise left Kan by \defref{density definition}. It follows that its composition with $f$ is too, by $\mathcal S$-cocontinuity of $f$ and condition (y) of \thmref{presheaf objects as free cocompletions}, so that the composite on the left"/hand side above factors uniquely as a cell $\phi'$ as shown. Composing both sides with the weakly cocartesian cell corresponding to the cartesian cell $\cart$ above, using the vertical companion identity (\lemref{companion identities lemma}) we conclude that $\phi'$ is unique such that $\phi = \phi' \of \yon$, as required. This completes the proof.
	\end{proof}
	
	\section{Yoneda embeddings in monoidal augmented virtual double categories} \label{yoneda embeddings in a monoidal augmented virtual double category}
	Let $\E$ be a finitely complete category with a subobject classifier $\Omega$. Recall from \exref{generic subobjects are yoneda embeddings} that $\Omega$ induces a Yoneda embedding $\map\yon 1\Omega$ in $\ModRel(\E)$. Also recall, from e.g.\ Section~A2 of \cite{Johnstone02}, that the following are equivalent for $\E$: (a)~$\E$~has power objects; (b) $\E$ has all exponentials of the form $\Omega^A$; (c) $\E$ is cartesian closed. In this final section we generalise the implications (a) $\Leftrightarrow$ (b) $\Leftarrow$ (c) to augmented virtual double categories as follows. Given a monoidal augmented virtual double category $\K = (\K, \tens, I)$ in the sense below consider a Yoneda embedding $\map\yon I{\ps I}$ for the monoidal unit as well as a unital object $A$ in $\K$. The main result of this section, \thmref{presheaf object equivalent to iota-small internal hom}, shows that under mild conditions the Yoneda embedding \mbox{$\map{\yon_A}A{\ps A}$} exists in $\K$ if and only if the `internal hom' $\inhom{\dl A, \ps I}$ (\defref{internal iota-small hom}) does, where $\dl A$ denotes the `horizontal dual' of $A$ (\defref{horizontal dual}), and, in that case, $\ps A \iso \inhom{\dl A, \ps I}$. This result was used in \exsref{enriched yoneda embedding summary}{yoneda embeddings for internal preorders} to obtain Yoneda embeddings in $\enProf\V$ and $\ModRel(\E)$, with the latter example recovering the classical implication (b) $\Rightarrow$ (a) above.
	
	Instead of proving the main result \thmref{presheaf object equivalent to iota-small internal hom} directly we will prove a generalisation in \thmsref{universal morphism from a yoneda embedding}{yoneda embedding from universal morphisms theorem} as follows; this generalisation is more widely applicable and allows for simpler diagrams in its proof. Given a functor \mbox{$\map F\K\L$} of augmented virtual double categories, a Yoneda embedding $\map{\yon_A}A{\ps A}$ in $\L$, a `locally universal horizontal morphism' $\hmap\iota A{FA'}$ from $A$ to $F$ in $\L$ (\defref{universal horizontal morphism}) and an object $P \in \K$, the generalisation asserts the equivalence of the existence of a Yoneda embedding $\map{\yon_{A'}}{A'}P$ in $\K$ and that of a relative universal morphism $\map\eps{FP}{\ps A}$ in $\L$ (\defref{universal vertical morphism}). Applying  the generalisation to the endofunctor $\map{F \dfn \dl A \tens \dash}\K\K$ recovers \thmref{presheaf object equivalent to iota-small internal hom}.
	
	The main result and its generalisation only depend on \secrref{Kan extension section}{yoneda embeddings section}. \thmref{free cocompletion of the monoidal unit} below, which describes conditions ensuring that the Yoneda embedding $\map\yon I{\ps I}$ defines $\ps I$ as a free cocompletion of $I$, depends on \secref{cocompleteness section}.
	
	\subsection{Monoidal augmented virtual double categories}
	We start by introducing the notion of monoidal augmented virtual double category. Notice that the $2$"/category $\AugVirtDblCat$ of augmented virtual double categories, their functors and the transformations between them (\augsecref 3) has finite products. In particular the square $\K \times \K$ of an augmented virtual double category $\K$ can be taken to have as objects, vertical morphisms, horizontal morphisms and cells, ordered pairs of those in $\K$ where, in the case of a pair of cells $(\phi, \phi')$, the arities of $\phi$ and $\phi'$ coincide. Likewise the terminal object in $\AugVirtDblCat$ is the strict double category $1$ (see \augpropref{7.8}) generated by a single object $*$. Analogous to the fact that monoidal categories are pseudomonoids (see Section~3 of \cite{Day-Street97}) in the $2$"/category of categories, monoidal augmented virtual double categories are pseudomonoids in the $2$"/category $\AugVirtDblCat$ as follows.
	\begin{definition} \label{monoidal augmented virtual double category}
		A \emph{monoidal augmented virtual double category} is an augmented virtual double category $\K$ equipped with a \emph{monoidal product} $\map\tens{\K \times \K}\K$ and a \emph{monoidal unit} object $I \in \K$, with specified horizontal unit $\hmap{I_I}II$, as well as invertible \emph{associator} and \emph{unitor} transformations
		\begin{displaymath}
			\as\colon \tens \of (\tens \times \id) \xRar\iso \tens \of (\id \times \tens), \quad \lu\colon \tens \of (I \times \id) \xRar\iso \id \quad\text{and}\quad \ru\colon \tens \of (\id \times I) \xRar\iso \id
		\end{displaymath}
		satisfying the usual axioms.
		
		We call $(\K, \tens, I)$ \emph{cartesian monoidal} if $\map\tens{\K \times \K}\K$ and $\map I1\K$ form right adjoints to the diagonal functor $\map\Delta\K{\K \times \K}$ and the terminal functor $\map !\K1$ respectively, in the $2$"/category $\AugVirtDblCat$ (see \exref{precartesian pseudo double category} below).
	\end{definition}
	
	By the definition of functor of augmented virtual double categories and that of transformation of such functors (\augsecref 3), a monoidal structure on an augmented virtual double category $\K$ restricts to a monoidal structure $(\tens_\textup v, I, \as_\textup v, \lu_\textup v, \ru_\textup v)$ on the category $\K_\textup v$ of its objects and vertical morphisms. Notice that the invertible cells of the associator and unitor transformations are of the forms below, where $\hmap JAB$, $\hmap{J'}{A'}{B'}$ and $\hmap{J''}{A''}{B''}$ are any horizontal morphisms in $\K$. To prevent confusing the monoidal unit \emph{object} $I$ with horizontal unit \emph{morphisms} $\hmap{I_A}AA$, throughout this section the latter are consistently denoted with their object $A$ as subscript.
	\begin{displaymath}
		\begin{tikzpicture}[textbaseline]
			\matrix(m)[math35, column sep={9em,between origins}]{A \tens (A' \tens A'') & B \tens (B' \tens B'') \\ (A \tens A') \tens A'' & (B \tens B') \tens B'' \\};
			\path[map]	(m-1-1) edge[barred] node[above, inner sep=6pt] {$J \tens (J' \tens J'')$} (m-1-2)
													edge node[left] {$\as$} (m-2-1)
									(m-1-2) edge node[right] {$\as$} (m-2-2)
									(m-2-1) edge[barred] node[below, inner sep=6pt] {$(J \tens J') \tens J''$} (m-2-2);
			\path[transform canvas={xshift=4.5em}]	(m-1-1) edge[cell] node[right] {$\as$} (m-2-1);
		\end{tikzpicture}\qquad\qquad
		\begin{tikzpicture}[textbaseline]
			\matrix(m)[math35, column sep={4em,between origins}]{I \tens A & I \tens B \\ A & B \\};
			\path[map]	(m-1-1) edge[barred] node[above, inner sep=6pt] {$I_I \tens J$} (m-1-2)
													edge node[left] {$\lu$} (m-2-1)
									(m-1-2) edge node[right] {$\lu$} (m-2-2)
									(m-2-1) edge[barred] node[below] {$J$} (m-2-2);
			\path[transform canvas={xshift=2em}]	(m-1-1) edge[cell] node[right] {$\lu$} (m-2-1);
		\end{tikzpicture}\qquad
		\begin{tikzpicture}[textbaseline]
			\matrix(m)[math35, column sep={4em,between origins}]{A \tens I & B \tens I \\ A & B \\};
			\path[map]	(m-1-1) edge[barred] node[above, inner sep=6pt] {$J \tens I_I$} (m-1-2)
													edge node[left] {$\ru$} (m-2-1)
									(m-1-2) edge node[right] {$\ru$} (m-2-2)
									(m-2-1) edge[barred] node[below] {$J$} (m-2-2);
			\path[transform canvas={xshift=2em}]	(m-1-1) edge[cell] node[right] {$\ru$} (m-2-1);
		\end{tikzpicture}
	\end{displaymath}
	Given a unital object $X$ and a path $\hmap{\ul J = (J_1, \dotsc, J_n)}{A_0}{A_n}$ in $\K$ it will be useful to abbreviate $\hmap{X \tens \ul J \dfn (I_X \tens J_1, \dotsc, I_X \tens J_n)}{X \tens A_0}{X \tens A_n}$.
	
	\begin{example} \label{V-Prof is monoidal}
		A symmetry $\sm\colon X \tens Y \xrar\iso Y \tens X$ for a monoidal category $\V$ induces a monoidal structure on the unital virtual equipment $\enProf\V$ of $\V$"/profunctors, with the usual monoidal product $A \tens A'$ of $\V$"/categories $A$ and $A'$ (see e.g.\ Section~1.4 of \cite{Kelly82}) and the monoidal product $\hmap{J \tens J'}{A \tens A'}{B \tens B'}$ of $\V$"/profunctors $J$ and $J'$ defined by $(J \tens J')\bigpars{(x, x'), (y, y')} = J(x, y) \tens J'(x', y')$. Likewise a symmetric universe enlargement $\V \subset \V'$ (\exref{enriched yoneda embedding summary}) induces a monoidal structure on the augmented virtual equipment $\enProf{(\V, \V')}$ of $\V$"/profunctors between $\V'$"/categories.
	\end{example}
	
	\begin{example} \label{V-sProf is monoidal}
		As long as the monoidal product $\tens$ of a symmetric monoidal category $\V$ preserves small colimits on either side the monoidal structure on $\enProf\V$ restricts to the unital virtual double category $\ensProf\V \subset \enProf\V$ of small $\V$"/profunctors. Indeed if $\hmap JAB$ and $\hmap{J'}{A'}{B'}$ are small $\V$"/profunctors, with each $J(\dash, y)$ and $J'(\dash, y')$ ``generated'' by their actions on small sub"/$\V$"/categories $A_y \subseteq A$ and $A'_{y'} \subseteq A'$ in the sense of \augexref{2.8}, then by using the `Fubini formula' for coends (see Section~2.1 of \cite{Kelly82} for its dual) one easily sees that each $(J \tens J')\bigpars{\dash, (y, y')}$ is generated by its action on $A_y \tens A_{y'}$, regarded as a small sub"/$\V$"/category of $A \tens A'$.
	\end{example}
	
	\begin{example} \label{monoidal pseudo double category}
		Let $\nlDblCat$ denote the $2$"/category of pseudo double categories, normal lax double functors and transformations; see e.g.\ Section~6 of \cite{Shulman08}, and let $\DblCat$ denote its locally full sub"/$2$"/category generated by pseudo double functors. Let $\VirtDblCat_\textup u$ denote the $2$"/category of unital virtual double categories, normal functors and transformations; see \augsecref{10}. We have embeddings of $2$"/categories
		\begin{displaymath}
			\DblCat \hookrightarrow \nlDblCat \hookrightarrow \VirtDblCat_\textup u \xrar[\simeq]N \AugVirtDblCat_\textup u \hookrightarrow \AugVirtDblCat;
		\end{displaymath}
		see Section~2 of \cite{Dawson-Pare-Pronk06} (where virtual double categories are called `lax double categories') for $\nlDblCat \hookrightarrow \VirtDblCat_\textup u$ and \augsecref{10} for the $2$"/equivalence $N$. Under this composite a pseudo double category $\K$, with horizontal composition denoted $\hc$, is mapped to the augmented virtual double category $N(\K)$ with the same objects and morphisms as $\K$, whose unary cells $(J_1, \dotsc, J_n) \Rar K$ are cells \mbox{$J_1 \hc \dotsb \hc J_n \Rar K$} in $\K$ and whose nullary cells $(J_1, \dotsc, J_n) \Rar C$ are cells \mbox{$J_1 \hc \dotsb \hc J_n \Rar I_C$} in $\K$, where $I_C$ denotes the horizontal unit of the object $C$. Except for $\DblCat \hookrightarrow \nlDblCat$ the embeddings above are full and faithful.
		
		Analogous to the definition above, in Definition~2.9 of \cite{Shulman10} a \emph{monoidal pseudo double category} is defined to be a pseudomonoid $(\K, \tens, I)$ in the $2$"/category $\DblCat$. Since the composite above preserves finite products any monoidal pseudo double category $\K$ can be regarded as a monoidal augmented virtual double category $N(\K)$ in our sense.
	\end{example}
	
	\begin{example} \label{precartesian pseudo double category}
		In Definition~4.2.1 of \cite{Aleiferi18} a pseudo double category $\K$ is defined to be \emph{cartesian} if it forms a \emph{cartesian object} in the $2$"/category $\DblCat$, that is the diagonal pseudo double functor $\map\Delta\K{\K \times \K}$ and the terminal pseudo double functor $\map\term\K1$ admit right adjoints $\map\times{\K \times \K}\K$ and $\map 11\K$ in $\DblCat$. The pseudo double category $\Span\E$ of a spans in a category $\E$ with finite limits (\augexref{2.9}) is cartesian under the cartesian product of objects and spans; see Proposition~4.2.7 of \cite{Aleiferi18}.
		
		Using that the embedding $\DblCat \hookrightarrow \AugVirtDblCat$ above preserves finite products, any cartesian pseudo double category can be regarded as a cartesian object in $\AugVirtDblCat$ and thus, since cartesian objects canonically are pseudomonoids (see Remark~2.11 of \cite{Shulman10}), as a cartesian monoidal augmented virtual double category in the sense of \defref{monoidal augmented virtual double category}. We conclude that $\Span\E$, as a unital virtual equipment, admits a cartesian monoidal structure $(\times, 1)$. Notice that the latter restricts to a cartesian monoidal structure on the unital virtual equipment $\Rel(\E)$ of relations in $\E$ (\exref{internal modular relations}).
	\end{example}
	
	\begin{example} \label{ModRel is monoidal}
		Recall the $2$"/functor $\map\Mod\VirtDblCat\VirtDblCat_\textup u$ that maps a virtual double category $\K$ to the unital virtual double category $\Mod(\K)$ of monoids and bimodules in $\K$; see \augdefref{2.1}. $\Mod$ preserves finite products and hence the endo"/$2$"/functor
		\begin{displaymath}
			\AugVirtDblCat_\textup u \xrar U \VirtDblCat \xrar\Mod \VirtDblCat_\textup u \xrar[\simeq]N \AugVirtDblCat_\textup u,
		\end{displaymath}
		where $U$ forgets the nullary cells (\augpropref{3.3}), preserves (cartesian) monoidal structures on unital virtual double categories. Hence the cartesian monoidal structure on $\Rel(\E)$ of the previous example induces a cartesian monoidal structure $(\times, 1)$ on the unital virtual equipment $\ModRel(\E) \dfn (N \of \Mod)\bigpars{\Rel(\E)}$ of internal modular relations in $\E$ (\exref{internal modular relations}).
	\end{example}
	
	\subsection{Free cocompletion of the monoidal unit}
	Having introduced the notion of monoidal augmented virtual double category $\K = (\K, \tens, I)$, we now pause to describe conditions ensuring that the Yoneda embedding $\map\yon I\ps I$ for the monoidal unit defines $\ps I$ as the free cocompletion of $I$, in the sense of \defref{cocompletion}. More precisely we will show that $\ps I$ is the free $\mathcal C$"/cocompletion, with $\mathcal C$ the ideal of left diagrams
	\begin{displaymath}
		\mathcal C = \set{(d, J) \mid J \text{ is left composable}}
	\end{displaymath}
	where $\hmap JAB$ is left composable in the sense of \exref{left composable horizontal morphisms}: the pointwise right composites $(H \hc J)$ (\defref{pointwise right cocartesian path}) exist for all $\hmap HCB$. In order to state the conditions consider the nullary cartesian cell $\cell\cart{I_I}I$ on the left below that defines the horizontal unit $I_I$ (\defref{cartesian cells}) and, for any pair of horizontal morphisms $\hmap JAI$ and $\hmap HIB$, the monoidal product cell $\chi_{(J, H)} \dfn (\id_J \hc \cart) \tens (\cart \hc \id_H)$ on the right below. Because $\cell\cart{I_I}I$ is cocartesian (\lemref{companion identities lemma}), so are $\id_J \hc \cart$ and $\cart \hc \id_H$ by the pasting lemma (\auglemref{7.7}).
	\begin{displaymath}
		\begin{tikzpicture}[textbaseline]
				\matrix(m)[math35, column sep={1.75em,between origins}]{I & & I \\ & I & \\};
				\path[map]	(m-1-1) edge[barred] node[above] {$I_I$} (m-1-3);
				\path				(m-1-3) edge[eq, transform canvas={xshift=2pt}] (m-2-2)
										(m-1-1) edge[eq, transform canvas={xshift=-2pt}] (m-2-2);
				\draw				([yshift=0.333em]$(m-1-2)!0.5!(m-2-2)$) node[font=\scriptsize] {$\cart$};				
			\end{tikzpicture} \qquad\qquad\qquad \chi_{(J, H)} \dfn \begin{tikzpicture}[textbaseline]
						\matrix(m)[math35, column sep={5.2em,between origins}]{ A \tens I & I \tens I & I \tens B \\ A \tens I & & I \tens B \\ };
						\path[map]	(m-1-1) edge[barred] node[above] {$J \tens I_I$} (m-1-2)
												(m-1-2) edge[barred] node[above] {$I_I \tens H$} (m-1-3)
												(m-2-1) edge[barred] node[below] {$J \tens H$} (m-2-3);
						\path				(m-1-1) edge[eq] (m-2-1)
												(m-1-3) edge[eq] (m-2-3)
												(m-1-2) edge[cell, transform canvas={xshift=-4.2em}] node[right] {$(\id_J \hc \cart) \tens (\cart \hc \id_H)$} (m-2-2);
					\end{tikzpicture}
	\end{displaymath}
	\begin{theorem} \label{free cocompletion of the monoidal unit}
		Let $\K$ and $\map\yon I{\ps I}$ be as above. Assume that $\K$ has restrictions on the right as well as restrictions along vertical isomorphisms on the left (\defref{augmented virtual equipment}). If the companion $\yon_*$ exists then $\yon$ defines $\ps I$ as the free $\mathcal C$"/cocompletion of $I$ (\defref{cocompletion}) whenever the following conditions hold for any triple of morphisms $\hmap JAI$, $\hmap HIB$ and $\map fXB$:
		\begin{enumerate}[label=\textup{(\alph*)}]
			\item the cell $\chi_{(J, H)}$ on the right above is cocartesian;
			\item the restriction $H(\id, f)$ is preserved by the assignment $\id_J \tens \dash$.
		\end{enumerate}		
	\end{theorem}
	Notice that condition (a) means that $\chi_{(J, H)}$ defines $J \tens H$ as the horizontal composite $\pars{(J \tens I_I) \hc (I_I \tens H)}$ (\defref{pointwise right cocartesian path}).
	\begin{proof}
		By \exref{left composable horizontal morphisms} it suffices to prove that the pointwise right composite $(J \hc H)$ exists for any pair $\hmap JAI$ and $\hmap HIB$. We will do so by showing that the unique factorisation $\xi$ in the right"/hand side below is right pointwise cocartesian (\augdefref{9.1}), thus defining the restriction \mbox{$\hmap{K \dfn (J \tens H)(\inv{\mathfrak r}, \inv{\mathfrak l})}AB$} as the required pointwise right composite $(J \hc H)$ by \remref{right pointwise cocartesian and pointwise right cocartesian comparison}. Notice that the top row of cells in the left"/hand side below is composable because (in any monoidal category) $\map{\mathfrak r = \mathfrak l}{I \tens I}I$; see e.g.\ Section~VII.1 of \cite{MacLane98}.
		\begin{displaymath}
			\begin{tikzpicture}[textbaseline]
				\matrix(m)[math35, column sep={2.2em,between origins}]
					{ A & & I & & B \\
						A \tens I & & I \tens I & & I \tens B \\
						& A \tens I & & I \tens B & \\ };
				\path[map]	(m-1-1) edge[barred] node[above] {$J$} (m-1-3)
														edge node[left] {$\inv{\mathfrak r}$} (m-2-1)
										(m-1-3) edge[barred] node[above] {$H$} (m-1-5)
														edge node[right, inner sep=1pt] {$\inv{\mathfrak l}$} (m-2-3)
										(m-1-5) edge node[right] {$\inv{\mathfrak l}$} (m-2-5)
										(m-2-1) edge[barred] node[below] {$J \tens I_I$} (m-2-3)
										(m-2-3) edge[barred] node[below, xshift=-1pt] {$I_I \tens H$} (m-2-5)
										(m-3-2) edge[barred] node[below, xshift=-2pt] {$J \tens H$} (m-3-4);
				\path				(m-1-2) edge[cell] node[right] {$\inv{\mathfrak r}$} (m-2-2)
										(m-1-4) edge[cell] node[right] {$\inv{\mathfrak l}$} (m-2-4)
										(m-2-1) edge[transform canvas={xshift=-1pt}, eq] (m-3-2)
										(m-2-3) edge[cell] node[right] {$\chi_{(J, H)}$} (m-3-3)
										(m-2-5) edge[transform canvas={xshift=1pt}, eq] (m-3-4);										
			\end{tikzpicture} \quad = \quad \begin{tikzpicture}[textbaseline]
				\matrix(m)[math35, column sep={2.2em,between origins}]
					{ A & & I & & B \\
						& A & & B & \\
						& A \tens I & & I \tens B & \\ };
				\path[map]	(m-1-1) edge[barred] node[above] {$J$} (m-1-3)
										(m-1-3) edge[barred] node[above] {$H$} (m-1-5)
										(m-2-2) edge[barred] node[below] {$K$} (m-2-4)
														edge node[left] {$\inv{\mathfrak r}$} (m-3-2)										
										(m-2-4)	edge node[right] {$\inv{\mathfrak l}$} (m-3-4)
										(m-3-2) edge[barred] node[below, xshift=-2pt] {$J \tens H$} (m-3-4);
				\path				(m-1-3) edge[cell] node[right] {$\xi$} (m-2-3)
										(m-1-1) edge[transform canvas={xshift=-1pt}, eq] (m-2-2)
										(m-1-5) edge[transform canvas={xshift=1pt}, eq] (m-2-4);
				\draw				($(m-2-3)!0.5!(m-3-3)$) node[font=\scriptsize] {$\cart$};
			\end{tikzpicture}
		\end{displaymath}
		
		We will start by showing that $\xi$ is cocartesian (\defref{cocartesian path}). First notice that the cartesian cell in the right"/hand side above admits an inverse $\cell{\inv{\cart}}{J \tens H}K$, which can be seen by factorising the identity cell $\id_{J \tens H}$ through $\cart$. Next notice that the path of cells $(\inv{\mathfrak r}, \inv{\mathfrak l})$, that makes up the top row of the left"/hand side, is cocartesian: this follows from the fact that $\cell{\inv{\mathfrak r}}J{J \tens I_I}$, $\cell{\inv{\mathfrak l}}H{I_I \tens H}$, as well as cartesian cells that define restrictions along $\map{\inv{\mathfrak r}}A{A \tens I}$ or $\map{\inv{\mathfrak l}}B{B \tens I}$, are invertible cells. Likewise $\inv{\cart}$ is cocartesian so that, if $\chi_{(J, H)}$ is cocartesian then $\xi = \inv{\cart} \of \chi_{(J, H)} \of (\inv{\mathfrak r}, \inv{\mathfrak l})$ is so too by the pasting lemma (\auglemref{7.7}).
		
		To show that $\xi$ is in fact right pointwise cocartesian fix a vertical morphism $\map fXB$ and consider the unique factorisation $\cell{\xi'}{\pars{J, H(\id, f)}}{K(\id, f)}$ in $\xi \of (\id_J, \cart_{H(\id, f)}) = \cart_{K(\id, f)} \of \xi'$, as in \augdefref{9.1} and where the cartesian cells define the restrictions $H(\id, f)$ and $K(\id, f)$; we have to show that $\xi'$ is again cocartesian. To do so consider the equality below, whose identities follow from the functoriality of $\tens$, the naturality of the cells $\inv{\mathfrak l}$ (\augdefref{3.2}), and the definitions of $\xi$ and $\xi'$.
		\begin{align*}
			\pars{\id_J \tens \cart_{H(\id, f)}} &\of \chi_{\pars{J, H(\id, f)}} \of \pars{\inv{\mathfrak r}, \inv{\mathfrak l}} \\
			&= \chi_{(J, H)} \of \bigpars{\id_{J \tens_{I_I}}, \id_{I_I} \tens \cart_{H(\id, f)}} \of (\inv{\mathfrak r}, \inv{\mathfrak l}) \\
			&= \chi_{(J, H)} \of \pars{\inv{\mathfrak r}, \inv{\mathfrak l}} \of \pars{\id_J, \cart_{H(\id, f)}} \\
			&= \cart \of \xi \of \pars{\id_J, \cart_{H(\id, f)}} = \cart \of \cart_{K(\id, f)} \of \xi'
		\end{align*}
		Now write $\cell{\cart'}{K(\id, f)}{J \tens H(\id, f)}$ for the factorisation of $\cart \of \cart_{K(\id, f)}$, in the right"/hand side above, through the cartesian cell $\id_J \tens \cart_{H(\id, f)}$; $\cart'$ is cartesian by the pasting lemma (\lemref{pasting lemma for cartesian cells}). We thus obtain
		\begin{displaymath}
			\pars{\id_J \tens \cart_{H(\id, f)}} \of \chi_{\pars{J, H(\id, f)}} \of \pars{\inv{\mathfrak r}, \inv{\mathfrak l}} = \pars{\id_J \tens \cart_{H(\id, f)}} \of \cart' \of \xi',
		\end{displaymath}
		so that $\chi_{(J, H(\id, f))} \of (\inv{\mathfrak r}, \inv{\mathfrak l}) = \cart' \of \xi'$ by the uniqueness of factorisations through cartesian cells. Since $\cart'$ is cartesian the previous argument, proving the cocartesianness of $\xi$, applies to the latter identity so that, because $\chi_{(J, H(\id, f))}$ is cocartesian by assumption, $\xi'$ is cocartesian too. This completes the proof.
	\end{proof}
	
	\subsection{Horizontal duals}
	In order to explain the notion of `weak horizontal right dual' defined below consider the dual $\dl A$ of a category $A$ enriched in a symmetric monoidal category $\V'$, with hom"/objects $\dl A(x, y) = A(y, x)$ (see Section~1.4 of \cite{Kelly82}). Assuming that $\V'$ has large colimits preserved by $\tens'$ on both sides, so that $\enProf{\V'}$ is an equipment (\augexref{9.2}), it is shown in Theorem~5.1 of \cite{Shulman10} that the monoidal structure on $\enProf{\V'}$ restricts to a monoidal structure on the horizontal bicategory $H(\enProf{\V'})$ that it contains, consisting of $\V'$"/categories, $\V'$"/profunctors and the horizontal cells (i.e.\ transformations) between them. In the monoidal bicategory $H(\enProf{\V'})$ $\dl A$ forms the `right bidual' of $A$, in the sense of Definition~6 of \cite{Day-Street97}, defined as such by the `exact copairing' $\hmap\iota I{\dl A \tens' A}$ given by $\iota\bigpars{*, (x,y)} = A(x, y)$ which, for any $\V'$"/categories $B$ and $C$, induces an equivalence of categories
	\begin{displaymath}
		\map{\iota^\flat}{H(\enProf{\V'})(A \tens' B, C)}{H(\enProf{\V'})(B, \dl A \tens' C)}
	\end{displaymath}
	given by $\iota^\flat(J) \dfn (\iota \tens' I_B) \hc (I_{\dl A} \tens' J)$ where $I_B$ and $I_{\dl A}$ are unit profunctors. Using the `Yoneda isomorphisms' (see e.g.\ Formula~3.71 of \cite{Kelly82}) we find that $\iota^\flat(J)$ can simply be defined as $\iota^\flat(J)\bigpars{y, (x, z)} \dfn J\bigpars{(x, y), z}$. It follows that, for a symmetric universe enlargement $\V \subset \V'$ (\exref{enriched yoneda embedding summary}), the equivalences above restrict to equivalences \mbox{$H\bigpars{\enProf{(\V, \V')}}(A \tens' B, C) \simeq H\bigpars{\enProf{(\V, \V')}}(B, \dl A \tens' C)$} of categories of $\V$"/profunctors between $\V'$"/categories.  Taking $B = I$ and restricting $A$ and $C$ to be $\V$"/categories, we obtain an equivalence between $\V$"/profunctors $A \brar C$ and $\V$"/profunctors $I \brar \dl A \tens C$. Notice that the latter (trivially) are small $\V$"/profunctors (\augexref{2.8}) while the former are not in general, so that these equivalences do not restrict to analogous equivalences for small $\V$"/profunctors. They do however restrict to full and faithful functors between categories of small $\V$"/profunctors, and we conclude that the copairing $\iota$ induces full and faithful functors \mbox{$\map{\iota^\flat}{H(\K)(A, C)}{H(\K)(I, \dl A \tens' C)}$} in each of the cases $\K = \enProf{\V'}$, $\K = \enProf{(\V, \V')}$ and $\K = \ensProf\V$. Since $\enProf{\V'}$ is a pseudo double category (\augexref{9.2}) it is straightforward to see that, in the cases $\K = \enProf{\V'}$ and $\K = \enProf{(\V, \V')}$, the full and faithfulness of $\iota^\flat$ is equivalent to condition (b) below. The same holds true for $\K = \ensProf\V$ whenever $\V$ is small cocomplete and $\tens$ preserves colimits on both sides, so that $\ensProf\V$ is a pseudo double category (\augexref{9.3}).
	\begin{definition} \label{horizontal dual}
		Let $A$ be an object of a monoidal augmented virtual double category $\K = (\K, \tens, I)$. A \emph{weak horizontal right dual} of $A$ (shortly \emph{weak horizontal dual} of $A$) is a unital object $\dl A$ equipped with a horizontal \emph{copairing} $\hmap\iota I{\dl A \tens A}$ satisfying the following conditions:
		\begin{enumerate}[label=\textup{(\alph*)}]
			\item for each $\hmap JAB$ the pointwise right composite (\defref{pointwise right cocartesian path}) below exists, called the \emph{adjunct} of $J$;
			\begin{displaymath}
				\hmap{\flad J \dfn \iota \hc (\dl A \tens J)}I{\dl A \tens B}
			\end{displaymath}
			\item the assignments below, between collections of horizontal cells in $\K$ as shown and where the cell $\cocart$ defines $\flad J$, is a bijection.
		\end{enumerate}
		\begin{displaymath}
			\begin{tikzpicture}
				\matrix(m)[math35, xshift=-14em]{A & B \\ A & B \\};
				\path[map]	(m-1-1) edge[barred] node[above] {$\ul H$} (m-1-2)
										(m-2-1) edge[barred] node[below] {$J$} (m-2-2);
				\path				(m-1-1) edge[eq] (m-2-1)
										(m-1-2) edge[eq] (m-2-2);
				\path[transform canvas={xshift=1.75em}]	(m-1-1) edge[cell] (m-2-1);
			\matrix(m)[math35, column sep={2.3em,between origins}, xshift=8.5em]{I & & \dl A \tens A & & \dl A \tens B\\ & I & & \dl A \tens B & \\};
				\path[map]	(m-1-1) edge[barred] node[above] {$\iota$} (m-1-3)
										(m-1-3) edge[barred] node[above, inner sep=6pt] {$\dl A \tens \ul H$} (m-1-5)
										(m-2-2) edge[barred] node[below] {$\flad J$} (m-2-4);
				\path				(m-1-1) edge[eq] (m-2-2)
										(m-1-5) edge[eq] (m-2-4)
										(m-1-3) edge[cell] (m-2-3);
				
				\draw[font=\Large] (-16.6em,0) node {$\lbrace$}
										(-11.4em,0) node {$\rbrace$}
										(2.6em,0) node {$\lbrace$}
										(14.2em,0) node {$\rbrace$};
				\path[map]	(-9.4em,0) edge node[above] {$\cocart \of (\id_\iota, \dl A \tens \dash)$} (0.6em,0);
			\end{tikzpicture}
		\end{displaymath}
		If moreover for every $\hmap KI{\dl A \tens B}$ in $\K$ there exists a morphism $\hmap JAB$ with $\flad J \iso K$ then we call $\dl A$ the \emph{horizontal dual} of $A$.
	\end{definition}
	
	\begin{example} \label{horizontal dual of internal preorders}
		The horizontal dual $\dl A$ of an internal preorder $A = (A, \alpha)$ in $\ModRel(\E)$ (\exref{ModRel is monoidal}) is $\dl A \dfn (A, \rev\alpha)$, where $\rev\alpha = (A \xlar{\alpha_1} \alpha \xrar{\alpha_0} A)$ is the reverse of $\alpha$. Its horizontal copairing $\hmap\iota 1{\dl A \times A}$ is the modular subobject $\alpha \xrightarrow{(\alpha_0, \alpha_1)} A \times A$ (\exref{yoneda embedding for the terminal object in ModRel(E)}), whose right action is induced by the multiplication $\cell{\bar\alpha}{(\alpha, \alpha)}\alpha$ of $A$; similarly the adjunct $\hmap{\flad J}1{\dl A \times B}$ of an internal modular relation $\hmap JAB$ is $J$ itself considered as the modular subobject $\flad J \dfn \bigbrks{J \xrar{(j_0, j_1)} A \times B}$. Indeed consider the cell $(\iota, \dl A \times J) \Rar \flad J$ induced by the action of $A$ on $J$. It is a split epimorphism whose section is induced by the unit cell $\cell{\tilde\alpha}A\alpha$ of $A$, so that it is (pointwise) cocartesian in $\Rel(\E)$ (\augexref{7.4}) and hence pointwise right cocartesian by \exref{internal modular relations} and \remref{right pointwise cocartesian and pointwise right cocartesian comparison}; this proves condition (a) above. Recall that $\ModRel(\E)$ is locally thin (\exref{internal modular relations}) so that, to prove (b), it suffices to show that the assignments above are surjective. But that follows immediately from the fact that there exists a cell $H_1 \hc \dotsb \hc H_n \Rar \iota \hc (\dl A \times H_1) \hc \dotsb \hc (\dl A \times H_n)$ in $\Span\E$, induced by the unit cell $\tilde\alpha$, combined with the fact that the forgetful functor $\map U{\ModRel(\E)}{\Span\E}$ is locally full and faithful.
	\end{example}
	
	\subsection{Internal homs}
	If besides a weak horizontal dual $\dl A$ the Yoneda embedding $\map\yon I{\ps I}$ for the monoidal unit $I$ exists then we can compose the assignment $\flad{(\dash)}$ of \defref{horizontal dual} with the assignment $\cur{(\dash)}$ given by the Yoneda axiom (\defref{yoneda embedding}), thus obtaining a composite assignment of morphisms of the form
	\begin{displaymath}
		\set{A \brar B} \quad \xrar{\flad{(\dash)}} \quad \set{I \brar \dl A \tens B} \quad \xrar{\cur{(\dash)}} \quad \set{\dl A \tens B \to \ps I}
	\end{displaymath}
	which, using \propref{equivalence from yoneda embedding}, is essentially surjective if and only if $\dl A$ is a horizontal dual and $\yon$ admits nullary restrictions (\defref{yoneda embedding}). In the definition below we define to be `$\iota$"/small' those morphisms $\dl A \tens B \to \ps I$ in its essential image, besides defining the internal hom for such morphisms. Alternative conditions equivalent to that of $\iota$"/smallness are given in \lemref{iota-small axioms} below.
	
	\begin{definition} \label{internal iota-small hom}
		Let $\K = (\K, \tens, I)$ be a monoidal augmented virtual double category. Assume that the Yoneda embedding $\map\yon I{\ps I}$ exists and let $A$ be any object equipped with a weak horizontal dual $\dl A$, defined by a copairing $\hmap\iota I{\dl A \tens A}$.
		\begin{enumerate}[label=-]
			\item A morphism $\map f{\dl A \tens B}{\ps I}$ is called \emph{$\iota$"/small} if there exists a morphism $\hmap JAB$ such that $f \iso \cur{(\flad J)}$.
			\item We denote by $(\dl A \tens \dash \vs \ps I)_\iota \subseteq \dl A \tens \dash \vs \ps I$ (\defref{universal vertical morphism}) the full subcategory generated by all $\iota$"/small morphisms $\dl A \tens B \to \ps I$.
			\item The \emph{internal $\iota$"/small hom} $\iotahom$, if it exists, is an object $\iotahom$ of $\K$ equipped with a universal morphism $\map\ev{\dl A \tens \iotahom}{\ps I}$ from $\dl A \tens \dash$ to $\ps I$ relative to $(\dl A \tens \dash \vs \ps I)_\iota$ (\defref{universal vertical morphism}); $\ev$ is called \emph{evaluation}.
		\end{enumerate}
	\end{definition}
	
	The next result follows from the discussion preceding the definition; it can also be obtained by instantiating the functor $F$ of \propref{universality and iota-smallness} below by $\dl A \tens \dash$.
	\begin{proposition} \label{iota-smallness and the horizontal dual}
		Let $\map\yon I{\ps I}$ and $\hmap\iota I{\dl A \tens A}$ be as above. Every morphism of the form $\dl A \tens B \to \ps I$ is $\iota$"/small if and only if $\dl A$ is the horizontal dual of $A$ (\defref{horizontal dual}) and $\yon$ admits nullary restrictions (\defref{yoneda embedding}).
	\end{proposition}
	Assume that every morphism of the form $\dl A \tens B \to \ps I$ is $\iota$"/small, that is we have \mbox{$(\dl A \tens \dash \vs \ps I)_\iota = \dl A \tens \dash \vs \ps I$} in \defref{internal iota-small hom}, and that the functor $\map{\dl A \tens \dash}\K\K$ admits a right adjoint $\brks{\dl A, \dash}$. It then follows from \exref{counit components are universal} that the evaluation \mbox{$\map{\eps_{\ps I}}{\dl A \tens \brks{\dl A, \ps I}}{\ps I}$} defines $\brks{\dl A, \ps I}$ as the internal $\iota$"/small hom $\iotahom$. We are thus interested in monoidal augmented virtual double categories that are `closed' in the following sense.
	\begin{definition} \label{closed monoidal augmented virtual double categories}
		Let $\K = (\K, \tens, I)$ be a (cartesian) monoidal augmented virtual double category (\defref{monoidal augmented virtual double category}). We call $\K$ \emph{closed (cartesian) monoidal} if, for every unital object $A \in \K$, the functor $\map{A \tens \dash}\K\K$ admits a right adjoint $\inhom{A, \dash}$ in the $2$"/category $\AugVirtDblCat$ (\augsecref 3).
	\end{definition}
	\begin{example}
		Let $\K$ be a cartesian pseudo double category (\exref{precartesian pseudo double category}). Using that the embedding $\map N\nlDblCat\AugVirtDblCat$ of $2$"/categories of \exref{monoidal pseudo double category} is full and faithful, it follows that $N(\K)$ is closed cartesian monoidal in the above sense if and only if $\K$ is \emph{precartesian closed} in the sense of Definition~4.2 of \cite{Niefield20} such that, for each $A \in \K$, the right adjoint $(\dash)^A$ to $\map{\dash \times A}\K\K$ is a normal lax functor (see e.g.\ Definition~6.1 of \cite{Shulman08}).
	\end{example}
	
	\begin{example} \label{internal enriched hom}
		Let $\V' = (\V', \tens', I')$ be a closed symmetric monoidal category that is large complete and let $A$ be any (large) $\V'$"/category. Recall from Sections~2.2 and 2.3 of \cite{Kelly82} that, restricted to the $2$"/category $\enCat{\V'} = V(\enProf{\V'})$ of $\V'$"/categories, the endo"/$2$"/functor $A \tens' \dash$ (\exref{V-Prof is monoidal}) admits a right adjoint $\map{\inhom{A, \dash}'}{\enCat{\V'}}{\enCat{\V'}}$, whose image $\brks{A, C}'$ of a $\V'$"/category $C$ is the $\V'$"/category of $\V'$"/functors $\map pAC$.
		
		It is straightforward to check that the adjoint functor pairs $A \tens' \dash \ladj \inhom{A, \dash}'$ on $\enCat{\V'}$ extend to adjoint functor pairs $A \tens' \dash \ladj \inhom{A, \dash}'$ on $\enProf{\V'}$, thus making $\enProf{\V'}$ into a closed monoidal unital virtual equipment. The extension of $\inhom{A, \dash}'$ maps a $\V'$"/profunctor $\hmap KCD$ to the $\V'$"/profunctor $\hmap{\inhom{A, K}'}{\inhom{A, C}'}{\inhom{A, D}'}$ given by the ends
		\begin{displaymath}
			\inhom{A, K}'(p, q) \dfn \int_{x \in A} K(px, qx) \qquad\qquad\qquad (p \in \inhom{A, C}', q \in \inhom{A, D}'),
		\end{displaymath}
		where $K$ is regarded as a $\V'$"/profunctor $\map K{\op C \tens' D}{\V'}$; these ends exist because $\V'$ is assumed to be large complete. Notice that if $K = I_C$ is a unit $\V'$"/profunctor then $\inhom{A, I_C}'$ consists of the hom objects of the functor $\V'$"/category $\brks{A, C}'$.
		
		Taking $\V' = \Set'$ in previous we recover Corollary~4.8 of \cite{Niefield20}, which proves that the pseudo double category $\enProf{\Set'}$ (\augexref{2.4}) of $\Set'$"/profunctors between (locally large) categories is precartesian closed; see the previous example.
	\end{example}
	
	\begin{example}	\label{ModRel(E) is closed cartesian monoidal}
		Let $\E$ be a cartesian closed category with pullbacks. As is shown in Section~2 of \cite{Carboni-Street86}, the locally thin $2$"/category $\PreOrd(\E) = V\bigpars{\ModRel(\E)}$ of internal preorders in $\E$ (\exref{internal modular relations}) inherits a closed cartesian monoidal structure from that of $\E$ as follows. The exponential $\inhom{A, C}$ of internal preorders $A = (A, \alpha)$ and $C = (C, \gamma)$ is constructed as the pullback on the left below, and its internal ordering is obtained by pulling the internal ordering $(C^A \xlar{\gamma_0^A} \gamma^A \xrar{\gamma_1^A} C^A)$ on $C^A$ back along \mbox{$\inhom{A, C} \times \inhom{A, C} \rightarrowtail C^A \times C^A$}. Notice that $\inhom{A, C}$ is an internal partial order (\exref{internal modular relations}) whenever $C$ is so and that $\inhom{A, C} \iso C^A$ whenever $A = (A, I_A)$ is discrete. Underlying the order preserving evaluation $\map\eps{A \times \inhom{A, C}}C$ in $\PreOrd(\E)$ is the composite $A \times \inhom{A, C} \rightarrowtail A \times C^A \xrar\eps C$ in $\E$.
		\begin{displaymath}
			\begin{tikzpicture}[baseline]
				\matrix(m)[math35, column sep={5.5em,between origins}]{\inhom{A, C} & C^A \\ \gamma^\alpha & C^\alpha \times C^\alpha \\};
				\path[map]	(m-1-1)	edge[mono] (m-1-2)
														edge (m-2-1)
										(m-1-2) edge node[right] {$(C^{\alpha_0}, C^{\alpha_1})$} (m-2-2)
										(m-2-1) edge[mono] node[below] {$(\gamma_0^\alpha, \gamma_1^\alpha)$} (m-2-2);
				\coordinate (hook) at ($(m-1-1)+(0.6,-0.4)$);
				\draw (hook)+(0,0.17) -- (hook) -- +(-0.17,0);
			\end{tikzpicture} \qquad \qquad \qquad \begin{tikzpicture}[baseline]
				\matrix(m)[math35, column sep={7.0em,between origins}]{\inhom{A, K} & K^A \\ \inhom{A, C} \times \inhom{A, D} & C^A \times D^A \\};
				\path[map]	(m-1-1)	edge[mono] (m-1-2)
														edge[mono] node[left] {$\bigpars{\inhom{A, K}_0, \inhom{A, K}_1}$} (m-2-1)
										(m-1-2) edge[mono] node[right] {$(k_0^A, k_1^A)$} (m-2-2)
										(m-2-1) edge[mono] (m-2-2);
				\coordinate (hook) at ($(m-1-1)+(0.6,-0.4)$);
				\draw (hook)+(0,0.17) -- (hook) -- +(-0.17,0);
			\end{tikzpicture}
		\end{displaymath}
		
		It is straightforward to show that the adjoint functor pairs $A \times \dash \ladj \inhom{A, \dash}$ on $\PreOrd(\E)$ extend to adjoint functor pairs $A \times \dash \ladj \inhom{A, \dash}$ on $\ModRel(\E)$, thus making $\ModRel(\E)$ into a closed cartesian monoidal unital virtual equipment. The image $\hmap{\inhom{A, K}}{\inhom{A, C}}{\inhom{A, D}}$ of an internal modular relation $\hmap KCD$ is given by the left morphism in the pullback square on the right above.
	\end{example}
	
	\subsection{Yoneda embeddings and internal homs}
	The following theorem is the main result of this section. As promised in its introduction, for an object $A$ in a monoidal augmented virtual double category $\K = (\K, \tens, I)$ it relates the existence of a Yoneda embedding $\map{\yon_A}A{\ps A}$ to that of the internal $\iota$"/small hom $\iotahom$.
	\begin{theorem} \label{presheaf object equivalent to iota-small internal hom}
	 	In a monoidal augmented virtual double category $\K = (\K, \tens, I)$ that has restrictions on the right (\defref{augmented virtual equipment}) let $\map\yon I{\ps I}$ be a Yoneda embedding (\defref{yoneda embedding}). Let $A$ be a unital object (\defref{cartesian cells}) whose weak horizontal dual $\dl A$ exists, with copairing $\hmap\iota I{\dl A \tens A}$ (\defref{horizontal dual}), such that the functor $\map{\dl A \tens \dash}\K\K$ preserves restrictions on the right. Consider any object $P$ equipped with morphisms $\map{\yon_A}AP$ and $\map\ev{\dl A \tens P}{\ps I}$. If the companion $\hmap{\yon_{A*}}AP$ exists then the following are equivalent:
	 	\begin{enumerate}
	 		\item[\textup{(ye)}] $\map{\yon_A}AP$ is a Yoneda embedding that defines $P$ as the object of presheaves on $A$ (\defref{yoneda embedding}) and there exists a cartesian cell as on the left below;
	 		\item[\textup{(ih)}] $\map\ev{\dl A \tens P}{\ps I}$ is the evaluation that defines $P$ as the internal $\iota$"/small hom $\iotahom$ (\defref{internal iota-small hom}) and there exists a cartesian cell as on the right below.
	 	\end{enumerate}
	 	\begin{displaymath}
			\begin{tikzpicture}[baseline]
				\matrix(m)[math35, column sep={1.8em,between origins}]{I & & \dl A \tens P \\ & \ps I & \\};
				\path[map]	(m-1-1) edge[barred] node[above] {$\flad \yon_{A*}$} (m-1-3)
														edge[transform canvas={xshift=-2pt}] node[left] {$\yon$} (m-2-2)
										(m-1-3) edge[transform canvas={xshift=2pt}] node[right] {$\ev$} (m-2-2);
				\draw				([yshift=0.333em]$(m-1-2)!0.5!(m-2-2)$) node[font=\scriptsize] {$\cart$};
			\end{tikzpicture} \qquad \qquad \qquad \qquad \qquad \qquad \begin{tikzpicture}[baseline]
				\matrix(m)[math35, column sep={0.9em,between origins}]
					{I & & & & \dl A \tens A \\ & & & \mspace{11mu} \dl A \tens P & \\ & & \ps I & & \\};
				\path[map]	(m-1-1) edge[barred] node[above] {$\iota$} (m-1-5)
														edge[transform canvas={xshift=-1pt}, ps] node[left] {$\yon$} (m-3-3)
										(m-1-5) edge[transform canvas={xshift=1pt}, ps] node[right] {$\dl A \tens \yon_A$} (m-2-4)
										(m-2-4) edge[transform canvas={xshift=1pt}, ps] node[right] {$\ev$} (m-3-3);
				\draw				([yshift=1.333em]$(m-1-3)!0.5!(m-3-3)$) node[font=\scriptsize] {$\cart$};
			\end{tikzpicture}
		\end{displaymath}
	\end{theorem}
	
	For reference we summarise the main definitions used in the proof of the implication (ih) $\Rar$ (ye) above; for further details apply \defref{yoneda embedding from universal morphisms} below to the endofunctor $F \dfn \dl A \tens \dash$. Let $\K$, $\map\yon I{\ps I}$, $A$ and $\hmap\iota I{\dl A \tens A}$ be as in the theorem, and let $\map\ev{\dl A \tens \iotahom}{\ps I}$ be the evaluation defining the $\iota$"/small hom $\iotahom$ (\defref{internal iota-small hom}). Like any functor $\dl A \tens \dash$ preserves the horizontal unit $I_A$ (\augcororef{5.5}), that is $\dl A \tens I_A \iso I_{\dl A \tens A}$. It follows that $\flad I_A \iso \iota$ so that $\map{\cur{(\flad I_A)} \iso \cur\iota}{\dl A \tens A}{\ps I}$ by the functoriality of $\cur{(\dash)}$ (\propref{equivalence from yoneda embedding}); hence $\cur\iota$ is $\iota$"/small. Using the universality of the evaluation we obtain a morphism $\map{\yon_A}A\iotahom$ such that $\ev \of (\dl A \tens \yon_A) \iso \cur\iota$. Composing the latter isomorphism with the cartesian cell defining $\cur\iota$ (\defref{yoneda embedding}) we obtain a nullary cartesian cell $\iota \Rar \ps I$ as on the right above, so that condition~(ih) above is satisfied. Next assume that the companion $\yon_{A*}$ exists so that we can apply the theorem; we conclude that $\yon_A$ forms a Yoneda embedding. The existence of $\yon_{A*}$ also means that $\yon_A$ has nullary restrictions (\defref{yoneda embedding}) so that it induces an equivalence between horizontal morphisms $A \brar B$ and vertical morphisms $B \to \iotahom$ (\propref{equivalence from yoneda embedding}). Under this equivalence a horizontal morphism $\hmap JAB$ corresponds to the vertical morphism $\map{\cur J \dfn \shad{\bigpars{\cur{(\flad J)}}}}B{\iotahom}$, that is $\cur{(\flad J)} \iso \ev \of \bigpars{\dl A \tens \cur J}$ where $\map{\cur{(\flad J)}}{\dl A \tens B}{\ps I}$ is induced by $\hmap{\flad J}I{\dl A \tens B}$, using the Yoneda axiom for $\map\yon I{\ps I}$ (\defref{yoneda embedding}).
	
	\begin{proof}[of \thmref{presheaf object equivalent to iota-small internal hom}]
		As described in the introduction to this section, instead of giving a direct proof we will prove a generalisation of the present theorem in \thmsref{universal morphism from a yoneda embedding}{yoneda embedding from universal morphisms theorem} below. Applied to $\map{F = \dl A \tens \dash}\K\K$ the latter imply the equivalence (ye)~$\Leftrightarrow$~(ih) above as follows. Applying \lemref{simplified locally universal axiom} to the copairing $\hmap\iota I{\dl A \tens A}$ shows that $\iota$ is a locally universal morphism from $I$ to the endofunctor $F \dfn \dl A \tens \dash$, in the sense of \defref{universal horizontal morphism} below. Assuming (ye), so that $\map{\yon_A}AP$ is a Yoneda embedding, we can use \thmref{universal morphism from a yoneda embedding} to obtain an evaluation $\map{\eps}{FP}{\ps I}$ that defines $P$ as the internal $\iota$"/small hom $\iotahom$. From the definition of $\eps$, the existence of the cartesian cell on the left above and the uniqueness of the cartesian cells provided by the Yoneda axiom (\defref{yoneda embedding}), we conclude that $\eps \iso \ev$. Using that $A$ is assumed to be unital, \thmref{universal morphism from a yoneda embedding} also supplies the cartesian cell on the right above. This proves (ye)~$\Rightarrow$~(ih). For the converse (ih)~$\Rightarrow$~(ye) notice that, similarly to the previous argument, the existence of the cartesian cell on the right above implies that $\ev \of (\dl A \tens \yon_A) \iso \cur\iota$. It follows that we can take $\shad{(\cur\iota)} \dfn \yon_A$ in \defref{yoneda embedding from universal morphisms}, so that the (ih) $\Rightarrow$ (ye) follows from \thmref{yoneda embedding from universal morphisms theorem}.
	\end{proof}
	
	\begin{remark}
		The theorem above is reminiscent of Weber's construction of a good Yoneda structure on any $2$"/topos $\mathcal C = (\mathcal C, \dl{(\dash)}, \tau)$, as given in Section~5 of \cite{Weber07}; see also \exref{yoneda embeddings in 2-topoi} above. Like the Yoneda embeddings $\map{\yon_A}A{\iotahom}$ obtained above, the Yoneda embeddings $\map{y_A}A{\brks{\dl A, \Omega}}$ of Weber's Yoneda structure on $\mathcal C$ also map into an inner hom-object, where $\Omega$ is the target of the classifying discrete opfibration $\tau$. Another similarity is the importance of the equivalence of the discrete two"/sided fibrations $A \brar B$ in $\mathcal C$ and those of the form $1 \brar \dl A \times B$, analogous to our notion of horizontal dual (\defref{horizontal dual}).
		
		Differences include our construction requiring the monoidal unit $I$ to be unital, while Weber's construction applies regardless of whether the cartesian unit $1$ is admissible in $\mathcal C$; see Example~8.3 of \cite{Weber07}. Weber's construction applies to $2$"/topoi only; unlike our approach it cannot, for instance, be used to obtain enriched Yoneda embeddings (see \exref{no good yoneda structure on (Cat, Cat')-Prof} and \exref{enriched yoneda embedding}). Our result, finally, generalises in the form of \thmsref{universal morphism from a yoneda embedding}{yoneda embedding from universal morphisms theorem} below, as described in the introduction to this section.
	\end{remark}
	
	\begin{example} \label{enriched yoneda embedding}
		Let $\V \subset \V'$ be a symmetric universe enlargement (\exref{enriched yoneda embedding summary}) and consider the Yoneda embedding $\map\yon I\V$ (\exref{yoneda embedding for unit V-category}) in the monoidal augmented virtual equipment $\enProf{(\V, \V')}$ (\exref{V-Prof is monoidal}). Let $A$ be a $\V$"/category and let the horizontal dual $\dl A \dfn \op A$ and the copairing $\hmap\iota I{\op A \tens' A}$ in $\enProf{(\V, \V')}$ be as defined preceding \defref{horizontal dual}. Recall that we can take the image $\hmap{\flad J}I{\op A \tens' B}$ of a $\V$"/profunctor $\hmap JAB$, under the assignment $J \mapsto \flad J$ of \defref{horizontal dual}, to be given by \mbox{$\flad J\bigpars{*, (x, y)} \dfn J(x, y)$} for all $x \in A$ and $y \in B$. By \propref{iota-smallness and the horizontal dual} any $\V'$"/functor $\op A \tens' B \to \V$ is $\iota$"/small (\defref{internal iota-small hom}).
		
		Consider the adjunction $\op A \tens' \dash \ladj \inhom{\op A, \dash}'$ of endofunctors on $\enProf{\V'}$ of \exref{internal enriched hom} and notice that, because $A$ is a $\V$"/category, $\op A \tens' \dash$ restricts to an endofunctor on $\enProf{(\V, \V')}$. As discussed before \defref{closed monoidal augmented virtual double categories} the component $\map{\eps_\V}{\op A \tens' \inhom{\op A, \V}'}\V$ of the evaluation defines $\inhom{\op A, \V}'$ as the internal $\iota$"/small hom in $\enProf{\V'}$. Moreover, while $\inhom{\op A, \dash}'$ is unlikely to restrict along the embedding $\enProf{(\V, \V')} \subset \enProf{\V'}$, the fact that this embedding is full (\augdefref{3.6}) implies that $\eps_\V$ is universal too from $\op A \tens' \dash$, as an endofunctor on $\enProf{(\V,\V')}$, to $\V$, which follows easily from the unpacking of \defref{universal vertical morphism}. We conclude that  $\map{\ev \dfn \eps_\V}{\op A \tens' \brks{\op A, \V}'}\V$ defines the functor"/$\V'$"/category $\brks{\op A, \V}'$ as an internal $\iota$"/small hom in $\enProf{(\V, \V')}$ too, in the sense of \defref{internal iota-small hom}.
		
		The universality of $\ev$ induces a $\V'$"/functor $\map{\yon_A}A{\brks{\op A, \V}'}$ as described after the statement of \thmref{presheaf object equivalent to iota-small internal hom}. Using that $\map{\cur\iota}{\op A \tens' A}\V$ is given by $\cur\iota (x,y) = A(x,y)$, and that $\yon_A \dfn \shad{(\cur\iota)} = \brks{\op A, \cur\iota}' \of \coev_A$ by \exref{counit components are universal}, where \mbox{$\map{\coev_A}A{\inhom{\op A, \op A \tens' A}'}$} is the coevaluation, it follows that $\map{\yon_A}A{\brks{\op A, \V}'}$ recovers the classical enriched Yoneda embedding given by $\yon_A x = A(\dash, x)$; see e.g.\ Section~2.4 of \cite{Kelly82}.
		
		Finally notice that the companion $\hmap{\yon_{A*}}A{\brks{\op A, \V}'}$ exists in $\enProf{(\V, \V')}$: this follows from the well known isomorphisms $\yon_*(x, p) = \brks{\op A, \V}'(\yon_A x, p) \iso px \in \V$, for any $x \in A$ and $p \in \brks{\op A, \V}'$, that are given by the `strong Yoneda lemma'; see e.g.\ Formula~2.31 of \cite{Kelly82}. Applying \thmref{presheaf object equivalent to iota-small internal hom} we conclude that, for each $\V$"/category $A$, the classical enriched Yoneda embedding \mbox{$\map{\yon_A}A{\brks{\op A, \V}'}$} forms a Yoneda embedding that admits nullary restrictions in the augmented virtual equipment $\enProf{(\V, \V')}$, in the sense of \defref{yoneda embedding}.
	\end{example}
	
	\begin{example} \label{yoneda embeddings for internal preorders summary}
		Let $\E$ be a cartesian closed category with finite limits. Let $\map\yon 1{\ps 1}$ be a Yoneda embedding in $\ModRel(\E)$ (\exref{yoneda embedding for the terminal object in ModRel(E)}) and let $A = (A, \alpha)$ be any internal preorder in $\E$. The modular subobject $\alpha$ of $\dl A \times A$, that defines the horizontal dual $\dl A$ (\exref{horizontal dual of internal preorders}), corresponds to an order preserving morphism $\dl A \times A \to \ps 1$ (\exref{yoneda embedding for the terminal object in ModRel(E)}) which, under the closed cartesian monoidal structure on $\ModRel(\E)$ (\exref{ModRel(E) is closed cartesian monoidal}), corresponds to an order preserving morphism \mbox{$\map{\yon_A}A{\inhom{\dl A, \ps 1}}$}. That condition (ih) of the theorem above holds follows from \propref{iota-smallness and the horizontal dual} and the discussion following it, so that $\yon_A$ forms a Yoneda embedding in $\ModRel(\E)$ by condition (ye).
	\end{example}
	
	\subsection{Universal horizontal morphisms}
	In the remainder of this section we state and prove a generalisation of \thmref{presheaf object equivalent to iota-small internal hom}, as described in the introduction to this section. We start by generalising the notion of horizontal dual (\defref{horizontal dual}) as a horizontal variant of the notion of locally universal vertical morphism (\defref{universal vertical morphism}).
	\begin{definition} \label{universal horizontal morphism}
		Let $\map F\K\L$ be a functor of augmented virtual double categories (\augdefref{3.1}) and let $A \in \L$ be an object. We call a morphism $\hmap\iota A{FA'}$ \emph{locally universal from $A$ to $F$} if the following conditions are satisfied:
		\begin{enumerate}[label=\textup{(\alph*)}]
			\item for each $\hmap J{A'}B$ the pointwise right composite (\defref{pointwise right cocartesian path}) below exists, called the \emph{adjunct} of $J$;
			\begin{displaymath}
				\hmap{\flad J \dfn \iota \hc FJ}A{FB}
			\end{displaymath}
			\item for each $\map hX{A'}$ the restriction $\iota(\id, Fh)$ exists;
			\item the assigments below, between collections of cells in $\K$ and cells in $\L$ as shown and where the cells $\cocart$ and $\cart$ define $\flad J$ and $\iota(\id, Fh)$, are bijections.
		\end{enumerate}
		\begin{displaymath}
			\begin{tikzpicture}
				\matrix(m)[math35, xshift=-14em]{X_0 & X_n \\ A' & B \\};
				\path[map]	(m-1-1) edge[barred] node[above] {$\ul H$} (m-1-2)
														edge node[left] {$h$} (m-2-1)
										(m-1-2) edge node[right] {$k$} (m-2-2)
										(m-2-1) edge[barred] node[below] {$J$} (m-2-2);
				\path[transform canvas={xshift=1.75em}]	(m-1-1) edge[cell] (m-2-1);
			\matrix(m)[math35, column sep={2em,between origins}, xshift=9em]{A & & FX_0 & & FX_n\\ & A & & FB & \\};
				\path[map]	(m-1-1) edge[barred] node[above, inner sep=4pt, xshift=2pt] {$\iota(\id, Fh)$} (m-1-3)
										(m-1-3) edge[barred] node[above, inner sep=4pt] {$F\ul H$} (m-1-5)
										(m-1-5) edge node[right] {$Fk$} (m-2-4)
										(m-2-2) edge[barred] node[below] {$\flad J$} (m-2-4);
				\path				(m-1-1) edge[eq] (m-2-2)
										(m-1-3) edge[cell] (m-2-3);
				
				\draw[font=\Large] (-17em,0) node {$\lbrace$}
										(-11em,0) node {$\rbrace$}
										(3.8em,0) node {$\lbrace$}
										(14.0em,0) node {$\rbrace$};
				\path[map]	(-8.5em,0) edge node[above] {$\cocart \of (\cart, F\dash)$} (1.3em,0);
			\end{tikzpicture}
		\end{displaymath}
		If moreover every $\hmap HA{FB}$ in $\L$ admits a morphism $\hmap J{A'}B$ in $\K$ with $\flad J \iso H$ then we say that $\hmap\iota A{FA'}$ is \emph{universal from $A$ to $F$}. 
	\end{definition}
	
	The assignment $J \mapsto \flad J$ preserves restrictions as follows.
	\begin{lemma} \label{adjuncts and restrictions}
		Consider $\hmap J{A'}B$ and $\map kYB$ such that the restrictions $J(\id, k)$ and $\flad J(\id, Fk)$ exist. If $J(\id, k)$ is preserved by $F$ then $\flad{\bigpars{J(\id, k)}} \iso \flad J(\id, Fk)$.
	\end{lemma}
	\begin{proof}
		Write $\cart_{J(\id, k)}$, $\cart_{\flad J(\id, Fk)}$ and $\cocart_{\flad J}$ for the (co)cartesian cells defining $J(\id, k)$, $\flad J(\id, Fk)$ and $\flad J$. Since $\flad J$ is a pointwise right composite, by \defref{pointwise right cocartesian path} there exists a cocartesian cell $\cell{\cocart_{\flad J(\id, Fk)}}{(\iota, FJ(\id, k))}{\flad J(\id, Fk)}$ unique such that
		\begin{displaymath}
			\cocart_{\flad J} \of (\id_\iota, F\cart_{J(\id, k)}) = \cart_{\flad J(\id, Fk)} \of \cocart_{\flad J(\id, Fk)},
		\end{displaymath}
		so that the assertion follows from the uniqueness of horizontal composites.
	\end{proof}
	
	Under mild conditions condition (c) above simplifies as follows; compare the condition satisfied by a locally universal vertical morphism (\defref{universal vertical morphism}).
	\begin{lemma} \label{simplified locally universal axiom}
		If $A'$ is unital and both $\K$ and $\L$ admit restrictions on the right, with those of $\K$ preserved by $F$, then condition \textup{(c)} of \defref{universal horizontal morphism} simplifies to the assignments below, between collections of horizontal cells in $\K$ and in $\L$ and where $\cocart$ defines $\flad L$, being bijections.
		\begin{displaymath}
			\begin{tikzpicture}
				\matrix(m)[math35, xshift=-16em]{A' & B \\ A' & B \\};
				\path[map]	(m-1-1) edge[barred] node[above] {$\ul K$} (m-1-2)
										(m-2-1) edge[barred] node[below] {$L$} (m-2-2);
				\path				(m-1-1) edge[eq] (m-2-1)
										(m-1-2) edge[eq] (m-2-2);
				\path[transform canvas={xshift=1.75em}]	(m-1-1) edge[cell] (m-2-1);
			\matrix(m)[math35, column sep={1.75em,between origins}, xshift=4.0em]{A & & FA' & & FB\\ & A & & FB & \\};
				\path[map]	(m-1-1) edge[barred] node[above] {$\iota$} (m-1-3)
										(m-1-3) edge[barred] node[above, inner sep=6pt] {$F\ul K$} (m-1-5)
										(m-2-2) edge[barred] node[below] {$\flad L$} (m-2-4);
				\path				(m-1-1) edge[eq] (m-2-2)
										(m-1-5) edge[eq] (m-2-4)
										(m-1-3) edge[cell] (m-2-3);
				
				\draw[font=\Large] (-18.6em,0) node {$\lbrace$}
										(-13.4em,0) node {$\rbrace$}
										(-0.4em,0) node {$\lbrace$}
										(8.2em,0) node {$\rbrace$};
				\path[map]	(-10.9em,0) edge node[above] {$\cocart \of (\id_\iota, F\dash)$} (-2.9em,0);
			\end{tikzpicture}
		\end{displaymath}
	\end{lemma}
	\begin{proof}
		Since $A'$ is unital and $\K$ has restrictions on the right the conjoint $h^*$ of any morphism $\map h{X_0}{A'}$ exists by \augcororef{4.16}; we write $\cart_{h^*}$ for the defining cartesian cell. By \augcororef{5.5} the functor $F$ preserves $\cart_{h^*}$ and, because $\iota(\id, Fh)$ exists, by \auglemref{8.1} there exists a cocartesian horizontal cell $\cell{\cocart_{\iota(\id, Fh)}}{(\iota, Fh^*)}{\iota(\id, Fh)}$ such that
		\begin{displaymath}
			\id_\iota \hc F\cart_{h^*} = \cart_{\iota(\id, Fh)} \of \cocart_{\iota(\id, Fh)},
		\end{displaymath}
		where $\cart_{\iota(\id, Fh)}$ defines $\iota(\id, Fh)$. It follows that postcomposing the assignment of condition (c) of \defref{universal horizontal morphism} with the assignment $\dash \of (\cocart_{\iota(\id, Fh)}, \id, \dotsc, \id)$ coincides with precomposing the assignment below with $\cart_{h^*} \hc \dash$. Since the assignments given by $\cocart_{\iota(\id, Fh)}$ and $\cart_{h^*}$ are bijections, by \defref{cocartesian path} and by using the conjoint identities for $h^*$ (\lemref{companion identities lemma}) respectively, it follows that condition~(c) is equivalent to the assignments below being bijections.
		\begin{displaymath}
			\begin{tikzpicture}
				\matrix(m)[math35, column sep={1.75em,between origins}, xshift=-12em]{A' & & X_0 & & X_n\\ & A' & & B & \\};
				\path[map]	(m-1-1) edge[barred] node[above] {$h^*$} (m-1-3)
										(m-1-3) edge[barred] node[above] {$\ul H$} (m-1-5)
										(m-1-5) edge node[right] {$k$} (m-2-4)
										(m-2-2) edge[barred] node[below] {$J$} (m-2-4);
				\path				(m-1-1) edge[eq] (m-2-2)
										(m-1-3) edge[cell] (m-2-3);
				\matrix(m)[math35, xshift=12em]{A & FA' & FX_0 & FX_n \\ & A & FB & \\};
				\path[map]	(m-1-1) edge[barred] node[above] {$\iota$} (m-1-2)
										(m-1-2) edge[barred] node[above] {$Fh^*$} (m-1-3)
										(m-1-3) edge[barred] node[above] {$F\ul H$} (m-1-4)
										(m-1-4) edge node[below right] {$Fk$} (m-2-3)
										(m-2-2) edge[barred] node[below] {$\flad J$} (m-2-3);
				\path				(m-1-1) edge[eq] (m-2-2);
				\path[transform canvas={xshift=1.75em}]	(m-1-2) edge[cell] (m-2-2);
				
				\draw[font=\Large] (-16.4em,0) node {$\lbrace$}
										(-7.8em,0) node {$\rbrace$}
										(5.8em,0) node {$\lbrace$}
										(18.1em,0) node {$\rbrace$};
				\path[map]	(-5.5em,0) edge node[above] {$\cocart \of (\id, F\dash)$} (3.5em,0);
			\end{tikzpicture}
		\end{displaymath}
		
		Next it follows from the identity derived in the proof of the previous lemma that precomposing the assignment above with $\cart_{J(\id, k)} \of \dash$ coincides with postcomposing the assignment below, between horizontal cells in $\K$ and in $\L$, with the assignment $\cart_{\flad J(\id, Fk)} \of \dash$. Since vertical composition with cartesian cells gives bijective assignments we conclude that condition (c) is equivalent to the assignments below being bijections.
		\begin{displaymath}
			\map{\cocart_{\flad J(\id, Fk)} \of (\id, F\dash)}{\bigbrcs{h^* \conc \ul H \Rightarrow J(\id, k)}}{\bigbrcs{(\iota, Fh^*) \conc F\ul H \Rightarrow \flad J(\id, Fk)}}
		\end{displaymath}
		To complete the proof notice that, under the isomorphism $\flad J(\id, Fk) \iso \flad{\bigpars{J(\id, k)}}$, the assignments above are of the form as in the statement, with $\ul K = h^* \conc \ul H$ and $L = J(\id, k)$.
	\end{proof}
	
	\begin{definition} \label{iota-small}
		Let $\map F\K\L$ be a functor between augmented virtual double categories (\augdefref{3.1}) and let $\hmap\iota A{FA'}$ be a morphism in $\L$ that satisfies conditions (a) and (b) of \defref{universal horizontal morphism}. Given a Yoneda morphism $\map{\yon_A}A{\ps A}$ (\defref{yoneda embedding}) we make the following definitions.
		\begin{enumerate}[label=-]
			\item A morphism $\map f{FB}{\ps A}$ in $\L$ is called \emph{$\iota$"/small} if there exists $\hmap J{A'}B$ in $\K$ such that $f$ is the left Kan extension (\defref{left Kan extension}) of $\yon_A$ along $\hmap{\flad J}A{FB}$.
			\item We denote by $(F \vs \ps A)_\iota \subseteq F \vs \ps A$ (\defref{universal vertical morphism}) the full subcategory generated by all $\iota$"/small morphisms $FB \to \ps A$.
		\end{enumerate}
	\end{definition}
	
	\thmref{universal morphism from a yoneda embedding} below shows that, under mild conditions, the existence of a Yoneda morphism $\map{\yon_{A'}}{A'}{\ps{A'}}$ induces a universal morphism $\map\eps{F\ps{A'}}{\ps A}$ from $F$ to $\ps A$ relative to $(F \vs \ps A)_\iota$ (\defref{universal vertical morphism}). The lemma below lists conditions equivalent to $\iota$"/smallness: notice that the original notion above is condition (c) below for some $\hmap J{A'}B$. In \propref{universality and iota-smallness} we use this to relate the universality of $\iota$ (\defref{universal horizontal morphism}) to the $\iota$"/smallness of all morphisms in $\L$ of the form $FB \to \ps A$.
	\begin{lemma} \label{iota-small axioms}
		Let $\map F\K\L$, $\hmap\iota A{FA'}$ and $\map{\yon_A}A{\ps A}$ be as in \defref{iota-small}. The following conditions are equivalent for morphisms $\map f{FB}{\ps A}$ in $\L$ and $\hmap J{A'}B$ in $\K$ and, in that case, the left Kan extensions of conditions \textup{(c)}, \textup{(d)} and \textup{(e)} are pointwise (\defref{pointwise left Kan extension}).
		\begin{enumerate}[label=\textup{(\alph*)}]
			\item $f \iso \cur{(\flad J)}$ (\defref{yoneda embedding});
			\item $\hmap{\flad J}A{FB}$ (\defref{universal horizontal morphism}) is the restriction of $\ps A$ along $\yon_A$ and $f$;
			\item $f$ is the left Kan extension of $\yon_A$ along $\hmap{\flad J}A{FB}$;
			\item $f$ is the left Kan extension of $\yon_A$ along $(\iota, FJ)$;
			\item $f$ is the left Kan extension of $\map{\cur\iota}{FA'}{\ps A}$ along $FJ$.
		\end{enumerate}
	\end{lemma}
	\begin{proof}
		(a) $\Leftrightarrow$ (c) follows from the fact that $\cur{(\flad J)}$ is the left Kan extension of $\yon_A$ along $\flad J$, by the density of $\yon_A$ (\defref{density definition}), and the uniqueness of left Kan extensions. (b)~$\Leftrightarrow$~(c) follows from \lemref{weak left Kan extensions of yoneda embeddings}. (c) $\Leftrightarrow$ (d) follows from applying the vertical pasting lemma (\lemref{vertical pasting lemma}) to the pointwise right cocartesian cell $(\iota, FJ) \Rar \flad J$ and (d) $\Leftrightarrow$ (e) follows from applying the horizontal pasting lemma (\lemref{horizontal pasting lemma}) to the cartesian cell that defines $\cur\iota$ (\defref{yoneda embedding}), which is left Kan by density of $\yon_A$. Finally notice that $f \iso \cur{(\flad J)}$ is the pointwise left Kan extension of $\yon_A$ along $\flad J$ by density of $\yon_A$ so that, because $\flad J$ is defined by a pointwise right cocartesian cell and by the vertical and horizontal pasting lemmas, the previous arguments in fact obtain $f$ as a pointwise left Kan extension in each of the conditions (c), (d) and (e). 
	\end{proof}
	
	\begin{proposition} \label{universality and iota-smallness}
		Let $\map F\K\L$, $\hmap\iota A{FA'}$ and $\map{\yon_A}A{\ps A}$ be as in \defref{iota-small}. Every morphism in $\L$ of the form $FB \to \ps A$ is $\iota$"/small if and only if $\iota$ is a universal morphism from $A$ to $F$ (\defref{universal horizontal morphism}) and $\yon_A$ admits nullary restrictions (\defref{yoneda embedding}).
	\end{proposition}
	\begin{proof}
		Using the equivalence (b) $\Leftrightarrow$ (c) of the previous lemma, here we take $\iota$"/smallness of any morphism $\map f{FB}{\ps A}$ in $\L$ to mean the existence of a horizontal morphism \mbox{$\hmap J{A'}B$} in $\K$ such that $\flad J \iso \ps A(\yon_A, f)$. For the `if'"/part, assume that for any \mbox{$\map f{FB}{\ps A}$} the restriction $\hmap{\ps A(\yon_A, f)}A{FB}$ exists. Assuming moreover that $\iota$ is universal from $A$ to $F$ there exists $\hmap J{A'}B$ such that $\flad J \iso \ps A(\yon_A, f)$, proving that $f$ is $\iota$"/small. For the `only if'"/part assume that all morphisms of the form $\map f{FB}{\ps A}$ are $\iota$"/small. This immediately implies that $\yon_A$ admits all restrictions $\ps A(\yon_A, f)$. To show that $\iota$ is universal too let $\hmap HA{FB}$ be any horizontal morphism in $\L$. We have $H \iso \ps A(\yon_A, \cur H)$ by the Yoneda axiom (\defref{yoneda embedding}) so that, using that $\map{\cur H}{FB}{\ps A}$ is $\iota$"/small, there exists $\hmap J{A'}B$ such that $\flad J \iso \ps A(\yon_A, \cur H) \iso H$ as required.
	\end{proof}
	
	\subsection{Universal vertical morphisms from Yoneda morphisms}
	The following theorem generalises the implication (ye) $\Rightarrow$ (ih) of \thmref{presheaf object equivalent to iota-small internal hom}.
	\begin{theorem} \label{universal morphism from a yoneda embedding}
		Let $\map F\K\L$ be a functor of augmented virtual double categories (\augdefref{3.1}) and $\hmap\iota A{FA'}$ a locally universal morphism from $A$ to $F$ (\defref{universal horizontal morphism}). Consider Yoneda morphisms $\map{\yon_A}A{\ps A}$ and $\map{\yon_{A'}}{A'}{\ps{A'}}$ (\defref{yoneda embedding}) such that, for any $\map fB{\ps{A'}}$, the restrictions $\ps{A'}(\yon_{A'}, f)$ and $\flad{\yon_{A'*}}(\id, Ff)$ exist.
		\begin{displaymath}
			\begin{tikzpicture}[baseline]
				\matrix(m)[math35, column sep={1.75em,between origins}]{A & & F \ps{A'} \\ & \ps A & \\};
				\path[map]	(m-1-1) edge[barred] node[above] {$\flad \yon_{A'*}$} (m-1-3)
														edge[transform canvas={xshift=-2pt}] node[left] {$\yon_A$} (m-2-2)
										(m-1-3) edge[transform canvas={xshift=2pt}] node[right] {$\eps$} (m-2-2);
				\draw				([yshift=0.333em]$(m-1-2)!0.5!(m-2-2)$) node[font=\scriptsize] {$\cart$};
			\end{tikzpicture} \qquad \qquad \qquad \qquad \qquad \qquad \qquad \begin{tikzpicture}[baseline]
				\matrix(m)[math35, column sep={0.875em,between origins}]
					{A & & & & FA' \\ & & & \mspace{6mu} F\ps{A'} & \\ & & \ps A & & \\};
				\path[map]	(m-1-1) edge[barred] node[above] {$\iota$} (m-1-5)
														edge[transform canvas={xshift=-1pt}, ps] node[left] {$\yon_A$} (m-3-3)
										(m-1-5) edge[transform canvas={xshift=1pt}, ps] node[right] {$F\yon_{A'}$} (m-2-4)
										(m-2-4) edge[transform canvas={xshift=1pt}, ps] node[right, inner sep=1pt] {$\eps$} (m-3-3);
				\draw				([yshift=1.25em]$(m-1-3)!0.5!(m-3-3)$) node[font=\scriptsize] {$\cart$};
			\end{tikzpicture}
		\end{displaymath}
		
		If $F$ preserves all $\ps{A'}(\yon_{A'}, f)$ then the morphism $\map{\eps \dfn \cur{(\flad\yon_{A'*})}}{F\ps{A'}}{\ps A}$, that is given by the Yoneda axiom (\defref{yoneda embedding}) and that comes equipped with a cartesian cell as on the left above, is universal from $F$ to $\ps A$ relative to $(F \vs \ps A)_\iota$ (\defsref{universal vertical morphism}{iota-small}). If moreover $A'$ is unital then there exists a cartesian cell as on the right.
	\end{theorem}
	\begin{proof}
		The proof consists of three steps: we need to show that the functor (\defref{universal vertical morphism})
		\begin{displaymath}
			\map{\eps \of F\dash}{\K \vs \ps{A'}}{F \vs \ps A}
		\end{displaymath}
		(1) factors through $(F \vs \ps A)_\iota \hookrightarrow F \vs \ps A$ (\defref{iota-small}), (2) is essentially surjective onto $(F \vs \ps A)_\iota$ and (3) is full and faithful. To prove (1) consider any \mbox{$\map f B{\ps{A'}}$}; we have to prove that $\eps \of Ff$ is $\iota$"/small. To do so write $\cell{\cart'}{\ps{A'}(\yon_{A'}, f)}{\yon_{A'*}}$ for the unique factorisation of the cartesian cell defining $\ps{A'}(\yon_{A'}, f)$ through that defining the companion $\yon_{A'*}$. By the pasting lemma (\lemref{pasting lemma for cartesian cells}) $\cart'$ is the cartesian cell that defines $\ps{A'}(\yon_{A'}, f)$ as the restriction of $\yon_{A'*}$ along $f$. Since $F$ preserves the cartesian cells defining $\ps{A'}(\yon_{A'}, f)$ and $\yon_{A'*}$ it follows that it preserves $\cart'$, again by the pasting lemma. We can thus apply \lemref{adjuncts and restrictions} to find that
		\begin{displaymath}
			\ps A(\yon_A, \eps \of Ff) \iso \ps A(\yon_A, \eps)(\id, Ff) \iso \flad{\yon_{A'*}}(\id, Ff) \iso \flad{\pars{\yon_{A'*}(\id, f)}},
		\end{displaymath}
		where the first isomorphism follows from the pasting lemma and the second from the cartesian cell defining $\eps$, on the left above. Using \lemref{iota-small axioms}(b) we conclude that $\eps \of Ff$ is $\iota$"/small.
		
		In proving (2) and (3) we will use composites $\zeta$ of the form below, where \mbox{$\map fB{\ps{A'}}$} is any morphism and where the cell $\eta$ defines $\eps$ as the pointwise left Kan extension of $\yon_A$ along $(\iota, F\yon_{A'*})$. That $\eta$ exists follows the fact that $\eps$ is $\iota$"/small and \lemref{iota-small axioms}(d). Notice that the composite $\zeta$ is pointwise left Kan too, because $F\cart'$ is cartesian and by \defref{pointwise left Kan extension}. To prove the final assertion for the moment assume that $A'$ is unital so that both $\yon_{A'}$ and $F\yon_{A'}$ are full and faithful, by \lemref{full and faithful yoneda embedding} and \augcororef{5.15} respectively. Precomposing $\eta$ with the $F$"/image of the cocartesian cell that defines $\yon_{A'*}$ (\lemref{companion identities lemma}), which is again cocartesian by \augcororef{5.5}, we obtain a nullary cell of the form as on the right above, that is left Kan by \propref{pointwise left Kan extension along full and faithful map} and hence cartesian by \lemref{weak left Kan extensions of yoneda embeddings}, as required.
		
		\begin{displaymath}
			\zeta \dfn \quad \begin{tikzpicture}[textbaseline]
				\matrix(m)[math35]{A & FA' & FB\\ A & FA' & F\ps{A'}\\ & \ps A & \\};
				\path[map]	(m-1-1) edge[barred] node[above] {$\iota$} (m-1-2)
										(m-1-2) edge[barred] node[above, inner sep=6pt] {$F\ps{A'}(\yon_{A'}, f)$} (m-1-3)
										(m-1-3) edge[ps] node[right] {$Ff$} (m-2-3)
										(m-2-1) edge[barred] node[above] {$\iota$} (m-2-2)
														edge[transform canvas={xshift=-3pt, yshift=-2pt}] node[below left] {$\yon_A$} (m-3-2)
										(m-2-2) edge[barred] node[below] {$F\yon_{A'*}$} (m-2-3)
										(m-2-3) edge[transform canvas={xshift=3pt, yshift=-2pt}] node[below right] {$\eps$} (m-3-2);
				\path				(m-1-1) edge[eq] (m-2-1)
										(m-1-2) edge[eq] (m-2-2)
										(m-2-2) edge[cell] node[right] {$\eta$} (m-3-2);
				\draw				($(m-1-2)!0.5!(m-2-3)$) node[font=\scriptsize] {$F\cart'$};								
			\end{tikzpicture}
		\end{displaymath}
		
		To show (2) consider any $\iota$"/small morphism $\map g{FB}{\ps A}$. By \lemref{iota-small axioms}(d) $g$ is the pointwise left Kan extension of $\yon_A$ along $(\iota, FJ)$ for some $\hmap J{A'}B$. Taking \mbox{$\map{f \dfn \cur J}B{\ps{A'}}$}, as supplied by the Yoneda axiom (\defref{yoneda embedding}), in the composite above and composing it with $FJ \iso F\ps{A'}(\yon_{A'}, \cur J)$ we find that $\eps \of F\cur J$ too is the pointwise left Kan extension of $\yon_A$ along $(\iota, FJ)$. By uniqueness of Kan extensions we conclude that $g \iso \eps \of F\cur J$, which proves (2).
						
		To prove (3) consider the diagram of assignments between collections of cells in $\K$ and $\L$ below, where $\cart'$ is the cartesian cell $\ps{A'}(\yon_{A'}, f) \Rar \yon_{A'*}$; $\cocart$ denotes the cocartesian cell defining $\flad\yon_{A'*}$ (\defref{universal horizontal morphism}); $\zeta$ is the composite above and $\cart$ is the cartesian cell defining $\eps = \cur{(\flad{\yon_{A'*}})}$. That this diagram commutes follows from the equality $\eta = \cart \of \cocart$ in definition of $\zeta$ above (see the proof of (b) $\Rar$ (d) in \lemref{iota-small axioms}), the axioms satisfied by horizontal composition (\auglemref{1.3}) and the fact that $F$, like any functor of augmented virtual double categories, preserves horizontal compositions of cells.
		\begin{displaymath}
			\begin{tikzpicture}
				\matrix(m)[math35, column sep={1.75em,between origins}]{X_0 & & X_n \\ & \ps{A'} & \\};
				\path[map]	(m-1-1) edge[barred] node[above] {$\ul H$} (m-1-3)
														edge[transform canvas={xshift=-2pt}] node[left] {$f$} (m-2-2)
										(m-1-3) edge[transform canvas={xshift=2pt}, ps] node[right] {$g$} (m-2-2);
				\path[transform canvas={yshift=0.333em}]				(m-1-2) edge[cell] (m-2-2);
				
				\matrix(m)[math35, column sep={1.75em,between origins}, xshift=-14em]{A' & & X_0 & & X_n\\ & A' & & \ps{A'} & \\};
				\path[map]	(m-1-1) edge[barred] node[above, inner sep=6pt] {$\ps{A'}(\yon_{A'}, f)$} (m-1-3)
										(m-1-3) edge[barred] node[above] {$\ul H$} (m-1-5)
										(m-1-5) edge[ps] node[right] {$g$} (m-2-4)
										(m-2-2) edge[barred] node[below] {$\yon_{A'*}$} (m-2-4);
				\path				(m-1-1) edge[eq] (m-2-2)
										(m-1-3) edge[cell] (m-2-3);
				
				\matrix(m)[math35, column sep={0.875em,between origins}, xshift=13em]{FX_0 & & & & FX_n \\ & F\ps{A'}\mspace{18mu} & & \mspace{18mu}F\ps{A'} & \\ & & \ps A & & \\};
				\path[map]	(m-1-1) edge[barred] node[above] {$F\ul H$} (m-1-5)
														edge[transform canvas={xshift=-2pt}, ps] node[left] {$Ff$} (m-2-2)
										(m-1-5) edge[transform canvas={xshift=2pt}, ps] node[right] {$Fg$} (m-2-4)
										(m-2-2) edge[transform canvas={xshift=-2pt}] node[left] {$\eps$} (m-3-3)
										(m-2-4) edge[transform canvas={xshift=2pt}] node[right] {$\eps$} (m-3-3);
				\path[transform canvas={yshift=-0.8em}]				(m-1-3) edge[cell] (m-2-3);
				
				\matrix(m)[math35, column sep={3.5em,between origins}, yshift=-12em, xshift=-12em]{A & FA' & FX_0 & FX_n\\ & A & F\ps{A'} & \\};
				\path[map]	(m-1-1) edge[barred] node[above] {$\iota$} (m-1-2)
										(m-1-2) edge[barred] node[above, inner sep=6pt] {$F\ps{A'}(\yon_{A'}, f)$} (m-1-3)
										(m-1-3) edge[barred] node[above] {$F\ul H$} (m-1-4)
										(m-1-4) edge[transform canvas={xshift=2pt}, ps] node[below right] {$Fg$} (m-2-3)
										(m-2-2) edge[barred] node[below] {$\flad{\yon_{A'*}}$} (m-2-3);
				\path				(m-1-1) edge[eq] (m-2-2);
				\path[transform canvas={xshift=1.75em}]	(m-1-2) edge[cell] (m-2-2);
				
				\matrix(m)[math35, column sep={1.75em,between origins}, yshift=-12em, xshift=9em]{A & & FA' & & FX_0 & & FX_n\\ & & & & & & \\ & & & \ps A & & & \\};
				\draw				($(m-1-7)!0.5!(m-3-4)$) node(r)[font=\scriptsize] {$F\ps{A'}$};
				\path[map]	(m-1-1) edge[barred] node[above] {$\iota$} (m-1-3)
														edge[transform canvas={xshift=-2pt}] node[below left] {$\yon_A$} (m-3-4)
										(m-1-3) edge[barred] node[above, inner sep=6pt] {$F\ps{A'}(\yon_{A'}, f)$} (m-1-5)
										(m-1-5) edge[barred] node[above] {$F\ul H$} (m-1-7)
										(m-1-7) edge[transform canvas={xshift=2pt}] node[below right] {$Fg$} (r)
										(r) edge[transform canvas={xshift=2pt}] node[below right] {$\eps$} (m-3-4);
				\path[transform canvas={yshift=-0.7em}]	(m-1-4) edge[cell] (m-2-4);
				
				\draw[font=\LARGE]	(-2.7em,0) node {$\lbrace$}
										(2.5em,0) node {$\rbrace$}
										(-9.9em,0) node {$\rbrace$}
										(-18.3em,0) node {$\lbrace$}
										(16.3em,0em) node {$\rbrace$}
										(9.8em,0em) node {$\lbrace$}
										(-6.0em,-12em) node {$\rbrace$}
										(-18.1em,-12em) node {$\lbrace$}
										(14.9em,-12em) node {$\rbrace$}
										(2.8em,-12em) node {$\lbrace$};
				\path[map]	(-4.7em,0) edge node[above] {$\cart' \hc \dash$} (-7.9em,0)
										(4.5em,0) edge node[above] {$\eps \of F\dash$} (7.8em, 0)
										(-13.0em,-3.25em) edge node[below left, yshift=4pt] {$\cocart \of (\id, F\dash)$} (-11.75em,-7.75em)
										(11.0em,-4.25em) edge node[right, yshift=-1pt] {$\zeta \hc \dash$} (10.0em,-6.25em)
										(-4em,-12em) edge node[below] {$\cart \of \dash$} (0.8em,-12em);
										
			\end{tikzpicture}
		\end{displaymath}
		Notice that showing (3), that is $\eps \of F\dash$ is full and faithful, amounts to showing that the assignment $\eps \of F\dash$ in the right leg above is a bijection. Since $\zeta$ is left Kan the assignment $\zeta \hc \dash$ is a bijection; hence it suffices to prove that the composite assignment of the left leg above, and thus that of the right leg as well, is a bijection. In the left leg $\cart \of \dash$ is a bijection by definition (\defref{cartesian cells}) and \mbox{$\cocart \of (\id, F\dash)$} is a bijection by \defref{universal horizontal morphism}(c) (where we take $h = \id_{A'}$), which leaves the top assignment $\cart' \hc \dash$.
		\begin{displaymath}
			\begin{tikzpicture}
				\matrix(m)[math35]{A' & X_0 & X_n \\ & \ps{A'} & \\};
				\path[map]	(m-1-1) edge[barred] node[above, inner sep=6pt] {$\ps{A'}(\yon_{A'}, f)$} (m-1-2)
														edge[transform canvas={xshift=-2pt}] node[below left] {$\yon_{A'}$} (m-2-2)
										(m-1-2) edge[barred] node[above] {$\ul H$} (m-1-3)
										(m-1-3) edge[transform canvas={xshift=2pt}] node[below right] {$g$} (m-2-2);
				\path[transform canvas={yshift=0.25em}]				(m-1-2) edge[cell] (m-2-2);
			\end{tikzpicture}
		\end{displaymath}
		 To see that it too is a bijection notice that $\cart \of (\cart' \hc \dash) = (\cart \of \cart') \hc \dash$, where $\cart$ denotes the cartesian cell defining $\yon_{A'*}$, is a bijective assignment onto the collection of nullary cells of the form above, because the cartesian cell $\cart \of \cart'$ is left Kan by the density of $\yon_{A'}$ (\defref{density definition}). Since $\cart \of \dash$ is a bijection too, from the collection of cells in the top left in the diagram above onto that of the cells of the form above, we conclude that $\cart' \hc \dash$ is a bijection as well. This concludes the proof.
	\end{proof}
	
	\subsection{Lifting Yoneda embeddings along universal morphisms}
	The final theorem of this section --- \thmref{yoneda embedding from universal morphisms theorem} below --- is, roughly, a converse to the previous theorem: given a functor $\map F\K\L$ and, in $\L$, a locally universal morphism $\hmap\iota A{FA'}$, a Yoneda morphism $\map{\yon_A}A{\ps A}$ and a universal morphism $\map\eps{F\ps A'}{\ps A}$, it gives conditions ensuring that the morphism \mbox{$\map{\yon_{A'}}{A'}{\ps A'}$} in $\K$ constructed below is a Yoneda embedding.
	\begin{definition} \label{yoneda embedding from universal morphisms}
		Let $\map F\K\L$ be a functor of augmented virtual double categories (\augdefref{3.1}) and $\map{\yon_A}A{\ps A}$ a Yoneda morphism in $\L$ (\defref{yoneda embedding}). Let $A'$ be a unital object in $\K$, with horizontal unit $I_{A'}$ (\defref{cartesian cells}), and $\hmap\iota A{FA'}$ a morphism satisfying conditions (a) and (b) of \defref{universal horizontal morphism}. Given a morphism $\map\eps{F\ps A'}{\ps A}$ that is universal from $F$ to $\ps A$ relative to $(F \vs \ps A)_\iota$ (\defsref{universal vertical morphism}{iota-small}) we make the following definitions.
		\begin{enumerate}[label=-]
			\item The morphism $\map{\cur\iota}{FA'}{\ps A}$, given by the Yoneda axiom (\defref{yoneda embedding}) and defined by the cartesian cell in the right-hand side below, is $\iota$"/small (\defref{iota-small}) because $\flad I_{A'} \iso \iota$ by \auglemref{8.1}. We define $\map{\yon_{A'}}{A'}{\ps A'}$ in $\K$ to be $\yon_{A'} \dfn \shad{(\cur\iota)}$, i.e.\ the chosen morphism $\map{\yon_{A'}}{A'}{\ps A'}$ such that $\eps \of F\yon_{A'} \iso \cur\iota$ (\defref{universal vertical morphism}).
			\item For every $\hmap J{A'}B$ in $\K$ we define $\map{\curp J}B{\ps A'}$ to be $\curp J \dfn \shad{\bigpars{\cur{(\flad J)}}}$, where $\map{\cur{(\flad J)}}{FB}{\ps A}$ is defined by the cartesian cell in the left"/hand side below, so that it is $\iota$"/small by \lemref{iota-small axioms}(b).
			\begin{displaymath}
				\begin{tikzpicture}[textbaseline]
					\matrix(m)[math35, column sep={1.75em,between origins}]
						{ A & & FA' & & FB \\
							& A & & FB & \\
							& & \ps A & & \\};
					\path[map]	(m-1-1) edge[barred] node[above] {$\iota$} (m-1-3)
											(m-1-3) edge[barred] node[above, inner sep=5pt] {$FJ$} (m-1-5)
											(m-2-2) edge[barred] node[above, inner sep=2pt] {$\flad J$} (m-2-4)
															edge[transform canvas={xshift=-2pt}] node[below left] {$\yon_A$} (m-3-3)
											(m-2-4) edge[transform canvas={xshift=2pt}] node[below right, inner sep=1pt] {$\cur{(\flad J)}$} (m-3-3);
					\path				(m-1-1) edge[eq, transform canvas={xshift=-1pt}] (m-2-2)
											(m-1-5) edge[eq, transform canvas={xshift=1pt}] (m-2-4);
					\draw[font=\scriptsize]	($(m-1-3)!0.5!(m-2-3)$) node {$\cocart$}
											([yshift=0.333em]$(m-2-3)!0.5!(m-3-3)$) node {$\cart$};
				\end{tikzpicture} \quad = \quad \begin{tikzpicture}[textbaseline]
					\matrix(m)[math35, column sep={4em,between origins}]{A & FA' & FB \\ & \ps A & \\};
					\path[map]	(m-1-1) edge[barred] node[above] {$\iota$} (m-1-2)
															edge[transform canvas={xshift=-2pt}] node[below left] {$\yon_A$} node[sloped, above, inner sep=6.5pt, font=\scriptsize, pos=0.55] {$\cart$} (m-2-2)
											(m-1-2) edge[barred] node[above, inner sep=5pt] {$FJ$} (m-1-3)
															edge[ps] node[left, inner sep=0.5pt, yshift=2pt] {$\cur\iota$} (m-2-2)
											(m-1-3) edge[transform canvas={xshift=2pt}] node[below right, inner sep=1pt] {$\cur{(\flad J)}$} (m-2-2);
					\path[transform canvas={xshift=1.1em, yshift=3pt}]	(m-1-2) edge[cell] node[right, yshift=2pt] {$\chi^J$} (m-2-2);
				\end{tikzpicture}
			\end{displaymath}
			\item For every $\hmap J{A'}B$ in $\K$ we denote by $\chi^J$ the unique factorisation in the right-hand side above, which exists because the cartesian cell defining $\cur\iota$ is left Kan (\defref{density definition}). Because $\eps \of F\dash$ is full and faithful onto $F \vs \ps A$ (\defref{universal vertical morphism}) there exists a unique cell $\chi^{J\sharp}$, of the form as on the left below, that satisfies the identity on the right.
			\begin{displaymath}
				\begin{tikzpicture}[textbaseline]
					\matrix(m)[math35, column sep={1.75em,between origins}]{A' & & B \\ & \ps A' & \\};
					\path[map]	(m-1-1) edge[barred] node[above] {$J$} (m-1-3)
															edge[transform canvas={xshift=-2pt}, ps] node[left] {$\yon_{A'}$} (m-2-2)
											(m-1-3) edge[transform canvas={xshift=2pt}, ps] node[right] {$\curp J$} (m-2-2);
					\path[transform canvas={yshift=0.333em, xshift=-4.5pt}]	(m-1-2) edge[cell] node[right, inner sep=3pt, yshift=2pt] {$\chi^{J\sharp}$} (m-2-2);
				\end{tikzpicture} \qquad \qquad \qquad \qquad \qquad \chi^J = \quad \begin{tikzpicture}[textbaseline]
					\matrix(m)[math35, column sep={3.5em,between origins}]{FA' & & FB \\ & F\ps A' & \\ & \ps A & \\};
					\path[map]	(m-1-1) edge[barred] node[above] {$FJ$} (m-1-3)
															edge[transform canvas={xshift=-1pt}] node[left, inner sep=1pt] {$F\yon_{A'}$} (m-2-2)
											(m-1-1.240)	edge[bend right=38] node[below left] {$\cur\iota$} (m-3-2)
											(m-1-3) edge[transform canvas={xshift=1pt}, ps] node[right, inner sep=2pt] {$F\curp J$} (m-2-2)
											(m-1-3.-60)	edge[bend left=38] node[below right] {$\cur{(\flad J)}$} (m-3-2)
											(m-2-2) edge[ps] node[right, ps] {$\eps$} (m-3-2);
					\path[transform canvas={xshift=-7pt, yshift=0.333em}] (m-1-2) edge[cell] node[right, yshift=2pt] {$F\chi^{J\sharp}$} (m-2-2);
					\draw				($(m-1-2)!0.57!(m-3-1)$) node[rotate=45, xshift=-4pt] {$\iso$}
											($(m-1-2)!0.57!(m-3-3)$) node[rotate=-45, xshift=4pt] {$\iso$};
				\end{tikzpicture}
			\end{displaymath}
		\end{enumerate}
	\end{definition}
	
	\begin{theorem} \label{yoneda embedding from universal morphisms theorem}
		The morphism $\map{\yon_{A'}}{A'}\ps A'$ defined above is a Yoneda embedding, with each of the cells $\chi^{J\sharp}$ being cartesian, if the following conditions are satisfied:
		\begin{enumerate}[label=\textup{(\alph*)}]
			\item the restrictions $\yon_{A'*}(\id, f)$ and $\flad\yon_{A'*}(\id, Ff)$ exist for any $\map fB{\ps A'}$ in $\K$ and $F$ preserves all $\yon_{A'*}(\id, f)$;
			\item $\hmap\iota A{FA'}$ is locally universal from $A$ to $F$ (\defref{universal horizontal morphism}).
		\end{enumerate}
		Moreover, in that case $\hmap{\flad \yon_{A'*}}A{F\ps A'}$ is the restriction of $\ps A$ along $\yon_A$ and $\eps$.
	\end{theorem}
	\begin{proof}
		To show that $\yon_{A'}$ is a Yoneda embedding we have to prove that it is dense (\defref{density definition}) and that it satisfies the Yoneda axiom (\defref{yoneda embedding}); that $\yon_{A'}$ is full and faithful then follows from \lemref{full and faithful yoneda embedding} as $A'$ is assumed to be unital (\defref{yoneda embedding from universal morphisms}). That $\yon_{A'}$ satisfies the Yoneda axiom is proved in \lemref{yoneda axiom and local universality} below; here we prove the density of $\yon_{A'}$. To do so consider the equation below, where the left cartesian cell in the left"/hand side defines $\cur\iota \iso \eps \of F\yon_{A'}$ (\defref{yoneda embedding from universal morphisms}) and where the nullary cell $\cell\xi{\flad \yon_{A'*}}{\ps A}$ is the unique factorisation through the pointwise right cocartesian cell that defines $\flad \yon_{A'*}$ (\defref{universal horizontal morphism}). We claim that $\xi$ is cartesian, and hence pointwise left Kan by density of $\map{\yon_A}A{\ps A}$. Before proving this claim let us use it to prove the density of $\yon_{A'}$. Notice also that the final assertion above simply follows from the fact that $\xi$ is cartesian.
		\begin{displaymath}
				\begin{tikzpicture}[textbaseline]
					\matrix(m)[math35, column sep={1.75em,between origins}]
						{ A & & FA' & & F\ps A' \\
							& & & F\ps A' & \\
							& & \ps A & & \\ };
					\path[map]	(m-1-1) edge[barred] node[above] {$\iota$} (m-1-3)
															edge[transform canvas={xshift=-2pt}] node[below left] {$\yon_A$} (m-3-3)
											(m-1-3) edge[barred] node[above, inner sep=5pt] {$F\yon_{A'*}$} (m-1-5)
															edge[transform canvas={xshift=-4pt}] node[left, inner sep=0.5pt] {$F\yon_{A'}$} (m-2-4)
											(m-2-4) edge[transform canvas={xshift=2pt}] node[right] {$\eps$} (m-3-3);
					\path				(m-1-5) edge[transform canvas={xshift=2pt}, eq, ps] (m-2-4);
					\draw[font=\scriptsize]	($(m-1-4)!0.5!(m-2-4)$) node[yshift=0.45em, xshift=-2pt] {$F\cart$}
											($(m-1-3)!0.4!(m-3-1)$) node[xshift=8pt] {$\cart$};
				\end{tikzpicture} \quad = \quad \begin{tikzpicture}[textbaseline]
					\matrix(m)[math35, column sep={1.75em,between origins}]
						{ A & & FA' & & F\ps A' \\
							& A & & F\ps A' & \\
							& & \ps A & & \\ };
					\path[map]	(m-1-1) edge[barred] node[above] {$\iota$} (m-1-3)
											(m-1-3) edge[barred] node[above, inner sep=5pt] {$F\yon_{A'*}$} (m-1-5)
											(m-2-2) edge[barred] node[above, inner sep=2pt] {$\flad\yon_{A'*}$} (m-2-4)
															edge[transform canvas={xshift=-2pt}] node[below left] {$\yon_A$} (m-3-3)
											(m-2-4) edge[transform canvas={xshift=2pt}] node[below right] {$\eps$} (m-3-3);
					\path				(m-1-1) edge[transform canvas={xshift=-1pt}, eq] (m-2-2)
											(m-1-5) edge[transform canvas={xshift=1pt}, eq, ps] (m-2-4)
											(m-2-3) edge[transform canvas={yshift=0.333em}, cell] node[right] {$\xi$} (m-3-3);
					\draw[font=\scriptsize]	($(m-1-3)!0.5!(m-2-3)$) node {$\cocart$};
				\end{tikzpicture}
			\end{displaymath}
		
		Using that the restrictions $(F\yon_{A'*})(\id, Ff) \iso F\bigpars{\yon_{A'*}(\id, f)}$ and $\flad\yon_{A'*}(\id, Ff)$ exist for any $\map fB{\ps A'}$ by condition (a), it follows from the vertical pasting lemma (\lemref{vertical pasting lemma}) that the right"/hand side above, and thus both sides, are left Kan cells that restrict along any $F$"/image $Ff$. Since the left cartesian cell in the left"/hand side is left Kan too, again by the density of $\yon_A$, it follows from the horizontal pasting lemma (\lemref{horizontal pasting lemma}) that the composite $\eps \of F\cart$ in the left"/hand side is a left Kan cell that restricts along any $Ff$ too. Using \propref{taking adjuncts preserves left Kan cells} we conclude that the adjunct of $\eps \of F\cart$, which is the cartesian cell $\cart$ defining $\yon_{A'*}$ itself, is pointwise left Kan. By \defref{density definition} this means that $\yon_{A'}$ is dense as required.
		
		It remains to prove the claim that the nullary cell $\xi$ in the right"/hand side above is cartesian. Notice that $\eps = \eps \of F(\id_{\ps A'})$ is $\iota$"/small because $\eps \of F\dash$ factors through $(F \vs \ps A)_\iota$ (\defsref{universal vertical morphism}{iota-small}). Hence, by \lemref{iota-small axioms}(b), the restriction $\hmap{\ps A(\yon_A, \eps)}A{F\ps A'}$ exists and $\ps A(\yon_A, \eps) \iso \flad J$ for some $\hmap J{A'}{\ps A'}$. To prove the claim we will show that the cartesian cell $\flad J \Rar \ps A$, that defines $\ps A(\yon_A, \eps)$, factors through $\xi$ as an invertible horizontal cell $\flad J \iso \flad \yon_{A'*}$. For this it suffices to prove that for any $\hmap H{A'}{\ps A'}$ in $\K$ the assignment $\xi \of \dash$ on the left below, between collections of cells in $\L$ as shown, is a bijection. Precomposing with the right cocartesian cells $(\iota, FH) \Rar \flad H$ (\defref{universal horizontal morphism}) and using the uniqueness of factorisations through cocartesian cells, we may equivalently prove that the assigment $\xi \of \dash$ on the right below is a bijection.
		\begin{displaymath}
			\begin{tikzpicture}
				\matrix(m)[math35, yshift=6em, xshift=-12em]{A & F\ps A' \\ A & F\ps A' \\};
				\path[map]	(m-1-1) edge[barred] node[above] {$\flad H$} (m-1-2)
										(m-2-1) edge[barred] node[below] {$\flad\yon_{A'*}$} (m-2-2);
				\path				(m-1-1) edge[eq] (m-2-1)
										(m-1-2) edge[eq, ps] (m-2-2)
										(m-1-1) edge[transform canvas={xshift=1.625em},cell] (m-2-1);
				
				\matrix(m)[math35, column sep={1.75em,between origins}, yshift=-6em, xshift=-12em]{A & & F\ps A' \\ & \ps A & \\};
				\path[map]	(m-1-1) edge[barred] node[above] {$\flad H$} (m-1-3)
														edge[transform canvas={xshift=-2pt}] node[left, xshift=2pt] {$\yon_A$} (m-2-2)
										(m-1-3) edge[transform canvas={xshift=2pt}, ps] node[right] {$\eps$} (m-2-2);
				\path[transform canvas={yshift=0.333em}]				(m-1-2) edge[cell] (m-2-2);
				
				\matrix(m)[math35, yshift=6em]{A' & \ps A' \\ A' & \ps A' \\};
				\path[map]	(m-1-1) edge[barred] node[above] {$H$} (m-1-2)
										(m-2-1) edge[barred] node[below] {$\yon_{A'*}$} (m-2-2);
				\path				(m-1-1) edge[eq] (m-2-1)
										(m-1-2) edge[eq, ps] (m-2-2)
										(m-1-1) edge[transform canvas={xshift=1.625em},cell] (m-2-1);
										
				\matrix(m)[math35, column sep={1.625em,between origins}, yshift=6em, xshift=17em]
					{ A & & FA' & & F\ps A' \\ & A & & F\ps A' & \\};
				\path[map]	(m-1-1) edge[barred] node[above] {$\iota$} (m-1-3)
										(m-1-3) edge[barred] node[above] {$FH$} (m-1-5)
										(m-2-2) edge[barred] node[below] {$\flad \yon_{A'*}$} (m-2-4);
				\path				(m-1-1) edge[eq] (m-2-2)
										(m-1-5) edge[eq, ps] (m-2-4)
										(m-1-3) edge[cell] (m-2-3);
				
				\matrix(m)[math35, column sep={0.8125em,between origins}, yshift=-6em]{FA' & & & & F\ps A' \\ & F\ps A' \mspace{9mu} & & & \\ & & \ps A & & \\};
				\path[map]	(m-1-1) edge[barred] node[above] {$FH$} (m-1-5)
														edge[transform canvas={xshift=-1pt}] node[left, inner sep=0pt] {$F\yon_{A'}$} (m-2-2)
										(m-1-5) edge[transform canvas={xshift=1pt}, ps] node[right] {$\eps$} (m-3-3)
										(m-2-2) edge[transform canvas={xshift=1pt}, ps] node[left] {$\eps$} (m-3-3);
				\path[transform canvas={yshift=-0.7em}]	(m-1-3) edge[cell] (m-2-3);
				
				\matrix(m)[math35, yshift=-6em, xshift=17em]{A & FA' & F\ps A' \\ & \ps A & \\};
				\path[map]	(m-1-1) edge[barred] node[above] {$\iota$} (m-1-2)
														edge[transform canvas={xshift=-1pt}] node[left] {$\yon_A$} (m-2-2)
										(m-1-2) edge[barred] node[above] {$FH$} (m-1-3)
										(m-1-3) edge[transform canvas={xshift=1pt}] node[right] {$\eps$} (m-2-2);
				\path				(m-1-2) edge[cell] (m-2-2);
				
				\draw[font=\LARGE]	(-14.7em,6em) node {$\lbrace$}
										(-9.4em,6em) node {$\rbrace$}
										(-14.8em,-6em) node {$\lbrace$}
										(-9.5em,-6em) node {$\rbrace$}
										(-2.5em,6em) node {$\lbrace$}
										(2.6em,6em) node {$\rbrace$}
										(-2.9em,-6em) node {$\lbrace$}
										(2.4em,-6em) node {$\rbrace$}
										(12.8em,6em) node {$\lbrace$}
										(20.9em,6em) node {$\rbrace$}
										(12.8em,-6em) node {$\lbrace$}
										(20.9em,-6em) node {$\rbrace$};
				
				\path[map]	(-12em,1.5em) edge node[left] {$\xi \of \dash$} (-12em,-2.0em)
										(4.6em,6em) edge node[above] {$\cocart \of (\id_\iota, F\dash)$} (10.8em,6em)
										(4.4em,-6em) edge node[below] {$\cart \hc \dash$} (10.8em,-6em)
										(0,1.5em) edge node[left] {$\eps \of F(\cart \of \dash)$} (0,0)
										(17em,1.5em) edge node[right] {$\xi \of \dash$} (17em,-2em);
			\end{tikzpicture}
		\end{displaymath}
		To do so, consider the whole diagram of assignments on the right above, where the cocartesian cell defines $\flad \yon_{A'*}$ and where the cartesian cells are those in the left"/hand side of equality that defines $\xi$ above; that the diagram commutes follows from the same equality and the interchange axioms (\auglemref{1.3}). We claim that, besides $\xi \of \dash$ on the right, all assignments above are bijections, so that $\xi \of \dash$ is a bijection as well. Indeed, $\cart \hc \dash$ is a bijection since $\cart$ is left Kan, by the density of $\yon_A$ (\defref{density definition}); $\eps \of F(\cart \of \dash)$ is a bijection because of the companion identities for $\cart$ (\lemref{companion identities lemma}) and the fact that $\map{\eps \of F\dash}{\K \vs\ps A'}{F \vs \ps A}$ is full and faithful (\defref{universal vertical morphism}); $\cocart \of (\id_\iota, F\dash)$ is a bijection by \defref{universal horizontal morphism}(c), where $h = \id_{A'}$ and $k = \id_{\ps A'}$. This completes the proof.
	\end{proof}
	
	\begin{lemma} \label{yoneda axiom and local universality}
		In \defref{yoneda embedding from universal morphisms} the morphism $\hmap\iota A{FA'}$ satisfies condition~\textup{(c)} of \defref{universal horizontal morphism} (i.e.\ $\iota$ is locally universal) precisely if, for each $\hmap J{A'}B$ in $\K$, the cell $\chi^{J\sharp}$ is cartesian (i.e.\ $\map{\yon_{A'}}{A'}{\ps A'}$ satisfies the Yoneda axiom of \defref{yoneda embedding}).
	\end{lemma}
	\begin{proof}
	Consider the diagram of assignments between collections of cells in $\K$ and $\L$ below, where $\cocart$ denotes the cocartesian cell definining $\flad J$ (\defref{universal horizontal morphism}) and the cartesian cells denoted by $\cart$, $\cart_{\cur{(\flad J)}}$ and $\cart_{\cur\iota}$ define the morphisms $\iota(\id, Fh)$, $\cur{(\flad J)}$ (\defref{yoneda embedding}) and $\cur\iota$ respectively. The vertical isomorphisms $\sigma_h$ and $\tau_k$ in the middle assignment of the right leg denote the composites $\sigma_h \dfn \sigma \of Fh$ and $\tau_k \dfn \tau \of Fk$, where $\sigma\colon \cur\iota \iso \eps \of F\yon_{A'}$ and $\tau\colon \eps \of F\curp J \iso \cur{(\flad J)}$ are the isomorphisms that equip the chosen morphisms $\map{\yon_{A'}}{A'}{\ps A'}$ and $\map{\curp J}B{\ps A'}$ (\defref{yoneda embedding from universal morphisms}).
		\begin{displaymath}
			\begin{tikzpicture}
				\matrix(m)[math35, xshift=-12em]{X_0 & X_n \\ A' & B \\};
				\path[map]	(m-1-1) edge[barred] node[above] {$\ul H$} (m-1-2)
														edge node[left] {$h$} (m-2-1)
										(m-1-2) edge node[right] {$k$} (m-2-2)
										(m-2-1) edge[barred] node[below] {$J$} (m-2-2);
				\path[transform canvas={xshift=1.625em}]				(m-1-1) edge[cell] (m-2-1);
				
				\matrix(m)[math35, column sep={1.625em,between origins}, yshift=-12em, xshift=-9em]{A & & FX_0 & & FX_n\\ & A & & FB & \\};
				\path[map]	(m-1-1) edge[barred] node[above, inner sep=6pt] {$\iota(\id, Fh)$} (m-1-3)
										(m-1-3) edge[barred] node[above] {$F\ul H$} (m-1-5)
										(m-1-5) edge[transform canvas={xshift=-2pt}] node[right] {$Fk$} (m-2-4)
										(m-2-2) edge[barred] node[below] {$\flad J$} (m-2-4);
				\path				(m-1-1) edge[eq] (m-2-2)
										(m-1-3) edge[cell] (m-2-3);
				
				\matrix(m)[math35, column sep={0.8125em,between origins}]{X_0 & & & & X_n \\ & A' \mspace{6mu} & & \mspace{6mu} B & \\ & & \ps A' & & \\};
				\path[map]	(m-1-1) edge[barred] node[above] {$\ul H$} (m-1-5)
														edge[transform canvas={xshift=-1pt}] node[left] {$h$} (m-2-2)
										(m-1-5) edge[transform canvas={xshift=1pt}] node[right] {$k$} (m-2-4)
										(m-2-2) edge[transform canvas={xshift=-1pt}, ps] node[left] {$\yon_{A'}$} (m-3-3)
										(m-2-4) edge[transform canvas={xshift=1pt}, ps] node[right] {$\curp J$} (m-3-3);
				\path[transform canvas={yshift=-0.7em}]				(m-1-3) edge[cell] (m-2-3);
				
				\matrix(m)[math35, column sep={0.8125em,between origins}, xshift=16em]{FX_0 & & & & FX_n \\ & FA' \mspace{12mu} & & \mspace{12mu} FB & \\ & & \ps A & & \\};
				\path[map]	(m-1-1) edge[barred] node[above] {$F\ul H$} (m-1-5)
														edge[transform canvas={xshift=-1pt}] node[left] {$Fh$} (m-2-2)
										(m-1-5) edge[transform canvas={xshift=1pt}] node[right] {$Fk$} (m-2-4)
										(m-2-2) edge[transform canvas={xshift=-1pt}, ps] node[left] {$\cur\iota$} (m-3-3)
										(m-2-4) edge[transform canvas={xshift=1pt}, ps] node[right] {$\cur{(\flad J)}$} (m-3-3);
				\path[transform canvas={yshift=-0.7em}]				(m-1-3) edge[cell] (m-2-3);
				
				\matrix(m)[math35, column sep={1.625em,between origins}, yshift=-12em, xshift=11em]
						{ A & & FX_0 & & FX_n \\
							& & & FB & \\
							& & \ps A & & \\};
					\path[map]	(m-1-1) edge[barred] node[above, inner sep=6pt] {$\iota(\id, Fh)$} (m-1-3)
															edge[transform canvas={xshift=-2pt}] node[below left] {$\yon_A$} (m-3-3)
											(m-1-3) edge[barred] node[above] {$F\ul H$} (m-1-5)
											(m-1-5)	edge[transform canvas={xshift=2pt}] node[right] {$Fk$} (m-2-4)
											(m-2-4) edge[transform canvas={xshift=2pt}] node[right] {$\cur{(\flad J)}$} (m-3-3);
					\path[transform canvas={yshift=-0.7em}]			(m-1-3) edge[cell] (m-2-3);
				
				\draw[font=\LARGE]	(-14.9em,0) node {$\lbrace$}
										(-9.2em,0) node {$\rbrace$}
										(-2.6em,0) node {$\lbrace$}
										(2.4em,0) node {$\rbrace$}
										(-13.2em,-12em) node {$\lbrace$}
										(-5em,-12em) node {$\rbrace$}
										(13.0em,0em) node {$\lbrace$}
										(19.0em,0em) node {$\rbrace$}
										(6.8em,-12em) node {$\lbrace$}
										(15.0em,-12em) node {$\rbrace$};
				\path[map]	(-7.2em,0) edge node[above] {$\chi^{J\sharp} \of \dash$} (-4.6em,0)
										(-11.5em,-4em) edge node[left, yshift=-3pt] {$\cocart \of (\cart, F\dash)$} (-10.25em,-7.5em)
										(4.4em,0) edge node[above] {$\sigma_h \hc (\eps \of F\dash) \hc \tau_k$} (11.0em,0)
										(-3em,-12em) edge node[below] {$\cart_{\cur{(\flad J)}} \of \dash$} (4.8em,-12em)
										(13.75em,-4.25em) edge node[right, yshift=-3pt] {$(\cart_{\cur\iota} \of \cart) \hc \dash$} (12.75em,-6.25em);
			\end{tikzpicture}
		\end{displaymath}
		
		Notice that $\cart_{\cur{(\flad J)}} \of \dash$ is a bijection by the definition of cartesian cell, and that $(\cart_{\cur\iota} \of \cart) \hc \dash$ is so too because $\cart_{\cur\iota}$ is pointwise left Kan, by the density of $\yon_A$ (\defref{density definition}). Since $\map{\eps \of F\dash}{\K \vs \ps A'}{F \vs \ps A}$ is full and faithful, by \defref{universal vertical morphism}, and because $\sigma_h$ and $\tau_k$ are isomorphisms, it follows that \mbox{$\sigma_h \hc \bigpars{\eps \of F\dash} \hc \tau_k$}, in the right leg above, is a bijection too. Finally notice that condition (c) of \defref{universal horizontal morphism} states that the assignment $\cocart \of (\cart, F\dash)$, in the left leg above, is a bijection for all $\hmap J{A'}B$, while the cells $\chi^{J\sharp}$ being cartesian amounts to the top assignment being a bijection for all $\hmap J{A'}B$. Hence the proof follows by showing that the diagram above commutes. That it does is shown by the equality below, whose left"/hand side is the composite of the right leg above and whose right"/hand side is that of the left leg.
		\begin{align*}
			(\cart&_{\cur\iota} \of \cart) \hc \sigma_h \hc \bigpars{\eps \of F(\chi^{J\sharp} \of \dash)} \hc \tau_k = (\cart_{\cur\iota} \of \cart) \hc \sigma_h \hc \bigpars{\eps \of F\chi^{J\sharp} \of F\dash} \hc \tau_k \\
			&= (\cart_{\cur\iota} \of \cart) \hc \bigbrks{\bigpars{\sigma \hc (\eps \of F\chi^{J\sharp}) \hc \tau} \of F\dash} = (\cart_{\cur\iota} \of \cart) \hc (\chi^J \of F\dash) \\
			&= (\cart_{\cur\iota} \hc \chi^J) \of (\cart, F\dash) = \cart_{\cur{(\flad J)}} \of \cocart \of (\cart, F\dash)
		\end{align*}
		The identities follow from the functoriality of $F$, the definitions of $\sigma_h$ and $\tau_k$ together with the interchange axiom (\auglemref{1.3}), the definition of $\chi^{J\sharp}$ (\defref{yoneda embedding from universal morphisms}), the interchange axiom, and the definition of $\chi^J$ (\defref{yoneda embedding from universal morphisms}).
	\end{proof}	
	

\end{document}